%% file: arXiv.tex
\documentclass[a4paper,11pt]{article}
\usepackage{import}
\usepackage{xr}

\usepackage{mathtools}
\usepackage{dsfont}
\usepackage[utf8]{inputenc}
\usepackage[greek,english]{babel}
\usepackage{alphabeta}
\usepackage[unicode=true,psdextra]{hyperref}
\hypersetup{
	colorlinks=true,
	citecolor=cyan,
	linkcolor=blue,
	filecolor=blue,      
	urlcolor=red,
}
\usepackage{accents}

\input{sty/paper-preamble2.tex}

\numberwithin{equation}{section}

\usepackage{silence}
\usepackage{etoc}
\usepackage{appendix}
\usepackage{comment}
\usepackage{minitoc}

\newcommand{\rmd}{\mathrm{d}}

\newcommand{\bbH}{\mathbb H}

\newcommand{\thetaMLE}[1][S]{\widehat{\theta}_{#1}^{\operatorname{\texttt{MLE}}}}
\newcommand{\thetaBest}[1][S]{\theta_{#1}^{\ast}}
\newcommand{\designRegular}[1][S]{\zeta_{n, #1}}
\newcommand{\localSet}[1][S]{\Theta_{#1}(r)}
\newcommand{\localSetRn}[2][S]{\Theta_{#1} (#2)}

\title{Advances in Bayesian model selection consistency for high-dimensional generalized linear models}
\author{Jeyong Lee\footnote{Department of Industrial and Management Engineering,
Pohang University of Science and Technology, {\tt jylee1024@postech.ac.kr} and {\tt mchae@postech.ac.kr}} \quad Minwoo Chae$^*$ \quad Ryan Martin\footnote{Department of Statistics, North Carolina State University, {\tt rgmarti3@ncsu.edu}}}
\date{\today}

\begin{document}

\maketitle

\begin{abstract}
Uncovering genuine relationships between a response variable of interest and a large collection of covariates is a fundamental and practically important problem.  In the context of Gaussian linear models, both the Bayesian and non-Bayesian literature is well-developed and there are no substantial differences in the model selection consistency results available from the two schools.  For the more challenging generalized linear models (GLMs), however, Bayesian model selection consistency results are lacking in several ways.  In this paper, we construct a Bayesian posterior distribution using an appropriate data-dependent prior and develop its asymptotic concentration properties using new theoretical techniques.  In particular, we leverage Spokoiny's powerful non-asymptotic theory to obtain sharp quadratic approximations of the GLM's log-likelihood function, which leads to tight bounds on the errors associated with the model-specific maximum likelihood estimators and the Laplace approximation of our Bayesian marginal likelihood.  In turn, these improved bounds lead to significantly stronger, near-optimal Bayesian model selection consistency results, e.g., far weaker beta-min conditions, compared to those available in the existing literature. In particular, our results are applicable to the Poisson regression model, in which the score function is not sub-Gaussian.

\smallskip

\end{abstract}

\addtocontents{toc}{\protect\setcounter{tocdepth}{-1}}

\section{Introduction}

Generalized linear models (GLMs), which include Gaussian, binomial, and Poisson regression models, are among the most powerful and widely used statistical tools; see, e.g., the classical text by \citet{McCullaghNelder:1989} for details.  Specifically, given independent observations $(x_1, Y_1), \ldots, (x_n, Y_n)$, where $x_i \in \bbR^p$ is a fixed covariate vector and $Y_i \in \mathcal{Y} \subseteq \bbR$ is the response variable, the GLM posits a conditional probability density/mass function of the form 
\begin{align} \label{eqn:glm_density_main}
p_{\theta}(y \mid x) = \exp \bigl\{ y x^{\top}\theta - b(x^{\top}\theta) + k(y) \bigr\},
\end{align}
where $b$ and $k$ are known functions and $\theta \in \mathbb{R}^p$ is the vector of unknown coefficients. We assume here that the model is well-specified, hence there exists a true coefficient $\theta_{0}$ to be inferred from the observable data $(x_1,Y_1),\ldots,(x_n, Y_n)$. Our focus is on the high-dimensional setting, where the number of parameters $p$ grows with the sample size $n$, possibly with $n \ll p$.

For the case $p > n$, a suitable low-dimensional structure on the model is necessary for the identifiability of the coefficient $\theta_0$. We assume that $\theta_0$ is {\em sparse} in the sense that most components of $\theta_0$ are zero. Statistical inference---including estimation of $\theta_0$, variable selection,  uncertainty quantification, etc.---under sparsity has been extensively studied over the last few decades. Various approaches have been developed, including those based on penalized regression \citep{tibshirani1996regression, fan2001variable, zou2006adaptive, zhang2010analysis} alongside computational methods \citep{breheny2011coordinate, mazumder2011sparsenet} and supporting theories \citep{chen2012extended, barber2015high, loh2017support, van2008high, fan2011nonconcave}. For a comprehensive introduction, see \citet{hastie2015statistical}, \citet{buhlmann2011statistics} and \citet{wainwright2019high}.

Significant advancements have been made in recent years in high-dimensional Bayesian analysis \citep{george2000variable, ishwaran2005spike, narisetty2014bayesian, carvalho2010horseshoe, piironen2017sparsity, van2017adaptive, johnson2012bayesian, rossell2017nonlocal, rovckova2018spike, rovckova2018bayesian, nie2023bayesian}. In parallel, computational methods \citep{hou2024laplace, ray2022variational, wan2021adaptive, hans2007shotgun, shin2018scalable} and corresponding asymptotic theory \citep{castillo2012needles, castillo2015sparse, yang2016computational, martin2014asymptotically, martin.walker.deb, martin2017empirical, belitser.ghosal.ebuq} have been rapidly developing.

Bayesian asymptotic theory has focused almost exclusively on the special case of high-dimensional Gaussian linear models; only a few theoretical studies have been dedicated to Bayesian GLMs more generally. Convergence rates of the posterior distributions have been investigated in \citet{jeong2021posterior}, and some model selection properties have been considered in \citet{narisetty2018skinny} and \citet{rossell2021approximate}. Works such as \citet{lee2021bayesian}, \citet{cao2022bayesian} and \citet{tang2024empirical} have extended the existing model selection consistency results to a wider class of GLMs, primarily by utilizing the proof techniques given in \citet{narisetty2018skinny}. The results obtained in these papers for model selection are not as sharp as those in the frequentist literature \citep[e.g.,][]{loh2017support} or those in Bayesian linear regression literature. In particular, existing Bayesian model selection results rely on the sub-Gaussianity of the score function through Hanson--Wright type inequalities \citep{hanson1971bound, hsu2012tail}, which are not applicable to important examples like the Poisson regression model. \citet{chae2019bayesian} addressed the Bayesian model selection problem in a linear regression model with a nonparametric error distribution, but their results still require sub-Gaussianity of the score function, a non-trivial restriction.

% eqn:fff.others

A main goal of the present paper is to close the significant gap between the extant Bayesian asymptotic theory for GLMs and that for the Gaussian linear model, particularly as it concerns model selection consistency. To this end, we lean heavily on several advanced techniques in, e.g., \citet{spokoiny2012parametric, spokoiny2017penalized} for analyzing the log-likelihood in parametric models. These techniques lead to sharp quadratic approximations of the log-likelihood ratio (Lemma \ref{lemma:quad_expansion_Theta}), sub-exponential tail bounds for the normalized score function (Lemma \ref{lemma:dev_ineq_score_func}), and precise Laplace approximations for the marginal likelihood (Theorem \ref{thm:LA_marginal_likelihood_main}). This refined analysis allows for significant improvements to the existing results on Bayesian model selection consistency in GLMs, notably in terms of the number of non-zero coefficients and the minimum magnitude of these coefficients.  In particular, the existing Bayesian model selection consistency results for GLMs (implicitly) work with the bound stated in \eqref{eqn:fff.others} below, which leads to the requirement that $s_\text{max}^4 \log p = o(n)$, where $s_\text{max}$ is the upper bound on the support of the prior on the model size, which must be (apparently far) less than the rank of the $n \times p$ design matrix.  Our refined analysis leads to a tighter bound, as stated in \eqref{eqn:fff.others} below, which implies much weaker constraints on the problem setting, i.e., $s_0^3 \log p = o(n)$, where $s_0$ is the size of the true model that includes only the important covariates.  These refinements also lead to substantially weaker demands---i.e., ``beta-min conditions''---on the minimum signal size required for consistent selection compared to what is  presently available in the Bayesian literature, thereby closing the current-but-unnecessary gap between the Bayesian and frequentist results.  Furthermore, all of these results hold for GLMs whose score function has sub-exponential---rather than sub-Gaussian---tails, making them applicable to Poisson regression models, among others. 

The remainder of this paper is organized as follows. Section \ref{sec:setup} introduces several notations and definitions regarding the model and design matrices. The empirical prior and the corresponding (fractional) posterior distributions are defined in Section \ref{sec:prior}. Section \ref{sec:convergence_of_posterior_distribution} considers the convergence rate of the posterior distribution. The main results concerning the model selection consistency are presented in Section \ref{sec:model_selection_consistency}, with specific examples of logistic and Poisson regression models provided in Section \ref{sec:examples}. Computational algorithms and hyperparameter selection are discussed in Section \ref{sec:computation_BVS}. Finally, concluding remarks are given in Section \ref{sec:discussion}.

All proofs and further technical details are deferred to the Appendix. In particular, detailed non-asymptotic statements are available in the Appendix, while we keep asymptotic statements in the main text for readability.

\section{Setup} \label{sec:setup}

\subsection{Notation} \label{sec:notations_main}

Table~\ref{table:notations} on page \pageref{table:notations} summarizes the notation used in the following sections.  This subsection briefly lists some of the basic notations and definitions.

For two real numbers $a$ and $b$, $a \vee b$ and $a \wedge b$ denote the maximum and minimum of $a$ and $b$, respectively.
For two positive sequences $(a_n)$ and $(b_n)$, $a_n \lesssim b_n$ (or $a_n = O(b_n)$) means that $a_n \leq C b_n$ for some constant $C \in (0, \infty)$.
Also, $a_n \asymp b_n$ indicates that $a_n \lesssim b_n$ and $b_n \lesssim a_n$.
The notation $a_n \ll b_n$ (or $a_n = o(b_n)$) implies that $a_n / b_n \rightarrow 0$ as $n \rightarrow \infty$. 

For a real random variable $Z$ and the function $\psi_{\alpha}(t) = e^{t^{\alpha}} -1$ with $\alpha > 0$, define the Orlicz norm $\|Z\|_{\psi_\alpha} = \inf\{ K > 0: \bbE \psi_\alpha(|Z|/K) \leq 1 \}$, 
%\begin{align*}
    %\left\| Z \right\|_{\psi_{\alpha}} = \inf\left\{ K > 0 : \bbE \exp \left( \dfrac{|Z|^{\alpha}}{K^{\alpha}} \right) \leq 2 \right\}, 
%\end{align*}
where $\inf \varnothing = \infty$ by convention.

All vectors are non-bold except for $n$-dimensional vectors which are bold. For $1 \leq q \leq \infty$, $\|\cdot\|_{q}$ indicates the $\ell_q$-norm of a vector. For a matrix $\mathbf{A} = (a_{ij}) \in \bbR^{n \times p}$, define $ \| \bA \|_{\max} = \max_{i \in [n], j \in [p]} |a_{ij}|$ and $\| \bA \|_{\infty} = \max_{i \in [n]} \sum_{j = 1}^{p} |a_{ij}|$. Let $\lambda_{\operatorname{min}}(\bA)$ and $\lambda_{\operatorname{max}}(\bA)$ denote the smallest and largest singular value of $\bA$, respectively.
For simplicity in notation, $\| \bA \|_{2}$ will often be used interchangeably with $\lambda_{\operatorname{max}}(\bA)$.
For two distinct matrices $\bA, \mathbf{B} \in \bbR^{n \times n}$, $\bA \succeq \mathbf{B}$ means $\bA - \mathbf{B}$ is positive semi-definite matrix.

Let $\bI_p$ be the $p\times p$ identity matrix, $\bY = (Y_i)_{i = 1}^{n} \in \mathcal{Y}^n \subseteq \bbR^{n}$ be the response vector and $\bX = (x_{ij}) \in \bbR^{n \times p}$ be the design matrix.
Let $x_i = (x_{i1}, \ldots, x_{ip})^{\top} \in \bbR^{p}$ be the $i$th row of $\bX$ and $\bx_{j} = (x_{1j}, \ldots, x_{nj})^{\top} \in \bbR^n$ be the $j$th column of $\bX$. 
For $S \subset [p] \defeq \left\{ 1, 2, ..., p \right\}$, let $x_{i, S} = \left(x_{ij}\right)_{j \in S}^{\top} \in \bbR^{|S|}$ and $\bX_S = (\bx_{j})_{j \in S} \in \bbR^{n \times |S|}$, where $|S|$ is the cardinality of $S$. The index set for the nonzero elements of $\theta \in \bbR^p$ is denoted as $S_{\theta} = \left\{i \in [p] : \theta_i \neq 0 \right\}$.
For $S \subseteq [p]$, let $\theta_S = (\theta_j)_{j \in S} \in \bbR^{|S|}$ and let
\begin{align} \label{def:tilde_theta}
    \widetilde{\theta}_{S} = (\widetilde{\theta}_{j})_{j \in [p]} =
    \begin{cases}
        \widetilde{\theta}_{j} = \theta_j, &\quad j \in S, \\
        \widetilde{\theta}_{j} = 0, &\quad j \in S^{\rm c}.
    \end{cases}
\end{align}
In words, $\widetilde\theta_S$ is the $p$-vector version of $\theta_S$ with zeros in for the entries corresponding to $S^{\rm c}$.

\renewcommand{\arraystretch}{1.3}
\begin{table}[hbt!]
\caption{Summary of notations and definitions. For lengthy definitions, refer to the main text.}
\centering
\begin{tabular}{c c p{8cm}} 
\hline
Symbol & Location & Definition \\ [0.5ex] 
\hline\hline
$C_{\rm dev}$ & \eqref{def:C_dev_GLM_main} & $\sup_{|y| \leq 1/2} b''(x + y) \leq  C_{\rm dev} b''(x)$ \\ 
$\thetaMLE, \thetaBest$ & \eqref{def:MLE_Best_main} & $\argmax_{\theta_S \in \bbR^{|S|}} L_{n, \theta_S}, \quad \argmax_{\theta_S \in \bbR^{|S|}} \bbE L_{n, \theta_S}.$ \\
$\rho_{\max, S}, \rho_{\min , S}$ & \eqref{def:rho_eigenvalues} & $\lambda_{\operatorname{max}} ( \bF_{n, \thetaBest[S]} ), \quad \lambda_{\operatorname{min}} ( \bF_{n, \thetaBest[S]} )$ \\
$\sigma_{\min}^2, \sigma_{\max}^2$  & \eqref{eqn:F_infinity_norm}, \eqref{def:A_4_2} & $\min_{i \in [n]} b''(x_i^{\top}\theta_0)$, $\quad \max_{i \in [n]} b''(x_i^{\top}\theta_0)$ \\
$\designRegular$ & \eqref{def:design_regularity} & $\max_{i \in [n]} \| \bF_{n, \thetaBest[S]}^{-1/2}x_{i, S} \|_{2}$ \\
$\xi_{n, S}$ & \eqref{def:score_func_normalized_main} & $\bF_{n, \thetaBest}^{-1/2} \dot{L}_{n, \thetaBest}$ \\
$\Delta_{{\rm mis}, S}$ & \eqref{def:misspecification_quantity_main} & 
$\Delta_{{\rm mis}, S} = \lambda_{\max} ( \bF_{n, \thetaBest}^{-1/2} 
\bV_{n, S} \bF_{n, \thetaBest}^{-1/2} )$, \\
$\widetilde{\Delta}_{{\rm mis}, S}$ & Lemma \ref{lemma:mis_on_posterior_concentration_set_main} & 
$\widetilde{\Delta}_{{\rm mis}, S} = \lambda_{\max} ( \bV_{n, S}^{-1/2} 
 \bF_{n, \thetaBest} \bV_{n, S}^{-1/2} )$, \\
$\bW_{\theta_S}, \bW_0$ & \eqref{def:W_matrix_main}, \eqref{def:W_0} & \\
$\bV_{n, S}$ & \eqref{def:V_matrix_main} &
$\sum_{i=1}^{n} \sigma_{i}^2 x_{i, S}x_{i, S}^{\top}$ \\
$\localSetRn[S]{r}$ & \eqref{def:local_neighborhood_elliptic_main} & 
$\{\theta_S \in \bbR^{|S|} : \| \bF_{n, \thetaBest[S]}^{1/2}(\theta_S - \thetaBest[S])  \|_{2} \leq r \}$ \\
$\pi_n(S), w_n(|S|)$ & \eqref{def:complexity_prior_main} & \\
$A_1$-$A_4$, $s_{\max}$ & \eqref{def:prior_S_penalty_main} & \\
$\scrS_{s}$ & \eqref{def:support_set_main} & $\{ S \subset [p] : |S| \leq s \}$ \\
$s_n$, $\widetilde{s}_n$ & Theorems \ref{thm:effective_dim_main}, \ref{thm:consistency_parameter_main} & $K_{\rm dim}s_0$, \: \: $(K_{\rm dim} + 1)s_0$ \\
$\phi_1(s; \bW)$, $\phi_2(s; \bW)$ & \eqref{def:compatibility_sparse_eigenvalue} & \\
$A_5, A_6, A_7$ & \eqref{A2:b} & \\
$A_8, K_{\rm cubic}$ & \eqref{A4:a}, \eqref{A4:b} & \\
$\cM_\alpha^n(S), \widehat{\cM}_\alpha^n(S)$ & \eqref{def:margin_likelihood_main}, \ref{def:marginal_likelihood_approx} & \\
$\scrS_{\rm eff}, \scrS_{\Theta_n}$  & \eqref{def:concentrated_support_eff_main}, \eqref{def:Theta_n_concentrated_support} & \\
$\widetilde{\scrS}_{\Theta_n}$, $\overline{\scrS}_{\Theta_n}$ & \eqref{def:A_4} & $\{ S \cup S_0 : S \in \scrS_{\Theta_n} \}$, $\scrS_{\Theta_n} \cup \widetilde{\scrS}_{\Theta_n}$ \\
$\scrS_{\rm sp}$ & \eqref{eqn:no_supset_claim} & $\{S \in \scrS_{\Theta_n} : S_0 \subsetneq S \}$ \\
$\scrS_{\rm fp}$ & \eqref{def:A6_notations} & $\{ S \cup S_0 : S \nsupseteq S_0, S \in \scrS_{\Theta_n} \}$ \\
$\kappa_n$, $\nu_n$, $\vartheta_{n, p}, K_{\rm min}$ & \eqref{def:A6_notations}, \eqref{A7:a} & $\vartheta_{n, p} = \min_{j \in S_0} |\theta_{0, j}|$ \\
\hline
\end{tabular}
\label{table:notations}
\end{table}

\subsection{Generalized linear models} \label{sec:likelihood_quantites}

This paper focuses on generalized linear models with canonical link functions. For a given $X=x$, suppose that the conditional density/mass function of the response variable $Y$ is given as in \eqref{eqn:glm_density_main}. Throughout this paper, we will assume the following without explicit restatement.
\begin{enumerate}
  \item The model is well-specified; hence there exists a ``true coefficient'' $\theta_0 \in \bbR^p$.
  \item $\theta_0$ is not the zero vector.
  \item $p \geq n^C$ for some constant $C > 0$.
  \item The covariates $x_1, \ldots, x_n$ in $\bbR^p$ are non-random.
  \item $b$ is strictly convex on $\bbR$ and three times differentiable, with derivatives $b', b''$ and $b'''$.
  \item There exists a constant $C_{\rm dev} \geq 1$, depending only on $b$, such that 
   \begin{align} \label{def:C_dev_GLM_main}
    \sup_{|y| \leq 1/2} b''(x + y) \leq  C_{\rm dev} b''(x), \quad \forall x \in \bbR.
   \end{align}
\end{enumerate}
The second assumption is only for convenience, and can easily be eliminated with additional statements in the main theorems. 
The third assumption is also made solely for notational convenience. Under this assumption, terms proportional to $\log n$ can be absorbed by terms proportional to $\log p$.
Verification of \eqref{def:C_dev_GLM_main} in standard GLMs is straightforward. For the Poisson regression model, for example, we have $b''(\cdot) = \exp(\cdot)$;
consequently, the constant $C_{\rm dev}$ in \eqref{def:C_dev_GLM_main} can be chosen as $e^{1/2}$.

The remainder of this subsection introduces some notation and background on GLMs. Let $\bbP^{(n)}_{\theta}$ be the joint probability measure corresponding to the product density $(y_1, \ldots, y_n) \mapsto \prod_{i=1}^n p_\theta(y_i \mid x_i)$.
It is well-known that $\bbE Y_i = b'(x_i^{\top}\theta_{0})$ and $\Var(Y_i) = b''(x_i^{\top}\theta_{0}) \defeq \sigma_{i}^{2}$, where $\bbE$ and $\Var$ denote expectation and variance under the true distribution $\bbP^{(n)}_{\theta_0}$.

Let $\ell_{\theta}(x, y) = \log p_{\theta}(y \mid x)$ be the log density and $\dot{\ell}_{\theta}(x, y) = \partial \ell_{\theta}(x, y) / \partial \theta$ be the score function. 
For convenience, we often write $p_{\theta}(Y_i \mid x_i)$, $\ell_{\theta}(x_i, Y_i)$, $\dot{\ell}_{\theta}(x_i, Y_i)$ as $p_{i, \theta}$, $\ell_{i, \theta}$, $\dot{\ell}_{i, \theta}$, respectively. 
Note that $\dot{\ell}_{i, \theta} = \left\{Y_i - b'(x_i^{\top}\theta)\right\}x_i = \epsilon_{i, \theta} \: x_i$, where $\epsilon_{i, \theta} = Y_i - b'(x_i^{\top}\theta)$.
Simply, we write $\epsilon_{i, \theta_0}$ as $\epsilon_i$. Let $L_{n, \theta} = L_{n, \theta}(\bX, \bY) = \sum_{i=1}^{n} \ell_{\theta}(x_i, Y_i)$ and
$$
L_{n, \theta_S} = L_{n, S, \theta_S} = \sum_{i=1}^{n} \ell_{\theta_S}(x_{i, S}, Y_i) =  \log p_{\theta_{S}}(Y_i \mid x_{i, S}).
$$
Define $\dot{L}_{n, \theta} = \sum_{i=1}^{n} \dot{\ell}_{i, \theta}$ and $\dot{L}_{n, \theta_S} = \sum_{i=1}^{n} \dot{\ell}_{i, \theta_S}$ similarly, where $\dot{\ell}_{i, \theta_S} = \{Y_i - b'(x_{i, S}^{\top}\theta_{S})\}x_{i, S}$.
Note that the notation $L_{n, \theta_S}$ (and $\dot{L}_{n, \theta_S}$, resp.) might be misleading because $L_{n, S, \theta_S}$ (and $\dot{L}_{n, S, \theta_S}$, resp.) depends not only on the vector $\theta_S$ but also on the model $S$. For convenience, we will continue to use the abbreviation $L_{n, \theta_S}$ (and $\dot{L}_{n, \theta_S}$, resp.), which should be understood as $L_{n, S, \theta_S}$ ($\dot{L}_{n, S, \theta_S}$, resp.). Similar abbreviations will be used elsewhere, e.g., see the definitions of $\bF_{n, \theta_S}$ and $\bW_{\theta_S}$ below.

Let $S_0$ be the index set for the nonzero entries of $\theta_0$ and $s_0 = |S_0| \geq 1$.
For $S \subseteq [p]$, set
\begin{align} \label{def:MLE_Best_main}
    \thetaMLE = \argmax_{\theta_S \in \bbR^{|S|}} L_{n, \theta_S} \quad \text{and} \quad 
    \thetaBest = \argmax_{\theta_S \in \bbR^{|S|}} \bbE L_{n, \theta_S}.
\end{align}
Recall that the corresponding $p$-vector versions, $\widetilde{\theta}_S^{\texttt{MLE}}$ and $\widetilde{\theta}_{S}^{\ast}$, are defined in \eqref{def:tilde_theta}.
Let
\begin{align} \label{def:Fisher_info_main}
	\bF_{n, \theta_S} = \bF_{n, S, \theta_S} 
    = -\dfrac{\partial^2}{\partial \theta_S   \partial \theta_S^{\top}} L_{n, \theta_S}
    = \bX_S^{\top} \bW_{\theta_S} \bX_S \in \bbR^{|S| \times |S|}
\end{align}
be the Fisher information matrix and 
\begin{align} \label{def:V_matrix_main}
  \bV_{n, S}
  = \sum_{i=1}^{n} \sigma_{i}^2 x_{i, S}x_{i, S}^{\top} = \bX_S^\top \bW_0 \bX_S,  
\end{align}
where $\bW_{\theta_S}$ is the diagonal matrix defined as
\begin{align} \label{def:W_matrix_main}
	\bW_{\theta_S} = \bW_{S, \theta_S} = \operatorname{diag}\bigl\{
	b''(\bx_{1, S}^{\top}\theta_S),
	..., 
	b''(\bx_{n, S}^{\top}\theta_S)
	\bigr\} \in \bbR^{n \times n}
\end{align}
and $\bW_0 = \bW_{\theta_0}$.
For $S \supseteq S_0$, we have $\widetilde{\theta}_{S}^{\ast} = \theta_0$, $\bF_{n, \thetaBest[S]} = \bV_{n, S}$ and
\begin{align} \label{def:W_0}
	\bW_{\thetaBest} = \bW_{\theta_{0} } = \operatorname{diag}\left\{
	\sigma_{1}^{2},
	..., 
	\sigma_{n}^{2}
	\right\} \in \bbR^{n \times n}.
\end{align}
However, $\bF_{n, \thetaBest[S]} = \bV_{n, S}$ is not guaranteed for $S \nsupseteq S_0$.

For $S \subset [p]$ with nonsingular $\bF_{n, \thetaBest}$, we introduce two important definitions from \cite{spokoiny2017penalized}. First, we define the normalized score function (evaluated at $\thetaBest[S]$) for model $S$ by
\begin{align} \label{def:score_func_normalized_main}
	\xi_{n, S} = \bF_{n, \thetaBest[S]}^{-1/2} \dot{L}_{n, \thetaBest[S]} = \bF_{n, \thetaBest[S]}^{-1/2} \sum_{i=1}^{n} \dot{\ell}_{i, \thetaBest[S]} = \bF_{n, \thetaBest[S]}^{-1/2} \sum_{i=1}^{n} \epsilon_{i, \thetaBest} x_{i, S}.
\end{align}
Regular behavior of $\xi_{n, S}$, such as (near) sub-Gaussianity, plays a central role in proving model selection consistency. We will discuss more about the regularity of $\xi_{n, S}$ in Section \ref{sec:no_superset}.
Second, define the local neighborhood of the optimal parameter $\thetaBest[S]$ as
\begin{align} \label{def:local_neighborhood_elliptic_main}
	\localSet[S] = \bigl\{\theta_S \in \bbR^{|S|} : \bigl\| \bF_{n, \thetaBest[S]}^{1/2}(\theta_S - \thetaBest[S])  \bigr\|_{2} \leq r \bigr\}, \quad r > 0.
\end{align}
Under regularity conditions, we will prove that $\thetaMLE$ concentrates on the local set $\localSetRn[S]{r}$, and the log-likelihood function $\theta_S \mapsto L_{n, \theta_S}$ can be approximated by a quadratic function within the local set $\Theta_S(r)$, with the radius $r$ of order $r \asymp (|S|\log p)^{1/2}$. Compared to the results in \cite{spokoiny2017penalized}, there is an additional term, $(\log p)^{1/2}$, which can be interpreted as the cost of requiring uniformity over $S$.  Furthermore, the adoption of such an elliptical set enables us to eliminate unnecessarily strong constraints related to the condition number of the matrix $\bF_{n, \thetaBest}$. In the Bayesian GLM literature \citep[e.g.,][]{barber2015high, ray2020spike, cao2022bayesian, tang2024empirical}, the condition number of $\bF_{n, \thetaBest}$ is often assumed to be bounded or not excessively large, primarily due to substantial technical difficulties. However, within the local set $\localSet[S]$, we can successfully remove these limitations, allowing the condition number of $\bF_{n, \thetaBest}$ to diverge up to a polynomial degree in $p$.

\subsection{Design matrix} \label{sec:design_matrix}

As mentioned above, we take the design matrix $\bX$ to be fixed. Given that we allow $p \gg n$, certain identifiability conditions are required to ensure the consistent estimation of $\theta_0$.
For $1 \leq s \leq p$ and $\bW \in \bbR^{n \times n}$, define the uniform compatibility number $\phi_1$ and the sparse singular value $\phi_2$ as
\begin{align} \label{def:compatibility_sparse_eigenvalue}
\begin{aligned}
	\phi_1^2(s; \bW) &= \inf \left\{\dfrac{ |S_{\theta}| \theta^{\top} \bSigma \theta  }{\| \theta \|_{1}^{2}} : 0 < |S_{\theta}| \leq s \right\} \\
	\phi_2^2(s; \bW) &= \inf \left\{\dfrac{\theta^{\top} \bSigma \theta  }{\| \theta \|_{2}^{2}} : 0 < |S_{\theta}| \leq s \right\},    
\end{aligned}
\end{align}
where $\bSigma = n^{-1} \bX^{\top}\bW\bX$. As in previous works \citep[e.g.,][]{jeong2021posterior}, the uniform compatibility number $\phi_1$ and the sparse singular value $\phi_2$ are concerned with recovery with respect to the $\ell_1$- and $\ell_2$-norms, respectively.  That is, suitable lower bounds on $\phi_1$ or $\phi_2$ make it possible to convert convergence in terms of the mean response to convergence of the parameter estimates to $\theta_0$.  Examples of \eqref{def:compatibility_sparse_eigenvalue} are presented in Section~\ref{sec:examples} and Appendix~\ref{sec:technical_lemmas_app}.

For $\bW = \bW_0$ and $S_{\theta} \supseteq S_0$, we have $\theta^{\top} (n \bSigma) \theta = \theta_{S_{\theta}}^{\top} \bF_{n, \thetaBest[S_{\theta}]} \theta_{S_{\theta}}$.
Therefore, the conditions on the eigenvalues of $\bF_{n, \thetaBest}$ are closely related to the estimation of $\theta$.
For $S \subset [p]$, let
\begin{align} \label{def:rho_eigenvalues}
	\rho_{\operatorname{max}, S} = \lambda_{\operatorname{max}} ( \bF_{n, \thetaBest[S]} ), \quad 
	\rho_{\operatorname{min}, S} = \lambda_{\operatorname{min}} ( \bF_{n, \thetaBest[S]} ).
\end{align}
The following inequalities can be directly derived from the definition:
\begin{align*}
    \bigl\| \bW_0^{1/2} \bX \theta \bigr\|_{2}^{2} & \geq n \phi_2^2(|S_{\theta}|; \bW_0) \| \theta \|_{2}^2 \\
	\rho_{\operatorname{min}, S} & \geq n \phi_2^2(|S'|; \bW_0) \qquad\quad \text{for} ~ S \supseteq S_0, ~~ |S'| \geq |S|.
\end{align*}

We follow \citet{spokoiny2017penalized} and define the design regularity quantity:
\begin{align} \label{def:design_regularity}
\designRegular = \max_{i \in [n]} \bigl\| \bF_{n, \thetaBest[S]}^{-1/2}x_{i, S} \bigr\|_{2}.
\end{align}
\citet{spokoiny2017penalized} showed that $\designRegular$ being sufficiently small ensures desirable properties of the log-likelihood and related quantities, in particular, $\designRegular \lesssim n^{-1/2}$ implies the quadratic expansion of the log-likelihood in a local neighborhood of $\theta_0$ remains valid for dimensions of order $s_0^3 \ll n$. (Note that \citet{spokoiny2017penalized} does not address a sparse setup; so, $s_0=p$ in his context, and the order $s_0^3 \ll n$ cannot be improved in general.) In Appendix \ref{sec:design_regularity_app}, we show that $\designRegular \lesssim n^{-1/2}$ holds with high probability in the case of Poisson regression, provided that $x_i$'s are \iid\ realizations from the standard normal distribution and $\|\theta_0\|_2$ is not too small.

However, the inequality $\designRegular \lesssim n^{-1/2}$ does not hold in general. For example, in logistic regression, it can be shown that $\rho_{\max, S} \lesssim n$ holds with high probability when $x_i$'s are \iid\ standard Gaussian; see Section~\ref{sec:examples} and Lemma \ref{lemma:least_eigenvalue_logit}.
%$b''(\cdot) \leq 1/4$,
Therefore,
\begin{align} \label{eqn:logit_design_regular_fail}
    \designRegular    
    \geq \rho_{\max, S}^{-1/2} \max_{i \in [n]} \left\| x_{i, S} \right\|_{2}  \gtrsim n^{-1/2} \max_{i \in [n]} \left\| x_{i, S} \right\|_{2}, 
\end{align}
hence $\designRegular \gg n^{-1/2}$ for $|S| \gg 1$ because $\max_{i \in [n]} \|x_{i, S}\|_2 \gtrsim |S|$.
In this case, Spokoiny's result only guarantees that the quadratic approximation of the log-likelihood remains valid up to an order of $s_0^4 \log p = o(n)$.
In Section 4, we consider a different approach to improve the required condition to $s_0^3 \log p = o(n)$, inspired by \citet[][Theorem~2.1]{barber2015high}. 

The approach in \citet{barber2015high} is not directly applicable to Poisson regression model with $s_0 \gg 1$. In this sense, the quadratic approximation of the log-likelihood in our paper combines the strengths of both \citet{spokoiny2017penalized} and \citet{barber2015high}, resulting in the sufficient condition $s_0^3 \log p = o(n)$ for both logistic and Poisson regression models.

\section{Prior and posterior distributions} \label{sec:prior}

\subsection{The prior}

Our sparsity-encouraging sequence of prior distributions for $\theta \in \bbR^p$, which we denote as $\Pi_n$, is defined hierarchically as follows.  Start by decomposing $\theta$ as $(S,\theta_S)$, where $S=S_\theta$ represents the configuration of zeros and non-zeros, and $\theta_S$ is the corresponding vector of non-zero values.  First, the marginal prior distribution for $|S|$ has mass function $w_n$ supported on the set $\{0, \ldots, s_{\max}\}$, where $s_{\max} \leq \operatorname{rank}(\bX)$ is a pre-specified upper bound for the number of nonzero coefficients. Here, we allow $s_{\max}$ to grow with $n$ and assume that $s_{\rm max} \geq s_0$. Next, the conditional prior for $S$, given the complexity $s$, is uniform over all such configurations.  Then the marginal prior for $S$ is 
\begin{align} 
\label{def:complexity_prior_main}
	\pi_{n}(S) = w_n(|S|) \, \textstyle\binom{p}{|S|}^{-1}.
\end{align}
Finally, the conditional prior for $\theta_S$, given $S$, has a density function $g_S$.  If we put this altogether, the prior distribution for $(S,\theta_S)$ has a ``density'' $(S, \theta) \mapsto \pi_n (S) \, g_{S}(\theta_S) \d\theta_S \times \delta_{0}(\d\theta_{S^{c}})$, where $\delta_{0}$ is the Dirac measure at zero on $\bbR^{p-|S|}$.  Of course, the prior $\Pi_n$ for $\theta$ is obtained by summing over $S$:
\[ \Pi_n(\d\theta) = \sum_S \big\{ \pi_n(S) \, g_S(\theta_S) \d\theta_S \times \delta_0(\d\theta_{S^{\rm c}}) \big\}. \]

For the prior to appropriately penalize the model size, a common assumption in the literature \citep[e.g.,][]{castillo2015sparse} is that there exist constants $A_1, A_2, A_3, A_4 > 0$ such that 
\begin{equation}
\begin{aligned} \label{def:prior_S_penalty_main}
    A_1 p^{-A_3} w_n(|S|-1) 
    \leq w_n(|S|) 
    \leq A_2 p^{- A_4} w_n(|S|-1),& \quad |S| \in  [s_{\operatorname{max}}] \\
    w_n(|S|) = 0,& \quad |S| > s_{\operatorname{max}}.
\end{aligned}
\end{equation}
With this prior, we can focus on the support set $\scrS_{s_{\max}}$ defined as
\begin{align} \label{def:support_set_main}
	\scrS_{s} = \left\{ S \subset [p] : |S| \leq s  \right\}
\end{align}
for a positive integer $s \leq p$. 

% For a prior $\Pi_n$ on $\theta \in \bbR^{p}$, we first select a dimension $s$ from a prior $\pi_p$ on the set $\{0, \ldots, s_{\max}\}$, where $s_{\max} \leq \operatorname{rank}(\bX)$ is a pre-specified upper bound for the number of nonzero coefficients. Then, we randomly choose a support $S \subset [p]$ with $|S| = s$. Given the support $S$, we select nonzero values $\theta_S$ from a continuous density $g_S$ on $\bbR^{|S|}$. Formally, the prior can be represented as
% $$
%  (S, \theta) \mapsto \pi_p (|S|) \binom{p}{|S|}^{-1} g_{S}(\theta_S)\delta_{0}(\theta_{S^{c}}),
% $$
% where $\delta_{0}$ is the Dirac measure at zero on $\bbR^{p-|S|}$. 
% With a slight abuse of notation, let
% \begin{align} \label{def:complexity_prior_main}
% 	\pi_{n}\left(S\right) = \pi_{p}(|S|) \binom{p}{|S|}^{-1}.
% \end{align}

For the prior density $g_{S}$, we follow \citet{martin2017empirical}, \citet{ebpred}, and \citet{tang2024empirical}; see, also, \citet{martin.walker.deb}. Specifically, here we take the $S$-specific prior density function to be
\begin{align} \label{def:prior_slab_main}
	g_S(\theta_S) = \mathcal{N}_{|S|}\bigl( \theta_S \mid \thetaMLE, \bigl\{\lambda \bF_{n, \thetaMLE}  \bigr\}^{-1} \bigr),
\end{align}
where $\cN_{s}(\cdot \mid \mu, \bSigma)$ denotes the $s$-dimensional multivariate normal density with mean $\mu$ and covariance matrix $\bSigma$.  What distinguishes this prior formulation from those in, e.g., \citet{castillo2015sparse} and \citet{jeong2021posterior}, is that this $S$-specific prior is {\em empirical} or {\em data-driven} in the sense that it depends on the data $(\bX, \bY)$.  The intuition behind this choice is as follows: we have no genuine prior information concerning the magnitudes of the non-zero entries in $\theta_0$, and we cannot use traditionally ``non-informative,'' improper priors for $\theta_S$---since model comparison and selection is one of our primary objectives---so we opt to let the data assist in choosing an appropriate center and spread for the prior density $g_S$.  At a more technical level, this data-driven prior centering alleviates the concerns expressed in e.g., \citet{castillo2015sparse}, about the heaviness of the prior density tails.  Again, the intuition is that the heaviness of the prior tails is less relevant if the prior center is informative.  

Lastly, some comments on the spread of the prior density $g_S$ are warranted.  Since the Fisher information $\bF_{n, \thetaMLE}$ is of order $n$, the prior density $g_S$ is fairly tightly concentrated around the $S$-specific MLE; this can, of course, be loosened to some extent via the choice of the scale factor $\lambda$.  It might seem contradictory for a sort of ``non-informative'' prior to be tightly concentrated, but that is not the case.  Indeed, there can be no benefit to the data-driven centering if the density itself is diffuse.  So, the relatively tight prior concentration is necessary to reap the benefits of the data-driven centering.  What matters most is that the corresponding posterior distribution has desirable properties, in particular, that it does not suffer---and perhaps even benefits---from the seemingly counter-intuitive, data-driven prior construction.  This has already been demonstrated in \citet{martin2017empirical} for the case of the Gaussian linear model, and in \citet{martin.walker.deb} more generally; in Sections~\ref{sec:convergence_of_posterior_distribution}--\ref{sec:model_selection_consistency} below, we show that the posterior distribution described next has very strong asymptotic properties in the context of GLMs.

\subsection{The (fractional) posterior}

Given the prior $\Pi_n$ and the likelihood $L_{n, \theta}$, we consider a $\alpha$-fractional posterior $\Pi_\alpha^n$ defined as
\begin{align} 
\label{eq:update}
\Pi_{\alpha}^n(\theta \in \cA) = \dfrac{
 		\int_{\cA} \exp(\alpha L_{n, \theta}) \, \Pi_n(\rmd\theta)
	}{
		\int \exp( \alpha L_{n, \theta} ) \, \Pi_{n} ( \rmd\theta )
	} \quad \text{ for any measurable $\cA \subset \bbR^{p}$},
\end{align}
where $\alpha \in (0, 1]$.  To help the reader with the notation, note that the subscript ``$n$'' in the prior $\Pi_n$ goes {\em up} to a superscript when it is {\em updated} to the posterior $\Pi_\alpha^n$ via the formula \eqref{eq:update}.  Use of a fractional or tempered likelihood was suggested in \citet{walker.hjort.2001} as a means to achieve posterior consistency under weaker-than-usual conditions.  Along these same lines, \citet{grunwald.ommen.scaling} and \citet{bhat.pati.yang.fractional} have argued that this tempering offers a degree of robustness to model misspecification; see, also, \citet{alquier.ridgway.2020}.  This robustness connection explains the necessity of the so-called {\em learning rate} or tempering in the construction of Gibbs posteriors when there is no model or likelihood function \citep[e.g.,][]{zhang2006b, martin.syring.chapter2022, gibbs.general}.  In \citet{martin2014asymptotically, martin.walker.deb} and \citet{martin2017empirical}, the tempering was explained as a technical device to prevent possible overfitting resulting from the use of the data in both the likelihood and the prior.  Like in the previous references, we will focus our attention here on the case $\alpha < 1$, just for simplicity. 
%But note that $\alpha$ can be arbitrarily close to 1, so there is no practical difference in the $\Pi_\alpha^n$-posterior inferences using, e.g., $\alpha=0.999$ and $\alpha=1$.
The theory presented here can be extended to cover the $\alpha=1$ case, just with some added assumptions and technical complications; see Section~\ref{sec:convergence_of_posterior_distribution}.

Given a posterior distribution for $\theta$, one can readily obtain a posterior for $S=S_\theta$ via marginalization.  Indeed, the marginal posterior of $S$ is given by the mass function 
\begin{align} \label{def:marginal_supp_posterior}
	\pi_\alpha^n(S) = 
	\dfrac{
		\pi_{n}(S) \int \exp( \alpha L_{n, \theta_S} ) \, g_{S}(\theta_{S}) \, \rmd \theta_S
	}{
		\sum_{S'}\pi_{n}(S') \int \exp( \alpha L_{n, \theta_{S'}} ) \, g_{S'}(\theta_{S'}) \, \rmd \theta_{S'}
	}.	
\end{align}
If we define the marginal likelihood as $\cM_\alpha^n(S) =  \int \exp(\alpha L_{n, \theta_S}) \, g_{S}(\theta_{S}) \, \rmd \theta_S$, then the marginal posterior mass function above can be represented by
\begin{align} \label{def:margin_likelihood_main}
\pi_\alpha^n(S) 
\propto \pi_{n}(S) \, \cM_\alpha^n(S).
\end{align}
This marginal posterior is what we will work with in the context of model selection.

\section{Posterior contraction}
\label{sec:convergence_of_posterior_distribution}

In this section, we demonstrate that the $\alpha$-fractional posterior distribution contracts to $\theta_0$ with a suitable rate. The main results and their proofs in this section are similar to those in \cite{jeong2021posterior} whose key idea is based on the general approach of \citet{ghosal2000convergence} and \citet{ghosal2007convergence}. A notable distinction in our theoretical analysis, compared to that in \citet{jeong2021posterior}, stems from our use of a data-dependent prior, which prevents the direct application of Fubini's theorem. \citet{martin.walker.deb} handle this in one way but, here, to overcome this technical obstacle, we initially establish fixed, non-data-dependent densities, $\overline{g}_S(\cdot)$ and $\underline{g}_S(\cdot)$, which satisfy 
%the following bounds 
\begin{align} \label{eqn:fixed_density_bound_upper_lower}
    p^{-c_1 s_0} \underline{g}_{S_0}(\cdot) \leq g_{S_0}(\cdot) , \quad
    g_S(\cdot) \leq p^{c_2 |S|} \overline{g}_{S}(\cdot) \quad \text{for all $S \in \scrS_{s_{\max}}$}, 
\end{align}
where $c_1$ and $c_2$ are positive constants. This facilitates the use of the general approach with Fubini's theorem. Importantly, the factors $p^{-c_1s_0}$ and $p^{c_2|S|}$ do not affect the rate of contraction; see Appendix \ref{sec:posterior_contraction_app} for details. For the inequalities \eqref{eqn:fixed_density_bound_upper_lower} to hold, assumption \textbf{(A1)} below is sufficient; see Lemma \ref{lemma:empirical_prior_bound} for the precise statement. 
\bed
\item[(\textbf{A1})]
There exist non-random $D_{n} > \sqrt{2}$ and non-random $\overline{\theta}_{S} \in \bbR^{|S|}$ such that $\bF_{n, \overline{\theta}_{S}}$ is nonsingular and
    \begin{align} \label{A1:a}
    \begin{aligned}
        \bbP_0^{(n)} \bigg( D_{n}^{-1} \bI_{|S|} \preceq \bF_{n, \thetaMLE} \bF_{n, \overline{\theta}_{S}}^{-1} &\preceq D_{n} \bI_{|S|}, \\
        \left\| \bF_{n, \overline{\theta}_{S}}^{1/2} \bigl( \thetaMLE - \overline{\theta}_{S} \bigr)  \right\|_{2}^{2} 
        &\leq D_{n} |S| \log p
        \quad \text{for all $S \in \scrS_{s_{\max}}$} \bigg) \geq 1 - p^{-1}.
    \end{aligned}
    \end{align}    
Furthermore, $\bF_{n, \thetaBest[S_0]}$ is nonsingular and
\begin{align}
    \begin{aligned} \label{A1:b}
        \designRegular[S_0]^2 s_0 \log p = o(1).
    \end{aligned}
\end{align}
\eed

Condition \eqref{A1:a} ensures that MLE $\thetaMLE$ does not deviate excessively 
from a fixed parameter $\overline{\theta}_S$ even when the models $S$ are misspecified, i.e., $S \nsupseteq S_0$. Also, condition \eqref{A1:b} guarantees the convergence of MLE for the true model $S_0$.
From the standard theory of maximum likelihood estimation, it is expected that $\thetaMLE$ is roughly close to $\thetaBest$. More specifically, under certain conditions, Lemma \ref{lemma:concentration_mle_score} establishes that
\begin{align} \label{eqn:MLE_concentration}
    \bigl\| \thetaMLE[S] - \thetaBest[S]  \bigr\|_2 \lesssim  \sqrt{\dfrac{ \Delta_{{\rm mis}, S} |S| \log p }{\rho_{\operatorname{min}, S}}}, \quad \text{for all $S \in \scrS_{s_{\max}}$},
\end{align}
with high probability, where
\begin{align} \label{def:misspecification_quantity_main}
    \Delta_{{\rm mis}, S} = \lambda_{\max} \bigl( \bF_{n, \thetaBest}^{-1/2} \bV_{n, S} \bF_{n, \thetaBest}^{-1/2}  \bigr)
\end{align}
denotes the magnitude of misspecification introduced in \cite{spokoiny2012parametric}.
Therefore, one can see that $\Delta_{{\rm mis}, S} \lesssim 1$ implies that $\thetaMLE$ contracts around $\thetaBest$ in a suitable sense. From this, one can prove that \eqref{A1:a} is satisfied with $\overline{\theta}_{S} = \thetaBest$ provided that $\max_{S \in \scrS_{s_{\max}}} \Delta_{{\rm mis}, S} \lesssim 1$ and $\max_{S \in \scrS_{s_{\max}}} \designRegular^2 |S| \log p = o(1)$; see Lemma \ref{lemma:concentration_mle_score} for the precise statement.

Note that $\Delta_{{\rm mis}, S} = 1$ for $S \supseteq S_0$, but $\Delta_{{\rm mis}, S}$ can become large for $S \nsupseteq S_0$. 
In Appendix \ref{sec:misspecified_estimator_example_app}, we prove under mild assumptions that \eqref{A1:a} is satisfied with high probability for a random matrix $\bX$.
Specifically, when $\| \theta_0 \|_2 \leq C$ and $x_{ij}$'s are \iid\ from $\cN(0, 1)$, the sufficient conditions can be summarized as follows:
\begin{align}
\begin{aligned} \label{case:both_GLM_random_design}
    &\text{ Poisson: $s_{\max} \log p = o(n^{1/2})$ implies \eqref{A1:a} with $\overline{\theta}_S = \thetaBest$ and $D_n = O(1)$.} \\
    &\text{ Logistic: $s_{\max} \log p = o(n^{2/3})$ implies \eqref{A1:a} with $\overline{\theta}_S = \thetaBest$ and $D_n = O(1)$.}
\end{aligned}    
\end{align}

Let $\overline{g}_{S}$ and $\underline{g}_{S_0}$ denote the densities corresponding to, respectively, 
\begin{align} \label{def:upper_lower_bound_empirical_prior}
    \cN \bigl( \overline{\theta}_{S}, \, \{ \tfrac{\lambda}{2} D_{n}^{-1} \, \bF_{n, \overline{\theta}_{S}} \}^{-1} \bigr), \quad \text{and} \quad
    \cN \bigl( \thetaBest[S_0], \, \{ 2 \lambda ( 1 + \delta_{n, S_0} ) \bF_{n, \thetaBest[S_0]} \}^{-1} \bigr),
\end{align} 
where $(\overline{\theta}_{S}, D_n)$, $\lambda$ and $\delta_{n, S_0}$ are defined in \textbf{(A1)}, \eqref{def:prior_slab_main} and Lemma \ref{lemma:Smoothness_of_the_Fisher_information_operator}, respectively. Specifically, under \eqref{A1:b}, we have $\delta_{n, S_0} = o(1)$.

\begin{lemma} \label{lemma:empirical_prior_bound_main}
Suppose that \textbf{(A1)} holds. 
Then, with $\bbP_{0}^{(n)}$-probability at least $1 - 2p^{-1}$, the following inequalities hold uniformly for all non-empty $S \in \scrS_{s_{\max}}$: 
\begin{align} \label{eqn:prior_bound_claim1_main}
   g_{S_0}(\theta_{S_0}) \geq p^{-(1 + \lambda C)s_{0}} \: \underline{g}_{S_0}(\theta_{S_0}),
   \quad
   g_{S}(\theta_S) \leq D_{n}^{2|S|} p^{ \lambda |S|/2} \: \overline{g}_S(\theta_S),
\end{align}
where $C > 0$ is a constant depending only on $C_{\rm dev}$, which is specified in \eqref{def:C_dev_GLM_main}.
\end{lemma}
\begin{proof}
See the proof of Lemma \ref{lemma:empirical_prior_bound}; Lemma \ref{lemma:empirical_prior_bound_main} is a special case of Lemma \ref{lemma:empirical_prior_bound}.
\end{proof}

Based on Lemma \ref{lemma:empirical_prior_bound_main}, we first provide a \textit{dimension reduction} theorem regarding the effective dimension of the posterior distribution. We need assumption (\textbf{A2}) for this. Recall that $\lambda$ and $D_n$ are specified in \eqref{def:prior_slab_main} and \eqref{A1:a}, respectively.
\bed
\item[(\textbf{A2})] 
The following asymptotic bounds hold:
\begin{align} \label{A2:a} 
\begin{aligned}
    \log \Bigg( 
    \left[ \max_{i \in [n]} b'' \left( x_{i}^{\top} \theta_{0} \right) \right]
    \vee 
    \|\bX_{S_0}\|_{\infty} 
    \vee
    \rho_{\min, S_{0}}^{-1} 
    \vee 
    \rho_{\operatorname{max}, S_0}
    \Bigg)
    &= O(\log p), \\
    s_0 \log p &= o(n).
\end{aligned}
\end{align}
Also, there exist constants $A_5, A_6 > 0$ and $A_7 \geq 0$ such that
\begin{align} \label{A2:b} 
    p^{-A_5} \leq \lambda \leq A_6 p^{-A_7}.
\end{align}
Finally, $\alpha \in (0, 1)$ and
\begin{align} \label{A2:c} 
    A_6p^{-A_7} < A_{4}, \quad \log_{p}(D_n) = o(1),
\end{align}
where $A_{4}$ is the constants specified in \eqref{def:prior_S_penalty_main}.
\eed
Condition \eqref{A2:a}, which is very mild, guarantees that sufficient prior mass is assigned to neighborhoods of $\theta_0$.
Conditions \eqref{A2:b} and \eqref{A2:c} ensure that the posterior will contract to the collection of models whose sizes are bounded by $K s_0$ for some constant $K > 0$. 
Note that $\lambda$ and $D_n$ cannot be excessively large. As illustrated in \eqref{case:both_GLM_random_design}, $D_n$ is typically of order $O(1)$.

Before stating the first of our posterior contraction theorems, we make two general remarks to fix the particular context.  First, as mentioned briefly above, here we focus on the case where $\alpha < 1$ for technical convenience. Extending to $\alpha=1$ is not difficult, but requires an additional assumption; see Assumption 2 in \citet{jeong2021posterior} and the related comments therein for more details.  Second, our results are stated for a fixed, true $\theta_0$ vector and the bounds involve features of that fixed $\theta_0$, such as the size/complexity $s_0$.  But just like the other papers on the present topic \citep[e.g.,][]{castillo2015sparse}, our results hold uniformly in $\theta_0$ that satisfy certain constraints on, say, the size/complexity or norm.  The specifics of the ``uniformity'' in each case can be readily gleaned from the finite-sample bounds presented in the Appendix.

\begin{theorem}[Effective dimension] \label{thm:effective_dim_main}
Suppose that \textbf{(A1)} and \textbf{(A2)} hold.
Then, there exists a constant $K_{\rm dim} > 1$ such that
$$
\bbE \, \Pi_n^{\alpha} \bigl\{\theta: |S_{\theta}| > K_{\rm dim} s_0\bigr\} \leq (s_0 \log p)^{-1} + 2p^{-1} + p^{-s_0}.
$$
\end{theorem}
\begin{proof}
See the proof of Theorem \ref{thm:effective_dim}; Theorem \ref{thm:effective_dim_main} is a special case of Theorem \ref{thm:effective_dim}.
\end{proof}

Define $s_n = K_{\rm dim}s_0$ and then set 
\begin{align} \label{def:concentrated_support_eff_main}
	\mathscr{S}_{\rm eff} = \{S \subset [p] : |S| \leq s_n \}.
\end{align}
Then, Theorem \ref{thm:effective_dim_main} implies that $\bbE \, \Pi_{\alpha}^n(\theta: S_\theta \in \scrS_{\rm eff}) \rightarrow 1$.
For two coefficient vectors $\theta_1, \theta_2 \in \bbR^{p}$, define the mean Hellinger distance by
\begin{align*}
	H_n\left( \theta_1, \theta_2  \right) = \left\{
	n^{-1} \sum_{i=1}^{n} H^2 \left( p_{i, \theta_1}, p_{i, \theta_2} \right)
	\right\}^{1/2},
\end{align*}
where $H^2 \left( p_{i, \theta_1}, p_{i, \theta_2} \right) = \int \left( \sqrt{p_{i, \theta_1}} - \sqrt{p_{i, \theta_2}} \right)^2 \rmd \mu$.

%Theorem 2 in \cite{tang2024empirical} presents a result identical to that in Theorem \ref{thm:effective_dim_main}. However, their approach does not achieve `dimension reduction' because Theorem 2 of \cite{tang2024empirical} relies on the quadratic expansion of the log-likelihood over $\scrS_{s_{\max}}$, which necessitates a sample size as large as that required for subsequent model selection proofs.

%To achieve the desired theoretical properties, it is essential that subsequent model selection theorems incorporate dimension reduction techniques on $\scrS_{\rm eff}$. Specifically, according to Theorem \ref{thm:effective_dim_main}, the set of possible models is reduced to $\mathscr{S}_{\rm eff}$, as defined in \eqref{def:concentrated_support_eff_main}. This reduction allows a valid quadratic approximation under the condition $s_0^3 \log p = o(n)$, instead of the more restrictive $s_{\max}^3 \log p = o(n)$. Given that $s_0$ is an unknown quantity, one might choose $s_{\max}$ such that $s_{\max} \gg s_0$ to ensure consistent estimation. Consequently, implementing an appropriate dimension reduction strategy is crucial for achieving the desired dependency on dimension.

\begin{theorem}[Consistency in Hellinger distance] \label{thm:consistency_Hellinger_main}
Suppose that \textbf{(A1)} and \textbf{(A2)} hold.
Then there exists a constant $K_{\operatorname{Hel}}>0$ such that
\begin{align*} %\label{eqn:hel_dist_claim}
    \bbE \, \Pi_{\alpha}^n \left\{\theta: H_n\left(\theta, \theta_0 \right)> K_{\operatorname{Hel}} \: \epsilon_n \right\} \leq  2(s_0 \log p)^{-1} + 4p^{-1} + 2p^{-s_0}     
\end{align*}
for sufficiently large $n$, where $\epsilon_n = \left(  s_0 \log p / n  \right)^{1/2}$.
\end{theorem}

\begin{proof}
See the proof of Theorem \ref{thm:consistency_Hellinger}; Theorem \ref{thm:consistency_Hellinger_main} is a special case of Theorem \ref{thm:consistency_Hellinger}.
\end{proof}

To ensure the convergence of $\theta$, we need the following assumption.

\bed
\item[(\textbf{A3})]
The following asymptotic bound holds:
\begin{align} \label{A3:a} 
    &\| \bX \|_{\rm max}^2 s_0^2 \log p / \phi_{2}^{2}\left( \widetilde{s}_n; \bW_0 \right) = o(n),
\end{align}
where $\widetilde{s}_n = (K_{\rm dim} + 1)s_0$ 
\eed

\begin{theorem}[Consistency in parameter $\theta$] \label{thm:consistency_parameter_main}
    Suppose that \textbf{(A1)}-\textbf{(A3)} hold.
    Then there exists a constant $K_{\operatorname{theta}}>0$ such that
    \begin{align*}
    \bbE \, \Pi_{\alpha}^n \left(\theta: \left\|\theta-\theta_0\right\|_1 > 
    \frac{K_{\operatorname{theta}}s_0 }{ \phi_1 \left( \widetilde{s}_n ; \bW_0 \right)} \sqrt{\dfrac{\log p}{n}} \right) &\leq  2(s_0 \log p)^{-1} + 4p^{-1} + 2p^{-s_0} \\
    \bbE \, \Pi_{\alpha}^n \left(\theta: \left\|\theta-\theta_0\right\|_2  > \frac{K_{\operatorname{theta}} }{ \phi_2 \left( \widetilde{s}_n ; \bW_0 \right)} \sqrt{\dfrac{s_0 \log p}{n}} \right) &\leq  2(s_0 \log p)^{-1} + 4p^{-1} + 2p^{-s_0} \\
    \bbE \, \Pi_{\alpha}^n \bigl(\theta: \bigl\| \bF_{n, \theta_0}^{1/2} (\theta-\theta_0)\bigr\|_2^2   > K_{\operatorname{theta}} s_0 \log p\bigr) &\leq  2(s_0 \log p)^{-1} + 4p^{-1} + 2p^{-s_0}.
    \end{align*}
\end{theorem}

\begin{proof}
See the proof of Theorem \ref{thm:consistency_parameter}; Theorem \ref{thm:consistency_parameter_main} is a special case of Theorem \ref{thm:consistency_parameter}.
\end{proof}

Theorem \ref{thm:consistency_parameter_main} yields contraction rates identical to those in \citet{jeong2021posterior}, where general but data-independent prior densities are considered. As discussed in the beginning of this section, the key difference between our approach and that of \citet{jeong2021posterior} lies in the data-dependency of the empirical prior distribution. 
    
From a technical perspective, the primary motivation for choosing an empirical prior is to eliminate unnecessary restrictions on the signal size of the true parameter $\theta_0$. As shown in Theorem 2.8 of \citet{castillo2012needles}, when the prior has a Gaussian tail, the resulting contraction rates may become suboptimal depending on $\| \theta_0 \|_2$; thereby necessitating certain restrictions on the signal size. For example, in the proof of Example 4 in \citet{jeong2021posterior}, they assumed $\lambda \| \theta_0 \|_{2}^{2} \lesssim s_0 \log p$ with a prior $g_{S}(\cdot) = \cN(\cdot \mid 0, \lambda^{-1} \bI_{|S|})$ to achieve (nearly) minimax-optimal contraction rates. Therefore, if $\| \theta_0 \|_{2}^{2} \gg s_0 \log p$, the minimax-optimality is not guaranteed with a constant $\lambda$. 

In the Gaussian model, the signal size restrictions mentioned above can be avoided by adopting a heavy-tailed prior \citep{castillo2012needles, castillo2015sparse}. For example, a Laplace prior on $\theta_S$ for each model $S$ does not impose specific restrictions on $\| \theta_0 \|_{1}$ or $\| \theta_0 \|_{2}$. Notably, \citet{castillo2012needles} and \citet{castillo2015sparse} rely on the explicit form of the Gaussian log-likelihood.

For GLMs, however, it is not easy to eliminate assumptions on the size of $\theta_0$. In the proof of Example 2 in \citet{jeong2021posterior}, it is still assumed that $\lambda \| \theta_0 \|_{1} \lesssim s_0 \log p$ even when Laplace prior is used for the slab part. This condition arises due to technical challenges in deriving a lower bound for the marginal likelihood.  Such a requirement is undesirable, as it undermines the rationale behind using a heavy-tailed prior.  In contrast, our theoretical framework does not impose any restrictions on the signal size of $\theta_0$. This is consistent with previous findings in the literature: similar results have been established by \citet{martin2017empirical} for Gaussian linear models and by \citet{tang2024empirical} for GLMs.

Before concluding this section, we present a lemma that plays an important role in establishing model selection consistency.
Let
\begin{align} 
\begin{aligned} \label{def:Theta_n_concentrated_support}
    \Theta_n &= \bigl\{ 
    \theta \in \bbR^p : 
    |S_{\theta}| \leq s_n, \quad
    \bigl\| \bF_{n, \theta_0}^{1/2} ( \theta - \theta_0 ) \bigr\|_{2}^{2} \leq K_{\rm theta} s_0 \log p
    \bigr\}, \\
	\scrS_{\Theta_n} &= \bigl\{ S \in \scrS_{s_n} : \big\| \bF_{n, \theta_0}^{1/2} (\widetilde{\theta}_S - \theta_0) \big\|_{2}^{2} \leq K_{\rm theta} s_0 \log p \: \text{ for some } \theta_S \in \bbR^{|S|} \bigr\}.
\end{aligned}
\end{align}
Then, Theorem~\ref{thm:consistency_parameter_main} implies that $\bbE \, \Pi_{\alpha}^n (\Theta_n) \rightarrow 1$ and $\bbE \, \Pi_{\alpha}^n (\theta: S_{\theta} \in \scrS_{\Theta_n}) \rightarrow 1$. 
Also, 
%one can easily see that 
\begin{align} \label{eqn:support_well}
    \| \theta_{0, S^{\rm c}} \|_2 \leq \dfrac{K_{\rm theta}}{ \phi_2 \left( \widetilde{s}_n ; \bW_0 \right) } \sqrt{\dfrac{s_0 \log p}{n }}
    \quad \forall S \in \scrS_{\Theta_n},
\end{align}
which implies that, for all $S \in \scrS_{\Theta_n}$, there exists $\theta_S \in \bbR^{|S|}$ such that $\widetilde{\theta}_S$ is sufficiently close to $\theta_0$. In other words, every model in $\scrS_{\Theta_n}$ is nearly well-specified.

The degree of model misspecification can be better expressed via the quantity $\Delta_{{\rm mis}, S}$, defined in \eqref{def:misspecification_quantity_main}. Recall that $\Delta_{{\rm mis}, S} = 1$ for $S \supseteq S_0$, but it can be large for a misspecified model $S$. Since we approximate the marginal likelihood using the Laplace approximation, an important step in achieving model selection consistency is to obtain a suitable convergence rate for the MLE $\thetaMLE$, e.g., \eqref{eqn:MLE_concentration}. Since the rate directly depends on $\Delta_{{\rm mis}, S}$, it is crucial to bound $\Delta_{{\rm mis}, S}$ appropriately. Lemma \ref{lemma:mis_on_posterior_concentration_set_main} provides an appropriate bound for this quantity.

%To analyze the impact of misspecification, we use $\Delta_{{\rm mis}, S}$, defined in \eqref{def:misspecification_quantity_main}, which quantifies the magnitude of misspecification. By leveraging the statistical properties within $\scrS_{\Theta_n}$, it can be shown that $\max_{S \in \scrS_{\Theta_n}} \| \bF_{n, \theta_0}^{1/2} (\widetilde{\theta}_{S}^{\ast} - \theta_0) \|_{2} \lesssim s_0 \log p$. This result forms the basis for the following lemma, which demonstrates that the misspecification effect, $\Delta_{{\rm mis}, S}$, can be suitably controlled within $\scrS_{\Theta_n}$.

\begin{lemma}[Misspecification on $\scrS_{\Theta_n}$] \label{lemma:mis_on_posterior_concentration_set_main}
Suppose that \textbf{(A1)}-\textbf{(A3)} hold.
Then, 
\begin{align} \label{eqn:mis_posterior_concentration_claim_main}
    \max_{S \in \scrS_{\Theta_n}} \{ \Delta_{{\rm mis}, S} \vee \widetilde{\Delta}_{{\rm mis}, S} \} 
    \leq 2,
\end{align} 
where $\widetilde{\Delta}_{{\rm mis}, S} = \| \bV_{n, S}^{-1/2} \bF_{n, \thetaBest} \bV_{n, S}^{-1/2}  \|_{2}$.
\end{lemma} 
\begin{proof}
See the proof of Lemma \ref{lemma:mis_on_posterior_concentration_set}; Lemma \ref{lemma:mis_on_posterior_concentration_set_main} is a special case of Lemma \ref{lemma:mis_on_posterior_concentration_set}.
\end{proof}

\section{Model selection consistency} \label{sec:model_selection_consistency}

This section presents our main results on model selection consistency for the posterior $\Pi_\alpha^n$. We focus here on the case $\alpha < 1$, but all the results are valid for $\alpha = 1$ once the posterior contraction results in the previous section have been established; the latter requires one additional assumption and some extra effort, as described in \citet{jeong2021posterior}.

\subsection{Laplace approximation} \label{sec:laplace_approx_main}

In this subsection, we provide results for a sharp Laplace approximation of the marginal likelihood $\cM_\alpha^n(S) := \int \exp( \alpha L_{n, \theta_S} ) \, g_{S}(\theta_{S}) \, \rmd \theta_S$.
Let 
\begin{align} \label{def:marginal_likelihood_approx}
\widehat{\cM}_\alpha^n (S) = \exp\bigl(\alpha L_{n, \thetaMLE[S]} \bigr) (1 + \alpha \lambda^{-1})^{-|S|/2}. 
\end{align}
be the Laplace approximation of $\cM_\alpha^n (S)$. 

To approximate the marginal likelihood, Laplace approximations have been widely considered in the literature on selection consistency in Bayesian GLMs; see, e.g.,  \citet{barber2015high}, \citet{narisetty2018skinny} \footnote{In \citet{narisetty2018skinny}, the Laplace approximation is used not for posterior inference but to approximate the marginal likelihood, which is then employed to bound the Bayes factor in their theoretical analysis (see the proof of Theorem 2 therein).}, \citet{rossell2021approximate}, \citet{cao2022bayesian}, and \citet{tang2024empirical}. The sharp convergence analysis in \cite{spokoiny2012parametric, spokoiny2017penalized} offers substantial benefits for obtaining an accurate approximation $\widehat{\cM}_n (S)$. To simplify the required conditions and statements, many statements in this section are written asymptotically. Detailed non-asymptotic statements for Laplace approximation can be found in Appendix \ref{sec:laplace_approximation_app}.

As shown in Section \ref{sec:convergence_of_posterior_distribution}, it suffices to consider the Laplace approximation for models $S \in \scrS_{\Theta_n}$, where $\scrS_{\Theta_n}$ is defined in \eqref{def:Theta_n_concentrated_support}.
However, in the proofs, we often need to consider models of the form $S \cup S_0$ with $S \in \scrS_{\Theta_n}$. To facilitate this, we introduce some related notation:
\begin{align}
\begin{aligned} \label{def:A_4}
    &\widetilde{\scrS}_{\Theta_n} = \left\{ S \cup S_0 : S \in \scrS_{\Theta_n} \right\}, \quad 
    \overline{\scrS}_{\Theta_n} = \scrS_{\Theta_n} \cup \widetilde{\scrS}_{\Theta_n}.
\end{aligned}
\end{align}
Additionally, let
\begin{align}
\begin{aligned} \label{def:A_4_2}
    &\cU_S = \big\{ u \in \bbR^{|S|} : \| u \|_{2} = 1 \big\}, \quad 
    \designRegular[\scrS_{\Theta_n}] = \max_{S \in \scrS_{\Theta_n}} \designRegular[S], \quad 
    \sigma_{\max}^2 = \max_{i \in [n]} b''(x_i^{\top}\theta_0).
\end{aligned}
\end{align}
To ensure the accuracy of $\widehat{\cM}_{\alpha}^{n}(S)$ for all $S \in \overline{\scrS}_{\Theta_n}$, we impose assumption (\textbf{A4}) below.
\bed
\item[(\textbf{A4})]
There exist constants $A_8, K_{\rm cubic} > 0$ such that
\begin{align} \label{A4:a} 
    \max_{S \in \scrS_{\Theta_n}} \rho_{\max, S} &\leq p^{A_8}, \\
    \max_{S \in \overline{\scrS}_{\Theta_n}} \sup_{u_S \in \cU_S} \dfrac{1}{n} \sum_{i=1}^{n} \bigl| x_{i, S}^{\top} u_S \bigr|^3 &\leq K_{\rm cubic}. \label{A4:b} 
\end{align}
Also, the following holds: 
\begin{align} \label{A4:c}
    \left[ (s_0^3 \log p)^{1/2} \  \designRegular[\scrS_{\Theta_n}] \right] \wedge \left[ \dfrac{ \sigma_{\max}^2}{ \phi_{2}^{3}\left( \widetilde{s}_n; \bW_0 \right) } \left( \dfrac{s_0^3 \log p}{n} \right)^{1/2} \right] &= o(1).
\end{align}
\eed

\begin{comment}
    More precisely, by setting the radius $r = C (|S| \log p)^{1/2}$ for a sufficiently large constant $C > 0$, the contributions of the tail integral become negligible:
\begin{align*}
    \int_{\left[\Theta_S(r) \right]^{\rm c}} \exp( \alpha L_{n, \theta_S} ) \, g_{S}(\theta_{S}) \, \rmd \theta_S
    \ll \widehat{\cM}_\alpha^n (S),
\end{align*}
where the local set $\Theta_S$ is defined in \eqref{def:local_neighborhood_elliptic_main}. 
Similarly, to ensure another type of tail control described in Lemma \ref{lemma:normality_truncated_support}, we impose the condition $C' |S| \log \rho_{\max, S} \leq r$ for some constant $C' > 0$. Consequently, combining \eqref{A4:a} with the setting $r = M_n (|S| \log p)^{1/2}$ significantly simplifies the analysis and facilitates the Laplace approximation.
\end{comment}

Let $\cM_{\alpha}^{n}(S, \cA) = \int_{\cA} \exp( \alpha L_{n, \theta_S} ) \, g_{S}(\theta_{S}) \, \rmd \theta_S$.
While condition \eqref{A4:a} is used to bound the tail part $\cM_{\alpha}^{n}(S, \Theta_{S}^{\rm c}(r))$ of the marginal likelihood as
\begin{align*}
    \frac{\cM_{\alpha}^{n}(S, \Theta_{S}^{\rm c}(r))}{\cM_{\alpha}^{n}(S, \Theta_{S}(r))} \approx 0
\end{align*}
with $r \asymp (|S| \log p)^{1/2}$, conditions \eqref{A4:b} and \eqref{A4:c} are used to approximate $\cM_{\alpha}^{n}(S, \Theta_{S}(r))$. To be more precise, we would like to mention that condition \eqref{A4:b} ensures \eqref{eqn:fff} below.

Condition \eqref{A4:a} is very mild, and condition \eqref{A4:b} also holds in many examples. For example, if $x_{ij}$'s are independent standard Gaussian and $s_0 \log p = o(n^{2/3})$, then \eqref{A4:b} holds with high probability; see Lemma \ref{lemma:cubic_poly_Gaussian}.
Additionally, condition \eqref{A4:c} holds under $s_0^3 \log p \ll n$, provided that either of the following conditions hold:
\begin{align}
\begin{aligned} \label{case:both_GLM}
\designRegular[\scrS_{\Theta_n}] \lesssim n^{-1/2} \quad \text{and} \quad \sigma_{\max}^2 \vee \phi_{2}^{-1}\left( \widetilde{s}_n; \bW_0 \right) = O(1).
    % 1.& \: \:  \designRegular[\scrS_{\Theta_n}] \lesssim n^{-1/2} \\
    % 2.& \: \:  \sigma_{\max}^2 \vee \phi_{2}^{-1}\left( \widetilde{s}_n; \bW_0 \right) = O(1).
\end{aligned}
\end{align}
Each condition in \eqref{case:both_GLM} corresponds to the first and second terms on the left-hand side of \eqref{A4:c}. For a logistic regression model, $\sigma_{\max}^2$ is bounded; hence the second condition holds if $\phi_{2}\left( \widetilde{s}_n; \bW_0 \right)$ is bounded away from zero. As mentioned in Section \ref{sec:design_matrix}, for Poisson regression, the condition $\designRegular[\scrS_{\Theta_n}] \lesssim n^{-1/2}$ is satisfied under a mild assumption on $\| \theta_0 \|_{2}$.

To compare with existing results, we would like to highlight that a crucial step in the proof of Theorem \ref{thm:LA_marginal_likelihood_main} is to establish that
\begin{align} \label{eqn:fff}
    \max_{S \in \scrS_{\Theta_n}} \sup_{\theta_S \in \localSetRn[S]{r}} \bigl\| \bF_{n, \thetaBest}^{-1/2} \bF_{n, \theta_S} \bF_{n, \thetaBest}^{-1/2} - \bI_{|S|} \bigr\|_{2} \lesssim \left( \dfrac{s_0 \log p}{n} \right)^{1/2},
\end{align}
which is closely related to the smoothness of the map $\theta_S \mapsto \bF_{n, \theta_S}$; see Lemma \ref{lemma:extended_Fisher_smooth2}.
Similar techniques have been considered in \citet{narisetty2018skinny}, \citet{lee2021bayesian}, \citet{cao2022bayesian} and \citet{tang2024empirical}. Although not explicitly stated in these papers, their quadratic approximation requires that $s_{\max}^4 \log p = o(n)$ under some conditions \citep[][Lemma~7.2]{lee2021bayesian}. This is because their results are based on 
\begin{align}
\label{eqn:fff.others}
    \max_{S \in \scrS_{s_{\max}}} \sup_{\theta_S \in \localSetRn[S]{r}} \bigl\| \bF_{n, \thetaBest}^{-1/2} \bF_{n, \theta_S} \bF_{n, \thetaBest}^{-1/2} - \bI_{|S|} \bigr\|_{2} \lesssim \left( \dfrac{s_{\max}^2 \log p}{n} \right)^{1/2},
\end{align}
which is a significantly looser bound compared to \eqref{eqn:fff}.
To the best of our knowledge, (\textbf{A4}) is the weakest condition for Laplace approximation to be valid in GLMs. Now, we state the main theorem for the Laplace approximation.

\begin{theorem}[Laplace approximation of the marginal likelihood] \label{thm:LA_marginal_likelihood_main}
Suppose that \textbf{(A1)}-\textbf{(A4)} hold.
Then, 
\begin{align*} 
\bbP_{0}^{(n)} \left(
\dfrac{\pi_{\alpha}^n(S) }{ \pi_{\alpha}^n(S_0) } 
\leq 
2
\dfrac{\pi_{n}(S) \, \widehat{\cM}_\alpha^n (S) }{\pi_{n}(S_0)  \, \widehat{\cM}_\alpha^n(S_0)}
\quad 
\text{for all } \ S \in \scrS_{\Theta_n} \setminus \varnothing 
\right) 
\geq 
1 - p^{-1},
\end{align*} 
where $\pi_{\alpha}^n(\cdot)$ is defined in \eqref{def:marginal_supp_posterior}.
\end{theorem}
\begin{proof}
See the proof of Theorem \ref{thm:LA_marginal_likelihood}; Theorem \ref{thm:LA_marginal_likelihood_main} is a special case of Theorem \ref{thm:LA_marginal_likelihood}.
\end{proof}

From the proof, one can deduce that the constant $2$ in Theorem \ref{thm:LA_marginal_likelihood_main} can be replaced by $1 + \epsilon$ for any arbitrarily small constant $\epsilon > 0$, provided that $n$ is sufficiently large; see Theorem \ref{thm:LA_marginal_likelihood} for the precise statement.

    A technical advantage to using an empirical prior is that it simplifies the form of the Laplace approximation. With additional effort, we conjecture that the Laplace approximation (Theorem \ref{thm:LA_marginal_likelihood_main}) and model selection consistency results in Sections \ref{sec:no_superset} and \ref{sec:no_false_negative} would also hold for data-independent priors, such as those considered in \citet{narisetty2018skinny}, \citet{barber2016laplace}, \citet{lee2021bayesian} and \citet{cao2022bayesian}.

    It is also worth mentioning that model selection consistency does not necessarily require an optimal posterior convergence rate.
    %, e.g., $(s_0 \log p /n)^{1/2}$ in the $\ell_2$-norm sense. 
    However, if the posterior convergence rate is sub-optimal, then a stronger condition (e.g., a condition on $s_0$) would be required. This is because a crucial step in proving model selection consistency is the quadratic approximation of the log-likelihood within a local set where the posterior contracts. Typically, the accuracy of this quadratic approximation strongly depends on the size of this local set. Consequently, the same condition may no longer be sufficient to ensure model selection consistency. 

\subsection{No supersets} \label{sec:no_superset}

Recall that for $S \supseteq S_0$, we have $\bF_{n, \thetaBest} = \bX_S^{\top} \bW_0 \bX_S$. 
We will show that $\bbE \, \Pi_{\alpha}^n(\theta: S_{\theta} = S_0) \rightarrow 1$ under suitable assumptions. One challenging part is to prove that
\begin{align} \label{eqn:no_supset_claim}
	\bbE \, \Pi_{\alpha}^n(\theta: S_{\theta} \in \scrS_{\rm sp}) \rightarrow 0,
\end{align}
where $\scrS_{\rm sp} = \{S \in \scrS_{\Theta_n} : S \supsetneq S_0 \}$ is the collection of supersets of $S_0$.
We first state the key assumption. 
Although condition \eqref{A5:b} below is slightly stronger than \eqref{A4:c}, under either of the conditions described in \eqref{case:both_GLM}, the condition $s_0^3 \log p = o(n)$ is sufficient to satisfy \eqref{A5:b}.

\bed
\item[(\textbf{A5})] 
The constants $A_4$ and $A_7$, specified in \eqref{def:prior_S_penalty_main} and \eqref{A2:b}, satisfy 
\begin{align} \label{A5:a} 
A_4 + A_7/2 > \alpha (16 C_{\rm dev}) + \log_{p}(s_0) + \delta_{1}   
\end{align}
for some (sufficiently small) constant $\delta_{1} > 0$ and 
\begin{align} \label{A5:b} 
\left[ \left( s_0^{3} \log p \right)^{1/2} \designRegular[\widetilde{\scrS}_{\Theta_n}] \right] \wedge \left[ \dfrac{ \sigma_{\max}^2 }{ \phi_{2}^{3}(\widetilde{s}_n; \bW_0)} \left( \dfrac{s_0^3 \log p}{n} \right)^{1/2} \right] = o(1),
\end{align}
where $\designRegular[\widetilde{\scrS}_{\Theta_n}] = \max_{S \in \widetilde{\scrS}_{\Theta_n}} \designRegular$.
\eed

Condition \eqref{A5:a} enables the posterior to eliminate unimportant variables. Specifically, $A_4$ directly penalizes the model complexity through the prior defined in \eqref{def:complexity_prior_main} while $A_7$ achieves a similar effect by shrinking the (approximated) marginal likelihood as described in \eqref{def:marginal_likelihood_approx}.
Consequently, \eqref{A5:a} describes the interplay between $A_4$ and $A_7$, resulting in an appropriate regularization effect on the model size $|S|$. 
See Section \ref{sec:theory_hyperparams} for further discussion.

\begin{theorem}[No superset] \label{thm:no_superset_main}
    Suppose that \textbf{(A1)}-\textbf{(A5)} hold.
    Then,
     \begin{align*} %\label{eqn:no_superset}
        \bbE \, \Pi_{\alpha}^n(\theta: S_{\theta} \in \scrS_{\rm sp}) \leq  2(s_0 \log p)^{-1} + 5p^{-1} + 2p^{-s_0} + 3p^{-\delta_{1}},
    \end{align*}
    where $\delta_{1}$ is the constant specified in \eqref{A5:a}.
\end{theorem}
\begin{proof}
See the proof of Theorem \ref{thm:no_superset}; Theorem \ref{thm:no_superset_main} is a special case of Theorem \ref{thm:no_superset}.
\end{proof}

Before presenting the key idea in our proof of Theorem \ref{thm:no_superset_main}, it is worth introducing the general proof strategy followed in the literature on Bayesian model selection consistency. For $S \supsetneq S_0$, by a Taylor expansion, we can approximate $L_{n, \thetaMLE} - L_{n, \thetaMLE[S_0]}$ by
\bean
  L_{n, \thetaMLE} - L_{n, \thetaMLE[S_0]} \approx \bigl\| \operatorname{Proj}_{\scrC_S} (\widetilde \cE) \bigr\|_2^2
\eean
for some linear space $\scrC_S$ with dimension $|S| - |S_0|$, where $\widetilde{\cE} = \bW_0^{-1/2} \cE$, $\cE = (\epsilon_i)_{i=1}^{n}$, $\epsilon_i = Y_i - b'(x_i^{\top} \theta_0)$ and $\operatorname{Proj}_{\scrC}$ is the orthogonal projection operator onto $\scrC$. More specifically, 
\bean
  L_{n, \thetaMLE} - L_{n, \thetaMLE[S_0]} \approx \bigl\| ( \bH_{S} - \bH_{S_0} ) \widetilde{\cE} \bigr\|_2^2,
\eean
where $\bH_S = \bW_0^{1/2} \bX_S \bF_{n, \thetaBest}^{-1} \bX_S^{\top} \bW_0^{1/2}$ is the orthogonal projection matrix onto the column space of $\bW_0^{1/2}\bX_S$.
If $\epsilon_i$ is a sub-Gaussian random variable, then one can establish 
\begin{align}
\begin{aligned} \label{eqn:sub_G_ratio}
    \bigl\| \operatorname{Proj}_{\scrC_S} ( \widetilde{\cE} ) \bigr\|_2^2 
    \lesssim |S \setminus S_0| \log p, \quad \forall S \supsetneq S_0
\end{aligned}
\end{align}
with high-probability; see \citet{narisetty2018skinny}, \citet{chae2019bayesian}, \citet{rossell2021approximate}, \citet{lee2021bayesian}, and \citet{tang2024empirical}. The proofs in these papers explicitly or implicitly rely on the concentration inequality of the quadratic form of sub-Gaussian variables, widely known as the Hanson--Wright inequality \citep{hanson1971bound, hsu2012tail}.  While there exists a Hanson--Wright type concentration inequality for sub-exponential variables \citep{Gotze2021subE}, this only leads to the conclusion $\widetilde{\cE}^{\top} ( \bH_{S} - \bH_{S_0} ) \widetilde{\cE} \lesssim (|S \setminus S_0| \log p)^2$, which is a substantially looser bound compared to \eqref{eqn:sub_G_ratio}.

The sub-Gaussian nature of $\epsilon_i$ is closely related to the sub-Gaussianity of the score $\dot L_{n, \thetaBest}$. When $Y_i$ is sub-exponential, the score vector $\dot L_{n, \thetaBest}$ is also sub-exponential. The crux of our proof lies in leveraging the near-sub-Gaussianity of the normalized score $\xi_{n, S} = \bF_{n, \thetaBest}^{-1/2} \dot{L}_{n, \thetaBest}$. More specifically, if $\xi_{n, S}$ is sub-exponential, there exists a (fixed) number $t_{n, S} > 0$ such that 
\begin{align*}
    \log \bbE \exp \{ u^{\top} \xi_{n, S} \} \lesssim \tfrac{1}{2} \| u  \|_{2}^{2}, \quad \text{for $\|u\|_2 \leq t_{n, S}$}. 
\end{align*}
Note that $t_{n, S} = \infty$ corresponds to the sub-Gaussian case. In Appendix \ref{sec:general_sub-exponential_tail_app}, we demonstrate that $t_{n, S}$ diverges to infinity as the sample size increases when $Y_i$ is sub-exponential, an important property emphasized in \citet{spokoiny2012parametric, spokoiny2023deviation}. Furthermore, \citet{barber2015high} have approximated $L_{n, \thetaMLE} - L_{n, \thetaMLE[S_0]}$ as
\bean
  L_{n, \thetaMLE} - L_{n, \thetaMLE[S_0]} \approx \bigl\| \operatorname{Proj}_{\scrC_S'} (\xi_{n, S}) \bigr\|_2^2
\eean
for some linear space $\scrC_S'$ with dimension $|S| - |S_0|$.
Based on these two facts, we prove that
\begin{align} \label{eqn:logLik_nosupset}
    L_{n, \thetaMLE} - L_{n, \thetaMLE[S_0]} \lesssim |S \setminus S_0| \log p, \quad \text{for all $S \in \scrS_{\Theta_n}$ with $S \supsetneq S_0$},   
\end{align}
which is the most challenging part in the proof of Theorem \ref{thm:no_superset_main}.

\subsection{No false negative} \label{sec:no_false_negative}

Here we present sufficient conditions under which the posterior distribution assigns nearly no mass to models with false negatives, i.e. $S$ with $S \nsupseteq S_0$. Combining this with the results in the previous sections leads to the strong model selection consistency, as stated in Theorem \ref{thm:selection_main}. We first briefly describe the proof strategy.

For $S \nsupseteq S_0$, according to our Laplace approximation, we only need to find a suitable upper bound for difference $L_{n, \thetaMLE} - L_{n, \thetaMLE[S_0]}$. Indeed, for all $S \in \scrS_{\Theta_n}$ with $S \nsupseteq S_0$, we can obtain 
\begin{align} \label{eqn:logLik_nofalse}
  L_{n, \thetaMLE} - L_{n, \thetaMLE[S_0]} & \leq
  -\dfrac{n}{4} \phi_2^2(\widetilde{s}_n; \bW_0) \bigl\| \widetilde{\theta}_S^{\texttt{MLE}} - \widetilde{\theta}_{S_{\texttt{+}}}^{\texttt{MLE}} \bigr\|_2^{2} + C|S \cap S_0^{\rm c}|\log p,
\end{align}
where $S_{\texttt{+}} = S \cup S_0$, $C = C(C_{\rm dev}) > 0$ and $\widetilde{\theta}_S^{\texttt{MLE}} \in \bbR^p$ is the $p$-vector version of $\thetaMLE$; see \eqref{eqn:no_false_eq1}.
Furthermore, it is not difficult to see that
\begin{align*}
    \bigl\| \widetilde{\theta}_{S}^{\texttt{MLE}}  -  \widetilde{\theta}_{S_{\texttt{+}}}^{\texttt{MLE}}  \bigr\|_{2}
    \geq 
    |S_0 \cap S^{\rm c}| \Bigl\{ \min_{j \in S_0} |\theta_{0, j}| - 
    \bigl\| \widehat{\theta}_{S_{\texttt{+}}}^{\texttt{MLE}} - \theta_{S_{\texttt{+}}}^{\ast} \bigr\|_{\infty}
    \Bigr\}.
\end{align*}
Therefore, the model selection problem boils down to the problem of obtaining a sharp convergence rate of $\thetaMLE$ with respect to $\ell_\infty$-norm.

%In the proof of Theorem \ref{thm:selection_main}, an important step is proving $\ell_{\infty}$ error of $\thetaMLE$. Specifically, this estimation error is devised to demonstrate that our posterior effectively rules out the possibility of a false negative case, say $S \nsupseteq S_0$. 
%By leveraging $\ell_{\infty}$ error of $\thetaMLE$, the selection consistency can be established under conditions that are weaker than those typically found in the existing Bayesian variable selection literature. A detailed discussion will be provided after we present precise statements of Lemma \ref{thm:L_infty_estimation}.

Let 
\begin{align} 
\begin{aligned} \label{def:A6_notations}
    \scrS_{\rm fp} &= \left\{ S \cup S_0 : S \nsupseteq S_0,  S \in \scrS_{\Theta_n} \right\},
    \\
    \designRegular[\scrS_{\rm fp}] &= \max_{S \in \scrS_{\rm fp}} \designRegular[S],
    \\    
    \nu_n &= \big(1 + 2/(e\log 2) \big) \Big( 1 + \frac{\sigma_{\max}^{2}}{\log 2} \Big).    
\end{aligned}
\end{align}
We use assumption {\bf (A6)} below to obtain $\ell_\infty$-convergence of $\thetaMLE$.

\bed
\item[(\textbf{A6})]
$\left\| \bX \right\|_{\max}^2 \log p = o(n)$, $\max_{j \in [p]} \left\| \bx_j \right\|_{2} = O(n^{1/2})$ and there exists $\kappa_n > 1$ such that
\begin{align} \label{A6:a}
\max_{S \in \scrS_{\rm fp}} \bigl\| \bF_{n, \thetaBest}^{-1} \bigr\|_{\infty} \leq \kappa_{n} n^{-1}    
\end{align}
and
\begin{align*}
\begin{aligned}
    &\left[ \dfrac{\left( s_0^{2} \log p \right)^{1/2} \designRegular[\scrS_{\rm fp}]}{\phi_2\left( \widetilde{s}_n ; \bW_0 \right) \nu_n \kappa_{n}} \right] \wedge 
    \left[ \dfrac{\sigma_{\max}^2 }{ \phi_{2}^{4}(\widetilde{s}_n; \bW_0) \nu_n \kappa_n } 
    \left( \dfrac{s_0^2 \log p}{n} \right)^{1/2} \right] 
    = o(1).
\end{aligned}
\end{align*}
\eed

For the case of a logistic regression model, we show that $\nu_n$ in {\bf(A5)} can be replaced by the constant $(1 + 2(e\log 2)^{-1})(4\sqrt{\log 2})^{-1}$; see \eqref{eqn:logit_Orlicz_bound} in Lemma \ref{lemma:orlic_bound}. Assumption \eqref{A6:a} appears in the literature on model selection and $\ell_\infty$-norm consistency in GLMs with penalized likelihood approaches \citep{wainwright2009sharp, fan2011nonconcave, loh2017support}.

In Lemma \ref{lemma:matrix_infty_norm}, we prove that if $x_{ij}$'s are \iid\ standard Gaussian variables and $s_0^2 \log p = o(n)$, then $\max_{S \in \scrS_{\rm fp}} \| ( \bX_S^{\top} \bX_S )^{-1} \|_{\infty} = O(n^{-1})$ with high probability. This implies that
\begin{align} \label{eqn:F_infinity_norm}
    \max_{S \in \scrS_{\rm fp}} \left\| \bF_{n, \thetaBest}^{-1} \right\|_{\infty} 
    \leq \sigma_{\min}^{-2} \max_{S \in \scrS_{\rm fp}} \left\| \left( \bX_S^{\top} \bX_S \right)^{-1} \right\|_{\infty} 
    \lesssim \sigma_{\min}^{-2} n^{-1},
\end{align}
where $\sigma_{\min}^{2} = \min_{i \in [n]} b''(x_{i}^{\top} \theta_0)$.
In this case, $\kappa_{n}$ can be chosen as a quantity of order $\sigma_{\min}^{-2}$.

\begin{theorem}[$\ell_{\infty}$-estimation error] \label{thm:L_infty_estimation_main}
    Suppose that \textbf{(A1)}-\textbf{(A6)} hold.
    Then, there exists a constant $K > 0$ such that
    \begin{align*}
        \max_{S \in \scrS_{\rm fp}} \bigl\| \thetaMLE - \thetaBest  \bigr\|_{\infty} 
        \leq  
        K \nu_n \kappa_{n} \sqrt{\dfrac{\log p}{n}}
    \end{align*}    
    with $\bbP_0^{(n)}$-probability at least $1 - 3p^{-1}$.
\end{theorem}
\begin{proof}
See the proof of Theorem \ref{thm:L_infty_estimation}; Theorem \ref{thm:L_infty_estimation_main} is a special case of Theorem \ref{thm:L_infty_estimation}.
\end{proof}

%If $\kappa_n \vee \nu_n = O(1)$, Theorem \ref{thm:L_infty_estimation_main} guarantees that
%$$
%\max_{S \in \scrS_{\rm fp}} \left\| \thetaMLE - \thetaBest  \right\|_{\infty}
%= O_{P_0}\left( \sqrt{\dfrac{\log p}{n}} \right).
%$$
Now, we are ready to prove
\begin{align} \label{eqn:no_false_negative}
	\bbE \, \Pi_{\alpha}^n (\theta: S_{\theta} \nsupseteq  S_0) = o(1).
\end{align}
Since 
$$
\bbE \, \Pi_{\alpha}^n( \theta: S_{\theta} \neq S_0)
    = 
    \bbE \, \Pi_{\alpha}^n(\theta: S_{\theta} \supsetneq S_0)
    + 
    \bbE \, \Pi_{\alpha}^n (\theta: S_{\theta} \nsupseteq S_0),    
$$
Theorem \ref{thm:no_superset_main} and \eqref{eqn:no_false_negative} gives the strong model selection consistency, i.e., 
$$
\bbE \, \Pi_{\alpha}^n (\theta: S_{\theta} = S_0 ) \rightarrow 1.
$$
For \eqref{eqn:no_false_negative}, we need the following assumption, widely known as the beta-min condition.

\bed
\item[(\textbf{A7})]
There exists a constant $K_{\rm min} > 0$ such that
\begin{align} \label{A7:a} 
    \vartheta_{n, p} = \min_{j \in S_0} | \theta_{0, j}  | \geq K_{\rm min} \Bigg[
    \left( \nu_n \kappa_{n} \sqrt{\frac{\log p}{ n }} \right) \wedge \left( \phi_2^{-1}\left( \widetilde{s}_n ; \bW_0 \right) \sqrt{\frac{s_0 \log p}{ n }} \right) \Bigg].
\end{align} 
and, furthermore, 
\begin{align} \label{A7:b} 
\begin{aligned}
    \kappa_{n} \nu_{n} \phi_2\left( \widetilde{s}_n ; \bW_0 \right) \gtrsim 1.
\end{aligned}
\end{align}
\eed

\begin{theorem}[Selection consistency] \label{thm:selection_main}
    Suppose that \textbf{(A1)}-\textbf{(A7)} hold, and $K_{\rm min}$ in \eqref{A7:a} is a large enough constant. Then,
    \begin{align*} %\label{eqn:selection_consistency}
    \bbE \, \Pi_{\alpha}^n(\theta: S_\theta = S_0)
    \geq 
    1 - \bigl\{ 4(s_0 \log p)^{-1} + 25p^{-1} + 4p^{-s_0} + 3p^{-\delta_{1}} \bigr\},
    \end{align*}
    where $\delta_{1}$ is the constant specified in \eqref{A5:a}.
\end{theorem}
\begin{proof}
See the proof of Theorem \ref{thm:selection_consistency}; Theorem \ref{thm:selection_main} is a special case of Theorem \ref{thm:selection_consistency}.
\end{proof}

It is shown in \citet[][Theorem~2]{wainwright2009information} that if $\min_{j \in S_0} |\theta_{0, j}| \ll \{ n^{-1} \log(p/s_0)\}^{1/2}$
% \begin{align*}
%     \min_{j \in S_0} |\theta_{0, j}| \ll \{ n^{-1} \log(p/s_0)\}^{1/2} %\sqrt{\dfrac{ \log(p/s_0) }{n } }
% \end{align*}
in a linear regression model, then $\theta_{0, j}$ cannot be consistently detected. In this sense, the amount $(n^{-1} \log p)^{1/2}$ can be understood as the minimum magnitude of signals to be consistently selected. \citet{loh2017support} obtained the selection consistency with the beta-min condition \eqref{A7:a} and, although not explicitly stated, their Corollary~3 assumes $\kappa_n$ and $\nu_{n}$ are both $O(1)$. Therefore, \eqref{A7:a} corresponds to the rate-optimal beta-min condition under the setting described in \citet{loh2017support}.    

In Bayesian linear regression, \citet{castillo2015sparse} obtained the model selection consistency with the beta-min condition $\min_{j \in S_0} |\theta_{0, j}| \gtrsim (n^{-1} \log p)^{1/2}$ under the mutual coherence condition. The mutual coherence condition is rather strong; it is relaxed to conditions on sparse singular values in, e.g., \citet{martin2017empirical}.
Proofs in these papers rely on the closed-form marginal likelihood of Gaussian models. \citet{chae2019bayesian} extended the result of \citet{martin2017empirical} to a non-Gaussian linear model, but their proof relies on the sub-Gaussianity of the score function, limiting their applicability in Poisson and other GLMs. There are other articles studying the model selection consistency in GLMs, but they require a substantially stronger beta-min condition $\min_{j \in S_0} |\theta_{0, j}| \gtrsim (n^{-1} s_0 \log p)^{1/2}$; see \citet{barber2015high}, \citet{narisetty2018skinny}, \citet{lee2021bayesian}, \citet{cao2022bayesian} and \citet{tang2024empirical}. In light of this, \eqref{A7:a} significantly improves upon the existing results.

In our theoretical framework, establishing model selection consistency relies on bounding likelihood ratios. Specifically, if $\sigma_{\min}^{-2} \vee \sigma_{\max}^{2} = O(1)$, the following inequality holds, and plays a crucial role in proving Theorems \ref{thm:no_superset_main} and \ref{thm:selection_main}: for all $S \in \scrS_{\Theta_n}$,
\begin{align} \label{eqn:MLE_bound}
\begin{aligned}
    L_{n, \thetaMLE} - L_{n, \thetaMLE[S_0]} 
    \leq C_1 \left| S \cap S_0^{\rm c} \right| \log p - 
    C_2 \left| S^{\rm c} \cap S_0 \right| n \min_{j \in S_0} \left| \theta_{0, j} \right|^{2}
\end{aligned}
\end{align}
for some constants $C_1, C_2 > 0$. This inequality combines the results of \eqref{eqn:logLik_nosupset} and \eqref{eqn:logLik_nofalse}, which represent significant contributions of the present paper.

Similar versions of \eqref{eqn:MLE_bound} have been established in the literature on GLMs. For example, \cite{barber2015high} demonstrated in Theorem 2.2 that  
\begin{align*}
    L_{n, \thetaMLE} - L_{n, \thetaMLE[S_0]} 
    &\leq C_3 \left| S \cap S_0^{\rm c} \right| \log p, &\forall S \in \scrS_{s_{\max}} \text{ with  } S \supseteq S_0, \\
    L_{n, \thetaMLE} - L_{n, \thetaMLE[S_0]} 
    &\leq
    -C_4 n \min_{j \in S_0} \left| \theta_{0, j} \right|^{2}, &\forall S \in \scrS_{s_{\max}} \text{ with  } S \nsupseteq S_0.
\end{align*}
for some constants $C_3, C_4 > 0$. These results necessitate the beta-min condition of order at least $\left| \theta_{0, j} \right| \asymp (s_{\max} \log p / n)^{1/2}$. Moreover, \citet{hou2024laplace} explicitly assume a stronger version of \eqref{eqn:MLE_bound} to ensure the selection consistency of their proposed estimator. To the best of our knowledge, \eqref{eqn:MLE_bound} represents the sharpest bound in the Bayesian GLM literature.

\section{Examples}
\label{sec:examples}

This section aims to summarize our main results in the context of two of the most common GLMs, namely, logistic and Poisson regressions; see Corollaries \ref{coro:model_selection_logit} and
\ref{coro:model_selection_poisson} for key summaries. Our theoretical analysis in previous sections was conditional on the design matrix but, in order to discuss the results that are expected for ``typical'' design matrices, here we consider the simple random matrix setup where each entry of the design matrix $\bX$ is an i.i.d.~standard normal random variable, i.e., $x_{ij} \overset{\iid}{\sim} \cN(0, 1)$.
Note that results in this section can be naturally extended to a more general setting where $X_{i} \overset{\iid}{\sim} \cN(0, \bSigma)$ and $\bSigma$ satisfies the following conditions:
\begin{align} \label{eqn:general_random_design}
    C_1 \leq \lambda_{\min}(\bSigma) \leq \lambda_{\max}(\bSigma) \leq C_2, \quad 
    \| \bSigma^{-1} \|_{\infty} \leq C_3
\end{align}
for some constants $C_1, C_2, C_3 > 0$.

With slight abuse of notation, let $\bbP$ and $\bbE$ be the joint probability measure and expectation corresponding to $(\bX, \bY)$, respectively. For readability, many of the results presented in this section will state that one thing or another happens with high probability when $n$ is sufficiently large.  For the precise non-asymptotic statements, see Appendices \ref{sec:misspecified_estimator_example_app} and \ref{sec:technical_lemmas_app}.

\subsection{Random design quantities} \label{sec:random_design_example}
The following corollary summarizes the asymptotic behavior of various quantities in the context of a random design.

\begin{corollary} \label{coro:example_design_matrix}
The following hold with $\bbP$-probability converging to $1$ as $n \rightarrow \infty$:
\begin{align} 
\begin{aligned} \label{eqn:example_design_matrix_1}
\| \bX \|_{\max} & \leq 2\sqrt{\log (np)} \\
\| \bX_{S_0} \|_{\infty} & \leq 2s_0\sqrt{\log (np)} \\
\max_{j \in [p]} \| \bX_{j} \|_{2} & \leq \sqrt{n} + 2\sqrt{\log p} \\ 
\max_{i \in [n]} | X_{i}^{\top} \theta_0 | & \leq 2\left\| \theta_0 \right\|_{2} \sqrt{\log n}. 
\end{aligned}
\end{align}
Furthermore, if $( s_0^{2} \log p ) \vee  (s_{0} \log p)^{3/2} = o(n)$, 
% \begin{align*}
%    ( s_0^{2} \log p ) \vee  (s_{0} \log p)^{3/2} = o(n), 
% \end{align*}
then the following hold with $\bbP$-probability converging to $1$ as $n \rightarrow \infty$:
\begin{align} \label{eqn:example_design_matrix_2}
\max_{S \in \scrS_{\widetilde{s}_n}}  \bigl\| ( \bX_{S}^{\top} \bX_{S} )^{-1} \bigr\|_{\infty} = O(n^{-1}) \quad 
\text{and} \quad \max_{S \in \overline{\scrS}_{\Theta_n}} \sup_{u_S \in \cU_S} \dfrac{1}{n} \sum_{i=1}^{n} \bigl| X_{i, S}^{\top} u_S \bigr|^3 = O(1).
\end{align}
\end{corollary}

A notable difference between linear regression and other kinds of GLMs is the variance term $b''$.  The specific effect of this variance term is that the posterior concentration properties depend on the magnitude $\| \theta_0 \|_{2}$ of the true coefficient vector.  To maintain lower bounds on the sparse singular value $\phi_{2}^{2}(s ; \bW_0)$, certain stochastic restrictions on the \textit{natural parameter} $X_{i}^{\top} \theta_0$ are crucial. For example, if $b''(X_{i}^{\top} \theta_0) > C$ for some constant $C > 0$ with positive probability for each $i \in [n]$, then for each $s \in [p]$, 
\begin{align} 
    \phi_{2}^{2}(s; \bW_0) & = \inf_{S \in \scrS_{s}} \lambda_{\min} \bigg( \sum_{i=1}^{n} b''(X_{i}^{\top} \theta_0) X_{i, S} X_{i, S}^{\top} \bigg) \notag \\
    & \geq 
    C \inf_{S \in \scrS_{s}} \lambda_{\min} \bigg( \sum_{i \in \cI_{C} } X_{i, S} X_{i, S}^{\top} \bigg), \label{eqn:design_example_section_eq1}
\end{align}
where $\cI_{C} = \{ i \in [n] : b''(X_{i}^{\top} \theta_0) > C \}$.
Since $b''(X_{i}^{\top} \theta_0)$ is bounded away from zero with positive probability, it follows that $|\cI_{C}| \geq c n$ for some $c \in (0, 1)$ with high probability.
Moreover, if $s \log p = o(n)$, it can be shown that
\begin{align*}
    \inf_{S \in \scrS_{s}} \lambda_{\min} \bigg( \sum_{i =1}^{n} X_{i, S} X_{i, S}^{\top} \bigg) \geq C' n
\end{align*}
for some constant $C' > 0$.
Therefore, combining these two results, the right-hand side in \eqref{eqn:design_example_section_eq1} is lower-bounded by a constant multiple of $n$, with high probability.

Note that the specific form of the restriction on $\| \theta_0 \|_{2}$ will depend on the choice of $b(\cdot)$. While Poisson regression models imposes no restriction on the signal size, boundedness of the signal size is crucial for the regularity of $\phi_{2}$ in logistic regression models; see Lemma \ref{lemma:Poisson_least_eigenvalue_V} and \ref{lemma:least_eigenvalue_logit} for precise statements.

As mentioned earlier, $\sigma_{\max}^{2}$ is closely related with the stochastic regularity of $\cE = (\epsilon_i)_{i \in [n]}$, where $\epsilon_i = Y_i - b'(X_i^\top \theta_0)$. 
Unlike in linear regression, where a homogeneous variance $\sigma^{2}$ is often assumed, the Orlicz norm of each $\epsilon_i$ in the GLM context depends on the natural parameter. In particular, for the Poisson model, $\sigma_{\max}^{2}$ is utilized to bound the Orlicz norm of $\epsilon_{i}$ uniformly over all observations. 
To control this value, it is necessary to obtain the maximal bound of $|X_{i}^{\top} \theta_0 |$ as in \eqref{eqn:example_design_matrix_1}. Additionally, $\sigma_{\min}^{-2}$ can be utilized to bound $\kappa_n$. Consequently, a very small $\sigma_{\min}^{2}$ may result in looser bounds that negatively affect the $\ell_{\infty}$-estimation error and/or beta-min condition.

\subsection{Logistic regression} \label{sec:logit_example}

In this subsection, we focus on the logistic regression model, where $b(\cdot) = \log\{1 + \exp(\cdot)\}$. The following corollaries provide theoretical verifications of the assumed conditions for Theorem \ref{thm:selection_main} under the random design setup. 

\begin{corollary} \label{coro:example_logit_1}
Suppose that $s_0 \log p  = o(n)$.
Then 
\begin{align} \label{claim:logit_0}
\begin{aligned}
    \phi_{1}^{-2}\left( \widetilde{s}_{n}; \bW_0 \right) \vee \phi_{2}^{-2}\left( \widetilde{s}_{n}; \bW_0 \right) & = O\bigl( e^{2\| \theta_0 \|_{2}} \bigr) \\ 
    \sigma_{\min}^{-2} & = O\big( e^{2 \| \theta_0 \|_{2} \sqrt{\log n}} \bigr) \\
    \max_{S \in \scrS_{\Theta_n}} \rho_{\max, S} & = O(n),
\end{aligned}
\end{align}
with $\bbP$-probability converging to $1$ as $n \rightarrow \infty$. 
Also, assume that
\begin{align} \label{eqn:example_logit_1_assume}
(s_{\max} \log p)^{3/2} = o(n), \quad \text{and} \quad \left\| \theta_0 \right\|_{2} = O(1).
\end{align}
Then, with $\bbP$-probability converging to $1$ as $n \rightarrow \infty$, \eqref{claim:logit_1} holds uniformly for all $S \in \scrS_{s_{\max}}$:
\begin{align}
\begin{aligned} \label{claim:logit_1}
    \left\| \bF_{n, \thetaBest}^{-1/2} \bF_{n, \thetaMLE} \bF_{n, \thetaBest}^{-1/2} \right\|_{2}
    \vee 
    \left\| \bF_{n, \thetaMLE}^{-1/2} \bF_{n, \thetaBest} \bF_{n, \thetaMLE}^{-1/2} \right\|_{2} &= O(1) \\
    \left\| \bF_{n, \thetaBest}^{1/2} \left( \thetaMLE - \thetaBest \right) \right\|_{2}
    &= O(|S| \log p).
\end{aligned}
\end{align}
Furthermore, for any $k > 0$, \eqref{claim:logit_2} holds:
\begin{align} \label{claim:logit_2}
\begin{aligned}
    &\phi_{1}^{-2}\left( \widetilde{s}_{n}; \bW_0 \right) \vee \phi_{2}^{-2}\left( \widetilde{s}_{n}; \bW_0 \right) = O(1), \quad  
    \sigma_{\min}^{-2} = O(n^{k}), \\
    &\max_{S \in \scrS_{\rm fp}} \bigl\| \bF_{n, \thetaBest}^{-1} \bigl\|_{\infty} = O \bigl( n^{-1 + k} \bigr), \quad
    \nu_n \leq \dfrac{1}{4\sqrt{\log 2}} \bigg( 1 + \dfrac{2}{e\log 2} \bigg),
\end{aligned}    
\end{align}
with $\bbP$-probability converging to $1$ as $n \rightarrow \infty$.
\end{corollary}

For $\eta \in \bbR$, note that $b''(\eta) = e^{\eta} / (1 + e^{\eta})^{2} \gtrsim e^{-|\eta|}$. As discussed in Section \ref{sec:random_design_example}, the boundedness of $\| \theta_0 \|_{2}$ is imposed to ensure that $\phi_{2}$ is bounded away from zero. Furthermore, this boundedness facilitates the control of $\sigma_{\min}^{2}$ while the maximum variance is automatically bounded, regardless of the signal size, with $\sigma_{\max}^{2} \leq b''(0) = 1/4$. This ensures the boundedness of $\nu_n$ in the context of the logistic model (see Lemma \ref{lemma:orlic_bound} and corresponding proofs).

\begin{corollary} \label{coro:model_selection_logit}
Suppose that the prior precision parameter $\lambda$ satisfies \eqref{A2:b} for some constants $A_5, A_6 > 0$ and $A_7 \geq 0$. Also, assume that
\begin{align*}
    &\| \theta_0 \|_{2} = O(1), \quad 
    \alpha \in (0, 1), \\
    &A_4 > A_6 p^{-A_7}, \quad 
    A_4 + A_7/2 > \alpha (16 e^{3/2}) + \log_{p}(s_0) + \delta_1
\end{align*}
for some small constant $\delta_1$, where $A_4$ is specified in \eqref{def:complexity_prior_main}.
Assume further that there exist constants $\beta, K_{\min} > 0$ such that
\begin{align} \label{assume:model_selection_logit}
\begin{aligned}
    &(s_0^{3} \log p) \vee (s_0^{2} \log p)^{1/(1-\beta)} \vee (s_0 \log p)^{2} \vee (s_{\max} \log p)^{3/2} = o(n) \\
    &\vartheta_{n, p} \geq K_{\min} \Bigg( \sqrt{\dfrac{\log p}{n^{1 - \beta}}} \wedge \sqrt{\dfrac{s_0 \log p}{n}} \Bigg).
\end{aligned}
\end{align}
If $K_{\min}$ is large enough, then $\bbE \, \Pi_{\alpha}^n(\theta: S_\theta = S_0) \rightarrow 1$.
% \begin{align*}
%     \bbE \, \Pi_{\alpha}^n(\theta: S_\theta = S_0) \rightarrow 1.
% \end{align*}
\end{corollary}

Note that the beta-min condition in the above corollary is arbitrarily close to the ideal bound ``$(n^{-1} \log p)^{1/2}$'' motivated by \citet[][Theorem~2]{wainwright2009information}.  This is a much weaker requirement, hence a much stronger model selection consistency result, compared to those in the existing Bayesian GLM literature \citep[e.g.,][]{tang2024empirical}.

\subsection{Poisson regression}

In this subsection, we focus on the Poisson regression model, where $b(\cdot) = \exp(\cdot)$. The following corollaries provide theoretical verifications of the assumed conditions for Theorem \ref{thm:selection_main} under the random design setup. 

For $\eta \in \bbR$, note that $\bbP\{ b''(X_{i}^{\top} \theta_0) \geq 1\} \geq 1/2$ without any restrictions of $\theta_0 \in \bbR^{p}$. 
In this model, the boundedness of $\| \theta_0 \|_{2}$ is imposed to ensure that $\sigma_{\min}^{-2} \vee \sigma_{\max}^{2}$ is not too large. 
Unlike the logistic model, for the Poisson model with $b''(\cdot) = \exp(\cdot)$, the variance can fluctuate severely depending on the size of the natural parameter. 
Therefore, to control the magnitude of $\max_{i \in [n]} |X_{i}^{\top} \theta_0|$, a certain restriction for $\| \theta_0 \|_{2}$ is imposed in Corollary \ref{coro:example_poisson_1}.

\begin{corollary} \label{coro:example_poisson_1}
Suppose that $s_0 \log p  = o(n)$.
Then, 
\begin{align} \label{eqn:example_poisson_1_1}
\begin{aligned}
\phi_{1}^{-2}\left( \widetilde{s}_{n}; \bW_0 \right) \vee \phi_{2}^{-2}\left( \widetilde{s}_{n}; \bW_0 \right) & = O\left( 1 \right) \\
\sigma_{\min}^{-2} \vee \sigma_{\max}^{2} & = O\bigl( e^{2 \| \theta_0 \|_{2} \sqrt{\log n}} \bigr)
\end{aligned}
\end{align}
with $\bbP$-probability converging to $1$ as $n \rightarrow \infty$. 
Also, assume that
\begin{align} \label{eqn:example_poisson_1_assume}
(s_0 \log p)^{2} \vee ( s_{\max} \log p )^{2} = o(n) \quad \text{and} \quad \| \theta_0 \|_{2} = O(1).
\end{align}
Then, with $\bbP$-probability converging to $1$ as $n \rightarrow \infty$, \eqref{claim:Poisson_1} holds uniformly for all $S \in \scrS_{s_{\max}}$:
\begin{align}
\begin{aligned} \label{claim:Poisson_1}
    \left\| \bF_{n, \thetaBest}^{-1/2} \bF_{n, \thetaMLE} \bF_{n, \thetaBest}^{-1/2} \right\|_{2}
    \vee 
    \left\| \bF_{n, \thetaMLE}^{-1/2} \bF_{n, \thetaBest} \bF_{n, \thetaMLE}^{-1/2} \right\|_{2} &= O(1) \\
    \left\| \bF_{n, \thetaBest}^{1/2} \left( \thetaMLE - \thetaBest \right) \right\|_{2}
    &= O(|S| \log p),
\end{aligned}
\end{align}
Furthermore, for any $k > 0$, \eqref{claim:Poisson_2} holds with $\bbP$-probability converging to $1$ as $n \rightarrow \infty$:
\begin{align} \label{claim:Poisson_2}
\begin{aligned}
    &\sigma_{\min}^{-2} \vee \sigma_{\max}^{2} = O(n^{k}), \quad 
    \max_{S \in \scrS_{\Theta_n}} \rho_{\max, S} = O(n), \\
    &\max_{S \in \scrS_{\rm fp}} \bigl\| \bF_{n, \thetaBest}^{-1} \bigl\|_{\infty} = O \bigl( n^{-1 + k} \bigr), \quad
    \nu_n = O(n^{k}).
\end{aligned}    
\end{align}
\end{corollary}

\begin{corollary} \label{coro:model_selection_poisson}
Suppose that the prior precision parameter $\lambda$ satisfies \eqref{A2:b} for some constants $A_5, A_6 > 0$ and $A_7 \geq 0$. Also, assume that
\begin{align*}
    &\| \theta_0 \|_{2} = O(1), \quad 
    \alpha \in (0, 1), \\
    &A_4 > A_6 p^{-A_7}, \quad 
    A_4 + A_7/2 > \alpha (16 e^{1/2}) + \log_{p}(s_0) + \delta_1
\end{align*}
for some small constant $\delta_1$, where $A_4$ is specified in \eqref{def:complexity_prior_main}.
Assume further that there exist constants $\beta, K_{\min} > 0$ such that
\begin{align} \label{assume:model_selection_Poisson}
\begin{aligned}
    &(s_0^{3} \log p)^{1/(1-\beta)} \vee (s_0 \log p)^{2} \vee (s_{\max} \log p)^{2} = o(n) \\
    &\vartheta_{n, p} \geq K_{\min} \Bigg( \sqrt{\dfrac{\log p}{n^{1 - \beta}}} \wedge \sqrt{\dfrac{s_0 \log p}{n}} \Bigg).
\end{aligned}
\end{align}
If $K_{\min}$ is large enough, then $\bbE \, \Pi_{\alpha}^n(\theta: S_\theta = S_0) \rightarrow 1$.
% \begin{align*}
%     \bbE \, \Pi_{\alpha}^n(\theta: S_\theta = S_0) \rightarrow 1.
% \end{align*}
\end{corollary}

In view of $s_0^{3} \log p = o(n^{1-\beta})$ and $\vartheta_{n, p} \gtrsim \sqrt{\log p / n^{(1 - \beta)}}$, the conditions in Corollaries \ref{coro:model_selection_logit} and \ref{coro:model_selection_poisson} are slightly more restrictive than those of Theorem \ref{thm:selection_main}. 
This result arises from a technical reason: specifically, the need to consider the maximum value of $|X_{i}^{\top} \theta_0|$. Thus, the undesirable $\beta$ can be eliminated by considering some random design setup where $|X_{i}^{\top} \theta_0| = O(1)$ with high probability. 
Nonetheless, since $\beta$ in \eqref{assume:model_selection_Poisson} can be chosen arbitrary small, Corollaries \ref{coro:model_selection_logit} and \ref{coro:model_selection_poisson} ``almost'' match the dimension dependency $s_0^{3} \log p = o(n)$ argued in Section \ref{sec:model_selection_consistency}.

\section{Computational strategies in Bayesian model selection} \label{sec:computation_BVS}

\subsection{Algorithms} \label{sec:computation_alg}
 
In Section \ref{sec:model_selection_consistency}, we show that our posterior distribution achieves model selection consistency. Computing this posterior distribution is challenging, however, due to the discrete nature of $\pi_{\alpha}^{n}(S)$. This has led to the development of various computational strategies.  This includes shotgun stochastic search (SSS) and its variants \citep{hans2007shotgun, shin2018scalable, cao2022bayesian}, Metropolis--Hastings Markov chain Monte Carlo (MH MCMC) \citep{yang2016computational, martin2017empirical, tang2024empirical}, and (approximate) Gibbs sampling \citep{narisetty2018skinny, hou2024laplace}.
We focus our discussion here on the algorithmic details of MH MCMC.

Let $q(S' \mid S)$ denote a proposal distribution, defined as:
\begin{align*}
    q( S' \mid S) = | \mathscr{N}(S) |^{-1} \, \mathds{1}_{S' \in \mathscr{N}(S)},
\end{align*}
where $\mathscr{N}(S)$ represents the neighborhoods of the model $S$:
\begin{align*}
    \mathscr{N}(S) = \big( 
    \mathscr{N}_{\rm add}(S) \cup
    \mathscr{N}_{\rm del}(S) \cup
    \mathscr{N}_{\rm swap}(S)
    \big)    
    \cap \scrS_{s_{\max}}.
\end{align*}
The components of $\mathscr{N}(S)$ are given by:
\begin{align*}
    \mathscr{N}_{\rm add}(S) &= \big\{ S \cup \{ j \} : j \in [p] \setminus S \big\}, \quad 
    \mathscr{N}_{\rm del}(S) = \big\{ S \setminus \{ j \} : j \in S \big\}, \\
    \mathscr{N}_{\rm swap}(S) &= \big\{ S \setminus \{ k \} \cup \{ j \} : j \in [p] \setminus S, k \in S \big\}.
\end{align*}
For a current model $S \in \scrS_{s_{\max}}$, a single iteration of the MH algorithm proceeds as follows:
\begin{itemize}
    \item[1.] Sample $S' \sim q( \cdot \mid S)$.
    \item[2.] Move to the next model $S'$ with probability
    \begin{align*}
        1 \wedge \dfrac{ \widehat{\pi}_{\alpha}^{n}(S') \, q(S \mid S') }{ \widehat{\pi}_{\alpha}^{n}(S) \, q(S' \mid S) },
    \end{align*}
    where $\widehat{\pi}_{\alpha}^{n}(\cdot) = \pi_{n}(\cdot) \, \widehat{\cM}_{\alpha}^{n}(\cdot)$.
    Otherwise, stay at the current model $S$.
\end{itemize}

A practical advantage to---and one of the original motivating factors behind---the empirical prior developments in the Gaussian linear regression problem is that the marginal posterior for $S$ is available in closed form.  For GLMs, however, $\cM_{\alpha}^{n}(\cdot)$ is not available in closed form, so it is common to replace it in the above algorithm with the approximation $\widehat{\cM}_{\alpha}^{n}(\cdot)$. At each iteration, computing $\widehat{\cM}_{\alpha}^{n}(S)$ requires evaluating $\thetaMLE$, which unfortunately entails considerable computational costs and, in turn, may limit the method's viability in large-scale data analysis problems. To address this, \citet{rossell2021approximate} proposed a computationally efficient inference technique called the approximate Laplace approximation, which employs a single step Newton-Raphson update under a suitable initial parameter. More recently, \citet{hou2024laplace} introduced a similar second-order refinement technique and an efficient Gibbs sampling algorithm. 
Their approach achieves polynomial complexity in both $n$ and $p$, making it scalable to large-scale problems.

\begin{comment}
Notably, the approach in \citet{hou2024laplace} remains theoretically attractive provided that the MLEs 
$\thetaMLE$ achieve suitable rates for certain likelihood ratios. We briefly outline their approach.
Suppose we have an initial estimator $\widehat{\theta}^{\rm init} \in \bbR^{p}$ (e.g., LASSO) that is computationally convenient and achieves the optimal $\ell_2$-norm rate, say $(s_0 \log p/n)^{1/2}$. Then its second-order refinements are given by
\begin{align*}
    \widehat{\theta}_{S}^{\rm ref} = \widehat{\theta}_{S}^{\rm init} + \bF_{n, \widehat{\theta}_{S}^{\rm init}} \dot{L}_{n, \widehat{\theta}_{S}^{\rm init}}, \quad  S \subseteq [p].
\end{align*}
This construction leads to
\begin{align*}
    L_{n, \widehat{\theta}_{S}^{\rm ref}} - L_{n, \widehat{\theta}_{S_0}^{\rm ref}} 
    \leq 
    L_{n, \thetaMLE} - L_{n, \thetaMLE[S_0]} + 
    \operatorname{Rem}_n(\widehat{\theta}_{S_0}^{\rm ref}) +
    \operatorname{Rem}_n(\thetaMLE[S_0]),
\end{align*}
where $\operatorname{Rem}_n(\cdot)$ denotes a remainder term after a quadratic approximation. 
Under $s_0^{4} \log p = o(n)$ with additional regularity conditions, these remainders can be suitably bounded, so they do not affect the rates of the likelihood ratios $L_{n, \thetaMLE} - L_{n, \thetaMLE[S_0]}$.
See Theorem 6 in \citet{hou2024laplace} and the corresponding conditions therein.    
\end{comment}

\subsection{Hyperparameter choice: some intuition and theory} \label{sec:theory_hyperparams}

An important practical consideration is the choice of hyperparameters, namely $A_1$--$A_7$ and $\alpha$. For simplicity, assume 
\begin{align*}
    \lambda = A_6 p^{-A_7}, \quad 
    \pi_{n}(S) \propto p^{-A_4} \textstyle\binom{p}{|S|}^{-1}, \quad 
    \forall S \in \scrS_{s_{\max}}.
\end{align*}
We start with some intuition based on previous experience using the proposed Bayesian method for model selection in simulation studies, in the context of GLMs and beyond.  Based on that experience, the model selection performance is largely insensitive to the choices of $\lambda$ and $\alpha$, especially the choice of $\alpha$.  There is some natural appeal to choosing $\alpha$ close to 1, so that it more closely resembles a genuine Bayesian posterior distribution, and our experience suggests that taking, say, $\alpha = 0.99$ works well.  Furthermore, taking $\lambda$ to be a small constant or decreasing not too rapidly generally worked well.  But the choice of how severely the prior should penalize model complexity, as quantified by $A_4$ in the expression above, plays a much more impactful role in the method's overall performance.  Previous experience suggests that, if $A_4$ is too large, so that the penalty is too severe, then the model selection procedure will tend to miss important variables.  So, previous papers have recommended choosing $A_4$ to be rather small, e.g., $A_4 = 0.05$.  What this intuitive analysis fails to offer, however, is an understanding of the interplay between these choices and the regularity conditions leading to the strong model selection consistency results presented above.  The more detailed analysis that follows is intended to help fill this technical gap.

To keep the analysis relatively simple, we combine  \eqref{def:prior_S_penalty_main}, \eqref{A2:b}, and the previous display, so that $A_1 = A_2 = 1$, $A_3 = A_4$ and $A_5 = A_7 - \log_{p}(A_6)$. Then, the sufficient condition for model selection consistency (\eqref{A2:c} and \eqref{A5:a}) can be summerized as
\begin{align} \label{eqn:hyperparams_requirements}
    A_{4} > A_6 p^{-A_7}, \quad 
    A_{4} + A_7/2 > \alpha (16 C_{\rm dev}) + \log_{p}(s_0) + \delta_1,
\end{align}
where $\delta_1$, as specified in \eqref{A5:a}, can be chosen as a sufficiently small constant. The two parts of \eqref{eqn:hyperparams_requirements} are assumed to ensure that Theorems \ref{thm:effective_dim_main} and \ref{thm:no_superset_main} hold, respectively. This is the setting we adopt in the subsequent analysis. 

Since $s_0 \leq p$, we have $\log_{p}(s_0) \leq 1$. If $p \asymp \exp(n^{c_1})$ and $s_0 \asymp n^{c_2}$ with $c_1 \in (0, 1)$ and $c_2 \in [0, 1/3)$, one can see that $\log_{p}(s_0) = o(1)$. Additionally, $16 C_{\rm dev}$ in \eqref{eqn:hyperparams_requirements} can be refined by a constant $C > 0$ satisfying
\begin{align*}
    L_{n, \thetaMLE} - L_{n, \thetaMLE[S_0]} \leq C \left| S \setminus S_0 \right| \log p, \quad \forall S \in \scrS_{\Theta_n} \text{ with } S \supsetneq S_0.
\end{align*}

Consider the \textit{constant} $\lambda$ regime: $A_7 = 0$. 
As discussed in Section \ref{sec:no_superset}, $A_4$ serves as a regularization parameter that suppresses the overfitting effect arising from $S \supsetneq S_0$.
When $A_4$ is large enough satisfying \eqref{eqn:hyperparams_requirements}, the posterior can effectively exclude undesirable supersets while still retaining dimension-reduction capabilities.
Given that the empirical prior is highly informative, the constant $\lambda$ regime is especially noteworthy.

Next, consider the \textit{polynomially decreasing} $\lambda$ regime: $A_7 > 0$. 
In this regime, $A_4$ can be set to a relatively small constant because the ``burden'' of penalizing large models ($S \supsetneq S_0$) is distributed between $A_4$ and $A_7$. Recall that a regularization effect of $A_7$ stems from 
$(1 + \alpha \lambda^{-1})^{-|S|/2}$ in \eqref{def:marginal_likelihood_approx}, whose dominant order scales with $\lambda^{|S|/2} \asymp p^{-A_7 |S|/2}$.
Thus, when $A_7$ is sufficiently large, a smaller $A_4$ is sufficient to maintain the desired posterior properties.

Furthermore, it is worth introducing an interesting effect of the fractional likelihood. 
When $\alpha \in (0, 1)$, the likelihood is effectively down-weighted, reducing model complexity. 
By taking $\alpha$ such that $\alpha (16 C_{\rm dev})$ is small enough, it becomes possible to use a small
$A_4$ even under the \textit{constant} $\lambda$ regime. Consequently, by balancing $(A_4, A_7, \alpha)$, one can flexibly control model complexity while maintaining theoretical validity.

These two regimes discussed above have been well-established in the literature. 
In high-dimensional Gaussian linear regression, the complexity prior in \citet{castillo2015sparse} and \citet{chae2019bayesian} demonstrated the necessity of a suitably large $A_4$ to avoid false positives.
Meanwhile, diffusive priors, corresponding to $A_7 > 0$, achieve comparable outcomes
\citep{narisetty2014bayesian}.
In the context of GLMs, \citet{narisetty2018skinny} and \citet{lee2021bayesian} have adopted the second regime with $A_4 \approx 0$ and sufficiently large $A_7$. Conversely, \citet{tang2024empirical} and \citet{hou2024laplace} considered large enough $A_4$ to establish a version of Theorem \ref{thm:no_superset_main}. 

Despite the heavy technical machinery used in the above analysis, we still {\em cannot} definitively answer the question of how to optimally set the critical hyperparameters $(\alpha,\lambda,A_4)$.  The issue is a disconnect---common in the literature on high-dimensional inference---between what works in theory and what works in practice.  The major obstacle here is that the theoretical analysis, e.g., \eqref{eqn:hyperparams_requirements}, effectively requires $A_4$ to be set rather large to achieve model selection consistency, but choosing $A_4$ to be large in practice tends to over-penalize the model size, resulting in poor model selection performance.  In the Gaussian linear regression model, \citet{martin2017empirical} recommended the following default choices of hyperparameters:
\[ \lambda = 10^{-3}, \quad \alpha = 0.999, \quad A_4 = 0.05. \]
This recommendation was based on a non-exhaustive search over different hyperparameter choices in several settings, in particular $(n, p, s_0) \in \{ (100, 500, 5), (200, 1000, 5) \}$.  That is, the recommendation in the above display corresponds to what those authors determined to offer the best overall model selection performance in their simulations.  Similar settings were used in other applications, e.g., in the logistic and Poisson regression simulation studies presented in \citet{tang2024empirical}.  This is by no means a definitive answer to the question of how to choose hyperparameters in applications, for at least two reasons.  First, their recommendation cannot be generalized beyond the moderate $(n,p)$ settings they considered in their experiments.  Second, while one can argue that Martin et al.'s settings roughly match the polynomially decreasing $\lambda$ regime and that their small $\lambda$ helps compensate for the penalization that is lost when choosing $A_4$ small, there is still the constant $16 \alpha C_\text{dev}$ in \eqref{eqn:hyperparams_requirements} that need not be small.  While our refined analysis still cannot definitively answer the hyperparameter choice question, what it does offer that previous analyses do not is a clearly and theoretically-grounded understanding of why and how the hyperparameters are related.  With the insights provided by the theoretical analysis here, we hope that further empirical investigations can shed more light on their practical choice.

\section{Discussion} \label{sec:discussion}

This paper presents new and improved results on posterior contraction and model selection consistency for a class of Bayesian (or at least ``Bayesian-like'') posterior distributions in the context of sparse, high-dimensional GLMs.  These improvements are made possible thanks to a refined analysis based in part on results of \citet{spokoiny2012parametric, spokoiny2017penalized}, originally employed in the context of likelihood-based inference in finite-dimensional parametric models.  These refinements, in particular, lead to precise quadratic approximations to the GLM's log-likelihood function which, in turn, is used to obtain Laplace approximations of the Bayesian marginal likelihood that are more precise than those obtained by other authors.  This increased precision leads to more relaxed conditions on the model inputs, e.g., $(n, p, s_0, \ldots)$, which broadens the scope of applications and, thereby, strengthens the conclusions.  Furthermore, the previous literature was lacking in terms of its coverage of the entire class of GLMs, including those (e.g., Poisson) models whose score function has sub-exponential rather than sub-Gaussian tails. The refined analysis also suggests that an answer to the practical question of how to choose the prior hyperparameters might be within theoretical reach.  While we cannot definitively answer this question about hyperparameter choice based on our analysis, this does shed new light on the problem and motivate further empirical (and perhaps theoretical) investigations. 

Given the new and powerful selection consistency results, it would be relatively straightforward to establish a version of the fundamental {\em Bernstein--von Mises theorem}---e.g., \citet[][Ch.~2]{ghoshramamoorthi} and \citet[][Ch.~12]{ghosal2017fundamentals}---which would give a large-sample approximation of the posterior distribution, $\Pi_\alpha^n$, by a multivariate Gaussian or a mixture thereof.  Indeed, under conditions sufficient for selection consistency, it should be relatively easy to show \citep[e.g.,][Theorem~5]{tang2024empirical}, perhaps under further conditions, that the full posterior can be approximated, asymptotically, by a single $s_0$-dimensional Gaussian distribution centered at the $S_0$-specific MLE.  More generally, under weaker conditions, a mixture-of-Gaussians approximation of the posterior along the lines of \citet[][Theorem~6]{castillo2015sparse} should be within reach.  
%We leave this as a topic for a follow-up paper.  

Some readers might find the added generality offered by the power $\alpha \leq 1$ to be unnecessary.  The choice $\alpha < 1$ does, however, offer non-negligible simplification in the theoretical analysis.  Also, \citet{walker.hjort.2001} showed that there are examples in which the posterior based on $\alpha < 1$ is consistent while the posterior based on $\alpha=1$ is inconsistent; see, also, \citet{grunwald.ommen.scaling}.  Moreover, at least in principle, the fraction power leads to faster posterior concentration rates since the proofs can proceed without consideration of the entropies that inevitably (albeit insignificantly) slow down the rate of concentration.  Beyond these relatively old and familiar points, it is worth asking if there is a concrete benefit to the choice of $\alpha < 1$.  
%Given the refined analysis in the present paper, perhaps some understanding of the role played by $\alpha$ is within reach.  
While $\alpha$ does not significantly affect concentration rates and selection consistency, one of us (RM) has conjectured elsewhere that a choice of $\alpha < 1$ may have an impact in higher-order properties like distributional approximations, uncertainty quantification, etc.  As it pertains to uncertainty quantification, i.e., posterior credible regions are asymptotically valid confidence regions, the modern proofs rely on a suitable inflation of credible ball's radius by some constant/negligible factor.  Since $\alpha < 1$ has the effect of flattening out the likelihood, thereby inflating posterior credible balls, the conjecture is that a choice of $\alpha < 1$ might automatically accommodate this inflation that currently appears necessary to prove asymptotically valid uncertainty quantification.  So far, no clear connection has emerged, although some limited results are presented in \citet{ebcvg}.  It is possible that the influence of $\alpha$ is confounded with the Gaussianity of all the previous examples considered, so we hope that the more refined analysis here in outside the Gaussian context can shed more light on this matter. Even more generally, when using a data-dependent prior, the lines between likelihood and prior are blurred, which is precisely what distinguishes the misspecified model (and Gibbs posterior) cases where a learning rate (like our $\alpha$) needs to be chosen carefully from the classical Bayesian cases where $\alpha \equiv 1$ suffices.  So, further investigation into the role that $\alpha$ plays here is warranted, perhaps from several different angles.

Finally, there are a number of other papers that have used similar kinds of data-dependent prior distributions.  When the prior is for aspects of the model's location parameter (e.g., in Gaussian linear regression), the technical complications created by the data-dependence is rather mild.  When the prior concerns aspects of the model beyond a location parameter, however, this data-dependence is more problematic, and other authors---in particular, \citet{eb.gwishart} and \citet{tang2024empirical}---have relied on certain proof techniques that may have negatively impacted the rates attained.  The proof technique employed in this paper, namely, bounding the prior data-dependent density by suitable deterministic sub- and super-probability densities, is new and broadly applicable.  It would be interesting to revisit the aforementioned applications, and dig into some yet-to-be-investigated applications, such as mixture density estimation, to see if/how this bounding technique might be beneficial. 

%%%%%%%%%%%%%%%%%%%%%%%%%%%%%%%%%%%%%%%%%%%%%%%%%%%%%%%%%%%%%%%%%%%%%%%%%%%%%%%%

%%%%%%%%%%%%%%%%%%%%%%%%%%%%%%%%%%%%%%%%%%%%%%%%%%%%%%%%%%%%%%%%%%%%%%%%%%%%%%%%
\bibliographystyle{apalike}
\bibliography{AOS_revision/references}       % Bibliography file (usually '*.bib')
\addtocontents{toc}{\protect\setcounter{tocdepth}{2}}
\pagebreak
%%%%%%%%%%%%%%%%%%%%%%%%%%%%%%%%%%%%%%%%%%%%%%%%%%%%%%%%%%%%%%%%%%%%%%%%%%%%%%%%

\tableofcontents

\begin{appendix}
    
\section{Notations}

We first introduce common notations used in Appendix. For distributions having densities with respect to a dominating measure $\mu$, define Kullback--Leibler (KL) divergence and the corresponding variance as
\begin{align*}
    \operatorname{KL}(p_{i, \theta_1}, p_{i, \theta_2}) &= \int p_{i, \theta_1}  \log \dfrac{p_{i, \theta_1}}{p_{i, \theta_2}} {\rm d} \mu, \\
    \operatorname{V_{KL}}(p_{i, \theta_1}, p_{i, \theta_2}) &= 
    \bbE  \left[  \left\{
    \log \dfrac{p_{i, \theta_1}}{p_{i, \theta_2}}
    - \operatorname{KL}(p_{i, \theta_1}, p_{i, \theta_2})
    \right\}^2\right].
\end{align*}
Let $\operatorname{Proj}_{\bbH}(x)$ be the orthogonal projection of $x$ onto a subspace $\bbH$.

For the convenience of readers, the main notations used in Appendix are summarized in Table \ref{tab:notation}.

\renewcommand{\arraystretch}{1.5} 
\begin{table}[hbt!]
\caption{Summary of notations and definitions.} \label{tab:notation}
\centering
\begin{tabular}{c c p{6cm}} 
\hline
Notation & Location \\ [0.5ex] 
\hline\hline
$C_{\rm col}$ & Lemma \ref{lemma:eps_deviation_ineq} \\
$K_{\rm score}, C_{\rm radius}$ & \eqref{eqn:score_concentration}, \eqref{eqn:mle_concentration} \\
$\omega_{\epsilon, p, s}, z_{\epsilon, p, S}, \widetilde{z}_{\epsilon, p, S}$ & Lemma \ref{lemma:projection_score_vec} \\
$C_{n, S}$ & Lemma \ref{lemma:Smoothness_of_the_Fisher_information_operator}  \\
$\gamma_{n} \left(\theta\right)$ & Lemma \ref{lemma:sufficient prior mass} \\
$\delta_{n, S}, \widetilde{\delta}_{n, S}$ & Lemma \ref{lemma:Smoothness_of_the_Fisher_information_operator}, \ref{lemma:extended_Fisher_smooth}  \\
$\bV_{S, {\rm low}}, \bV_{S, {\rm up}}$ & Lemma \ref{lemma:normality_truncated_support}  \\
$\widetilde{\scrS}_{s_{\max}}$ & \eqref{def:supp_set_MLE_concentration} \\
$\scrC(\cdot, \cdot)$ & \eqref{def:proj_coordinate_subspace} \\
\hline
\end{tabular}
\end{table}

\section{Parametric estimation theory} \label{sec:param_estimation_app}

For the exponential family, we have that the moment generating function of $Y_i$ is given by
\begin{align} \label{assume:mgf_finite}
    \bbE e^{t Y_i} = \exp \bigl\{ b(x_i^{\top}\theta_0 + t) - b(x_i^{\top}\theta_0) \bigr\}, \quad \forall t \in \bbR.
\end{align}
It should be noted that \eqref{assume:mgf_finite} can be applied to generalized linear models with canonical link functions, such as Poisson regression and logistic regression.

The following two lemmas are modified versions of Lemma B.1 in \citet{barber2015high}.

\begin{lemma}[Deviation of normalized score function] \label{lemma:dev_ineq_score_func} 
For $S \subset [p]$ and $\omega > 0$, suppose that $\bF_{n, \thetaBest}$ is nonsingular and 
\begin{align} \label{assume:dev_assume}
    \dfrac{\sqrt{2} \omega \designRegular}{ \sqrt{ C_{\rm dev} \Delta_{{\rm mis}, S}}} \leq \frac{1}{2},
\end{align}
where $\Delta_{{\rm mis}, S}$ is defined in \eqref{def:misspecification_quantity_main}.
Then, for any $u \in \bbR^{|S|}$ with $\| u \|_2 = 1$,
\begin{align*}
    \bbP_{0}^{(n)} \bigg( u^{\top} \xi_{n, S} > \sqrt{2 C_{\rm dev} \Delta_{{\rm mis}, S} \omega^2 }  \bigg) \leq e^{-\omega^2}.
\end{align*}
\end{lemma}

\begin{proof}
Note that $\sum_{i=1}^n (\epsilon_i - \epsilon_{i, \theta_S^*}) x_{i, S} = - \sum_{i=1}^n \{b'(x_{i}^\top \theta_0 ) - b'( x_{i, S}^\top \thetaBest) \} x_{i, S}$ is non-random and its expectation is zero because $\E \epsilon_i = 0$ and $\bbE \dot{L}_{n, \thetaBest} = 0$.
Therefore, $\sum_{i=1}^n (\epsilon_i - \epsilon_{i, \theta_S^*}) x_{i, S} = 0$ and
\begin{align*}
  \xi_{n, S} 
  = \sum_{i=1}^{n} \bF_{n, \thetaBest[S]}^{-1/2} (\epsilon_i + \epsilon_{i, \theta_S^*} - \epsilon_i ) x_{i, S}
  = \sum_{i=1}^{n} \bF_{n, \thetaBest[S]}^{-1/2} \epsilon_i x_{i, S}.
\end{align*}
Let $\widetilde{\omega} = \sqrt{2 C_{\rm dev} \Delta_{{\rm mis}, S} \omega^2 }$.
For $u \in \bbR^{|S|}$ with $\| u \|_{2} = 1$ and $t > 0$, note that
\begin{align}
\begin{aligned} \label{eqn:dev_ineq_eqn1}
  \bbP_{0}^{(n)} \bigl\{ u^{\top} \xi_{n, S} > \widetilde{\omega} \bigr\} 
  & = \bbP_{0}^{(n)} \left\{ u^{\top} \bF_{n, \thetaBest[S]}^{-1/2}  \sum_{i=1}^{n}  \left[ Y_i - b'(x_i^{\top} \theta_0) \right] x_{i, S}   > \widetilde{\omega} \right\} \\
  &= \bbP_{0}^{(n)} \left\{ t \sum_{i=1}^{n} u^{\top} \bF_{n, \thetaBest[S]}^{-1/2} x_{i, S} Y_i > t\sum_{i=1}^{n}  u^{\top} \bF_{n, \thetaBest[S]}^{-1/2} b'(x_i^{\top} \theta_0) x_{i, S} + t\widetilde{\omega} \right\}.
\end{aligned}
\end{align}
By Markov inequality and \eqref{assume:mgf_finite}, the logarithm of the probability in \eqref{eqn:dev_ineq_eqn1} is bounded by
\begin{align} \label{eqn:dev_ineq_eqn2}
\begin{aligned}
       & -\sum_{i=1}^{n} \left[ t u^{\top} \bF_{n, \thetaBest[S]}^{-1/2} b'(x_i^{\top} \theta_0) x_{i, S} \right] - t \widetilde{\omega}
    + \sum_{i=1}^{n} \left[ b\left( x_i^{\top} \theta_0  + t u^{\top} \bF_{n, \thetaBest[S]}^{-1/2} x_{i, S} \right) - b(x_i^{\top} \theta_0)
    \right]
    \\
    &= \sum_{i=1}^{n} \left[ 
        b\left( x_i^{\top} \theta_0  + t u^{\top} \bF_{n, \thetaBest[S]}^{-1/2} x_{i, S} \right) - b(x_i^{\top} \theta_0)
        - b'(x_i^{\top} \theta_0) t u^{\top} \bF_{n, \thetaBest[S]}^{-1/2} x_{i, S}
    \right] - t \widetilde{\omega} \\
    &= \dfrac{1}{2}
    \sum_{i=1}^{n} \left[ b'' \left( x_i^{\top} \theta_0  + \eta t u^{\top} \bF_{n, \thetaBest[S]}^{-1/2} x_{i, S}\right) \left( t u^{\top} \bF_{n, \thetaBest[S]}^{-1/2} x_{i, S} \right)^2   \right] - t \widetilde{\omega} \\
    &= \dfrac{t^2}{2} 
    u^{\top} \bF_{n, \thetaBest[S]}^{-1/2} 
    \left[ \sum_{i=1}^{n}  b'' \left( x_i^{\top} \theta_0  + \eta t u^{\top} \bF_{n, \thetaBest[S]}^{-1/2} x_{i, S}\right) x_{i, S}x_{i, S}^{\top} \right] 
    \bF_{n, \thetaBest[S]}^{-1/2} u
    - t \widetilde{\omega},
\end{aligned}
\end{align}
where the second equality holds for some $\eta \in (0,1)$ by Taylor's theorem.

By taking $t = \left( 2\omega^2 / C_{\rm dev} \Delta_{{\rm mis}, S} \right)^{1/2}$, we have
\begin{align*}
    \left|\eta t u^{\top} \bF_{n, \thetaBest[S]}^{-1/2} x_{i, S} \right|
    = \left| \eta \frac{ \sqrt{2} \omega}{ \sqrt{C_{\rm dev} \Delta_{{\rm mis}, S}}  } u^{\top} \bF_{n, \thetaBest[S]}^{-1/2} x_{i, S} \right| 
    \leq \dfrac{ \sqrt{2} \omega \designRegular}{ \sqrt{C_{\rm dev} \Delta_{{\rm mis}, S}}  } \leq 1/2,
\end{align*}
which, combining with \eqref{def:C_dev_GLM_main}, implies that
\begin{align*}
    \sum_{i=1}^{n}  b'' \left( x_i^{\top} \theta_0  + \eta t u^{\top} \bF_{n, \thetaBest[S]}^{-1/2} x_{i, S}\right) x_{i, S}x_{i, S}^{\top} 
    \preceq C_{\rm dev} \sum_{i=1}^{n}  b'' \left( x_i^{\top}\theta_0 \right) x_{i, S}x_{i, S}^{\top} = C_{\rm dev} \bV_{n, S}.
\end{align*}
Therefore, \eqref{eqn:dev_ineq_eqn2} is bounded by
\begin{align*}
    \dfrac{C_{\rm dev}}{2} \dfrac{ 2\omega^2 }{C_{\rm dev} \Delta_{{\rm mis}, S}}
    u^{\top} \bF_{n, \thetaBest[S]}^{-1/2} \bV_{n, S} \bF_{n, \thetaBest[S]}^{-1/2} u - \dfrac{ \sqrt{2} \omega }{ \sqrt{C_{\rm dev} \Delta_{{\rm mis}, S}}}  \widetilde{\omega} 
    \leq  \omega^2 - 2\omega^2 = -\omega^2.
\end{align*}
This completes the proof.
\end{proof}

\begin{remark}
For $S \in \scrS_{s_{\max}}$, suppose that $\Delta_{{\rm mis}, S}$ is bounded away from zero and $\designRegular \lesssim n^{-1/2}$ . 
Let 
$$
\omega = \left[ (2s + 1) \log p + s \log(6) \right]^{1/2}.
$$
Then, one can see that
\begin{align*}
    \max_{S \in \scrS_{s_{\max}}} \omega \designRegular \left(\dfrac{2 }{C_{\rm dev} \Delta_{{\rm mis}, S}} \right)^{1/2} 
    \lesssim  \max_{S \in \scrS_{s_{\max}}} \omega \designRegular
    \lesssim \max_{S \in \scrS_{s_{\max}}}  \left( \dfrac{ |S| \log p }{ n } \right)^{1/2} = o(1)
\end{align*}
provided that $ \max_{S \in \scrS_{s_{\max}}} |S| \log p = o(n)$. Hence, the condition for Lemma \ref{lemma:dev_ineq_score_func} is satisfied for sufficiently small $\designRegular$,
which is proportional to the sample size $n$.
\end{remark}

For a given $S \supseteq S^{'}$, define 
\begin{align} \label{def:proj_coordinate_subspace}
    \scrC(S, S^{'}) = \left\{ \bF_{n, \thetaBest[S]}^{1/2} x : x = (x_j)_{j=1}^{|S|} \in \bbR^{|S|} \text{ with } x_j = 0 \text{ for all } j \in S \setminus S^{'} \right\}.
\end{align}
For $\epsilon \in (0, 1)$, let 
\begin{align*}
\overline{\scrS}_{\epsilon, s_{\max}} = \left\{ S \in \scrS_{s_{\max}} : 
\bF_{n, \thetaBest[S]} \succ 0, \quad
\dfrac{\sqrt{2} \omega_{\epsilon, p, |S|} \designRegular[S] }{ \sqrt{C_{\rm dev} \Delta_{{\rm mis}, S}} } \leq 1/2
\right\},   
\end{align*}
where $\Delta_{{\rm mis}, S}$ is defined in \eqref{def:misspecification_quantity_main} and $\omega_{\epsilon, p, s} = \left[ (2s + 1) \log p + s \log(3/\epsilon) \right]^{1/2}$.

\begin{lemma} \label{lemma:projection_score_vec}    
    Suppose that $p \geq 2$.
    Then, for any $\epsilon \in (0, 1)$,
    \begin{align} \label{eqn:proj_score_claim1}
        \bbP_0^{(n)} \bigg( \| \xi_{n, S} \|_{2} > z_{\epsilon, p, S} \quad 
        \text{for some } S \in \overline{\scrS}_{\epsilon, s_{\max}}  \bigg) &\leq p^{-1}, \\
        \label{eqn:proj_score_claim2}
        \bbP_0^{(n)} \bigg( \left\| \operatorname{Proj}_{\scrC(S, S_0)^{\perp}}  \left( \xi_{n, S} \right) \right\|_{2} >  \widetilde{z}_{\epsilon, p, S} \quad 
        \text{for some } S \in \overline{\scrS}_{\epsilon, s_{\max}} \text{ with } S \supsetneq S_0  \bigg) &\leq p^{-1},
    \end{align}
    where 
    \begin{align*}
     &z_{\epsilon, p, S} = \sqrt{2  C_{\rm dev} \Delta_{{\rm mis}, S} } (1-\epsilon)^{-1} \omega_{\epsilon, p, |S|}, \quad
     \widetilde{z}_{\epsilon, p, S} = \sqrt{2  C_{\rm dev} } (1-\epsilon)^{-1} \omega_{\epsilon, p, |S \setminus S_0|}.
    \end{align*}
\end{lemma}
\begin{proof}
    For $\epsilon \in (0, 1)$ and $S \in \overline{\scrS}_{\epsilon, s_{\max}}$, let $\cU_{S} = \left\{ u \in \bbR^{|S|} :  \| u \|_{2} = 1 \right\}$ and $\widehat{\cU}_{S, \epsilon}$ be the $\epsilon$-cover of $\cU_{S}$. One can choose $\widehat{\cU}_{S, \epsilon}$ so that $| \widehat{\cU}_{S, \epsilon} | \leq (3/\epsilon)^{|S|}$; see Proposition 1.3 of Section 15 in \citet{lorentz1996constructive}.
    For $y \in \bbR^{|S|}$, we can choose $x \in \widehat{\cU}_{S, \epsilon}$ such that
    \begin{align} \label{eqn:proj_score_eps_net}
        x^{\top} \dfrac{y}{\| y \|_2} = \left( \dfrac{y}{\| y \|_2} \right)^{\top} \dfrac{y}{\| y \|_2} + \left( x-\dfrac{y}{\| y \|_2} \right)^{\top} \dfrac{y}{\| y \|_2} \geq 1 - \epsilon,
    \end{align}    
    so we have $x^{\top}y \geq (1-\epsilon)\| y \|_{2}$.
    It follows that
    \begin{align} 
    \begin{aligned} \label{eqn:score_vec_event}
        &\bbP_0^{(n)} \big( \| \xi_{n, S} \|_{2} > z_{\epsilon, p, S} \big) \\
        &\leq \bbP_0^{(n)} \left\{  \max_{u \in \widehat{\cU}_{S, \epsilon} }  u^{\top} \xi_{n, S}  > (1 - \epsilon) z_{\epsilon, p, S} \right\} \\
        &\leq \left| \widehat{\cU}_{S, \epsilon} \right|
        \max_{u \in \widehat{\cU}_{S, \epsilon}}\bbP_0^{(n)} \left\{ u^{\top} \xi_{n, S}  > (1 - \epsilon) z_{\epsilon, p, S} \right\}
        \\
        &\leq \left( \dfrac{3}{\epsilon}\right)^{|S|} e^{-\omega_{\epsilon, p, |S|}^2}
        = \left( \dfrac{3}{\epsilon}\right)^{|S|} 
        \exp \left[ 
            - \log p  - |S| \left\{ 2\log p + \log  \left(\frac{3}{\epsilon} \right) \right\} 
        \right] \\
        &= p^{- (1 +2|S|)}        
    \end{aligned}
    \end{align}
    where the last inequality holds by Lemma \ref{lemma:dev_ineq_score_func}.
    Therefore, 
    \begin{align*}
        \bbP_0^{(n)} \left( \| \xi_{n, S} \|_{2} > z_{\epsilon, p, S} \text{ for some } S \in \overline{\scrS}_{\epsilon, s_{\max}} \right) 
        \leq \sum_{s = 1}^{\infty} \binom{p}{s} p^{-1 - 2s}
        \leq p^{-1} \sum_{s = 1}^{\infty} p^{- s} \leq p^{-1},
    \end{align*}
    where the second inequality holds because $\binom{p}{s} \leq p^{s}$, completing the proof of \eqref{eqn:proj_score_claim1}.   
    
    To prove \eqref{eqn:proj_score_claim2}, suppose that $S \in \overline{\scrS}_{\epsilon, s_{\max}}$ with $S \supsetneq S_0$ and let
    \begin{align*}
        \cV (S, S_0) = \left\{ \dfrac{u}{\| u \|_{2}}  \in \bbR^{|S|} : u \in \scrC(S, S_0)^{\perp}  \right\},
    \end{align*}
    Let $\widehat{\cV}_\epsilon(S, S_0)$ be an $\epsilon$-cover of $\cV (S, S_0)$ with $|\widehat{\cV}_\epsilon(S, S_0)| \leq (3/\epsilon)^{|S\backslash S_0|}$. One can choose such a cover by Proposition 1.3 of Section 15 in \citet{lorentz1996constructive}.
    As before, for $y \in \scrC(S, S^{'})^{\perp}$, we have $x^{\top}y \geq (1-\epsilon)\| y \|_{2}$ for some $x \in \widehat\cV_\epsilon(S, S_0)$.
    Note that $\Delta_{{\rm mis}, S} = 1$ for all $S \supseteq S_0$.
    Therefore,
    \begin{align*}
        &\bbP_0^{(n)} \left( \left\| \operatorname{Proj}_{\scrC(S, S_0)^{\perp}}  \left( \xi_{n, S} \right) \right\|_{2} > \widetilde{z}_{\epsilon, p, S} \right)  
        \leq \bbP_0^{(n)} \left\{  \max_{u \in \widehat{\cV}_\epsilon(S, S_0)}  u^{\top} \xi_{n, S}   > (1 - \epsilon) \widetilde{z}_{\epsilon, p, S} \right\} \\
        &\leq \left| \widehat{\cV}_\epsilon(S, S_0) \right|
        \max_{u \in \widehat{\cV}_\epsilon(S, S_0)} \bbP_0^{(n)} \left\{  u^{\top} \xi_{n, S}  > (1 - \epsilon) \widetilde{z}_{\epsilon, p, S} \right\} \\
        &\leq  \left( \dfrac{3}{\epsilon}\right)^{|S \setminus S_0|} 
        e^{-\omega_{\epsilon, p, |S \setminus S_0|}^2}
        =  \exp \left( - \log p  - 2|S \setminus S_0|\log p \right) = p^{- (1 + 2|S \setminus S_0|)}.
    \end{align*}
    It follows that
    \begin{align*}
        &\bbP_0^{(n)} \left\{ \left\| \operatorname{Proj}_{\scrC(S, S_0)^{\perp}}  \left( \xi_{n, S_{\texttt{+}}} \right) \right\|_{2} > \widetilde{z}_{\epsilon, p, S}
         \text{ for some } S \in \overline{\scrS}_{\epsilon, s_{\max}} \text{ with } S \supsetneq S_0 \right\} \\
        \leq &\sum_{r = 1}^{\infty} \binom{p - s_0}{r} p^{-1 - 2r}
        \leq p^{-1} \sum_{r = 1}^{\infty} p^{- r} \leq p^{-1},
    \end{align*}
    where the second inequality holds because $\binom{p-s_0}{r} \leq p^{r}$.
    This completes the proof of \eqref{eqn:proj_score_claim2}.
\end{proof}
From here on, we set $\epsilon = 1/2$ for simplicity in notation. Consequently, we represent $z_{\epsilon, p, S}$, $\widetilde{z}_{\epsilon, p, S}$ and $\omega_{\epsilon, p, s}$ with $\epsilon = 1/2$ as $z_{p, S}$ and $\omega_{p, s}$.

The following lemma is a modified version of Lemma 3.8 in \citet{spokoiny2017penalized} and Proposition 2.1 in \citet{barber2015high}.
\begin{lemma} [Smoothness of the Fisher information operator]
\label{lemma:Smoothness_of_the_Fisher_information_operator}
Let $r_{p, S} = 4z_{p, S}$. 
For $S \in \scrS_{s_{\max}}$, suppose that there exists $C_{n, S} > 0$ such that
\begin{align*}
 \sup_{\theta_S \in \localSetRn[S]{ r_{p, S} }} \max_{i \in [n]} \:  \exp \left( 3\left| x_{i, S}^{\top} \left[ \theta_{S} - \thetaBest \right] \right| \right) \leq C_{n, S},    
\end{align*}
and $\bF_{n, \thetaBest}$ is nonsingular.
Then, for all $\theta_{S} \in \localSetRn[S]{ r_{p, S} }$,
\begin{align} \label{eqn:fisher_smooth_ineq}
    (1 - \delta_{n,S}) \bF_{n, \thetaBest[S]} \preceq 
    \bF_{n, \theta_S} \preceq 
    (1 + \delta_{n,S}) \bF_{n, \thetaBest[S]},
\end{align}
where $\delta_{n, S} = \delta_{n, p, S} =  C_{n, S} r_{p, S} \designRegular$. 
\end{lemma}

\begin{proof}
For given $\theta_{S} \in \localSetRn[S]{r_{p, S}}$, 
\begin{align*}
  \bF_{n, \theta_S} - \bF_{n, \thetaBest[S]} 
  &= \sum_{i=1}^{n} \left\{ b''(x_{i, S}^{\top}\theta_S) - b''(x_{i, S}^{\top}\thetaBest[S]) \right\}x_{i, S}x_{i, S}^{\top}.
\end{align*}    
By Taylor's theorem, there exists $\theta_{S}^{\circ}(i) \in \localSetRn[S]{r_{p, S}}$ on the line segment between $\theta_{S}$ and $\thetaBest[S]$ such that
\begin{align} 
 \begin{aligned} \label{eqn:taylor_formular_eq}
    &\left| b''(x_{i, S}^{\top}\theta_S) - b''(x_{i, S}^{\top}\thetaBest[S]) \right|
    = \dfrac{\left| b'''(x_{i, S}^{\top} \theta_{S}^{\circ}(i)) \right| }{b''(x_{i, S}^{\top} \thetaBest[S])	}
    \left| x_{i, S}^{\top}\theta_S - x_{i, S}^{\top}\thetaBest[S] \right| b''(x_{i, S}^{\top} \thetaBest[S]) \\
    &\leq \dfrac{ b''(x_{i, S}^{\top} \theta_{S}^{\circ}(i)) }{b''(x_{i, S}^{\top} \thetaBest[S])	}
    \left| x_{i, S}^{\top}\theta_S - x_{i, S}^{\top}\thetaBest[S] \right| b''(x_{i, S}^{\top} \thetaBest[S]) \\
    &\leq \exp \left( 3 \left| x_{i, S}^{\top} \left[ \theta_{S}^{\circ}(i) - \thetaBest \right] \right| \right)
    \left| x_{i, S}^{\top}\theta_S - x_{i, S}^{\top}\thetaBest[S] \right| b''(x_{i, S}^{\top} \thetaBest[S]),
 \end{aligned}
\end{align}
where the inequalities hold by $| b'''(\cdot) | \leq b''(\cdot)$ (see Section 2.1 in \citet{ostrovskii2021finite}) and Lemma \ref{lemma:GLM_b_ratio}. 
Also, we have
	\begin{align}
		\begin{aligned} \label{eqn:fit_value_ineq}
        \left| x_{i, S}^{\top}\theta_{S} - x_{i, S}^{\top}\thetaBest[S] \right| 
		&= \left| \left\{\bF_{n, \thetaBest[S]}^{-1/2}x_{i, S}\right\}^{\top} \bF_{n, \thetaBest[S]}^{1/2}\left(\theta_{S} - \thetaBest[S]\right) \right| \\
		&\leq r_{p, S} \left\| \bF_{n, \thetaBest[S]}^{-1/2}x_{i, S} \right\|_{2} 
		\leq r_{p, S} \designRegular,
		\end{aligned}
	\end{align}
where two inequalities in the second line hold by the definitions of $\localSetRn[S]{r_{p, S}}$ and $\zeta_{n, S}$.
By \eqref{eqn:taylor_formular_eq} and \eqref{eqn:fit_value_ineq}, we have
	\begin{align*}
        \max_{i \in [n]} \left| b''(x_{i, S}^{\top}\theta_S) - b''(x_{i, S}^{\top}\thetaBest[S]) \right|
		\leq 
        C_{n, S} r_{p, S} \designRegular b''(x_{i, S}^{\top} \thetaBest[S]).
	\end{align*}	
	It follows that
	\begin{align}
		-\delta_{n,S} \sum_{i=1}^{n} b''(x_{i, S}^{\top} \thetaBest[S])x_{i, S}x_{i, S}^{\top}
		\preceq \bF_{n, \theta_S} - \bF_{n, \thetaBest[S]} 
		\preceq \delta_{n,S} \sum_{i=1}^{n} b''(x_{i, S}^{\top} \thetaBest[S])x_{i, S}x_{i, S}^{\top},
	\end{align}
completing the proof of \eqref{eqn:fisher_smooth_ineq}.
\end{proof}

\begin{remark}
    By \eqref{eqn:fit_value_ineq}, note that
    \begin{align*}
        \sup_{\theta_S \in \localSetRn[S]{ r_{p, S} }} \max_{i \in [n]} \:  \exp \left( 3\left| x_{i, S}^{\top} \left[ \theta_{S} - \thetaBest \right] \right| \right)
        \leq 
        \exp \left( 3 \designRegular r_{p, S} \right).
    \end{align*}
    If $\designRegular r_{p, S} \leq C$ for $S \in \scrS_{s_{\max}}$ and $C > 0$, one can see that 
    \begin{align*}
        \sup_{\theta_S \in \localSetRn[S]{ r_{p, S} }} \max_{i \in [n]} \:  \exp \left( 3\left| x_{i, S}^{\top} \left[ \theta_{S} - \thetaBest \right] \right| \right) 
        \leq 
        e^{3C}
    \end{align*}
    where the inequality holds for both Poisson and logistic regression models.
\end{remark}

The following lemma is a modified version of Theorem 3.4 and 3.5 in \citet{spokoiny2017penalized}.

Let 
\begin{align} \label{def:supp_set_MLE_concentration}
\widetilde{\scrS}_{s_{\max}} = \left\{ S \in \scrS_{s_{\max}} : 
\delta_{n, S} \leq 1/2, \quad
\bF_{n, \thetaBest[S]} \succ 0, \quad
\dfrac{\sqrt{2} \omega_{p, |S|} \designRegular[S] }{ \sqrt{C_{\rm dev} \Delta_{{\rm mis}, S}} } \leq 1/2
\right\},   
\end{align}
where $\Delta_{{\rm mis}, S}, \omega_{p, |S|}$ and $\delta_{n, S}$ are defined in Lemmas \ref{lemma:dev_ineq_score_func}, \ref{lemma:projection_score_vec} and \ref{lemma:Smoothness_of_the_Fisher_information_operator}, respectively.

\begin{lemma} \label{lemma:concentration_mle_score}
Suppose that $p \geq 2$.
Then,
\begin{align}
    \bbP_{0}^{(n)} \left( \thetaMLE[S] \notin \localSetRn[S]{r_{p, S}} \quad
    \text{for some } S \in \widetilde{\scrS}_{s_{\max}}
    \right) &\leq p^{-1} \nonumber \\
    \label{eqn:Fisher_theory}
    \bbP_0^{(n)} \bigg(  \left\| \bF_{n, \thetaBest}^{1/2} \left[ \thetaMLE - \thetaBest \right] - \xi_{n, S} \right\|_{2} > r_{p, S} \delta_{n, S} \quad \text{for some } S \in \widetilde{\scrS}_{s_{\max}} \bigg) &\leq p^{-1}.
\end{align}
\end{lemma}

\begin{proof}
For $S \in \widetilde{\scrS}_{s_{\max}}$, Theorem 3.4 and 3.5 in \citet{spokoiny2017penalized} implies that
\begin{align*} 
\bbP_0^{(n)} \left( \thetaMLE[S] \notin \localSetRn[S]{r_{p, S}} \right) \leq p^{-2|S|-1}
\end{align*}
and
\begin{align*} 
\bbP_0^{(n)} \left(  \left\| \bF_{n, \thetaBest}^{1/2} \left[ \thetaMLE - \thetaBest \right] - \xi_{n, S} \right\|_{2} > r_{p, S} \delta_{n, S} \right) \leq p^{-2|S|-1},
\end{align*}
respectively.
Here, for $S \in \widetilde{\scrS}_{s_{\max}}$, note that the above deviation results hold under the same event where \eqref{eqn:score_vec_event} in Lemma \ref{lemma:projection_score_vec} hold.

Since $\binom{p}{|S|} \leq p^{|S|}$,
\begin{align*}
    \bbP_{0}^{(n)} \left(
       \thetaMLE[S] \notin \localSetRn[S]{r_{p, S}} \text{ for some } S \in \widetilde{\scrS}_{s_{\max}}
    \right) 
    \leq 
    \sum_{s = 1}^{\infty} \binom{p}{s} p^{-2s -1} 
    \leq 
    p^{-1}\sum_{s = 1}^{\infty} p^{-s} \leq p^{-1}
\end{align*}    
and 
\begin{align*}
\bbP_0^{(n)} \left(  \left\| \bF_{n, \thetaBest}^{1/2} \left[ \thetaMLE - \thetaBest \right] - \xi_{n, S} \right\|_{2} > r_{p, S} \delta_{n, S} \: \text{ for some } S \in \widetilde{\scrS}_{s_{\max}} \right) 
&\leq 
\sum_{s = 1}^{\infty} \binom{p}{s} p^{-2s -1} \\
&\leq 
p^{-1},
\end{align*}
which completes the proof.
\end{proof}

\begin{remark} [Concentration properties of the MLE and the normalized score function] 
    From the results of Lemmas \ref{lemma:projection_score_vec} and \ref{lemma:concentration_mle_score}, with $\bbP_0^{(n)}$-probability at least $1 - p^{-1}$, the following ineqaulities hold simultaneously for all $S \in \widetilde{\scrS}_{s_{\max}}$:
	\begin{align} \label{eqn:score_concentration}
		&\left\| \xi_{n, S} \right\|_{2}^{2} \leq z_{p, S}^{2} \leq K_{\operatorname{score}} \Delta_{{\rm mis}, S} |S| \log p \\ 
        \label{eqn:mle_concentration}
		&\left\| \bF_{n, \thetaBest}^{1/2} \left( \thetaMLE - \thetaBest \right) \right\|_{2}^{2} \leq r_{p, S}^2 \leq C_{\rm radius} \Delta_{{\rm mis}, S} |S| \log p 
	\end{align}
    where $K_{\operatorname{score}} = 32 C_{\rm dev}$ and $C_{\rm radius} = 512 C_{\rm dev}$.    
    For $S \in \scrS_{s_{\max}}$ with $S \supseteq S_0$, note that $\Delta_{{\rm mis}, S} = 1$. Correspondingly, \eqref{eqn:mle_concentration} implies that
	\begin{align} \label{eqn:mle_concentration_param_vec}
		\left\| \thetaMLE[S] - \thetaBest[S]  \right\|_2 \leq C \sqrt{\dfrac{ |S| \log p }{\rho_{\operatorname{min}, S}}}, \quad \forall S \in \widetilde{\scrS}_{s_{\max}} \text{ with } S \supseteq S_0
	\end{align}
	for some constant $C > 0$, depending only on $C_{\rm dev}$, with $\bbP_0^{(n)}$-probability at least $1 - p^{-1}$. Note that \citet{tang2024empirical} provides a similar concentration result given by
	\begin{align} \label{eqn:tang_mle_concentration}
		\left\| \thetaMLE[S] - \thetaBest[S]  \right\|_2 \lesssim \sqrt{ \dfrac{|S| \log p}{ \rho_{\min, S} } \left( \dfrac{\rho_{\max, S}}{\rho_{\min, S}} \right)
        }.
	\end{align}
	The bound \eqref{eqn:tang_mle_concentration} might be looser than \eqref{eqn:mle_concentration_param_vec} since $\rho_{\operatorname{max}, S }/ \rho_{\operatorname{min}, S }$ may diverge.
	In particular, for $S \supseteq S_0$, the concentration of $\thetaMLE[S]$ within the local set $\localSetRn[S]{r_{p, S}}$ is useful for proving the posterior contraction results. 
\end{remark}

\begin{lemma} \label{lemma:eps_deviation_ineq}
Let $\cE = (\epsilon_i)_{i \in [n]}$, where $\epsilon_i = Y_i - b'(x_i^{\top}\theta_0)$. Suppose that there exists a constant $C_{\rm col} > 1$ such that
\begin{align} \label{eqn:eps_dev_cond}
    4 C_{\rm col}^{-1} \left\| \bX \right\|_{\max}^2 \log p \leq n, \quad \max_{j \in [p]} \left\| \bx_{j} \right\|_{2} \leq C_{\rm col} n^{1/2}.
\end{align}
Then,
\begin{align*}
    \bbP_{0}^{(n)}\left( \max_{j \in [p]} \left| \bx_j^{\top} \cE  \right| \geq 4\sqrt{2 C_{\rm col}} \nu_n (n \log p)^{1/2} \right) 
    \leq 2p^{-1},
\end{align*}
where 
\begin{align*}
    \nu_n = \big(1 + 2/(e\log 2) \big) \big( 1 + \sigma_{\max}^{2} (\log 2)^{-1} \big), \quad
    \sigma_{\max}^2 = \max_{i \in [n]} b''(x_i^{\top}\theta_0).
\end{align*}
\end{lemma}

\begin{proof}
    By Lemma \ref{lemma:orlic_bound}, we have
    \begin{align} \label{eqn:eps_dev_eq1}
        \left\| \epsilon_{i}  \right\|_{\psi_{1}} 
        \leq \big(1 + 2/(e\log 2) \big) \big( 1 + \sigma_{i}^{2} (\log 2)^{-1} \big),
    \end{align}
    where $\sigma_{i}^{2} = b''(x_i^{\top}\theta_0)$.
    Also, for all $i \in [n]$,
    \begin{align*}
        \bbE e^{t \epsilon_i} \leq e^{t^2 (2\sqrt{2} \| \epsilon_i \|_{\psi_{1}} )^2/2}, \quad |t| \leq 1/(2\sqrt{2}\| \epsilon_i \|_{\psi_{1}})
    \end{align*}
    by Proposition 4.1 in \citet{zhang2020concentration} with a slightly modified constant. 
    By the concentration inequality for sub-exponential random variables (see Corollary 4.2 in \citet{zhang2020concentration}), for any $t \geq 0$ and $j \in [p]$,
    \begin{align*}
        &\bbP_{0}^{(n)} \big( \left| \bx_j^{\top} \cE  \right| \geq t \big) 
        = 
        \bbP_{0}^{(n)}\left( \left| \sum_{i = 1}^{n} x_{ij} \epsilon_i  \right| \geq t \right) \\
        &\leq 
        2\exp\left( - \dfrac{1}{2} 
        \left[ \dfrac{t^2}{  \| \bx_j \|_{2}^{2} \left\{ 8 \max_{i \in [n]} \| \epsilon_i \|_{\psi_{1}}^{2} \right\} } 
        \wedge
         \dfrac{t}{  \| \bx_j \|_{\infty} \left\{ 2\sqrt{2} \max_{i \in [n]} \| \epsilon_i \|_{\psi_{1}} \right\} } \right]
        \right) \\
        &\leq 
        2\exp\left( - \dfrac{1}{2} 
        \left[ \dfrac{t^2}{  8 C_{\rm col} n  \nu_n^2 } 
        \wedge
         \dfrac{t}{  2\sqrt{2} \| \bX \|_{\max}  \nu_n } \right]
        \right),
    \end{align*}
    where the second inequality holds by \eqref{eqn:eps_dev_cond} and \eqref{eqn:eps_dev_eq1}. Since \eqref{eqn:eps_dev_cond} implies
    \begin{align*}
        \dfrac{ \left[4\sqrt{2 C_{\rm col}} \nu_n (n \log p)^{1/2} \right]^2}{  8 C_{\rm col} n  \nu_n^2 }
        \leq
        \dfrac{\left[4\sqrt{2 C_{\rm col}} \nu_n (n \log p)^{1/2} \right]}{  2\sqrt{2} \| \bX \|_{\max}  \nu_n },
    \end{align*}    
    we have
    \begin{align*}
       \bbP_{0}^{(n)}\left( \left| \bx_j^{\top} \cE  \right| \geq 4\sqrt{2 C_{\rm col}} \nu_n (n \log p)^{1/2} \right) \leq 2e^{-2\log p}.
    \end{align*}
    by taking $t = 4\sqrt{2 C_{\rm col}} \nu_n (n \log p)^{1/2}$. Note that
    \begin{align*}
        &\bbP_{0}^{(n)}\left( \max_{j \in [p]} \left| \bx_j^{\top} \cE  \right| \geq 4\sqrt{2 C_{\rm col}} \nu_n (n \log p)^{1/2} \right) \\
        &\leq p \max_{j \in [p]} \bbP_{0}^{(n)}\left( \left| \bx_j^{\top} \cE  \right| \geq 4\sqrt{2 C_{\rm col}} \nu_n (n \log p)^{1/2} \right)  
        \leq 2e^{-2\log p + \log p} = 2p^{-1},
    \end{align*}
    which completes the proof.
\end{proof}

\section{Posterior contraction} 
\label{sec:posterior_contraction_app}

In this subsection, our proof strategy is largely inspired by \citet{jeong2021posterior}, with certain modifications to accommodate a data-dependent prior. 
A notable challenge with such priors arises because we can't directly employ Fubini's theorem, a standard technique for proving posterior consistency.
To overcome this, one can consider replacing the density $g_{S}(\cdot)$ with two alternative prior densities: $\overline{g}_{S}(\cdot)$ and $\underline{g}_S(\cdot)$. These alternatives facilitate deriving appropriate upper and lower bounds for $g_{S}(\cdot)$.
If the replaced prior densities $\overline{g}_{S}(\cdot)$ and $\underline{g}_S(\cdot)$ do not depend on the data $\bY$, one can apply Fubini's theorem and standard techniques.

\begin{lemma} \label{lemma:empirical_prior_bound}
    Suppose that $p \geq 3$ and 
    \begin{align} \label{assume:empirical_prior_bound}
    \bF_{n, \thetaBest[S_0]} \succ 0, \quad
    \designRegular[S_0]^{2} s_0 \log p \leq ( \left[ C_{\rm dev}/64 \right] \wedge 0.05).
    \end{align}
    Also, assume that there exist non-random $\overline{\theta}_S \in \bbR^{|S|}$ and $D_{n} > \sqrt{2}$ satisfying \eqref{A1:a}.
    Then, with $\bbP_{0}^{(n)}$-probability at least $1 - 2p^{-1}$, the following inequalities hold uniformly for all non-empty $S \in \scrS_{s_{\max}}$:
    \begin{align} \label{eqn:prior_bound_claim1}
        g_{S}(\theta_S) \leq D_{n}^{2|S|} p^{\lambda |S|/2} \: \overline{g}_S(\theta_S), 
    \end{align}
    and
    \begin{align} \label{eqn:prior_bound_claim2}
        g_{S_0}(\theta_{S_0}) \geq p^{- (1 + 3 \lambda C_{\rm radius}/2) s_0}  \: \underline{g}_{S_0}(\theta_{S_0}),
    \end{align}
    where $C_{\rm radius}$ is the constant defined in \eqref{eqn:mle_concentration}, and $\overline{g}_{S}$ and $\underline{g}_{S}$ are the densities defined in \eqref{def:upper_lower_bound_empirical_prior}.
\end{lemma}
\begin{proof}
By the assumption, there exists an event $\Omega_{n, 1}$ such that $\bbP_0^{(n)} \left(\Omega_{n, 1} \right) \geq 1 - p^{-1}$ and on $\Omega_{n, 1}$, \eqref{A1:a} holds for all $S \in \scrS_{s_{\max}}$. 
Also, we can apply the results of Lemma \ref{lemma:concentration_mle_score} by the assumption \eqref{assume:empirical_prior_bound}. Hence, there exists an event $\Omega_{n, 2}$ such that $\bbP_0^{(n)} \left(\Omega_{n, 2} \right) \geq 1 - p^{-1}$ and on $\Omega_{n, 2}$, $\thetaMLE[S_0] \in \localSetRn[S_0]{r_{p, S_0}}$.
Let $\Omega_n = \Omega_{n, 1} \cap \Omega_{n, 2}$. Then, $\bbP_0^{(n)} \left(\Omega_{n} \right) \geq 1 - 2p^{-1}$.
In the remainder of this proof, we work on the event $\Omega_n$.

For $S \in \scrS_{s_{\max}}$ and $\theta_S \in \bbR^{|S|}$,
    \begin{align} \label{eqn:prior_bound_eqn1}
    \begin{aligned}
        &g_S(\theta_S) \\
        &=(2\pi)^{-|S|/2} \operatorname{det} \left\{ \lambda \bF_{n, \thetaMLE[S]} \right\}^{1/2} 
        \exp \left[ - \dfrac{\lambda }{2} \left( \theta_S - \thetaMLE[S] \right)^{\top} \bF_{n, \thetaMLE[S]}  \left( \theta_S - \thetaMLE[S] \right)\right] \\
        &\leq (2\pi)^{-|S|/2} \operatorname{det} \left\{ \lambda D_{n} \bF_{n, \overline{\theta}_S} \right\}^{1/2} 
        \exp \left[ - \dfrac{\lambda D_{n}^{-1} }{2} \left\| \bF_{n, \overline{\theta}_{S}}^{1/2} \left( \theta_S - \thetaMLE[S] \right) \right\|_{2}^{2} \right], 
    \end{aligned}
    \end{align}
    by \eqref{A1:a}.
Since
\begin{align*}
    \left\| \bF_{n, \overline{\theta}_{S}}^{1/2} \left( \theta_S - \thetaMLE[S] \right) \right\|_{2}^{2}
    \geq
    \dfrac{1}{2}\left\| \bF_{n, \overline{\theta}_{S}}^{1/2} \left( \theta_S - \overline{\theta}_{S} \right) \right\|_{2}^{2} 
    -\left\| \bF_{n, \overline{\theta}_{S}}^{1/2} \left( \overline{\theta}_{S} - \thetaMLE[S] \right) \right\|_{2}^{2},
\end{align*}
the right hand side of \eqref{eqn:prior_bound_eqn1} is further bounded by
\begin{align*}
     &(2\pi)^{-|S|/2}\operatorname{det} \left\{ \lambda D_{n} \bF_{n, \overline{\theta}_{S} } \right\}^{1/2} \\
    &\times \exp \left[ 
        - \dfrac{\lambda D_{n}^{-1} }{4} \left\| \bF_{n, \overline{\theta}_{S}}^{1/2} \left( \theta_S - \overline{\theta}_{S} \right) \right\|_{2}^{2} 
        + \dfrac{\lambda D_{n}^{-1} }{2} \left\| \bF_{n, \overline{\theta}_{S}}^{1/2} \left( \overline{\theta}_{S} - \thetaMLE[S] \right) \right\|_{2}^{2}    
    \right] \\
    &= \overline{g}_S(\theta_S) \times \underbrace{\left(  \dfrac{ 2 D_{n} }{ D_{n}^{-1} } \right)^{|S|/2} 
    \exp \left[ \dfrac{\lambda D_{n}^{-1} }{2} \left\| \bF_{n, \overline{\theta}_{S}}^{1/2} \left( \overline{\theta}_{S} - \thetaMLE[S] \right) \right\|_{2}^{2} \right]}_{(\ast)},
\end{align*}
where $\overline{g}_S(\cdot)$ is defined in \eqref{def:upper_lower_bound_empirical_prior}.
By \eqref{A1:a}, $(\ast)$ is bounded by 
\begin{align*}
    (\sqrt{2})^{|S|} D_{n}^{|S|}
    \exp \left[ \dfrac{\lambda}{2} D_{n}^{-1} D_{n} |S| \log p \right]
    \leq
    D_{n}^{2|S|} p^{\lambda |S|/2},
\end{align*}
where the inequalities hold by $D_{n} \geq \sqrt{2}$.
This completes the proof of \eqref{eqn:prior_bound_claim1}.

Next, we will prove \eqref{eqn:prior_bound_claim2}.
Note that the density $g_{S_0}(\theta_{S_0})$ is bounded below by
\begin{align} \label{eqn:prior_eqn_lower1}
\begin{aligned}
    &(2\pi)^{-s_0/2} \operatorname{det} \left\{ \lambda \left( 1 - \delta_{n, S_0} \right) \bF_{n, \thetaBest[S_0]} \right\}^{s_0/2} \\
    &\qquad \times \exp \left[ - \dfrac{\lambda \left( 1 + \delta_{n, S_0} \right)}{2} \left\| \bF_{n, \thetaBest}^{1/2} \left( \theta_{S_0} - \thetaMLE[S_0] \right) \right\|_{2}^{2} \right].    
\end{aligned}
\end{align}
Since we have 
\begin{align*}
    \left\| \bF_{n, \thetaBest[S_0]}^{1/2} \left( \theta_{S_0} - \thetaMLE[S] \right) \right\|_{2}^{2}
    \leq
    2\left\| \bF_{n, \thetaBest[S_0]}^{1/2} \left( \theta_{S_0} - \thetaBest[S_0] \right) \right\|_{2}^{2} 
    +2\left\| \bF_{n, \thetaBest[S_0]}^{1/2} \left( \thetaBest[S_0] - \thetaMLE[S_0] \right) \right\|_{2}^{2},
\end{align*}
\eqref{eqn:prior_eqn_lower1} is further bounded below by
\begin{align*}
    &(2\pi)^{-s_0/2} \operatorname{det} \left\{ \lambda \left( 1 - \delta_{n, S_0} \right) \bF_{n, \thetaBest[S_0]} \right\}^{1/2} \\
    & \qquad \times \exp \left[ 
        - \lambda \left( 1 + \delta_{n, S_0} \right) \left\| \bF_{n, \thetaBest[S_0]}^{1/2} \left( \theta_{S_0} - \thetaBest[S_0] \right) \right\|_{2}^{2} 
        - \lambda \left( 1 + \delta_{n, S_0} \right) \left\| \bF_{n, \thetaBest[S_0]}^{1/2} \left( \thetaBest[S_0] - \thetaMLE[S_0] \right) \right\|_{2}^{2}    
    \right] \\
    &= \underline{g}_{S_0}(\theta_{S_0}) \times
    \underbrace{\left(  \dfrac{1 - \delta_{n, S_0}}{2 \left[1 + \delta_{n, S_0}\right] } \right)^{s_{0}/2} 
    \exp \left[ -\lambda \left( 1 + \delta_{n, S_0} \right) \left\| \bF_{n, \thetaBest[S_0]}^{1/2} \left( \thetaBest[S_0] - \thetaMLE[S_0] \right) \right\|_{2}^{2} \right]}_{(\ast\ast)}.       
\end{align*}
Since \eqref{assume:empirical_prior_bound} implies that $S_{0} \in \widetilde{\scrS}_{s_{\max}}$ defined in \eqref{def:supp_set_MLE_concentration}, we have
$\delta_{n, S_0} \leq 1/2$ and $\thetaMLE[S_0] \in \localSetRn[S_0]{r_{p, S_0}}$.
It follows that 
\begin{align*}
\dfrac{1 - \delta_{n, S_0}}{2 \left[1 + \delta_{n, S_0}\right]} \geq 1/6, \quad 
\left\| \bF_{n, \thetaBest[S_0]}^{1/2} \left( \thetaBest[S_0] - \thetaMLE[S_0] \right) \right\|_{2}^{2}
\leq C_{\rm radius} s_{0} \log p,
\end{align*}
where $C_{\rm radius} = 512 C_{\rm dev}$ is the constant specified in \eqref{eqn:mle_concentration}.    
Therefore, $(\ast\ast)$ is bounded below by
\begin{align*}
    (\sqrt{6})^{-s_0} \exp \bigg( - \dfrac{3}{2} \lambda C_{\rm radius} s_0 \log p \bigg)
    &\geq 
    \exp \left( - s_0 - \dfrac{3}{2} \lambda C_{\rm radius} s_0 \log p \right) \\
    &\geq 
    p^{-(1 + 3 \lambda C_{\rm radius}/2) s_{0}},
\end{align*}
where the last inequality holds by $p \geq 3$.
This completes the proof of \eqref{eqn:prior_bound_claim2}.
\end{proof}

The following lemma verifies Assumption 1 in \citet{jeong2021posterior}. Based on the following Lemma, we shall show in Lemma \ref{lemma:elbo} that the empirical prior of \citet{tang2024empirical}, defined in \eqref{def:prior_slab_main}, has a sufficient prior mass near the true parameter. Let $\underline{\bbG}_{S}$ be the probability measure which allows the density $\underline{g}_{S}$ with respect to the Lebesgue measure.

\begin{lemma}[Sufficient prior mass] \label{lemma:sufficient prior mass}
Let $\gamma_{n} \left(\theta\right) =  1 + (1 + C_{\rm dev}/2) \max_{i \in [n]}  b'' \left(x_i^{\top} \theta\right)$ for the constant $C_{\rm dev}$ defined in \eqref{def:C_dev_GLM_main}.
Suppose that \eqref{A2:b} hold for some constants $A_5, A_6 > 0$ and $A_7 \geq 0$. Furthermore, assume that $p \geq 3$ and
\begin{align}
\begin{aligned} \label{cond:sufficient_mass}
    &\max_{i \in [n]} \log \left\{ b'' \left( x_{i, S_0}^{\top} \theta_{0, S_0} \right) \right\} \lesssim \log p, \quad 
    \log \|\bX_{S_0}\|_{\infty} \lesssim \log p, \\
    &\log (\rho_{\min, S_0}^{-1} \vee \rho_{\max, S_0}) \lesssim \log p, \quad 
    \delta_{n, S_0} \leq 1.
\end{aligned}
\end{align}
Then, for all $m_1 > 0$, there exists $m_2 > 0$ such that
\begin{align*} 
    \underline{\bbG}_{S_0} \left\{ \theta_{S_0}: 
    \left\|\bX_{S_0} \left( \theta_{S_0} - \theta_{0, S_0} \right) \right\|_{\infty}^{2} \leq 
    \dfrac{m_1 s_0 \log p}{\gamma_{n}\left(\theta_{0}\right) n}
    \right\}
    \geq \exp\left(-m_2 s_0 \log p \right).
\end{align*}
\end{lemma}
\begin{proof}
We may assume that 
\begin{align*} 
    &\log \|\bX_{S_0}\|_{\infty} \leq c_{1} \log p, \quad 
    \log (\rho_{\min, S_0}^{-1} \vee \rho_{\max, S_0}) \leq c_{1} \log p, \\
    &\log n \leq c_{1} \log p, \quad 
    \max_{i \in [n]} \log \left\{ b'' \left( x_{i, S_0}^{\top} \theta_{0, S_0} \right) \right\} \leq c_{1} \log p, \quad 
    m_1 \geq p^{-c_{1}}
\end{align*}
for some constant $c_{1} > 0$.
Let $Z_{S_0} \in \bbR^{|S_0|}$ be a random vector following $\underline{\bbG}_{S_0}$.
Note that
\begin{align*}
    \left\|\bX_{S_0} \left( Z_{S_0} - \theta_{0, S_0} \right) \right\|_{\infty}
    \leq \left\| \bX_{S_0} \right\|_{\infty} \left\| Z_{S_0} - \theta_{0, S_0}  \right\|_{\infty}, \quad 
    \dfrac{m_1 s_0 \log p}{\gamma_{n}\left(\theta_{0}\right) \| \bX_{S_0} \|_{\infty}^{2} n}
    \geq 
    p^{-(5c_{1} + 1)}.
\end{align*}
It follows that, for $m_1 > 0$,
\begin{align} \label{eqn:hyper_cube_ineq_1}
\begin{aligned}
    &\underline{\bbG}_{S_0} \left\{
    \|\bX_{S_0} \left(Z_{S_0} - \theta_{0, S_0} \right) \|_{\infty}^{2} \leq 
    \dfrac{m_1 s_0 \log p}{\gamma_{n}\left(\theta_{0}\right) n}
    \right\} \\
    &\geq 
    \underline{\bbG}_{S_0} \left\{
    \| Z_{S_0} - \theta_{0, S_0} \|_{\infty}^{2} \leq 
    \dfrac{m_1 s_0 \log p}{\gamma_{n}\left(\theta_{0}\right) \| \bX_{S_0} \|_{\infty}^{2} n}
    \right\} 
    \geq 
    \underline{\bbG}_{S_0} \left\{
    \| Z_{S_0} - \theta_{0, S_0} \|_{\infty}^{2} \leq c_{n}^{2}
    \right\},
\end{aligned}
\end{align}
where $c_n^{2} = p^{-c_{2}}$ for some constant $c_{2} > 5c_{1} + 1$.
Since
\begin{align} \label{eqn:hyper_cube_ineq_2}
    \underline{\bbG}_{S_0} \left\{
    \|Z_{S_0} - \theta_{0, S_0} \|_{\infty}^{2} \leq c_n^{2}		
    \right\} \geq
    \left( 2c_n \right)^{s_0} 
    \inf_{\eta \in \bbR^{s_{0}} : \| \eta \|_{\infty} < c_n} \underline{g}_{S_0} \left( \theta_{0, S_0} + \eta \right),
\end{align}
it suffices to prove that the logarithm of the right hand side of \eqref{eqn:hyper_cube_ineq_2} is bounded below by $-m_2 s_0 \log p$ for some constant $m_2 > 0$. 
In other words, we only need to prove that
\begin{align} \label{eqn:hyper_cube_ineq_3}
    -s_0 \log (2c_n)
     + \sup_{\eta \in \bbR^{s_0} : \| \eta \|_{\infty} < c_n} \left[-\log \left\{ \underline{g}_{S_0} \left( \theta_{0, S_0} + \eta \right) \right\}\right] 
     \lesssim s_0 \log p.
\end{align}

Firstly, by the definition of $c_{n}^{2}$, we have
\begin{align*}
    &-\log \left(2c_n\right) 
    \lesssim -\log \left(c_n^{2}\right)
    = - \log \left( p^{-c_{2}} \right) = c_{2} \log p.
\end{align*}
To bound the second term in \eqref{eqn:hyper_cube_ineq_3}, since $\theta_{0, S_0} = \theta^*_{S_0}$, we have
\begin{align*}
&-\log \left\{ \underline{g}_{S_0} \left( \theta_{0, S_0} + \eta \right) \right\} \\
&=	-\log \left[
    \left(\dfrac{ 2\lambda \left( 1 + \delta_{n, S_0} \right) }{2\pi}\right)^{ s_0 /2} 
    \operatorname{det}\left\{\bF_{n, \thetaBest[S_0]}\right\}^{1/2}
    \exp\left\{
    -\lambda \left( 1 + \delta_{n, S_0} \right)
    \left\| 
        \bF_{n, \thetaBest[S_0]}^{1/2}  \eta
    \right\|_{2}^{2}
    \right\}
\right] \\
&= \underbrace{-\dfrac{s_0}{2}\log \left\{ \dfrac{\lambda \left( 1 + \delta_{n, S_0} \right) }{\pi} \right\} -\dfrac{1}{2} \log \operatorname{det}\left\{\bF_{n, \thetaBest[S_0]}\right\} }_{(\ast)}
+  
\underbrace{ \lambda \left( 1 + \delta_{n, S_0} \right)
    \left\| 
    \bF_{n, \thetaBest[S_0]}^{1/2}  \eta
    \right\|_{2}^{2}}_{(\ast\ast)}.
\end{align*}
Also,
\begin{align*}
\left(\ast\right) 
&\leq \dfrac{s_0}{2}\log \lambda^{-1} + \dfrac{s_0}{2} \log \left(\pi\right) - \dfrac{s_0}{2} \log \left( 1 + \delta_{n, S_0} \right) -\dfrac{s_0}{2} \log \rho_{{\rm min}, S_0}  \\
&\leq  \dfrac{A_5}{2} s_0 \log p + \dfrac{s_0}{2} \log \left(\pi\right) + \dfrac{s_0}{2} \log \rho_{{\rm min}, S_0}^{-1}   
\lesssim s_0 \log p,   
\end{align*}
where the last two inequalities hold by \eqref{A2:b} and \eqref{cond:sufficient_mass}.

Since $1 + \delta_{n , S_0} \leq 2$, if $\| \eta \|_{\infty} < c_n$,
\begin{align*}
(\ast\ast) 
&\leq 2\lambda \left\| \bF_{n, \thetaBest[S_0]}^{1/2}  \eta  \right\|_{2}^{2} 
\leq 2\lambda \rho_{{\rm max}, S_0} s_0 c_{n}^2
\leq 2 A_6 p^{-A_7} p^{c_1} p^{-c_{2}} s_0 \log p \\
&= 2 A_6 p^{-A_7 - c_{2} + c_1} s_0 \log p
\lesssim s_0 \log p,
\end{align*}
where the last two inequalities hold by \eqref{A2:b} and the definition of $c_{n}^{2}$.
\end{proof}

In Appendix \ref{sec:posterior_contraction_app}-\ref{sec:selection_consistency_app}, we address conditions that are either easily met in the asymptotic regime (where both $n$ and $p$ tend towards infinity) or are of relatively minor importance. These specific conditions are identified with the tag (section.AS.number) next to the relevant statements

\begin{lemma}[Evidence lower bound] \label{lemma:elbo}
Suppose that conditions in Lemmas \ref{lemma:empirical_prior_bound} and \ref{lemma:sufficient prior mass} hold. Also, assume that
\begin{align} \label{assume:elbo_conditions} 
    4s_0 \log p \leq n.
\end{align}
and
\begin{align} \label{assume:elbo_conditions_AS} \tag{C.AS.1}
    A_{1}^{-1} \vee (2A_{2})^{A_{4}^{-1}} \leq p
\end{align}
Then, there exists a constant $K_{{\rm elbo}} > 0$ such that
\begin{align} \label{eqn:ELBO_claim}
    \bbP_{0}^{(n)} \left\{
    \int_{\bbR^{p}} \Lambda_{n}^{\alpha}(\theta) \, \Pi_{n}(\rmd \theta) \geq \exp(-K_{{\rm elbo}}s_0 \log p )
    \right\} \geq 1 - \frac{1}{s_0 \log p} - \frac{2}{p},
\end{align}
where $\Lambda_{n}^{\alpha}(\theta) = \bigl( \prod_{i = 1}^{n} p_{i, \theta}/ p_{i, \theta_0} \bigr)^{\alpha}$.
\end{lemma}
\begin{proof}
%Let a constant $m_1 = 1$ and choose the corresponding constant $m_2 > 0$ by Lemma \ref{lemma:sufficient prior mass}.
Let
\begin{align*} 
    \mathscr{K}_n = \left\{
        \theta_{S_0} \in \bbR^{s_0} : 
        \dfrac{1}{n} \sum_{i = 1}^{n} \operatorname{KL}\left( p_{i, \theta_{0}}, p_{i, \theta_{S_0}} \right) \leq \dfrac{s_0 \log p}{n}, \quad
        \dfrac{1}{n} \sum_{i = 1}^{n} \operatorname{V}_{\operatorname{KL}}\left( p_{i, \theta_{0}}, p_{i, \theta_{S_0}} \right) \leq \dfrac{s_0 \log p}{n}
    \right\}
\end{align*}
and $\Omega_n = \Omega_{n, 1} \cup \Omega_{n, 2}$, where $\Omega_{n, 1}, \Omega_{n, 2}$ are the events in the proof of Lemma \ref{lemma:empirical_prior_bound}.
Then, $\bbP_0^{(n)}(\Omega_n) \geq 1 - 2p^{-1}$ and \eqref{eqn:prior_bound_claim2} holds on $\Omega_n$.
On $\Omega_n$, we have
\begin{align}
\begin{aligned} \label{eqn:elbo_eqn3}
&\int_{\bbR^p} \Lambda_{n}^{\alpha}(\theta) \Pi_{n}(\rmd \theta) \\
    & = \sum_{S \in \scrS_{s_{\max}} }\dfrac{w_n(|S|) }{ \binom{p}{|S|} } \int_{\bbR^{|S|}} \Lambda_{n}^{\alpha}(\theta_{S}) g_{S} \left( \theta_{S} \right) \rmd \theta_{S} \\
    & \geq \dfrac{w_n(s_0) }{ \binom{p}{s_0} } \int_{\mathscr{K}_n } \Lambda_{n}^{\alpha}(\theta_{S_0}) g_{S_0} \left( \theta_{S_0} \right) \rmd \theta_{S_0} \\
    & \geq p^{-(1 + 3\lambda C_{\rm radius}/2)s_{0}} \dfrac{w_n(s_0) }{ \binom{p}{s_0} } \int_{\mathscr{K}_n } \Lambda_{n}^{\alpha}(\theta_{S_0}) \underline{g}_{S_0} \left( \theta_{S_0} \right) \rmd \theta_{S_0} \\
    & = w_n(s_0) \exp \left[- \left( 1 + \dfrac{3}{2} \lambda C_{\rm radius} \right) s_0 \log p  - \log \binom{p}{s_0} \right] \int_{\mathscr{K}_n } \Lambda_{n}^{\alpha}(\theta_{S_0}) \underline{g}_{S_0} \left( \theta_{S_0} \right) \rmd \theta_{S_0} \\
    & \geq w_n(s_0) \exp \left( -\left[ 512 A_6 C_{\rm dev} + 2 \right] s_0 \log p \right) \int_{\mathscr{K}_n } \Lambda_{n}^{\alpha}(\theta_{S_0}) \underline{g}_{S_0} \left( \theta_{S_0} \right) \rmd \theta_{S_0},    
\end{aligned}
\end{align}
where the second inequality is by Lemma \ref{lemma:empirical_prior_bound} and the last inequality holds because $\lambda C_{\rm radius} = \lambda (512 C_{\rm dev}) \leq 512 A_6 C_{\rm dev}$ and $\binom{p}{s_0} \leq p^{s_0}$. By slightly modifying Lemma 10 of \citet{ghosal2007convergence}, one can easily prove that, for any $C > 0$,
\begin{align} \label{eqn:elbo_eqn0}
    \bbP_{0}^{(n)} \left\{
    \int_{\scrK_n} \Lambda_{n}^{\alpha}(\theta_{S_0} ) \underline{g}_{S_0}(\theta_{S_0}) \rmd \theta_{S_0} \geq  e^{-\alpha (1 + C) s_0 \log p} \; \underline{\bbG}_{S_0}(\scrK_n)
    \right\} \geq 1 - \dfrac{1}{C^2 s_0 \log p},
\end{align}
where $\underline{\bbG}_{S_0}$ is the probability measusre with the density $\underline{g}_{S_0}$. Suppose \eqref{eqn:elbo_eqn0} holds for $C = 1$. 

We next prove that 
\begin{align} \label{eqn:elbo_intermediate_condition}
    \underline{\bbG}_{S_0}(\scrK_n) 
    \geq \underline{\bbG}_{S_0}\left\{
    \theta_{S_0} \in \bbR^{s_0} : \| \bX_{S_0} (\theta_{S_0} - \theta_{0, S_0})   \|_{\infty}^2 \leq \dfrac{s_0 \log p}{n \gamma_n (\theta_0)}
    \right\}.
\end{align}
Suppose that $\theta_{S_0}$ satisfies the inequality in the right hand side of \eqref{eqn:elbo_intermediate_condition}.
Then, since $\gamma_n(\theta_0) \geq 1$ and $4s_0 \log p \leq n$, we have $\| \bX_{S_0} (\theta_{S_0} - \theta_{0, S_0}) \|_{\infty} \leq 1/2$.
Note that
\begin{align*}
    \operatorname{KL}\left( p_{i, \theta_{0}}, p_{i, \theta_{S_0}} \right) &=  
    -\left( x_{i, S_0}^{\top} \theta_{S_0} - x_{i, S_0}^{\top} \theta_{0, S_0} \right) b' \left( x_{i, S_0}^{\top} \theta_{0, S_0} \right)  - b \left( x_{i, S_0}^{\top} \theta_{0, S_0} \right) + b \left( x_{i, S_0}^{\top} \theta_{S_0} \right), \\
    \operatorname{V}_{\operatorname{KL}}\left( p_{i, \theta_{0}}, p_{i, \theta_{S_0}} \right) 
    &= b'' \left( x_{i, S_0}^{\top} \theta_{0, S_0} \right)  \left( x_{i, S_0}^{\top} \theta_{0, S_0} - x_{i, S_0}^{\top} \theta_{S_0}  \right)^2,
\end{align*}
see page 2 of the supplementary material in \citet{jeong2021posterior}.
Also, by Taylor's theorem,
\begin{align*}
    \operatorname{KL}\left( p_{i, \theta_{0}}, p_{i, \theta_{S_0}} \right)
    &= \dfrac{ 1 }{2} b'' \left( \eta_{i, \theta_{S_0}} \right) \left( x_{i, S_0}^{\top} \theta_{0, S_0} - x_{i, S_0}^{\top} \theta_{S_0} \right)^2
\end{align*}
for some $\eta_{i, \theta_{S_0}}$ between $x_{i, S_0}^{\top} \theta_{0, S_0}$ and $x_{i, S_0}^{\top} \theta_{S_0}$. 
Since
\begin{align*}
    \left| \eta_{i, \theta_{S_0}} - x_{i, S_0}^{\top} \theta_{0, S_0}  \right|
    \leq 
    \left| x_{i, S_0}^{\top} \theta_{S_0} - x_{i, S_0}^{\top} \theta_{0, S_0}  \right|
    \leq 
    \| \bX_{S_0} (\theta_{S_0} - \theta_{0, S_0}) \|_{\infty} \leq \dfrac{1}{2},
\end{align*}
we have
\begin{align*}
    \dfrac{1}{2} b'' \left( \eta_{i, \theta_{S_0}} \right) \left( x_{i, S_0}^{\top} \theta_{0, S_0} - x_{i, S_0}^{\top} \theta_{S_0} \right)^2
    \leq 
    \dfrac{ C_{\rm dev} }{2}
    b''  \left( x_{i, S_0}^{\top} \theta_{0, S_0} \right)   \left( x_{i, S_0}^{\top} \theta_{0, S_0} - x_{i, S_0}^{\top} \theta_{S_0}  \right)^2,
\end{align*}
by \eqref{def:C_dev_GLM_main}.
Hence,
\begin{align*}
    &\max \left\{
    \operatorname{KL}\left( p_{i, \theta_0 }, p_{i, \theta_{S_0}} \right), \operatorname{V}_{\operatorname{KL}}\left( p_{i, \theta_0 }, p_{i, \theta_{S_0}} \right) \right\} \\
    &\leq 
    \left( 1 + \dfrac{C_{\rm dev}}{2} \right)
    b'' \left( x_{i, S_0}^{\top} \theta_{0, S_0} \right)  \left( x_{i, S_0}^{\top} \theta_{0, S_0} - x_{i, S_0}^{\top} \theta_{S_0}  \right)^2 \\
    &\leq \gamma_n(\theta_0) \| \bX_{S_0} (\theta_{S_0} - \theta_{0, S_0})   \|_{\infty}^2 
    \leq \dfrac{s_0 \log p}{n},
\end{align*}
which proves \eqref{eqn:elbo_intermediate_condition}.

By Lemma \ref{lemma:sufficient prior mass} with $m_1=1$, there exists a constant $m_2 > 0$ such that
\begin{align} \label{eqn:elbo_eqn4}
    \underline{\bbG}_{S_0}(\scrK_n) 
    \geq
    \underline{\bbG}_{S_0} \left\{
    \left\| \bX_{S_0} \left(Z_{S_0} - \theta_{0, S_0} \right) \right\|_{\infty}^{2} \leq 
    \dfrac{s_0 \log p}{\gamma_{n}\left(\theta_{0}\right) n}
    \right\}
    \geq \exp\left(-m_2 s_0 \log p \right).
\end{align}
By \eqref{eqn:elbo_eqn0} and \eqref{eqn:elbo_eqn4}, one can see that
\begin{align*}
    \bbP_{0}^{(n)} \left\{
    \int_{\scrK_n} \Lambda_{n}^{\alpha}(\theta_{S_0} ) \underline{g}_{S_0}(\theta_{S_0}) \rmd \theta_{S_0} \geq  e^{-(2\alpha + m_2) s_0 \log p} 
    \right\} \geq 1 - \dfrac{1}{s_0 \log p}.   
\end{align*}
Combining with \eqref{eqn:elbo_eqn3}, we have 
\begin{align*}
    \bbP_{0}^{(n)} \left\{
    \int_{\bbR^{p}} \Lambda_{n}^{\alpha}(\theta) \Pi_{n}(\rmd \theta)  
    \geq  
    w_n(s_0) e^{-\left[ 2\alpha + m_2 + 512 A_6 C_{\rm dev} + 2 \right]s_0 \log p} 
    \right\} \geq 1 - \dfrac{1}{s_0 \log p} - \dfrac{2}{p},
\end{align*}
where the term $2p^{-1}$ in the right hand side arises because \eqref{eqn:elbo_eqn3} holds on $\Omega_n$ with $\bbP_0^{(n)}(\Omega_n^{\rm c}) \leq 2p^{-1}$.

To complete the proof, we need a lower bound of $w_n(s_0)$.
Since $A_2 p^{-A_4} \leq 1/2$ by \eqref{assume:elbo_conditions_AS}, it is easy to see that $w_n(0) \geq 1/2$.
Since $w_n(s_0) \geq A_1^{s_0} p^{- A_3 s_0 } w_n(0)$ and \eqref{assume:elbo_conditions_AS} holds, we have
\begin{align*}
\log w_n(s_0) 
&\geq s_0 \log A_1 - A_3 s_0 \log p + \log w_n(0) \geq -s_0 \log p - A_3 s_0 \log p - \log 2 \\
& \geq -s_0 \log p - A_3 s_0 \log p - s_0\log p  
= - ( A_3 + 2 ) s_0 \log p.
\end{align*}
The proof is complete by taking  $K_{\rm elbo} = 2\alpha + m_2 + 4 + A_3 + 512 A_6 C_{\rm dev}$.
\end{proof}

\begin{theorem}[Effective dimension] \label{thm:effective_dim}
Suppose that conditions in Lemma \ref{lemma:elbo} hold. Also, assume that
\begin{align} \label{assume:sharp_A4}
    A_6p^{-A_7} + 4\log_{p} \big( \left[ A_{2} \vee 3\right] D_{n} \big) \leq A_{4}.
\end{align}
Then, there exists a constant $K_{\rm dim} \geq 2A_4^{-1} (K_{\rm elbo} + 2)$ such that
$$
\bbE \, \Pi_{\alpha}^n (\theta: |S_{\theta}| > K_{\rm dim} s_0) \leq (s_0 \log p)^{-1} + 2p^{-1} + p^{-s_0}.
$$
\end{theorem}
\begin{proof}
Let $\mathscr{D}_n(s) = \left\{ \theta \in \bbR^p : |S_{\theta}| > s \right\}$ for $s \in \bbN$ with $s \geq s_0$ and $\Omega_n$ be the event such that the results of Lemmas \ref{lemma:empirical_prior_bound} and \ref{lemma:elbo} hold. 
Then, $\bbP_0^{(n)} \left\{ \Omega_n^{\rm c} \right\} \leq (s_0 \log p)^{-1} + 2p^{-1}$.
Also,
\begin{align*}
\begin{aligned}
     \bbE \, \Pi_{\alpha}^n \{\scrD_n(s)\}
    \leq 
    \bbE \, \Pi_{\alpha}^n \{ \scrD_n(s)\} \mathds{1}_{\Omega_n}  + \bbP_0^{(n)}(\Omega_n^{\rm c}).   
\end{aligned}
\end{align*}
and
\begin{align}
\begin{aligned} \label{eqn:eff_dim_eqn2}  
    &\bbE \, \bigl( \Pi_{\alpha}^n \{ \scrD_n(s) \} \, \mathds{1}_{\Omega_n} \bigr)  \\
    &= 
    \bbE \left\{  \dfrac{
        \int_{\scrD_n(s)} \Lambda_{n}^{\alpha}(\theta) \rmd \Pi_{n}(\theta)
    }{
        \int_{\bbR^p} \Lambda_{n}^{\alpha}(\theta) \rmd \Pi_{n}(\theta)
    } \mathds{1}_{\Omega_n} \right\} \\
    &\leq 
    e^{K_{\operatorname{elbo}}s_0\log p} \, \bbE \left\{ \int_{\scrD_n (s)} \Lambda_{n}^{\alpha}(\theta) \rmd \Pi_{n}(\theta) \, \mathds{1}_{\Omega_n} \right\} \\
    &= 
    e^{K_{\operatorname{elbo}}s_0\log p} \, \bbE \left\{ 
        \sum_{S \in \scrS_{s_{\max}} : |S| > s}  \dfrac{w_n(|S|)}{\binom{p}{|S|}} \int_{\bbR^{|S|}} \Lambda_{n}^{\alpha}(\theta_S) g_{S} \left( \theta_S \right) \rmd \theta_S \, \mathds{1}_{\Omega_n}
    \right\} \\
    &\leq 
    e^{K_{\operatorname{elbo}}s_0\log p} \, \bbE \left\{ 
        \sum_{S \in \scrS_{s_{\max}} : |S| > s}  \dfrac{ w_n(|S|)}{\binom{p}{|S|}} 
        D_{n}^{2|S|} p^{ \lambda |S|/2} 
        \int_{\bbR^{|S|}} \Lambda_{n}^{\alpha}(\theta_S) \overline{g}_{S} \left( \theta_S \right) \rmd \theta_S
    \right\},     
\end{aligned}
\end{align}
where the first and second inequalities hold by Lemmas \ref{lemma:elbo} and \ref{lemma:empirical_prior_bound}, respectively.
Note that
\begin{align*}
    \int_{\bbR^{|S|}} \bbE \Lambda_{n}^{\alpha}(\theta_S) \overline{g}_{S} \left( \rmd \theta_S \right) 
    &= \int_{\bbR^{|S|}} \left[ \prod_{i=1}^{n} \int \left( \dfrac{p_{i, \theta_S} }{p_{i, \theta_0} } \right)^{\alpha}  p_{i, \theta_0}  \rmd \mu \right] \;  \overline{g}_{S} \left( \theta_S \right) \rmd \theta_S\\
    &= \int_{\bbR^{|S|}} \prod_{i=1}^{n} \left[ \int  p_{i, \theta_S}^{\alpha}  p_{i, \theta_0}^{1 - \alpha}  \rmd \mu \right]   \overline{g}_{S} \left( \theta_S \right)  \rmd \theta_S
    \\
    & \leq \int_{\bbR^{|S|}} \overline{g}_{S} \left( \theta_S \right)\rmd \theta_S \\
    & = 1,
\end{align*}
where the inequality holds because the Hellinger transform, $\int p_{1}^{\alpha_1} \cdots p_{N}^{\alpha_N} \rmd \mu $ for densities $p_1, \ldots, p_N$ with $\alpha_1 + \cdots + \alpha_N = 1$, is bounded by $1$; see Section B.2 of \citet{ghosal2017fundamentals}.
By applying Fubini theorem, \eqref{eqn:eff_dim_eqn2} is further bounded by
\begin{align} \label{eqn:eff_dim_eqn_3.5}
\begin{aligned}
    &e^{K_{\operatorname{elbo}}s_0\log p} \sum_{S \in \scrS_{s_{\max}}: |S| > s}  \dfrac{ w_n(|S|)}{\binom{p}{|S|}} D_{n}^{2|S|} p^{ \lambda |S|/2} \\
    &= e^{K_{\operatorname{elbo}}s_0\log p}  
    \sum_{S \in \scrS_{s_{\max}}: |S| > s}  \dfrac{ w_n(|S|)}{\binom{p}{|S|}} \exp \bigg( \left[ \lambda/2 \right] |S| \log p  + 2|S| \log D_{n} \bigg) \\
    &\leq
    e^{K_{\operatorname{elbo}}s_0\log p}
    \sum_{\tilde{s} > s}^{s_{\rm max}}  w_n(\tilde{s}) \exp \bigg( \left[A_6 p^{-A_7}/2\right] \tilde{s}\log p + 2 \tilde{s}\log D_{n} \bigg),    
\end{aligned}
\end{align}
where the last inequality holds by \eqref{A2:b}. 
Since \eqref{def:prior_S_penalty_main} imply that
\begin{align*}
    w_n(\tilde{s}) \leq \pi_p(0) A_2^{\tilde{s}} p^{-A_4 \tilde{s}} 
    \leq \left( A_2 p^{-A_4} \right)^{\tilde{s}} 
    = \exp \left( -A_4 \tilde{s} \log p + \tilde{s} \log A_{2} \right),
\end{align*}
\eqref{eqn:eff_dim_eqn_3.5} is further bounded by
\begin{align*}
&e^{K_{\operatorname{elbo}}s_0\log p}
\sum_{\tilde{s} > s}^{s_{\rm max}} 
\exp \bigg( - A_{4} \tilde{s} \log p + \tilde{s} \log A_{2} + \left[ A_6 p^{-A_7} /2\right] \tilde{s}\log p + 2 \tilde{s}\log D_{n} \bigg) \\
&\leq
\sum_{\tilde{s} > s} 
\exp \Bigg( \bigg[ \dfrac{A_6 p^{-A_7}}{2} + 2\log_{p} \big( \left\{ A_{2} \vee 3 \right\} D_{n} \big) - A_{4} \bigg] \tilde{s} \log p + K_{\rm elbo}s_0 \log p \Bigg) \\
&\leq 
\sum_{\tilde{s} > s} \exp \left\{ -\dfrac{A_4}{2} s \log p + K_{\rm elbo}s_0 \log p \right\},
\end{align*}
where the last inequality holds by \eqref{assume:sharp_A4}.
By taking $s = K_{\rm dim}s_0$ with $K_{\rm dim} \geq 2A_4^{-1} (K_{\rm elbo} + 2)$, the right hand side of the last display is equal to
\begin{align*}
\sum_{\tilde{s} \geq s} \exp \left\{ -2s_0 \log p \right\}
&\leq
p \exp \left\{ -2s_0 \log p \right\} 
= 
e^{-(2s_0 - 1) \log p} 
\leq p^{-s_0}. %\qedhere
\end{align*}
This completes the proof.
\end{proof}

Theorem \ref{thm:effective_dim}  implies that $\bbE \, \Pi_{\alpha}^n(\theta: S_\theta \in \scrS_{\rm eff}) \rightarrow 1$, where $\scrS_{\rm eff}$ is defined in \eqref{def:concentrated_support_eff_main}.
Let $\widetilde{s}_n = (K_{\rm dim} + 1)s_0$. Here, the additive $s_0$ arises from a technical reason. Specifically, we often consider the concatenated support $S_{\texttt{+}} = S \cup S_0$ for some $|S| \leq K_{\rm dim}s_0$ and statistical properties corresponding to $\theta$ with $S_{\theta} = S_{\texttt{+}}$. 

\begin{theorem}[Consistency in Hellinger distance] \label{thm:consistency_Hellinger}
Let $\epsilon_n = (  n^{-1}s_0 \log p )^{1/2}$.
Suppose that conditions in Theorem \ref{thm:effective_dim} hold and $\alpha \in (0, 1)$.
Then, there exists a constant $K_{\operatorname{Hel}}>0$ such that
\begin{align} \label{eqn:hel_dist_claim}
    \bbE \, \Pi_{\alpha}^n \{\theta: H_n(\theta, \theta_0)> K_{\operatorname{Hel}} \: \epsilon_n\} \leq  2(s_0 \log p)^{-1} + 4p^{-1} + 2p^{-s_0}     
\end{align}
\end{theorem}
\begin{proof}
Let $\Theta_{\rm eff} = \left\{ \theta \in \bbR^p : |S_{\theta}| \leq s_n \right\}$ and $\Omega_n$ is the event on which the results of Lemmas \ref{lemma:empirical_prior_bound} and \ref{lemma:elbo} hold.
By Lemmas \ref{lemma:empirical_prior_bound}, \ref{lemma:elbo} and Theorem \ref{thm:effective_dim}, we have 
\begin{align*}
    \bbE \, \Pi_{\alpha}^n(\Theta_{\rm eff}^{\rm c})
    + \bbP_{0}^{(n)}(\Omega_n^{\rm c})
    \leq 
    2(s_0 \log p)^{-1} + 4p^{-1} + p^{-s_0}.
\end{align*}
Also, for $\epsilon > 0$,
\begin{align}
\begin{aligned} \label{eqn:hel_dist_eqn1}
&\bbE \Pi_{\alpha}^n\{ \theta \in \bbR^p : H_{n} (\theta, \theta_0) > \epsilon\} \\
    &\leq \bbE \bigl[ \Pi_{\alpha}^n \{  \theta \in \Theta_{\rm eff} : H_{n} (\theta, \theta_0) > \epsilon\} \, \mathds{1}_{\Omega_n} \bigr] 
    + \bbE \Pi_{\alpha}^n(\Theta_{\rm eff}^{\rm c})  
    + \bbP_{0}^{(n)}(\Omega_n^{\rm c}) \\
    &\leq 
    e^{K_{\rm elbo} s_0 \log p } \: \bbE \bigg[ \int_{ \{ \theta \in \Theta_{\rm eff} : H_n(\theta, \theta_0)  > \epsilon \}}  \Lambda_{n}^{\alpha}(\theta) \, \Pi_{n}(\rmd\theta) \, \mathds{1}_{\Omega_n} \bigg] \\ 
    & \qquad + 2(s_0 \log p)^{-1} + 4p^{-1} + p^{-s_0},
\end{aligned}
\end{align}
where the second inequality holds by Lemma \ref{lemma:elbo}. 
By Lemma \ref{lemma:empirical_prior_bound}, the expected value of the term in the bracket in the right hand side of \eqref{eqn:hel_dist_eqn1} is bounded by
\begin{align} 
\begin{aligned} \label{eqn:hel_dist_eqn_Fubini}
\bbE & \left[
        \sum_{|S| \leq s_n} \int_{  \left\{ \theta_S \in \bbR^{|S|}: H ( \Tilde{\theta}_S, \theta_0) > \epsilon \right\} } 
        \Lambda_{n}^{\alpha}(\theta_S) D_{n}^{2|S|} p^{ \lambda |S|/2}  \dfrac{w_n(|S|)}{\binom{p}{|S|}} \overline{g}_S(\theta_S) \rmd \theta_S.
    \right] \\
    &\leq 
    p^{ (2 \log_{p}(D_n) + A_6/2) s_n }
    \bbE \left[
        \sum_{|S| \leq s_n} \int_{  \left\{ \theta_S \in \bbR^{|S|}: H ( \Tilde{\theta}_S, \theta_0) > \epsilon \right\} } 
        \Lambda_{n}^{\alpha}(\theta_S)  
         \dfrac{w_n(|S|)}{\binom{p}{|S|}} \overline{g}_S(\theta_S) \rmd \theta_S.
    \right] \\
    &\leq 
    \exp \big( \left[ 2 \log_{p}(D_n) + A_6/2 \right] K_{\rm dim} s_0 \log p \big)
    \int_{ \left\{ \theta \in \bbR^p : H_n(\theta, \theta_0)  > \epsilon \right\}}  \bbE \Lambda_{n}^{\alpha}(\theta) \,\overline{\Pi}( \rmd\theta),
\end{aligned}
\end{align}
where the second inequality holds by Fubini's theorem, and $\log_{p}(D_n) \leq A_4/4$ by \eqref{assume:sharp_A4}. 
Here, $\overline{\Pi}(\cdot)$ is the prior obtained from $\Pi$ by first replacing $g_S$ with $\overline g_S$ and then restricting and renormalizing it on $\Theta_{\rm eff}$.
Also,
\begin{align*}
    \bbE \, \Lambda_n^{\alpha}(\theta) = \int \prod_{i=1}^{n} p_{i, \theta}^{\alpha} p_{i, \theta_0}^{1-\alpha} \rmd \mu
    = \exp \left\{ \log  \prod_{i=1}^{n} \int p_{i, \theta}^{\alpha} p_{i, \theta_0}^{1-\alpha} \rmd \mu \right\}
    = \exp \left\{ -n R_{n, \alpha}(\theta, \theta_0) \right\},
\end{align*}
where $R_{n, \alpha}(\theta, \theta_0) = -n^{-1} \sum_{i=1}^{n} \log \int p_{i, \theta}^{\alpha} p_{i, \theta_0}^{1-\alpha} \rmd \mu$ is the averaged R\'enyi divergence of order $\alpha$. 
Since $\min\left\{ \alpha, 1 - \alpha \right\} H_n^2 (\theta, \theta_0) \leq R_{n, \alpha}(\theta, \theta_0)$ \citep[e.g.,][Lemma~B.5]{ghosal2017fundamentals}, we have
\begin{align*}
    -nR_{n, \alpha}(\theta, \theta_0) 
    \leq -n\min\left\{ \alpha, 1 - \alpha \right\} H_n^2 (\theta, \theta_0) 
    \leq -n\min\left\{ \alpha, 1 - \alpha \right\} \epsilon^2
\end{align*}
provided that $H_n(\theta, \theta_0)  > \epsilon$.
Hence, the right hand side of \eqref{eqn:hel_dist_eqn_Fubini} is equal to
\begin{align} \label{eqn:hel_dist_eqn2.5}
\begin{aligned}
    &e^{ (2 \log_{p}(D_n) + A_6/2)K_{\rm dim} s_0 \log p} 
    \int_{ \left\{ \theta \in \bbR^p : H_n(\theta, \theta_0)  > \epsilon \right\}}
    e^{-n R_{n, \alpha}(\theta, \theta_0)} \, \overline{\Pi}(\rmd \theta) \\
    &\leq 
    \exp \bigg( \left[ 2 \log_{p}(D_n) + A_6/2 \right] K_{\rm dim} s_0 \log p -\min\left\{ \alpha, 1 - \alpha \right\} n \epsilon^2 \bigg).
\end{aligned}
\end{align}
Therefore, \eqref{eqn:hel_dist_eqn1} is bounded by
\begin{align*} 
    &\exp \bigg[ 
    \big( K_{\rm elbo} + \left[ 2 \log_{p}(D_n) + A_6/2 \right] K_{\rm dim} \big) s_0 \log p - \min\left\{ \alpha, 1 - \alpha \right\} n \epsilon^2
    \bigg] \\
    &\qquad + 2(s_0 \log p)^{-1} + 4p^{-1} + p^{-s_0}.
\end{align*}
By taking $\epsilon$ and $K_{\operatorname{Hel}}$ as
\begin{align*}
    \epsilon &= \left\{ \left( K_{\rm elbo} + \left[ 2 \log_{p}(D_n) + A_6/2 \right] K_{\rm dim} + 1 \right) \min\left\{ \alpha, 1 - \alpha \right\}^{-1} \dfrac{s_0 \log p}{n} \right\}^{1/2}, \\
    K_{\operatorname{Hel}} &= \left\{ \left( K_{\rm elbo} + \left[ 2 \log_{p}(D_n) + A_6/2 \right] K_{\rm dim} + 1 \right) \min\left\{ \alpha, 1 - \alpha \right\}^{-1}\right\}^{1/2},
\end{align*}
this completes the proof of \eqref{eqn:hel_dist_claim}.
\end{proof}

\begin{lemma}[Lemma A1 in \cite{jeong2021posterior}]
    Let
    \begin{align*}
        h_{i}(\eta_{i, \theta}) 
        = H^2 \left( p_{i, \theta}, p_{i, \theta_0} \right) 
        = 
            1 - \exp \left\{ b\left( \dfrac{ \eta_{i, \theta} + \eta_{i, \theta_0}}{2} \right) - \dfrac{b(\eta_{i, \theta}) + b(\eta_{i, \theta_0})}{2} \right\},
    \end{align*}
    where $\eta_{i, \theta} = x_i^{\top} \theta$. Then,
    there exist constants $K_1, K_2 > 0$ such that
    \begin{align*}
        h_{i}(\eta_{i, \theta}) \geq h_{i}''(\eta_{i, \theta_0})
        \min \left\{ K_1 \left(x_i^{\top} \theta - x_i^{\top} \theta_0 \right)^2, K_2 \right\},
    \end{align*}
    where $h_{i}''$ is the second derivative of $\eta \mapsto h_i(\eta)$.
\end{lemma}
\begin{proof}
    See Lemma A1 in \cite{jeong2021posterior}.
\end{proof}

\begin{theorem}[Consistency in parameter $\theta$] \label{thm:consistency_parameter}
    Suppose that conditions in Theorem \ref{thm:consistency_Hellinger} hold and 
    \begin{align*} 
     \dfrac{8 (K_1 \vee 1) K_{\rm Hel}^2 (K_{\rm dim } + 1) }{K_2 \phi_1^{2} \left( \widetilde{s}_n ; \bW_0 \right)}  \| \bX \|_{\max}^{2} s_0^{2} \log p \leq n.
    \end{align*}
    Then, there exists a constant $K_{\operatorname{theta}}>0$ such that
    \begin{align*}
    \bbE \, \Pi_{\alpha}^n \left(\theta : \|\theta-\theta_0\|_1 > 
    \frac{K_{\operatorname{theta}}s_0 }{ \phi_1 \left( \widetilde{s}_n ; \bW_0 \right)} \sqrt{\dfrac{\log p}{n}} \right) & \leq  2(s_0 \log p)^{-1} + 4p^{-1} + 2p^{-s_0} \\
    \bbE \, \Pi_{\alpha}^n \left(\theta: \|\theta-\theta_0\|_2  > \frac{K_{\operatorname{theta}} }{ \phi_2 \left( \widetilde{s}_n ; \bW_0 \right)} \sqrt{\dfrac{s_0 \log p}{n}} \right) &\leq  2(s_0 \log p)^{-1} + 4p^{-1} + 2p^{-s_0} \\
    \bbE \, \Pi_{\alpha}^n \bigl(\theta: \|\bW_0^{1/2}\mathbf{X} (\theta-\theta_0 ) \|_2^2 > K_{\operatorname{theta}} s_0 \log p \bigr) &\leq  2(s_0 \log p)^{-1} + 4p^{-1} + 2p^{-s_0}.
    \end{align*}
\end{theorem}
\begin{proof}
    Based on Theorem \ref{thm:consistency_Hellinger}, the proof for Theorem \ref{thm:consistency_parameter} aligns with Theorem 3 provided by \cite{jeong2021posterior}. We refer the reader there for details.
\end{proof}

\begin{lemma} \label{lemma:extended_Fisher_smooth_general}
    For $S \in \scrS_{s_{\max}}$, assume that $\bF_{n, \thetaBest}$ is nonsingular.
    Then, for any $R > 0$ and $\theta_{S} \in \localSetRn[S]{ R }$,
    \begin{align*}
     (1 - \overline{\delta}_{n, S, R})  \bF_{n, \thetaBest[S]} \preceq 
    \bF_{n, \theta_S} 
    \preceq (1 + \overline{\delta}_{n, S, R}) \bF_{n, \thetaBest[S]}, 
    \end{align*}                                                            
    where 
    \begin{align} \label{def:delta_nSR}
        \overline{\delta}_{n, S, R} = 
        \left[ \sup_{\theta_S \in \localSetRn[S]{ R }} \max_{i \in [n]} \exp \left( 3 \left| x_{i, S}^{\top} \left[ \theta_S - \thetaBest \right]  \right| \right) \right]
        \designRegular R.
    \end{align}
\end{lemma}

\begin{proof}
Since the proof of this Lemma is similar to Lemma \ref{lemma:Smoothness_of_the_Fisher_information_operator}, we provide a sketch of the proof.
Let $\theta_{S} \in \localSetRn[S]{ R }$. 
Note that 
\begin{align*}
    \bF_{n, \theta_S} -  \bF_{n, \thetaBest}
    = \sum_{i=1}^{n} \left\{ b''( x_{i, S}^{\top} \theta_{S} ) - b''( x_{i, S}^{\top} \thetaBest ) \right\} x_{i, S}x_{i, S}^{\top}.
\end{align*}
By Taylor's theorem, there exists $\theta_S^{\circ}(i) \in \localSetRn[S]{ R }$, on the line segment between $\theta_S$ and $\thetaBest$, such that
\begin{align*}
    &\left| b''(x_{i, S}^{\top}\theta_S) - b''(x_{i, S}^{\top}\thetaBest[S]) \right|
    = \dfrac{\left| b'''(x_{i, S}^{\top} \theta_{S}^{\circ}(i)) \right| }{b''(x_{i, S}^{\top} \thetaBest[S])	}
    \left| x_{i, S}^{\top}\theta_S - x_{i, S}^{\top}\thetaBest[S] \right| b''(x_{i, S}^{\top} \thetaBest[S]) \\
    &\leq \dfrac{ b''(x_{i, S}^{\top} \theta_{S}^{\circ}(i)) }{b''(x_{i, S}^{\top} \thetaBest[S])	}
    \left| x_{i, S}^{\top}\theta_S - x_{i, S}^{\top}\thetaBest[S] \right| b''(x_{i, S}^{\top} \thetaBest[S]) \\
    &\leq \exp \left( 3 \left| x_{i, S}^{\top} \left[ \theta_{S}^{\circ}(i) - \thetaBest \right] \right| \right)
    \left| x_{i, S}^{\top}\theta_S - x_{i, S}^{\top}\thetaBest[S] \right| 
    b''\left(x_{i, S}^{\top} \thetaBest[S] \right),         
\end{align*}
where the inequalities hold by $| b'''| \leq b''$ \citep[e.g.,][Sec.~2.1]{ostrovskii2021finite} and Lemma \ref{lemma:GLM_b_ratio}. 
Since
\begin{align*}
    \left| x_{i, S}^{\top}\theta_S - x_{i, S}^{\top}\thetaBest  \right|
    = 
    \left| \left( \bF_{n, \thetaBest}^{-1/2}x_{i, S} \right)^{\top}  \bF_{n, \thetaBest}^{1/2}\left( \theta_S - \thetaBest \right) \right| 
    \leq \designRegular R,
\end{align*}
we have
\begin{align*}
    \left| b''( x_{i, S}^{\top} \theta_{S} ) - b''( x_{i, S}^{\top} \thetaBest ) \right|
    &\leq 
    \Bigg(
    \left[ \sup_{\theta_S \in \localSetRn[S]{ R }} \max_{i \in [n]} \exp \left( 3 \left| x_{i, S}^{\top} \left[ \theta_S - \thetaBest \right]  \right| \right) \right]
    \designRegular R 
    \Bigg)
    b''\left(x_{i, S}^{\top} \thetaBest[S] \right) \\
    &= \overline{\delta}_{n, S, R} \ b'' (x_{i, S}^{\top} \thetaBest).
\end{align*}
Therefore,
\begin{align*}
   -\overline{\delta}_{n, S, R} \sum_{i=1}^{n} b''( x_{i, S}^{\top} \thetaBest) x_{i, S}x_{i, S}^{\top}   
   \preceq
   \bF_{n, \theta_S} -  \bF_{n, \thetaBest} 
   \preceq
   \overline{\delta}_{n, S, R} \sum_{i=1}^{n} b''( x_{i, S}^{\top} \thetaBest) x_{i, S}x_{i, S}^{\top},   
\end{align*}
completing the proof.
\end{proof}

\begin{lemma}[Misspecification on $\scrS_{\Theta_n}$] \label{lemma:mis_on_posterior_concentration_set}
Suppose that 
\begin{align} \label{eqn:mis_posterior_concentration_cond}
    n \geq \left[ \dfrac{200 K_{\rm theta}( K_{\rm dim} + 1 )}{ \phi_{2}^{2} \left( \widetilde{s}_n ; \bW_0 \right)}  \left( \left\| \bX \right\|_{\max}^2 \vee 1 \right) \right] s_0^2 \log p.
\end{align}
Then, 
\begin{align}
\begin{aligned} \label{eqn:mis_posterior_concentration_claim1}
    \max_{S \in \scrS_{\Theta_n}} \left\| \bF_{n, \theta_0}^{1/2} \left( \widetilde{\theta}_{S}^{\ast} - \theta_0 \right) \right\|_{2}^{2} &\leq  8 K_{\rm theta} s_0 \log p, \\
    \max_{S \in \scrS_{\Theta_n}} \left\{ \Delta_{{\rm mis}, S} \vee \widetilde{\Delta}_{{\rm mis}, S} \right\} 
    &\leq \exp \left( C \dfrac{\left\| \bX \right\|_{\max}}{\phi_{2}(\widetilde{s}_n ; \bW_0) } \left[ \dfrac{ s_0^{2} \log p }{n} \right]^{1/2}  \right)   
    \leq 2,    
\end{aligned}
\end{align}
where $\widetilde{\Delta}_{{\rm mis}, S} = \| \bV_{n, S}^{-1/2} \bF_{n, \thetaBest} \bV_{n, S}^{-1/2}  \|_{2}$ and $C = C(K_{\rm dim}, K_{\rm theta}) > 0$. 
\end{lemma}

\begin{proof}
Let $S \in \scrS_{\Theta_n}$, $S_{\texttt{+}} = S \cup S_0$ and $R_n = 8K_{\rm theta} s_0 \log p$. 
Note that
\begin{align*}
   \designRegular[S_{\texttt{+}}] \leq \rho_{\min, S_{\texttt{+}}}^{-1/2} \max_{i \in [n]} \left\| x_{i, S_{\texttt{+}}} \right\|_{2}
   &\leq \dfrac{(K_{\rm dim} + 1)^{1/2} }{\phi_{2}(\widetilde{s}_n ; \bW_0 )} \left( \dfrac{s_0 \| \bX \|_{\max}^{2} }{n} \right)^{1/2}, \\
   \sup_{\theta_{S_{\texttt{+}}} \in \localSetRn[S_{\texttt{+}}]{ R_n }}
   \left\| \bF_{n, \thetaBest[S_{\texttt{+}}]}^{1/2} \left[ \theta_{S_{\texttt{+}}} - \thetaBest[S_{\texttt{+}}] \right]  \right\|_{2}^2 &\leq 8K_{\rm theta} s_0 \log p.
\end{align*}
For $\theta_{S_{\texttt{+}}} \in \localSetRn[S_{\texttt{+}}]{ R_n }$, we have
\begin{align} 
\begin{aligned} \label{eqn:mis_posterior_concentration_tech_1}
&\max_{i \in [n]} \left| x_{i, S_{\texttt{+}}}^{\top} \left[ \theta_{S_{\texttt{+}}}  - \thetaBest[S_{\texttt{+}}] \right] \right|
=
\max_{i \in [n]} \left| x_{i, S_{\texttt{+}}}^{\top} \bF_{n, \thetaBest[S_{\texttt{+}}]}^{-1/2} \bF_{n, \thetaBest[S_{\texttt{+}}]}^{1/2} \left[ \overline{\theta}_S^{\ast} - \thetaBest[S_{\texttt{+}}] \right] \right| \\
&\leq
\max_{i \in [n]} \left\| \bF_{n, \thetaBest[S_{\texttt{+}}]}^{-1/2} x_{i, S_{\texttt{+}}} \right\|_{2}  \left\| \bF_{n, \thetaBest[S_{\texttt{+}}]}^{1/2} \left[ \theta_{S_{\texttt{+}}} - \thetaBest[S_{\texttt{+}}] \right]  \right\|_{2} \\
&\leq \designRegular[S_{\texttt{+}}] R_n 
= \designRegular[S_{\texttt{+}}] \left( 8 K_{\rm theta} s_0 \log p \right) \\
&\leq \left[ \dfrac{(K_{\rm dim} + 1)^{1/2} }{\phi_{2}(\widetilde{s}_n ; \bW_0 )} \left( \dfrac{s_0 \| \bX \|_{\max}^{2} }{n} \right)^{1/2} \right] \bigg( 8 K_{\rm theta} s_0 \log p \bigg)^{1/2}
\leq 1/5,
\end{aligned}
\end{align}
where the last inequality holds by \eqref{eqn:mis_posterior_concentration_cond}. 
Recall that $\overline{\delta}_{n, S, R}$ defined in \eqref{def:delta_nSR}. By the last display, we have
\begin{align*}
\overline{\delta}_{n, S_{\texttt{+}}, R_n} 
&= 
\left[ \sup_{\theta_{S_{\texttt{+}}} \in \localSetRn[S_{\texttt{+}}]{ R_n }} \max_{i \in [n]} \exp \left( 3 \left| x_{i, S}^{\top} \left[ \theta_{S_{\texttt{+}}} - \thetaBest[S_{\texttt{+}}] \right]  \right| \right) \right]
\designRegular[S_{\texttt{+}}] R_n \\
&\leq
e^{3/5}/5 \leq 1/2, 
\end{align*}
which completes the proof of $\max_{S \in \widetilde{\scrS}_{\Theta_n} } \overline{\delta}_{n, S, R_n} \leq 1/2$, where $\widetilde{\scrS}_{\Theta_n} = \left\{ S \cup S_0 : S \in \scrS_{\Theta_n} \right\}$.

By the definition of $\scrS_{\Theta_n}$, there exists a parameter $\theta_S^{\circ} \in \bbR^{|S|}$ such that
\begin{align*}
    \left\| \bF_{n, \theta_0}^{1/2} \left( \widetilde{\theta}_S^{\circ} - \theta_0 \right)  \right\|_{2} \leq K_{\rm theta} s_0 \log p.
\end{align*}
Given a suitable ordering of the indices, let $\overline{\theta}_S^{\ast} = (\overline{\theta}_{j}^{\ast})_{j=1}^{|S_{\texttt{+}}|}$, where $\overline{\theta}_{j}^{\ast} = \theta_{S, j}^{\ast}$ for $j \in S$ and $\overline{\theta}_{j}^{\ast} = 0$ for $j \in S_{\texttt{+}} \setminus S$.
Let us define $\overline{\theta}_{S}^{\circ} \in \bbR^{|S_{\texttt{+}}|}$ as we define $\overline{\theta}_S^{\ast}$. Then, we have
\begin{align*}
    \left\| \bF_{n, \theta_0}^{1/2} \left( \widetilde{\theta}_{S}^{\ast} - \theta_0 \right) \right\|_{2} &= 
    \left\| \bF_{n, \thetaBest[S_{\texttt{+}}]}^{1/2} \left( \overline{\theta}_{S}^{\ast} - \thetaBest[S_{\texttt{+}}] \right) \right\|_{2}, \\
    \left\| \bF_{n, \theta_0}^{1/2} \left( \widetilde{\theta}_S^{\circ} - \theta_0 \right) \right\|_{2} &= 
    \left\| \bF_{n, \thetaBest[S_{\texttt{+}}]}^{1/2} \left( \overline{\theta}_{S}^{\circ} - \thetaBest[S_{\texttt{+}}] \right) \right\|_{2}
\end{align*}

We will prove the first assertion in \eqref{eqn:mis_posterior_concentration_claim1} by the contradiction.
Suppose that
\begin{align*}
    \left\| \bF_{n, \thetaBest[S_{\texttt{+}}]}^{1/2} \left( \overline{\theta}_{S}^{\ast} - \thetaBest[S_{\texttt{+}}] \right) \right\|_{2}^{2} > R_n.
\end{align*}
For $\theta_S \in \bbR^{|S|}$, let $\mathbb{L}_{n, \theta_S} = \bbE L_{n, \theta_S} = \sum_{i = 1}^{n} b'(X_{i}^{\top} \theta_0)X_{i, S}^{\top} \theta_S - b(X_{i, S}^{\top} \theta_S)$ and $\dot{\mathbb{L}}_{n, \theta_S} = \bbE \dot{L}_{n, \theta_S}$. 
To prove \eqref{eqn:mis_posterior_concentration_claim1}, firstly we will obtain an upper bound of $\mathbb{L}_{n, \thetaBest} - \mathbb{L}_{n, \thetaBest[S_{\texttt{+}}]}$.
Let 
\begin{align*}
    \partial \localSetRn[S_{\texttt{+}}]{ R_n } = \left\{\theta_{S_{\texttt{+}}} \in \bbR^{|S_{\texttt{+}}|} : \left\| \bF_{n, \thetaBest[S_{\texttt{+}}]}^{1/2} \left( \theta_{S_{\texttt{+}}} - \thetaBest[S_{\texttt{+}}] \right)  \right\|_{2}^{2} = R_n \right\}.
\end{align*}
Let $\check{\theta}_{S_{\texttt{+}}} \in \partial \localSetRn[S_{\texttt{+}}]{ R_n }$.
By Taylor's theorem, there exists $\widetilde{\theta}_{S_{\texttt{+}}} \in \localSetRn[S_{\texttt{+}}]{ R_n } $ such that
\begin{align*}
    \mathbb{L}_{n, \check{\theta}_{S_{\texttt{+}}} } - \mathbb{L}_{n, \thetaBest[S_{\texttt{+}}]}
    &= \left( \check{\theta}_{S_{\texttt{+}}} - \thetaBest[S_{\texttt{+}}] \right)^{\top} \dot{\mathbb{L}}_{n, \thetaBest[S_{\texttt{+}}]}
     -\dfrac{1}{2} ( \check{\theta}_{S_{\texttt{+}}} - \thetaBest[S_{\texttt{+}}])^{\top} \bF_{n, \widetilde{\theta}_{S_{\texttt{+}}}}  ( \check{\theta}_{S_{\texttt{+}}} - \thetaBest[S_{\texttt{+}}]) \\
    &= -\dfrac{1}{2} ( \check{\theta}_{S_{\texttt{+}}} - \thetaBest[S_{\texttt{+}}])^{\top} \bF_{n, 
    \widetilde{\theta}_{S_{\texttt{+}}}}  ( \check{\theta}_{S_{\texttt{+}}} - \thetaBest[S_{\texttt{+}}]) \\
    &\leq -\dfrac{1- \overline{\delta}_{n, S_{\texttt{+}}, R_n}}{2} \left\|  \bF_{n, \thetaBest[S_{\texttt{+}}] }^{1/2} \left( \check{\theta}_{S_{\texttt{+}}} - \thetaBest[S_{\texttt{+}}] \right)   \right\|_2^2 \quad (\because \text{ Lemma \ref{lemma:extended_Fisher_smooth_general}}) \\
    &\leq -\dfrac{1}{4} \left\|  \bF_{n, \thetaBest[S_{\texttt{+}}]}^{1/2} \left( \check{\theta}_{S_{\texttt{+}}} - \thetaBest[S_{\texttt{+}}] \right)   \right\|_2^2. \quad (\because \overline{\delta}_{n, S_{\texttt{+}}, R_n} \leq 1/2) \\
    &=  -\dfrac{ R_n }{4}.
\end{align*} 
Since $\theta \mapsto \mathbb{L}_{n, \theta}$ is concave, for any $\theta_{S_{\texttt{+}}} \in \left[ \localSetRn[S]{ R_n } \right]^{\rm c}$,
\begin{align*}
    \mathbb{L}_{n, \underline{\theta}_S} \geq \omega \mathbb{L}_{n, \theta_{S_{\texttt{+}}}} + (1 - \omega) \mathbb{L}_{n, \thetaBest[S_{\texttt{+}}]},
\end{align*}
where $\omega = \sqrt{ R_n} / \| \bF_{n, \thetaBest[S_{\texttt{+}}]}^{1/2} \left( \theta_{S_{\texttt{+}}} - \thetaBest[S_{\texttt{+}}] \right) \|_2$ and $\underline{\theta}_{S_{\texttt{+}}} = \omega \theta_{S_{\texttt{+}}} + (1 - \omega) \thetaBest[S_{\texttt{+}}] \in \partial\localSetRn[S_{\texttt{+}}]{ R_n }$.
Hence, 
\begin{align*}
-\dfrac{ R_n }{4}
\geq 
\sup_{ \check{\theta}_{S_{\texttt{+}}} \in \partial \localSetRn[S_{\texttt{+}}]{ R_n }} \mathbb{L}_{n, \check{\theta}_{S_{\texttt{+}}} } - \mathbb{L}_{n, \thetaBest[S_{\texttt{+}}]}
\geq 
\omega \left( \mathbb{L}_{n, \theta_{S_{\texttt{+}}}} - \mathbb{L}_{n, \thetaBest[S_{\texttt{+}}]} \right)
\geq 
\mathbb{L}_{n, \theta_{S_{\texttt{+}}}} - \mathbb{L}_{n, \thetaBest[S_{\texttt{+}}]}
\end{align*}
for all $\theta_{S_{\texttt{+}}} \notin \localSetRn[S_{\texttt{+}}]{ R_n }$.
Since we assume that $\overline{\theta}_{S}^{\ast} \notin \localSetRn[S_{\texttt{+}}]{ R_n }$, therefore, we have
\begin{align} \label{eqn:mis_posterior_concentration_eq1}
\mathbb{L}_{n, \thetaBest} - \mathbb{L}_{n, \theta_0} 
= \mathbb{L}_{n, \overline{\theta}_{S}^{\ast}} - \mathbb{L}_{n, \thetaBest[S_{\texttt{+}}]}
\leq -\dfrac{ R_n }{4}.
\end{align}

Secondly, we will obtain the lower bound of $\mathbb{L}_{n, \overline{\theta}_S^{\circ} } - \mathbb{L}_{n, \thetaBest[S_{\texttt{+}}]}$. 
Since $\overline{\theta}_S^{\circ} \in \localSetRn[S_{\texttt{+}}]{ R_n }$, by Taylor's theorem, there exists $\widetilde{\theta}_{S_{\texttt{+}}} \in \localSetRn[S_{\texttt{+}}]{ R_n }$ such that
\begin{align} 
\begin{aligned}\label{eqn:mis_posterior_concentration_eq2}
    \mathbb{L}_{n, \overline{\theta}_S^{\circ} } - \mathbb{L}_{n, \thetaBest[S_{\texttt{+}}]}
    &= \left( \overline{\theta}_S^{\circ} - \thetaBest[S_{\texttt{+}}] \right)^{\top} \dot{\mathbb{L}}_{n, \thetaBest[S_{\texttt{+}}]}
     -\dfrac{1}{2} ( \overline{\theta}_S^{\circ} - \thetaBest[S_{\texttt{+}}])^{\top} \bF_{n, \widetilde{\theta}_{S_{\texttt{+}}}}  ( \overline{\theta}_S^{\circ} - \thetaBest[S_{\texttt{+}}]) \\
    &= -\dfrac{1}{2} ( \check{\theta}_{S_{\texttt{+}}} - \thetaBest[S_{\texttt{+}}])^{\top} \bF_{n, 
    \widetilde{\theta}_{S_{\texttt{+}}}}  ( \check{\theta}_{S_{\texttt{+}}} - \thetaBest[S_{\texttt{+}}]) \\
    &\geq -\dfrac{1 + \overline{\delta}_{n, S_{\texttt{+}}, R_n} }{2} \left\|  \bF_{n, \thetaBest[S_{\texttt{+}}] }^{1/2} \left( \overline{\theta}_S^{\circ} - \thetaBest[S_{\texttt{+}}] \right)   \right\|_2^2 \quad (\because \text{ Lemma \ref{lemma:extended_Fisher_smooth_general}}) \\
    &\geq - \left\|  \bF_{n, \thetaBest[S_{\texttt{+}}]}^{1/2} \left( \overline{\theta}_S^{\circ} - \thetaBest[S_{\texttt{+}}] \right)   \right\|_2^2. \quad (\because \overline{\delta}_{n, S_{\texttt{+}}, R_n} \leq 1) \\
    &\geq  - K_{\rm theta} s_0 \log p.    
\end{aligned}
\end{align} 

Combining \eqref{eqn:mis_posterior_concentration_eq1} and \eqref{eqn:mis_posterior_concentration_eq2}, we have
\begin{align*}
    \mathbb{L}_{n, \thetaBest[S_{\texttt{+}}]} - K_{\rm theta} s_0 \log p 
    &\leq \mathbb{L}_{n, \overline{\theta}_S^{\circ} } = \mathbb{L}_{n, \theta_S^{\circ} } 
    \overset{\eqref{def:MLE_Best_main}}{\leq} \mathbb{L}_{n, \thetaBest[S]} 
    = \mathbb{L}_{n, \overline{\theta}_{S}^{\ast}}
    \leq \mathbb{L}_{n, \thetaBest[S_{\texttt{+}}]} -\dfrac{ R_n }{4} \\
    &= \mathbb{L}_{n, \thetaBest[S_{\texttt{+}}]} - 2K_{\rm theta} s_0 \log p, 
\end{align*}
which yields the contradiction. This completes the proof of the first assertion in \eqref{eqn:mis_posterior_concentration_claim1}.

Next, we will prove $\max_{S \in \scrS_{\Theta_n}} \Delta_{{\rm mis}, S} \leq 2$. For $S \in \scrS_{\Theta_n}$, note that
\begin{align*}
    \bV_{n, S} 
    &= \sum_{i = 1}^{n} b''\left( x_{i}^{\top} \theta_0 \right) x_{i, S}x_{i, S}^{\top} 
    = \sum_{i = 1}^{n} \dfrac{b'' ( x_{i, S_{\texttt{+}}}^{\top} \thetaBest[S_{\texttt{+}}] )}{b'' ( x_{i, S_{\texttt{+}}}^{\top} \overline{\theta}_{S}^{\ast} )} b''\left( x_{i, S}^{\top} \thetaBest \right) x_{i, S}x_{i, S}^{\top}  \\
    &\preceq \max_{i \in [n]}  \exp \left( 3 \left| x_{i, S_{\texttt{+}}}^{\top} \left[ \overline{\theta}_S^{\ast} - \thetaBest[S_{\texttt{+}}] \right] \right| \right) \sum_{i = 1}^{n} b''\left( x_{i, S}^{\top} \thetaBest \right) x_{i, S}x_{i, S}^{\top} \\
    &= \max_{i \in [n]} \exp \left( 3 \left| x_{i, S_{\texttt{+}}}^{\top} \left[ \overline{\theta}_S^{\ast} - \thetaBest[S_{\texttt{+}}] \right] \right| \right) \bF_{n, \thetaBest}
\end{align*}
where $S_{\texttt{+}} = S \cup S_0$ and the first inequality holds by Lemma \ref{lemma:GLM_b_ratio}. 
By similar technique in \eqref{eqn:mis_posterior_concentration_tech_1}, we have
\begin{align} 
\begin{aligned} \label{eqn:mis_pred_error_app}
\max_{i \in [n]} \exp \left( 3 \left| x_{i, S_{\texttt{+}}}^{\top} \left[ \overline{\theta}_S^{\ast} - \thetaBest[S_{\texttt{+}}] \right] \right| \right) 
\leq \exp \big( 3 \designRegular[S_{\texttt{+}}] R_n \big) 
\leq \exp(3/5) \leq 2.
\end{aligned}
\end{align}
This completes the proof of $\max_{S \in \scrS_{\Theta_n}} \Delta_{{\rm mis}, S} \leq 2$.

The proof of $\max_{S \in \scrS_{\Theta_n}} \widetilde{\Delta}_{{\rm mis}, S} \leq 2$ is similar. Hence, we will give a sketch of the proof. For $S \in \scrS_{\Theta_n}$, note that
\begin{align*}
    \bF_{n, \thetaBest} 
    &= \sum_{i = 1}^{n} b''\left( x_{i, S}^{\top} \thetaBest \right) x_{i, S}x_{i, S}^{\top} 
    = \sum_{i = 1}^{n} 
    \dfrac{b'' ( x_{i, S_{\texttt{+}}}^{\top} \overline{\theta}_{S}^{\ast} )}{ b'' ( x_{i, S_{\texttt{+}}}^{\top} \thetaBest[S_{\texttt{+}}] )} 
    b''\left(  x_{i, S_{\texttt{+}}}^{\top} \thetaBest[S_{\texttt{+}}] \right) x_{i, S}x_{i, S}^{\top}  \\
    &\preceq \max_{i \in [n]}  \exp \left( 3 \left| x_{i, S_{\texttt{+}}}^{\top} \left[ \overline{\theta}_S^{\ast} - \thetaBest[S_{\texttt{+}}] \right] \right| \right) 
    \sum_{i = 1}^{n} b''\left(  x_{i, S_{\texttt{+}}}^{\top} \thetaBest[S_{\texttt{+}}] \right) x_{i, S}x_{i, S}^{\top} \\
    &= \max_{i \in [n]} \exp \left( 3 \left| x_{i, S_{\texttt{+}}}^{\top} \left[ \overline{\theta}_S^{\ast} - \thetaBest[S_{\texttt{+}}] \right] \right| \right) \bV_{n, S}
    \preceq 2 \bV_{n, S},
\end{align*}
which completes the proof.
\end{proof}

\section{Laplace approximation} \label{sec:laplace_approximation_app}

For a given sequence $(M_n)$, define
\begin{align} \label{def:M_n}
    \widetilde{r}_{p, s} = (M_n^2 s \log p)^{1/2}.
\end{align}
By Lemma \ref{lemma:extended_Fisher_smooth}, for all $S \in \scrS_{\Theta_n}$, we have $\localSetRn[S]{r_{p, S}} \subset \localSetRn[S]{ \widetilde{r}_{p, |S|}}$ provided that $M_n^2 > 2C_{\rm radius}$, where $C_{\rm radius}$ is the constant specified in \eqref{eqn:mle_concentration}.
Therefore, the assertion of the following lemma is slightly more general than that of Lemma \ref{lemma:Smoothness_of_the_Fisher_information_operator}. 
Hereafter, note that $M_n$ can be regarded as an arbitrarily large constant.

Recall the following definitions:
\begin{align*} 
    \widetilde{\scrS}_{\Theta_n} = \left\{ S \cup S_0 : S \in \scrS_{\Theta_n} \right\}, \quad 
    \overline{\scrS}_{\Theta_n} = \scrS_{\Theta_n} \cup \widetilde{\scrS}_{\Theta_n}.
\end{align*}

\begin{lemma} \label{lemma:extended_Fisher_smooth}
Suppose that 
\begin{align} \label{assume:extended_Fisher_smooth}
    n \geq \left[ \dfrac{ 200K_{\rm theta}( K_{\rm dim} + 1 )}{ \phi_{2}^{2} \left( \widetilde{s}_n ; \bW_0 \right)}
    \left( \left\| \bX \right\|_{\max}^2 \vee 1 \right) \right] s_0^2 \log p, \quad 
    \max_{S \in \overline{\scrS}_{\Theta_n}} \widetilde{r}_{p, |S|} \designRegular \leq 1/5.
\end{align}
Also, assume that there exists a constant $K_{\rm cubic} > 0$ such that
\begin{align} \label{assume:extended_Fisher_smooth_K_cubic}
 \max_{S \in \overline{\scrS}_{\Theta_n}} \sup_{u_S \in \cU_{S}} 
 \dfrac{1}{n} \sum_{i = 1}^{n} \left| x_{i, S}^{\top} u_{S} \right|^3 
 \leq K_{\rm cubic}.
\end{align}
Then, for any $\theta_{S} \in \localSetRn[S]{ \widetilde{r}_{p, |S|} }$ and $S \in \overline{\scrS}_{\Theta_n}$,
\begin{align} \label{eqn:fisher_smooth_ineq_extended}
    (1 - \widetilde{\delta}_{n,S}) \bF_{n, \thetaBest[S]} \preceq 
    \bF_{n, \theta_S} \preceq 
    (1 + \widetilde{\delta}_{n,S}) \bF_{n, \thetaBest[S]},
\end{align}
where 
\begin{align*}
    \widetilde{\delta}_{n, S} 
    = 
    \bigg( 2 \widetilde{r}_{p, |S|} \designRegular \bigg) \wedge 
    \Bigg( \left\{ \dfrac{8\sqrt{2} K_{\rm cubic}}{\phi_{2}^{3}\left( \widetilde{s}_n ; \bW_0 \right)} \sigma_{\max}^2 \right\} \widetilde{r}_{p, |S|} n^{-1/2} \Bigg).
\end{align*}
\end{lemma}

\begin{proof}
Since the assumed conditions imply the sufficient conditions in Lemma \ref{lemma:mis_on_posterior_concentration_set}, we have
\begin{align} \label{eqn:extended_Fisher_smooth_eq0}
    \max_{S \in \scrS_{\Theta_n}} 
    \{ \Delta_{{\rm mis}, S} \vee \widetilde{\Delta}_{{\rm mis}, S} \} 
    \leq 2.    
\end{align}
Let $S \in \overline{\scrS}_{\Theta_n}$.
For given $\theta_{S} \in \localSetRn[S]{ \widetilde{r}_{p, |S|} }$, 
\begin{align*}
  \bF_{n, \theta_S} - \bF_{n, \thetaBest[S]} 
  &= \sum_{i=1}^{n} \left\{ b''(x_{i, S}^{\top}\theta_S) - b''(x_{i, S}^{\top}\thetaBest[S]) \right\}x_{i, S}x_{i, S}^{\top}.
\end{align*}    
By Taylor's theorem, there exists $\theta_{S}^{\circ}(i) \in \localSetRn[S]{ \widetilde{r}_{p, |S|} }$ on the line segment between $\theta_{S}$ and $\thetaBest[S]$ such that
\begin{align} 
 \begin{aligned} \label{eqn:taylor_formular_eq_extended}
    &\left| b''(x_{i, S}^{\top}\theta_S) - b''(x_{i, S}^{\top}\thetaBest[S]) \right|
    = \dfrac{\left| b'''(x_{i, S}^{\top} \theta_{S}^{\circ}(i)) \right| }{b''(x_{i, S}^{\top} \thetaBest[S])	}
    \left| x_{i, S}^{\top}\theta_S - x_{i, S}^{\top}\thetaBest[S] \right| b''(x_{i, S}^{\top} \thetaBest[S]) \\
    &\leq \dfrac{ b''(x_{i, S}^{\top} \theta_{S}^{\circ}(i)) }{b''(x_{i, S}^{\top} \thetaBest[S])	}
    \left| x_{i, S}^{\top}\theta_S - x_{i, S}^{\top}\thetaBest[S] \right| b''(x_{i, S}^{\top} \thetaBest[S]) \\
    &\leq \exp \left( 3 \left| x_{i, S}^{\top} \left[ \theta_{S}^{\circ}(i) - \thetaBest \right] \right| \right)
    \left| x_{i, S}^{\top}\theta_S - x_{i, S}^{\top}\thetaBest[S] \right| b''(x_{i, S}^{\top} \thetaBest[S]),     
 \end{aligned}
\end{align}
where the inequalities hold by $| b'''| \leq b''$ \citep[e.g.,][Sec.~2.1]{ostrovskii2021finite} and Lemma \ref{lemma:GLM_b_ratio}. 
Also, we have
	\begin{align}
		\begin{aligned} \label{eqn:fit_value_ineq_extended}
        \left| x_{i, S}^{\top}\theta_{S} - x_{i, S}^{\top}\thetaBest[S] \right| 
		&= \left| \left\{\bF_{n, \thetaBest[S]}^{-1/2}x_{i, S}\right\}^{\top} \bF_{n, \thetaBest[S]}^{1/2}\left(\theta_{S} - \thetaBest[S]\right) \right| \\
		&\leq \widetilde{r}_{p, |S|} \left\| \bF_{n, \thetaBest[S]}^{-1/2}x_{i, S} \right\|_{2} 
		\leq \widetilde{r}_{p, |S|} \designRegular,
		\end{aligned}
	\end{align}
where two inequalities in the second line hold by the definitions of $\localSetRn[S]{\widetilde{r}_{p, |S|}}$ and $\zeta_{n, S}$.
The last display implies that
\begin{align*}
    \exp \left( 3 \left| x_{i, S}^{\top} \left[ \theta_{S}^{\circ}(i) - \thetaBest \right] \right| \right)
    \leq 
    \exp( 3\widetilde{r}_{p, |S|} \designRegular )
    \leq 
    e^{3/5}
    \leq 2,
\end{align*}
where the second inequality holds by \eqref{assume:extended_Fisher_smooth}.
By \eqref{eqn:taylor_formular_eq_extended} and \eqref{eqn:fit_value_ineq_extended}, we have
\begin{align*}
    \max_{i \in [n]} \left| b''(x_{i, S}^{\top}\theta_S) - b''(x_{i, S}^{\top}\thetaBest[S]) \right|
    \leq 
    2 \widetilde{r}_{p, |S|} \designRegular b''(x_{i, S}^{\top} \thetaBest[S]).
\end{align*}	
It follows that
\begin{align}
    -\delta_{n,S} \sum_{i=1}^{n} b''(x_{i, S}^{\top} \thetaBest[S])x_{i, S}x_{i, S}^{\top}
    \preceq \bF_{n, \theta_S} - \bF_{n, \thetaBest[S]} 
    \preceq \delta_{n,S} \sum_{i=1}^{n} b''(x_{i, S}^{\top} \thetaBest[S])x_{i, S}x_{i, S}^{\top},
\end{align}
completing the proof of \eqref{eqn:fisher_smooth_ineq_extended} for the case where $\widetilde{\delta}_{n, S} = 2 \widetilde{r}_{p, |S|} \designRegular$.

Next, we will prove that \eqref{eqn:fisher_smooth_ineq_extended} holds with
\begin{align*}
    \widetilde{\delta}_{n, S} =  \left\{ \dfrac{8\sqrt{2} K_{\rm cubic}}{\phi_{2}^{3}\left( \widetilde{s}_n ; \bW_0 \right)} \sigma_{\max}^2 \right\} \widetilde{r}_{p, |S|} n^{-1/2}.
\end{align*}
For given $\theta_{S} \in \localSetRn[S]{\widetilde{r}_{p, |S|}}$ and $u_S \in \cU_S$,
\begin{align} \label{eqn:Fisher_diff_eq1_extended}
    u_{S}^{\top} \left( \bF_{n ,\theta_S} - \bF_{n, \thetaBest} \right) u_{S} = \sum_{i=1}^{n} \left[ b''(x_{i, S}^{\top}\theta_S) - b''(x_{i, S}^{\top}\thetaBest[S]) \right] \left( x_{i, S}^{\top} u_{S} \right)^{2}
\end{align}
As proved in \eqref{eqn:taylor_formular_eq_extended}, for some $t \in [0, 1]$, we have
\begin{align*} 
    &\left| b''(x_{i, S}^{\top}\theta_S) - b''(x_{i, S}^{\top}\thetaBest[S]) \right| 
    = \left| b'''\left( x_{i, S}^{\top} \thetaBest + t x_{i, S}^{\top} \left[ \theta_S - \thetaBest \right]  \right) \right| 
    \left| x_{i, S}^{\top}\theta_S - x_{i, S}^{\top}\thetaBest[S]  \right| \\
    &\leq  b''\left( x_{i, S}^{\top} \thetaBest + t x_{i, S}^{\top} \left[ \theta_S - \thetaBest \right]  \right)
    \left| x_{i, S}^{\top}\theta_S - x_{i, S}^{\top}\thetaBest[S]  \right| \\
    &= \dfrac{b''\left( x_{i, S}^{\top} \thetaBest + t x_{i, S}^{\top} \left[ \theta_S - \thetaBest \right]  \right)}{b''\left( x_{i, S}^{\top} \thetaBest\right)} 
    \left| x_{i, S}^{\top}\theta_S - x_{i, S}^{\top}\thetaBest[S]  \right| b''\left( x_{i, S}^{\top} \thetaBest\right) \\
    &\leq \exp \left( 3\left| x_{i, S}^{\top} \left[ \theta_S - \thetaBest \right] \right| \right)
    \left| x_{i, S}^{\top}\theta_S - x_{i, S}^{\top}\thetaBest[S]  \right| b''\left( x_{i, S}^{\top} \thetaBest\right).
\end{align*}
By \eqref{eqn:fit_value_ineq_extended}, we have, for all $\theta_{S} \in \localSetRn[S]{\widetilde{r}_{p, |S|}}$,
\begin{align*}
    \exp \left( 3\left| x_{i, S}^{\top} \left[ \theta_S - \thetaBest \right] \right| \right)
    \leq 
    \exp( 3\widetilde{r}_{p, |S|} \designRegular ) 
    \leq 
    2,
\end{align*}
where the last inequality holds by \eqref{assume:extended_Fisher_smooth}.
Also, by the equation \eqref{eqn:mis_pred_error_app} in the proof of Lemma \ref{lemma:mis_on_posterior_concentration_set} and \eqref{eqn:extended_Fisher_smooth_eq0}, we have, for all $S \in \overline{\scrS}_{\Theta_n}$,
\begin{align} 
\begin{aligned} \label{eqn:Fisher_diff_eq2_extended}
    &b''\left( x_{i, S}^{\top} \thetaBest\right) 
    = \dfrac{b''\big( x_{i, S}^{\top} \thetaBest\big)}{ b''\big( x_{i, S_{\texttt{+}}}^{\top} \thetaBest[S_{\texttt{+}}] \big) } b''\left( x_{i, S_{\texttt{+}}}^{\top} \thetaBest[S_{\texttt{+}}] \right)
    \leq 2 b''\left( x_{i, S_{\texttt{+}}}^{\top} \thetaBest[S_{\texttt{+}}] \right),  \\
    &n \phi_{2}^{2}\left( \widetilde{s}_n ; \bW_0 \right) \leq \lambda_{\min} \left( \bV_{n, S} \right) \leq 2 \lambda_{\min} \left( \bF_{n, \thetaBest} \right) = 2 \rho_{\min, S},    
\end{aligned}
\end{align}
where $S_{\texttt{+}} = S \cup S_0$. Let $\nu_S = (\theta_S - \thetaBest)/\left\| \theta_S - \thetaBest[S]  \right\|_{2}$.
Hence, \eqref{eqn:Fisher_diff_eq1_extended} is bounded by
\begin{align*}
&\max_{i \in [n]} \left\{ \exp \left( 3\left| x_{i, S}^{\top} \left[ \theta_S - \thetaBest \right] \right| \right) b''\left( x_{i, S}^{\top} \thetaBest\right) \right\}
\sum_{i=1}^{n} 
\left| x_{i, S}^{\top}\theta_S - x_{i, S}^{\top}\thetaBest[S]  \right|
\left( x_{i, S}^{\top} u_{S} \right)^{2}    \\
&\leq
4\sigma_{\max}^2
\left\| \theta_S - \thetaBest[S]  \right\|_{2}
\sum_{i=1}^{n} 
\left| x_{i, S}^{\top} \nu_S \right|
\left( x_{i, S}^{\top} u_{S} \right)^{2}  \\
&\leq 
4\sigma_{\max}^2
\left\| \theta_S - \thetaBest[S]  \right\|_{2}
n 
\bigg( \dfrac{1}{n} \sum_{i=1}^{n} \left| x_{i, S}^{\top} u_S \right|^3 \bigg)^{2/3}
\bigg( \dfrac{1}{n} \sum_{i=1}^{n} \left| x_{i, S}^{\top} \nu_S \right|^3 \bigg)^{1/3} \\
&\leq 
4\sigma_{\max}^2
\left\| \theta_S - \thetaBest[S]  \right\|_{2}
n 
\left[ \max_{S \in \overline{\scrS}_{\Theta_n}} \sup_{u_S \in \cU_{S}} \bigg( \dfrac{1}{n} \sum_{i=1}^{n} \left| x_{i, S}^{\top} u_S \right|^3 \bigg) \right] \\
&\leq 
4K_{\rm cubic}\sigma_{\max}^2
\left\| \theta_S - \thetaBest[S]  \right\|_{2} n 
= 4K_{\rm cubic}\sigma_{\max}^2
\left\| \bF_{n, \thetaBest}^{-1/2} \bF_{n, \thetaBest}^{1/2} \left( \theta_S - \thetaBest[S] \right)  \right\|_{2} n \\
&\leq 
4K_{\rm cubic}\sigma_{\max}^2
\rho_{\min, S}^{-1/2} \widetilde{r}_{p, |S|} n \\
&\leq 
4K_{\rm cubic}\sigma_{\max}^2
\left[ \dfrac{\sqrt{2}}{\sqrt{n} \phi_{2}\left( \widetilde{s}_n ; \bW_0 \right) } \right]
\widetilde{r}_{p, |S|} n
\quad (\because \eqref{eqn:Fisher_diff_eq2_extended})
\\
&= 
\bigg( 
\dfrac{4\sqrt{2} K_{\rm cubic}}{\phi_{2}\left( \widetilde{s}_n ; \bW_0 \right)} \sigma_{\max}^2 
\bigg) 
\widetilde{r}_{p, |S|} n^{1/2},
\end{align*}
which implies that
\begin{align*}
    \sup_{\theta_S \in \localSetRn[S]{ \widetilde{r}_{p, |S|} }} \left\| \bF_{n, \theta_S} - \bF_{n, \thetaBest} \right\|_{2} 
    \leq 
    \bigg( \dfrac{4\sqrt{2} K_{\rm cubic}}{\phi_{2}\left( \widetilde{s}_n ; \bW_0 \right)} \sigma_{\max}^2 \bigg) \widetilde{r}_{p, |S|} n^{1/2}.
\end{align*}
Therefore,
\begin{align*}
    \sup_{\theta_S \in \localSetRn[S]{ \widetilde{r}_{p, |S|} }} \left\| \bF_{n, \thetaBest}^{-1/2} \bF_{n, \theta_S} \bF_{n, \thetaBest}^{-1/2} - \bI_{|S|} \right\|_{2}
    &\leq
    \rho_{\min, S}^{-1}
    \sup_{\theta_S \in \localSetRn[S]{ \widetilde{r}_{p, |S|} }} \left\| \bF_{n, \theta_S} - \bF_{n, \thetaBest} \right\|_{2} \\
    &\leq  \dfrac{2}{n\phi_{2}^{2}\left( \widetilde{s}_n ; \bW_0 \right)}
    \sup_{\theta_S \in \localSetRn[S]{ \widetilde{r}_{p, |S|} }} \left\| \bF_{n, \theta_S} - \bF_{n, \thetaBest} \right\|_{2}.
\end{align*}
where the second inequality holds by \eqref{eqn:Fisher_diff_eq2_extended}. It follows that \eqref{eqn:fisher_smooth_ineq_extended} holds with
\begin{align*}
    \widetilde{\delta}_{n, S} = 
    \bigg( \dfrac{8\sqrt{2} K_{\rm cubic}}{\phi_{2}^{3}\left( \widetilde{s}_n ; \bW_0 \right)} \sigma_{\max}^2 \bigg)
    \widetilde{r}_{p, |S|} n^{-1/2}, 
\end{align*}
which completes the proof.
\end{proof}

\begin{lemma} \label{lemma:extended_Fisher_smooth2}
Suppose that the conditions in Lemma \ref{lemma:extended_Fisher_smooth} hold and $M_n^{2} \geq 2C_{\rm radius}$, where $C_{\rm radius}$ is the constant specified in \eqref{eqn:mle_concentration}. 
Then, for any $\theta_{S} \in \localSetRn[S]{ r_{p, S} }$ and $S \in \overline{\scrS}_{\Theta_n}$,
\begin{align} \label{eqn:fisher_smooth_ineq_extended2}
    (1 - \delta_{n,S}) \bF_{n, \thetaBest[S]} \preceq 
    \bF_{n, \theta_S} \preceq 
    (1 + \delta_{n,S}) \bF_{n, \thetaBest[S]},
\end{align}
where 
\begin{align*}
    \delta_{n, S} = \bigg( 2 r_{p, |S|} \designRegular \bigg) \wedge 
    \Bigg( \left\{ \dfrac{8\sqrt{2} K_{\rm cubic}}{\phi_{2}^{3}\left( \widetilde{s}_n ; \bW_0 \right)} \sigma_{\max}^2 \right\} r_{p, |S|} n^{-1/2} \Bigg).
\end{align*}
\end{lemma}

\begin{proof}
    The proof is similar to Lemma \ref{lemma:extended_Fisher_smooth}, but replaces $\widetilde{r}_{p, |S|}$ with $r_{p, S}$.
\end{proof}

\begin{remark}
Under the conditions in Lemma \ref{lemma:extended_Fisher_smooth}, if 
\begin{align*}
\max_{S \in \overline{\scrS}_{\Theta_n}} \designRegular = O(n^{-1/2}) \quad  
\text{ or } \quad
\phi_{2}^{-1}\left( \widetilde{s}_n ; \bW_0 \right) \vee \sigma_{\max}^2 = O(1)   
\end{align*}
then 
\begin{align*}
\max_{S \in \overline{\scrS}_{\Theta_n}} \widetilde{\delta}_{n, S} 
= O\left( M_n \left[ \dfrac{s_0 \log p}{n} \right]^{1/2} \right), \quad 
\max_{S \in \overline{\scrS}_{\Theta_n}} \delta_{n, S} 
= O\left( \left[ \dfrac{s_0 \log p}{n} \right]^{1/2} \right),    
\end{align*}
which plays a crucial role to obtain the desired rate $s_0^3 \log p = o(n)$.
\end{remark}

%%%%%%%%%%%%%%%%%

For Lemma \ref{lemma:normality_truncated_support}, we define the following notations:
\begin{align} \label{def:V_matrices}
\bV_{S, {\rm low}} = \alpha(1 - \widetilde{\delta}_{n, S} ) \bF_{n, \thetaBest}  + \lambda \bF_{n, \thetaMLE}, \quad 
\bV_{S, {\rm up}} = \alpha(1 + \widetilde{\delta}_{n, S} ) \bF_{n, \thetaBest}  + \lambda \bF_{n, \thetaMLE}.
\end{align}

\begin{lemma} \label{lemma:normality_truncated_support}
Suppose that \eqref{assume:extended_Fisher_smooth_K_cubic} holds for some constant $K_{\rm cubic} > 0$ and
\begin{align} \label{assume:normality_truncated_support_sample}
    n \geq C 
    \left[ \phi_{2}^{-2} \left( \widetilde{s}_n ; \bW_0 \right)
    \left( \left\| \bX \right\|_{\max}^2 \vee 1 \right) \right] s_0^2 \log p, \quad  
    \max_{S \in \overline{\scrS}_{\Theta_n}} \widetilde{r}_{p, |S|} \designRegular \leq 1/5,
\end{align}
where $C = C(K_{\rm dim}, K_{\rm theta})$ is a large enough constant.
Also, assume that \eqref{A2:b} holds for some constants $A_5, A_6 > 0$, $A_7 \geq 0$, and
\begin{align} \label{assume:smooth_rho_max2}
\begin{aligned}
\quad \max_{S \in \scrS_{\Theta_n}} \rho_{\operatorname{max}, S} \leq p^{A_8},
\end{aligned}
\end{align}
where $A_8 > 0$ is a constant. Assume further that
\begin{align} \label{assume:normality_truncated_support_asymp} \tag{D.AS.2} 
\begin{aligned}
C' & \leq M_n, \quad 
C' M_n \leq p, \quad 
\alpha \in (0, 1],
\end{aligned}
\end{align}
where $C' = C'(C_{\rm dev}, \alpha, A_6, A_8)$ is a large enough constant.
Then, with $\bbP_0^{(n)}$-probability at least $1 - p^{-1}$,
the following inequalities hold uniformly for all non-empty $S \in \scrS_{\Theta_n}$:
\begin{align}
\begin{aligned} \label{eqn:normality_truncated_support_claim}
& \dfrac{ \displaystyle
\int_{\localSetRn[S]{ \widetilde{r}_{p, |S|}}^{\rm{c}}} \exp\left\{
    - \dfrac{1}{2} (\theta_{S} - \thetaMLE[S])^{\top} \bV_{S, {\rm low}} (\theta_{S} - \thetaMLE[S]) 
\right\} \rmd \theta_{S}
}{ \displaystyle
    \int_{\bbR^{|S|}} \exp\left\{
    - \dfrac{1}{2} (\theta_{S} - \thetaMLE[S])^{\top} \bV_{S, {\rm low}} (\theta_{S} - \thetaMLE[S])
    \right\} \rmd \theta_{S}
} &\leq p^{-\alpha M_n^2|S|/64}, \\
& \dfrac{ \displaystyle
\int_{\localSetRn[S]{ \widetilde{r}_{p, |S|}}^{\rm{c}}} \exp\left\{
    - \dfrac{1}{2} (\theta_{S} - \thetaMLE[S])^{\top} \bV_{S, {\rm up}} (\theta_{S} - \thetaMLE[S]) 
\right\} \rmd \theta_{S}
}{ \displaystyle
    \int_{\bbR^{|S|}} \exp\left\{
    - \dfrac{1}{2} (\theta_{S} - \thetaMLE[S])^{\top} \bV_{S, {\rm up}} (\theta_{S} - \thetaMLE[S])
    \right\} \rmd \theta_{S}
} &\leq p^{-\alpha M_n^2|S|/64}.
\end{aligned}
\end{align}
\end{lemma}

\begin{proof} 
By the assumptions, one can easily check that
\begin{align*}
    \max_{S \in \overline{\scrS}_{\Theta_n} } \widetilde{\delta}_{n, S} \leq 1/2, \quad 
    \scrS_{\Theta_n} \subseteq \widetilde{\scrS}_{s_{\max}}, 
\end{align*}
where $\widetilde{\scrS}_{s_{\max}}$ and $\overline{\scrS}_{\Theta_n}$ are defined in \eqref{def:supp_set_MLE_concentration} and \eqref{def:A_4}, respectively.
This implies that, by Lemma \ref{lemma:concentration_mle_score}, there exists an event $\Omega_n$ such that $\bbP_0^{(n)} \left(\Omega_n\right) \geq 1- p^{-1}$ and $\thetaMLE[S] \in \localSetRn[S]{r_{p, S}}$ for all $S \in \scrS_{\Theta_n}$ on $\Omega_n$. In the remainder of this proof, we work on the event $\Omega_n$.

Let $S \in \scrS_{\Theta_n} \setminus \varnothing$.
Since the denominators in \eqref{eqn:normality_truncated_support_claim} are bounded below by $\operatorname{det} \left( \bV_{S, {\rm low}} \right)^{-1/2}$ and $\operatorname{det} \left( \bV_{S, {\rm up}} \right)^{-1/2}$, it suffices to show that 
\begin{align} \label{eqn:lemma_meas_zero_claim}
\begin{aligned}
    \operatorname{det} \left(  \bV_{S, {\rm low}} \right)^{1/2} 
    \int_{\localSetRn[S]{ \widetilde{r}_{p, |S|}}^{\rm{c}}} \exp\left\{
      - \dfrac{1}{2} \left\| \bV_{S, {\rm low}}^{1/2} (\theta_{S} - \thetaMLE[S]) \right\|_{2}^{2} 
    \right\} \rmd \theta_{S} 
    &\leq p^{-\alpha M_n^2|S|/64}, \\
    \operatorname{det} \left(  \bV_{S, {\rm up}} \right)^{1/2} 
    \int_{\localSetRn[S]{ \widetilde{r}_{p, |S|}}^{\rm{c}}} \exp\left\{
    - \dfrac{1}{2} \left\| \bV_{S, {\rm up}}^{1/2} (\theta_{S} - \thetaMLE[S]) \right\|_{2}^{2} 
    \right\} \rmd \theta_{S} 
    &\leq p^{-\alpha M_n^2|S|/64} 
\end{aligned}
\end{align}
with $\bbP_0^{(n)}$-probability at least $1 - p^{-1}$.
We prove only the first inequality in \eqref{eqn:lemma_meas_zero_claim}; the proof of the second inequality is analogous, with the replacement of $1- \widetilde \delta_{n, S}$ by $1 + \widetilde \delta_{n, S}$.

For $\theta_{S} \notin \localSetRn[S]{ \widetilde{r}_{p, |S|}}$, note that
\begin{align*}
    \left\| \bF_{n, \thetaBest}^{1/2} \left( \thetaBest[S] - \thetaMLE[S] \right) \right\|_{2} 
    &\leq \sqrt{2 C_{\rm radius} |S| \log p}
    \leq 
    \left( 1 - \dfrac{1}{\sqrt{2}} \right) M_n \sqrt{|S| \log p} \\
    &\leq \left( 1 - \dfrac{1}{\sqrt{2}} \right) 
    \left\| \bF_{n, \thetaBest}^{1/2}  \left( \theta_{S} - \thetaBest[S] \right)  \right\|_{2}, 
\end{align*}
where the first inequality holds by \eqref{eqn:mle_concentration} and Lemma \ref{lemma:mis_on_posterior_concentration_set}, and second inequality holds by \eqref{assume:normality_truncated_support_sample} with large enough $C' = C'(C_{\rm dev})$.
It follows that
\begin{align}
\begin{aligned}
    \label{eqn:truncated_eqn1}
    \left\| \bF_{n, \thetaBest}^{1/2} \left( \theta_S - \thetaMLE[S] \right) \right\|_{2} 
    &\geq 
    \left\| \bF_{n, \thetaBest}^{1/2} \left( \theta_{S} - \thetaBest[S] \right) \right\|_{2}  
    - \left\| \bF_{n, \thetaBest}^{1/2} \left( \thetaBest[S] - \thetaMLE[S] \right) \right\|_{2} \\
    &\geq 
    \dfrac{1}{\sqrt{2}} \left\| \bF_{n, \thetaBest}^{1/2} \left( \theta_S - \thetaBest[S] \right) \right\|_{2}. 
\end{aligned}
\end{align} 
Also, Lemma \ref{lemma:extended_Fisher_smooth} implies that
\begin{align} \label{eqn:truncated_eqn2}
\bV_{S, {\rm low}}
= 
\alpha(1 - \widetilde{\delta}_{n, S} ) \bF_{n, \thetaBest}  + \lambda \bF_{n, \thetaMLE}
\succeq
(\alpha + \lambda) \left[1 - \widetilde{\delta}_{n, S} \right] \bF_{n, \thetaBest}.
\end{align}
Hence, we have on $\Omega_n$,
\begin{align*}
\begin{aligned} 
&\int_{\localSetRn[S]{ \widetilde{r}_{p, |S|}}^{\rm{c}}} \exp\left\{
- \dfrac{1}{2} (\theta_{S} - \thetaMLE[S])^{T} \bV_{S, {\rm low}} (\theta_{S} - \thetaMLE[S]) 
\right\} \rmd \theta_{S} \\
&\leq 
\int_{\localSetRn[S]{ \widetilde{r}_{p, |S|}}^{\rm{c}}} \exp\left\{
- \dfrac{1}{2}
(\alpha + \lambda) \left[1 - \widetilde{\delta}_{n, S} \right]
\left\| \bF_{n, \thetaBest}^{1/2} \left( \theta_S - \thetaMLE[S] \right)  \right\|_{2}^{2} 
\right\} \rmd \theta_{S} \quad &(\because \eqref{eqn:truncated_eqn2}) \\
&\leq
\int_{\localSetRn[S]{ \widetilde{r}_{p, |S|}}^{\rm{c}}} \exp\left\{
- \dfrac{1}{4}
(\alpha + \lambda) \left[1 - \widetilde{\delta}_{n, S} \right]
\left\| \bF_{n, \thetaBest}^{1/2} \left( \theta_S - \thetaBest[S] \right)  \right\|_{2}^{2} 
\right\} \rmd \theta_{S} \quad &(\because \eqref{eqn:truncated_eqn1}) \\
&\leq 
\int_{\localSetRn[S]{ \widetilde{r}_{p, |S|}}^{\rm{c}}} \exp\left\{
- \dfrac{\alpha + \lambda}{8}  \left\| \bF_{n, \thetaBest}^{1/2} \left( \theta_S - \thetaBest[S] \right)  \right\|_{2}^{2} 
\right\} \rmd \theta_{S} \quad &(\because \eqref{assume:smooth_rho_max2}) \\
&\leq 
\int_{\localSetRn[S]{ \widetilde{r}_{p, |S|}}^{\rm{c}}} \exp\left\{
- \dfrac{\alpha}{8}  \left\| \bF_{n, \thetaBest}^{1/2} \left( \theta_S - \thetaBest[S] \right)  \right\|_{2}^{2} 
\right\} \rmd \theta_{S}. 
\end{aligned}
\end{align*}
With $h_S = \bF_{n, \thetaBest}^{1/2} \left( \theta_S - \thetaBest[S] \right)$ and Lebesgue measure $\mu$, the last display is bounded by 
\begin{align}
\begin{aligned} \label{eqn:normality_eq2.5}
&\sum_{k=1}^{\infty} 
\exp\left\{
- \dfrac{\alpha k}{8}  M_n^2 |S| \log p 
\right\} \mu\bigg\{ 
    k M_n^2 |S| \log p \leq  \| h_S \|_{2}^{2} \leq (k+1) M_n^2 |S| \log p
\bigg\}  \\    
&\leq 
\sum_{k=1}^{\infty} 
\exp\left\{
- \dfrac{\alpha k}{8}  M_n^2 |S| \log p
\right\} 
\mu\left\{ 
    h_S \in \bbR^{|S|} : 
    \| h_S \|_{2}^{2} \leq (k+1) M_n^2 |S| \log p
\right\}  \\   
&= 
\sum_{k=1}^{\infty} 
\exp\left\{
- \dfrac{\alpha k}{8}  M_n^2 |S| \log p 
\right\} 
\dfrac{\pi^{|S|/2}}{ \Gamma(|S|/2 + 1) } \left\{  (k + 1) M_n^2 |S| \log p \right\}^{|S|} \\    
&\leq
\left\{ \sqrt{\pi} M_n^2 |S| \log p \right\}^{|S|}
\sum_{k=1}^{\infty} (k+1)^{|S|}
\exp\left\{
- \dfrac{\alpha k}{8} M_n^2 |S| \log p
\right\} \\
&= %4
\left\{ \sqrt{\pi}M_n^2 |S| \log p \right\}^{|S|}
\sum_{k=1}^{\infty} 
\exp\left\{
- \dfrac{\alpha k}{8} M_n^2 |S| \log p + |S|\log(k+1) 
\right\} \\
&\leq
\left\{ \sqrt{\pi}M_n^2 |S| \log p \right\}^{|S|}
\sum_{k=1}^{\infty} 
\exp\left\{
- \dfrac{\alpha k}{8} M_n^2 |S| \log p + |S|k
\right\} \\
&\leq
\left\{ \sqrt{\pi}M_n^2 |S| \log p \right\}^{|S|}
\sum_{k=1}^{\infty} 
\exp\left\{
- \dfrac{\alpha k}{16} M_n^2 |S| \log p
\right\} ,
\end{aligned}
\end{align}
where the last inequality holds by \eqref{assume:normality_truncated_support_asymp}. 
Also, one can see that \eqref{assume:normality_truncated_support_asymp} implies that
\begin{align*}
\exp(- \alpha M_n^2 |S| \log p / 16 ) \leq 1/2.
\end{align*}
Hence, the right hand side of \eqref{eqn:normality_eq2.5} is further bounded by
\begin{align*}
\underbrace{  \left\{ \sqrt{\pi} M_n^2 |S| \log p   \right\}^{|S|}
 \exp\left\{- \dfrac{\alpha}{32} M_n^2 |S| \log p \right\}  }_{(\ast)}.   
\end{align*}

To obtain \eqref{eqn:lemma_meas_zero_claim}, it suffices to prove that 
\begin{align} \label{eqn:normality_claim}
\operatorname{det} \left( \bV_{S, {\rm low}} \right)^{1/2}  \times (\ast) 
\leq p^{-\alpha M_n^2|S|/64}.
\end{align}
Since $\thetaMLE \in \localSetRn[S]{r_{p, S}} \subset \localSetRn[S]{ \widetilde{r}_{p, |S|}}$, we have 
\begin{align*}
\lambda_{\operatorname{max}} \left( \bV_{S, {\rm low}} \right)
&\leq
\lambda_{\operatorname{max}}
\left\{
\left[\alpha \dfrac{1 - \widetilde{\delta}_{n, S} }{1 + \widetilde{\delta}_{n, S} } + \lambda \right] \left[1 + \widetilde{\delta}_{n, S} \right] \bF_{n, \thetaBest}
\right\} \\
&\leq
(\alpha + \lambda) (1 + \widetilde{\delta}_{n, S} ) \rho_{\operatorname{max}, S} 
\leq
\dfrac{3}{2}(\alpha + \lambda)\rho_{\operatorname{max}, S} 
\leq 
\dfrac{3}{2}(\alpha + A_6) p^{A_8}
\end{align*}
It follows that
\begin{align*}
\operatorname{det} \left( \bV_{S, {\rm low}} \right)^{1/2} 
\leq
\left( \dfrac{3 \left[ \alpha + A_6 \right] }{2} p^{A_8} \right)^{|S|/2}.    
\end{align*}
Hence, the logarithm of the left hand side of \eqref{eqn:normality_claim} is bounded by
\begin{align*}
\begin{aligned} 
&\dfrac{|S|}{2} \log \left( \dfrac{3 \left[ \alpha + A_6 \right] }{2} p^{A_8} \right) 
+ |S| \left\{\log(\sqrt{\pi}) + \log \left( M_n^2 |S| \log p \right) \right\} - \dfrac{\alpha}{32} M_n^2 |S|\log p \\
&=
\dfrac{|S|}{2} \log \left( \dfrac{3 \pi \left[ \alpha + A_6 \right] }{2} \right) 
+ \dfrac{A_8}{2} |S| \log p
+ |S| \bigg[ 
\log \left( M_n^2 \right) 
+ \log \left( |S| \right) 
+ \log \left( \log p \right) \bigg]
- \dfrac{\alpha}{32} M_n^2 |S|\log p \\
&\leq 
\bigg[ 
\dfrac{1}{2}\log \left( \dfrac{3 \pi \left[ \alpha + A_6 \right] }{2} \right) 
+ \dfrac{A_8}{2} + 1 + 1 + 1 
- \dfrac{\alpha}{32} M_n^2
\bigg]
|S|\log p \\
&\leq
- \dfrac{\alpha}{64} M_n^2 |S|\log p,
\end{aligned}
\end{align*}
where the last inequality holds by \eqref{assume:normality_truncated_support_sample} with large enough $C' = C'(\alpha, A_6, A_8)$.
\end{proof}

\begin{lemma} \label{lemma:margin_probability}
Suppose that conditions in Lemma \ref{lemma:normality_truncated_support} hold. 
Also, assume that 
\begin{align} \label{assume:margin_prob_conditions} \tag{D.AS.3}
    \quad C \leq M_n^2,
\end{align}
where $C = C(C_{\rm dev}, \alpha, A_5) > 0$ is a large enough constant.
Then, with $\bbP_{0}^{(n)}$-probability at least $1 - p^{-1}$, the following inequality holds uniformly for all non-empty $S \in \scrS_{\Theta_n}$:
\begin{align*}
    \int_{\localSetRn[S]{ \widetilde{r}_{p, |S|}}^{\rm{c}}} \exp\left(\alpha L_{n, \theta_{S}} \right) g_{S}(\theta_{S}) \rmd \theta_S
    \leq p^{-|S|} \exp\left(\alpha L_{n, \thetaMLE[S]} \right) (1 + \alpha \lambda^{-1})^{-|S|/2}.
\end{align*}
\end{lemma}
\begin{proof}
By the assumptions, one can easily check that
\begin{align*}
    \max_{S \in \overline{\scrS}_{\Theta_n} } \widetilde{\delta}_{n, S} \leq 1/2, \quad 
    \scrS_{\Theta_n} \subseteq \widetilde{\scrS}_{s_{\max}}, 
\end{align*}
where $\widetilde{\scrS}_{s_{\max}}$ and $\overline{\scrS}_{\Theta_n}$ are defined in \eqref{def:supp_set_MLE_concentration} and \eqref{def:A_4}, respectively.
This implies that, by Lemma \ref{lemma:projection_score_vec}, there exists an event $\Omega_n$ such that,  $\bbP_0^{(n)}(\Omega_n) \geq 1 - p^{-1}$ and on $\Omega_n$
\begin{align*}
    \| \xi_{n, S} \|_{2}^2 \leq 2K_{\rm score} |S| \log p, \quad \forall S \in \scrS_{\Theta_n},
\end{align*}
where $K_{\rm score} = K_{\rm score}(C_{\rm dev})$ is the constant specified in \eqref{eqn:score_concentration}. 
In the remainder of this proof, we work on the event $\Omega_n$ with a non-empty $S \in \scrS_{\Theta_n}$.

Let $S \in \scrS_{\Theta_n} \setminus \varnothing$. Since 
\begin{align*}
    &\int_{\localSetRn[S]{ \widetilde{r}_{p, |S|}}^{\rm{c}}} \exp\left(\alpha L_{n, \theta_{S}} \right) g_{S}(\theta_{S}) \rmd \theta_S \\
    &= \exp\left(\alpha L_{n, \thetaMLE[S]} \right)
    \int_{\localSetRn[S]{ \widetilde{r}_{p, |S|}}^{\rm{c}}} 
        \exp\left(
            \alpha L_{n, \theta_{S}} - \alpha L_{n, \thetaMLE[S]} 
            \right) g_{S}(\theta_{S}) \rmd \theta_S \\
    &\leq \exp\left(\alpha L_{n, \thetaMLE[S]} \right)
    \int_{\localSetRn[S]{ \widetilde{r}_{p, |S|}}^{\rm{c}}} 
    \exp\left(
    \alpha L_{n, \theta_{S}} - \alpha L_{n, \thetaBest[S]} 
    \right) g_{S}(\theta_{S}) \rmd \theta_S,
\end{align*}
it suffices to prove that
\begin{align} \label{eqn:claim2_margin_prob}
    (1 + \alpha \lambda^{-1})^{|S|/2}
    \int_{\localSetRn[S]{ \widetilde{r}_{p, |S|}}^{\rm{c}}} 
    \exp\left(
    \alpha L_{n, \theta_{S}} - \alpha L_{n, \thetaBest[S]} 
    \right) g_{S}(\theta_{S}) \rmd \theta_S 		
    \leq p^{-|S|}.		
\end{align}
Note that
\begin{align} \label{eqn:margin_prob_claim2}
\begin{aligned}
    &\int_{\localSetRn[S]{ \widetilde{r}_{p, |S|}}^{\rm{c}}} 
    \exp\left(
    \alpha L_{n, \theta_{S}} - \alpha L_{n, \thetaBest[S]} 
    \right) g_{S}(\theta_{S}) \rmd \theta_S \\
    &\leq
    \sup_{\theta_{S} \in \localSetRn[S]{ \widetilde{r}_{p, |S|}}^{\rm c}} \left[	\exp\left(
    \alpha L_{n, \theta_{S}} - \alpha L_{n, \thetaBest[S]} 
    \right) \right].        
\end{aligned}
\end{align}

At the end of this proof, we will prove that
\begin{align} \label{eqn:margin_eqn0}
    \sup_{\theta_{S}^{\circ} \in \partial \localSetRn[S]{ \widetilde{r}_{p, |S|}}} L_{n, \theta_{S}^{\circ}} - L_{n, \thetaBest[S]} \leq - \dfrac{1}{8} M_n^{2} |S| \log p,
\end{align}
where $\partial \localSetRn[S]{ \widetilde{r}_{p, |S|}} = \{\theta_{S} \in \bbR^{|S|} : \|
\bF_{n, \thetaBest}^{1/2} \left( \theta_{S} - \thetaBest[S] \right)  \|_{2} = M_n \sqrt{|S| \log p} \}$ is the boundary of $\localSetRn[S]{ \widetilde{r}_{p, |S|}}$.
Since $\theta \mapsto L_{n, \theta}$ is concave, for any $\theta_S \in \localSetRn[S]{ \widetilde{r}_{p, |S|}}^{\rm c}$,
\begin{align*}
    L_{n, \overline{\theta}_S} \geq \omega L_{n, \theta_S} + (1 - \omega) L_{n, \thetaBest},
\end{align*}
where $\omega = M_n \sqrt{|S| \log p} / \| \bF_{n, \thetaBest}^{1/2} \left( \theta_{S} - \thetaBest[S] \right) \|_2$ and $\overline{\theta}_S = \omega \theta_S + (1 - \omega) \thetaBest \in \partial\localSetRn[S]{ \widetilde{r}_{p, |S|}}$.
Hence, 
\begin{align*}
- \dfrac{1}{8} M_n^2 |S| \log p
\geq 
\sup_{\theta_{S}^{\circ} \in \partial \localSetRn[S]{ \widetilde{r}_{p, |S|}}}  L_{n, \theta_{S}^{\circ}} - L_{n, \thetaBest[S]}
\geq 
\omega \left( L_{n, \theta_{S}} - L_{n, \thetaBest[S]} \right)
\geq 
L_{n, \theta_{S}} - L_{n, \thetaBest[S]}
\end{align*}
for $\theta_S \in \localSetRn[S]{ \widetilde{r}_{p, |S|}}^{\rm c}$.
Combining with \eqref{eqn:margin_prob_claim2}, the left hand side of \eqref{eqn:claim2_margin_prob} is bounded by
\begin{align*}
    & (1 + \alpha \lambda^{-1})^{|S|/2} \exp \left( - \dfrac{\alpha M_n^2}{8} |S| \log p \right) \\
    &= \exp \left( \dfrac{|S|}{2} \log \left\{ 1 + \alpha \lambda^{-1} \right\} 
    -\dfrac{\alpha M_n^2}{8} |S| \log p \right) \\
    &\leq \exp \left( \dfrac{|S|}{2} \log \left\{ 2(1 \vee \lambda^{-1}) \right\}  -\dfrac{\alpha M_n^2}{8} |S| \log p \right) \\
    &\leq \exp \left( \dfrac{|S|}{2} \log2 + \dfrac{A_5}{2}|S|\log p  -\dfrac{\alpha M_n^2}{8} |S| \log p \right) &\quad (\because \text{\eqref{A2:b}}) \\    
    &\leq \exp \left( -|S| \log p \right) = p^{-|S|}. &\quad (\because \text{\eqref{assume:margin_prob_conditions}})
\end{align*}

To complete the proof, we only need to prove \eqref{eqn:margin_eqn0}.
By Taylor's theorem, for $\theta_{S}^{\circ} \in \partial \localSetRn[S]{ \widetilde{r}_{p, |S|}}$, there exists $\overline{\theta}_{S} \in \localSetRn[S]{ \widetilde{r}_{p, |S|}}$ such that
\begin{align*}
    L_{n, \theta_{S}^{\circ}} - L_{n, \thetaBest[S]}
    &= (\theta_{S}^{\circ} - \thetaBest[S])^{\top} \dot{L}_{n, \thetaBest[S]}
     -\dfrac{1}{2} (\theta_{S}^{\circ} - \thetaBest[S])^{\top} \bF_{n, \overline{\theta}_{S}}  (\theta_{S}^{\circ} - \thetaBest[S]) \\
    &= \xi_{n, S}^{\top}  \bF_{n, \thetaBest}^{1/2} \left( \theta_{S}^{\circ} - \thetaBest[S] \right)
    -\dfrac{1}{2} (\theta_{S}^{\circ} - \thetaBest[S])^{\top} \bF_{n, 
    \overline{\theta}_{S}}  (\theta_{S}^{\circ} - \thetaBest[S]) \\
    &\leq \xi_{n, S}^{\top}  \bF_{n, \thetaBest}^{1/2} \left( \theta_{S}^{\circ} - \thetaBest[S] \right)
    -\dfrac{1- \widetilde{\delta}_{n, S} }{2} \left\|  \bF_{n, \thetaBest}^{1/2} \left( \theta_{S}^{\circ} - \thetaBest[S] \right)   \right\|_2^2 &\quad (\because \text{ Lemma  \ref{lemma:extended_Fisher_smooth}}) \\
    &\leq \xi_{n, S}^{\top}  \bF_{n, \thetaBest}^{1/2} \left( \theta_{S}^{\circ} - \thetaBest[S] \right)
    -\dfrac{1}{4} \left\|  \bF_{n, \thetaBest}^{1/2} \left( \theta_{S}^{\circ} - \thetaBest[S] \right)   \right\|_2^2. &\quad (\because \widetilde{\delta}_{n, S} \leq 1/2)
\end{align*} 
Also, we have on $\Omega_n$
\begin{align*}
    \xi_{n, S}^{\top} \bF_{n, \thetaBest}^{1/2} \left( \theta_{S}^{\circ} - \thetaBest[S] \right) 
    &\leq \left\| \xi_{n, S} \right\|_2 \left\|  \bF_{n, \thetaBest}^{1/2} \left( \theta_{S}^{\circ} - \thetaBest[S] \right) \right\|_2  \\
    &\leq \left( 2K_{\operatorname{score}} |S| \log p \right)^{1/2} \left\|  \bF_{n, \thetaBest}^{1/2} \left( \theta_{S}^{\circ} - \thetaBest[S] \right)   \right\|_2.
\end{align*}
Hence, $L_{n, \theta_{S}^{\circ}} - L_{n, \thetaBest[S]}$ is bounded by
\begin{align*}
    &\left[ (2K_{\operatorname{score}} |S| \log p)^{1/2}
    -\dfrac{1}{4}  \left\|  \bF_{n, \thetaBest}^{1/2} \left( \theta_{S}^{\circ} - \thetaBest[S] \right)   \right\|_2
    \right]
    \left\|  \bF_{n, \thetaBest}^{1/2} \left( \theta_{S}^{\circ} - \thetaBest[S] \right)   \right\|_2 \\
    &\leq
    \left[ \sqrt{2K_{\operatorname{score}} |S| \log p} -\dfrac{M_n}{4} \sqrt{|S| \log p}
    \right] 
    \left\|  \bF_{n, \thetaBest}^{1/2} \left( \theta_{S}^{\circ} - \thetaBest[S] \right) \right\|_2 &(\because \theta_S^{\circ} \in \partial \localSetRn[S]{ \widetilde{r}_{p, |S|}}) \\
    &\leq
    -\dfrac{M_n}{8} \sqrt{|S| \log p}
    \left\|  \bF_{n, \thetaBest}^{1/2} \left( \theta_{S}^{\circ} - \thetaBest[S] \right) \right\|_2 &(\because \eqref{assume:margin_prob_conditions})  \\  
    &\leq
    -\dfrac{M_n^2}{8} |S| \log p, &(\because \theta_S^{\circ} \in \partial \localSetRn[S]{ \widetilde{r}_{p, |S|}})
\end{align*}
which completes the proof.
\end{proof}

Recall the definition of the approximated marginal likelihood:
\begin{align*}
\widehat{\cM}_\alpha^n (S) = \exp\bigl(\alpha L_{n, \thetaMLE[S]} \bigr) \, (1 + \alpha \lambda^{-1})^{-|S|/2}. 
\end{align*}
The following theorem justifies the use of the Laplace approximation for the marginal likelihood.

\begin{theorem}[Laplace approximation of the marginal likelihood] \label{thm:LA_marginal_likelihood}
Suppose that conditions in Lemmas \ref{lemma:normality_truncated_support}, \ref{lemma:margin_probability}. Also, assume that
\begin{align} \label{assume:LA_conditions}
    \max_{S \in \scrS_{\Theta_n}} |S| \widetilde{\delta}_{n, S} \leq \dfrac{1}{36}
\end{align}
Then, with $\bbP_{0}^{(n)}$-probability at least $1 - p^{-1}$, the following inequality holds uniformly for all non-empty $S \in \scrS_{\Theta_n}$:
\begin{align} \label{eqn:LA_fin_claim1}
\begin{aligned}
\left| 1 - \dfrac{ \cM_\alpha^n  (S) }{ \widehat{\cM}_\alpha^n (S) }\right| 
\leq 
\tau_{n, p, S},
\end{aligned}
\end{align}
where $\tau_{n, p, S} = 6|S|\widetilde{\delta}_{n, S} + 2 p^{-1} \leq 1/3$. 
Consequently, we have
\begin{align*}
\bbP_{0}^{(n)} \left( 
\dfrac{\pi_\alpha^n(S)}{ \pi_\alpha^n(S_0)} 
\leq 
\left( \dfrac{1 + \tau_{n, p, S}}{1 - \tau_{n, p, S}} \right)
\dfrac{\pi_{n} (S) \,  \widehat{\cM}_\alpha^n (S) }{\pi_{n}(S_0) \,  \widehat{\cM}_\alpha^n(S_0)}
\quad 
\text{for all } S \in \scrS_{\Theta_n} \setminus \varnothing
\right) 
\geq 
1 - p^{-1}.
\end{align*}
\end{theorem}

\begin{proof}
By the conditions in Lemma \ref{lemma:normality_truncated_support}, we have
\begin{align*}
    \max_{S \in \overline{\scrS}_{\Theta_n}} \widetilde{\delta}_{n, S} \leq 1/2.
\end{align*}
From the proofs, one can see that the assertions of Lemmas \ref{lemma:projection_score_vec} and \ref{lemma:concentration_mle_score} hold on the same event.
Hence, there exists an event $\Omega_n$ such that $\bbP_0^{(n)} (\Omega_n) \geq 1 - p^{-1}$, and on $\Omega_n$,
\begin{align*}
    \thetaMLE \in \localSetRn[S]{r_{p, S}}, \quad \| \xi_{n, S} \|_{2}^{2} \leq 2K_{\rm score} |S| \log p
\end{align*}
for all non-empty $S \in \scrS_{\Theta_n}$, where $K_{\rm score}$ is the constant specified in \eqref{eqn:score_concentration}, depending only on $C_{\rm dev}$. 
In the remainder of this proof, we work on the event $\Omega_n$ with a non-empty $S \in \scrS_{\Theta_n}$.

Since  $\thetaMLE \in \localSetRn[S]{r_{p, S}} \subset \localSetRn[S]{ \widetilde{r}_{p, |S|}}$, for $\theta_{S} \in \localSetRn[S]{ \widetilde{r}_{p, |S|}}$, there exists $\overline{\theta}_{S} \in \localSetRn[S]{ \widetilde{r}_{p, |S|}}$ such that
\begin{align*}
L_{n, \theta_S}
	&= L_{n, \thetaMLE[S]} + (\theta_{S} - \thetaMLE[S])^{\top} \dot{L}_{n, \thetaMLE[S]}  
	 - \dfrac{1}{2} (\theta_{S} - \thetaMLE[S])^{\top} \bF_{n, \overline{\theta}_{S}} (\theta_{S} - \thetaMLE[S]) \\  
	&= L_{n, \thetaMLE[S]} - \dfrac{1}{2} (\theta_{S} - \thetaMLE[S])^{\top} \bF_{n, \overline{\theta}_{S}} (\theta_{S} - \thetaMLE[S]).
\end{align*}
For $\cA \subset \bbR^{|S|}$, let $\cM_{\alpha}^{n}(S, \cA) = \int_{\cA} \exp( \alpha L_{n, \theta_S} ) \, g_{S}(\theta_{S}) \, \rmd \theta_S$.
Then, the last display gives
\begin{align*}
    &\cM_\alpha^n(S, \localSetRn[S]{ \widetilde{r}_{p, |S|}}) \\
    &=  \int_{\localSetRn[S]{ \widetilde{r}_{p, |S|}}} \exp\left[ \alpha \Bigl\{
        L_{n, \thetaMLE[S]} - \dfrac{1}{2} (\theta_{S} - \thetaMLE[S])^{\top} \bF_{n, \overline{\theta}_{S}} (\theta_{S} - \thetaMLE[S])
        \Bigr\} \right] \, g_{S}(\theta_{S}) \, \rmd \theta_S \\	
    &=  \exp\big(\alpha L_{n, \thetaMLE[S]} \big)
        \int_{\localSetRn[S]{ \widetilde{r}_{p, |S|}}} \exp \bigg( - \dfrac{\alpha}{2}  (\theta_{S} - \thetaMLE[S])^{\top} \bF_{n, \overline{\theta}_{S}} (\theta_{S} - \thetaMLE[S] )
        \bigg) \, g_{S}(\theta_{S}) \, \rmd \theta_S \\	
    &=   \int_{\localSetRn[S]{ \widetilde{r}_{p, |S|}}} 
        \exp\left\{ - \dfrac{1}{2}  (\theta_{S} - \thetaMLE[S])^{\top} \bigl(
        \alpha\bF_{n, \overline{\theta}_{S}} + \lambda \bF_{n, \thetaMLE} \bigr) (\theta_{S} - \thetaMLE[S]) 
        \right\}
     \rmd \theta_{S} \\
    & \qquad \times \underbrace{
        \exp\bigl(\alpha L_{n, \thetaMLE[S]} \bigr) 
        \operatorname{det}\bigl\{ 2\pi \bigl( \lambda \bF_{n, \thetaMLE} \bigr)^{-1} \bigr\}^{-1/2}
        }_{(\ast)}.
\end{align*}
From the definitions of $\bV_{S, {\rm low}}$ and $\bV_{S, {\rm up}}$ in \eqref{def:V_matrices}, we have
\begin{align*}
    &\cM_\alpha^n(S, \localSetRn[S]{ \widetilde{r}_{p, |S|}}) \leq (\ast) \times \int_{\localSetRn[S]{ \widetilde{r}_{p, |S|}}} 
        \exp\Bigl\{ - \dfrac{1}{2} (\theta_{S} - \thetaMLE[S])^{\top} \bV_{S, {\rm low}} (\theta_{S} - \thetaMLE[S]) 
        \Bigr\} \, \rmd \theta_S, \\
    &\cM_\alpha^n(S, \localSetRn[S]{ \widetilde{r}_{p, |S|}}) \geq (\ast) \times \int_{\localSetRn[S]{ \widetilde{r}_{p, |S|}}} 
        \exp\Bigl\{ - \dfrac{1}{2} (\theta_{S} - \thetaMLE[S])^{\top} \bV_{S, {\rm up}} (\theta_{S} - \thetaMLE[S]) 
        \Bigr\} \, \rmd \theta_S.
\end{align*}
Also, 
\begin{align*}
    \int_{\localSetRn[S]{ \widetilde{r}_{p, |S|}}} 
    \exp\left\{ - \dfrac{1}{2} 
    \left\| \bV_{S, {\rm low}}^{1/2} \left(\theta_{S} - \thetaMLE[S]\right) \right\|_{2}^{2}
    \right\} 
    &\leq (2\pi)^{|S|/2} \operatorname{det} \left( \bV_{S, {\rm low}} \right)^{-1/2}, \\
    \int_{\localSetRn[S]{ \widetilde{r}_{p, |S|}}} 
    \exp\left\{ - \dfrac{1}{2} \left\| \bV_{S, {\rm up}}^{1/2} \left(\theta_{S} - \thetaMLE[S]\right) \right\|_{2}^{2} 
    \right\} 
    &\geq (2\pi)^{|S|/2} \operatorname{det} \left( \bV_{S, {\rm up}} \right)^{-1/2} 
    \left( 1- p^{-\alpha M_n^2/64} \right). 
\end{align*}
where the second inequality holds by Lemma \ref{lemma:normality_truncated_support}.
It follows that
\begin{align*}
    \cM_\alpha^n(S, \localSetRn[S]{ \widetilde{r}_{p, |S|}}) & \leq 
        \exp\left(\alpha L_{n, \thetaMLE[S]} \right) 
        \operatorname{det}\left( \lambda \bF_{n, \thetaMLE}    \right)^{1/2} 
        \operatorname{det} \left( \bV_{S, {\rm low}} \right)^{-1/2} , \\
    \cM_\alpha^n(S, \localSetRn[S]{ \widetilde{r}_{p, |S|}}) &\geq 
        \exp\left(\alpha L_{n, \thetaMLE[S]} \right) 
        \operatorname{det}\left( \lambda \bF_{n, \thetaMLE} \right)^{1/2}
        \operatorname{det} \left( \bV_{S, {\rm up}} \right)^{-1/2} \left( 1- p^{-\alpha M_n^2/64} \right).
\end{align*}
For all non-empty $S \in \scrS_{\Theta_n}$, let $\epsilon_S = 2|S| \widetilde{\delta}_{n, S}$.
We next prove the following inequalities: 
\begin{align} \label{eqn:LA_claim1}
    \operatorname{det}\left( \lambda \bF_{n, \thetaMLE} \right)^{1/2} 
    \operatorname{det} \left( \bV_{S, {\rm low}} \right)^{-1/2}
        &\leq (1 + \alpha \lambda^{-1})^{-|S|/2} e^{\epsilon_S}, \\
    \operatorname{det}\left( \lambda \bF_{n, \thetaMLE} \right)^{1/2} 
    \operatorname{det} \left( \bV_{S, {\rm up}} \right)^{-1/2} 
        &\geq (1 + \alpha \lambda^{-1})^{-|S|/2} e^{-\epsilon_S}. \label{eqn:LA_claim2}
\end{align}
Firstly, by Lemma \ref{lemma:extended_Fisher_smooth}, the left hand side of \eqref{eqn:LA_claim1} is bounded above by
\begin{align*}
\left[
\dfrac{\operatorname{det} \left\{
    \lambda \left[1 + \widetilde{\delta}_{n, S} \right] \bF_{n, \thetaBest[S]}
    \right\} }
{
    \operatorname{det} 
    \left\{
     (\alpha + \lambda)  \left[1 - \widetilde{\delta}_{n, S} \right] \bF_{n, \thetaBest}
    \right\}  }
\right]^{1/2}    
&= \left[
\dfrac{1 + \widetilde{\delta}_{n, S} }
{ (1 + \alpha \lambda^{-1})  (1 - \widetilde{\delta}_{n, S} )   }
\right]^{|S|/2} \\
&=
(1 + \alpha \lambda^{-1})^{-|S|/2} \left[\dfrac{1 + \widetilde{\delta}_{n, S}  }{1 - \widetilde{\delta}_{n, S} }\right]^{|S|/2}.
\end{align*}
Combining \eqref{assume:LA_conditions} with the inequality $(1 + x/t)^{t} \leq e^x$ for $|x| \leq t$, we have
\begin{align} \label{eqn:delta_ratio_eqn}
\begin{aligned}
\left[\dfrac{1 + \widetilde{\delta}_{n, S}  }{1 - \widetilde{\delta}_{n, S} }\right]^{|S|/2}
&= \left( 1 + \dfrac{2 \widetilde{\delta}_{n, S} }{ 1 - \widetilde{\delta}_{n, S} } \right)^{|S|/2}  
\leq 
\exp \left( \dfrac{  |S| \widetilde{\delta}_{n, S} }{1 - \widetilde{\delta}_{n, S}  } \right) \\
&\leq 
\exp \left( 2|S| \widetilde{\delta}_{n, S} \right) 
= \exp( \epsilon_{S} ),    
\end{aligned}
\end{align}
implying \eqref{eqn:LA_claim1}. By \eqref{assume:LA_conditions}, we have $\epsilon_S = 2|S| \widetilde{\delta}_{n, S} \leq 1/18$ for all non-empty $S \in \scrS_{\Theta_n}$.

Similarly, the left hand side of \eqref{eqn:LA_claim2} is bounded below by
\begin{align*}
\left[
\dfrac{\operatorname{det} \left\{
    \lambda \left[1 - \widetilde{\delta}_{n, S} \right] \bF_{n, \thetaBest[S]}
    \right\} }
{
    \operatorname{det} 
    \left\{
     (\alpha + \lambda)  \left[1 + \widetilde{\delta}_{n, S} \right] \bF_{n, \thetaBest}
    \right\}  }
\right]^{1/2}    
&= \left[
\dfrac{1 - \widetilde{\delta}_{n, S} }
{ (1 + \alpha \lambda^{-1})  (1 + \widetilde{\delta}_{n, S} )   }
\right]^{|S|/2} \\
= (1 + \alpha \lambda^{-1})^{-|S|/2} \left[\dfrac{1 - \widetilde{\delta}_{n, S}  }{1 + \widetilde{\delta}_{n, S} }\right]^{|S|/2}     
&\geq (1 + \alpha \lambda^{-1})^{-|S|/2} \exp( -\epsilon_{S} ),     
\end{align*}
where the last equality holds by \eqref{eqn:delta_ratio_eqn}. This completes the proof of \eqref{eqn:LA_claim2}.

By \eqref{eqn:LA_claim1} and \eqref{eqn:LA_claim2}, we have
\begin{align*}
    \cM_\alpha^n(S, \localSetRn[S]{ \widetilde{r}_{p, |S|}}) 
    &\leq \exp\left(\alpha L_{n, \thetaMLE[S]} \right) (1 + \alpha \lambda^{-1})^{-|S|/2} 
    \exp( \epsilon_{S} ), \\
    \cM_\alpha^n(S, \localSetRn[S]{ \widetilde{r}_{p, |S|}}) 
    &\geq \exp\left(\alpha L_{n, \thetaMLE[S]} \right) (1 + \alpha \lambda^{-1})^{-|S|/2} 
    \exp( -\epsilon_{S} ) \left( 1- p^{-\alpha M_n^2/64} \right),
\end{align*}
which implies that
\begin{align} \label{eqn:Approx_Margin_Lik}
\begin{aligned}
\max_{S \in \mathscr{S}_0} 
\left| 1 - \dfrac{
		\cM_\alpha^n  (S, \localSetRn[S]{ \widetilde{r}_{p, |S|}})
}{
		\widehat{\cM}_\alpha^n (S)
}\right| 
&\leq 
\bigg( 1 - \exp( -\epsilon_{S} ) + p^{-\alpha M_n^2/64} \bigg) \vee \bigg( \exp( \epsilon_{S} ) - 1 \bigg) \\
&\leq 
\big( \epsilon_{S} + p^{-\alpha M_n^2/64} \big) \vee \big( 2\epsilon_{S} \big)
\eqqcolon 
\widetilde{\tau}_{n, p, S},
\end{aligned}
\end{align}
where the last inequality holds by $1 - e^{-x} \leq x$ and $e^{x} \leq 1 + 2x$ for $x \in (0, 1)$.
Accordingly, we have a lower bound
\begin{align*}
    \cM_\alpha^n  (S) \geq \cM_\alpha^n(S, \localSetRn[S]{ \widetilde{r}_{p, |S|}}) \geq  \widehat{\cM}_\alpha^n (S)(1 - \widetilde{\tau}_{n, p, S}).
\end{align*}
An upper bound of $\cM_n(S)$ can be obtained by
\begin{align*}
	\cM_\alpha^n(S) 
	&= \cM_\alpha^n  (S, \localSetRn[S]{ \widetilde{r}_{p, |S|}}) + \cM_\alpha^n  (S, \localSetRn[S]{ \widetilde{r}_{p, |S|}}^{\rm{c}})  \\
	&\leq \widehat{\cM}_\alpha^n (S) \left( 1 + \widetilde{\tau}_{n, p, S} \right) +  \cM_\alpha^n  (S, \localSetRn[S]{ \widetilde{r}_{p, |S|}}^{\rm{c}}) \qquad (\because \eqref{eqn:Approx_Margin_Lik}) \\
    &\leq \widehat{\cM}_\alpha^n (S)\left( 1 + \widetilde{\tau}_{n, p, S} \right) +  p^{-|S|} \widehat{\cM}_\alpha^n (S) \qquad (\because \text{ Lemma \ref{lemma:margin_probability}} ) \\
    &\leq \widehat{\cM}_\alpha^n (S)\left( 1 + \widetilde{\tau}_{n, p, S} + p^{-1} \right) \\
    &\leq \widehat{\cM}_\alpha^n (S) \left( 1 + 3\epsilon_S + 2 p^{-1} \right) \\
    &=
    \widehat{\cM}_\alpha^n (S) \left( 1 + 6|S| \widetilde{\delta}_{n, S} + 2p^{-1} \right)
    = \widehat{\cM}_\alpha^n (S) \left( 1 + \tau_{n, p, S} \right).
\end{align*}
Combining the upper and lower bounds, we have
\begin{align*}
\max_{S \in \mathscr{S}_0} \left| 1 - \dfrac{
		\cM_\alpha^n  (S)
}{
		\widehat{\cM}_\alpha^n (S)
}\right| 
\leq \tau_{n, p, S},
\end{align*}
which completes the proof of \eqref{eqn:LA_fin_claim1}.

Note that
\begin{align*}
    \tau_{n, p, S}
    =
    6|S| \widetilde{\delta}_{n, S} + 2p^{-1} 
    \leq 
    1/6 + 1/6 = 1/3,
\end{align*}
where the last inequality holds by $p \geq 12$ and \eqref{assume:LA_conditions}.
Therefore, it holds that 
\begin{align*} 
\dfrac{ \pi_\alpha^n(S) }{ \pi_\alpha^n(S_0) } 
&= \dfrac{ \pi_n(S) \, {\cM}_\alpha^n  (S)}{ \pi_n(S_0) \, {\cM}_\alpha^ n (S_0)} 
\\
& \leq 
\left( \dfrac{1 + \tau_{n, p, S}}{1 - \tau_{n, p, S}} \right)
\dfrac{\pi_n(S) \, \widehat{\cM}_\alpha^n (S) }{\pi_n(S_0) \, \widehat{\cM}_\alpha^n  (S_0)} \\
&\leq 2
\dfrac{\pi_n(S) \, \widehat{\cM}_\alpha^n (S) }{\pi_n(S_0) \, \widehat{\cM}_\alpha^n  (S_0)} 
\\
& = 2
\dfrac{\pi_n(S) }{\pi_n(S_0)} \left(1 + \alpha \lambda^{-1}\right)^{-\left(|S| - s_0\right)/2}
\exp \Bigl(\alpha L_{n, \thetaMLE[S]} - \alpha L_{n, \thetaMLE[S_0]} \Bigr).
\end{align*}
This completes the proof.
\end{proof}

\section{Model selection consistency} \label{sec:selection_consistency_app}

Define 
\begin{align} \label{def:tilde_Theta_n}
\widetilde{\Theta}_n = \bigl\{ 
\theta \in \bbR^p : 
|S_{\theta}| \leq s_n, \quad
\bigl\| \bF_{n, \theta_0}^{1/2} ( \theta - \theta_0 ) \bigr\|_{2}^{2} \leq M_n^2 s_0 \log p
\bigr\}.
\end{align}
Note that $\widetilde\Theta_n$ is slightly larger than $\Theta_n$ defined in \eqref{def:Theta_n_concentrated_support}.

\begin{lemma}[Quadratic expansion on $\widetilde{\Theta}_{n}$] \label{lemma:quad_expansion_Theta}
    Suppose that conditions in Lemma \ref{lemma:extended_Fisher_smooth} hold. Define 
    \begin{align*}
        r_n(\theta) = L_{n, \theta} - L_{n, \theta_0} - (\theta - \theta_0)^{\top} \dot{L}_{n, \theta_0} + \dfrac{1}{2}(\theta - \theta_0)^{\top} \bF_{n, \theta_0} (\theta - \theta_0).
    \end{align*}
    %where $\bF_{n, \theta} = \bX^{\top} \bW_{\theta} \bX \in \bbR^{p \times p}$ for $\theta \in \bbR^p$. 
    Then, with $\bbP_0^{(n)}$-probability at least $1 - p^{-1}$,
    \begin{align} \label{eqn:remainder_claim}
        \sup_{\theta \in \widetilde{\Theta}_{n}} \left| r_n(\theta) \right|
        \leq
        \dfrac{M_n^2}{2} \widetilde{\delta}_{n, \widetilde{\scrS}_{\Theta_n}}  s_0 \log p,    
    \end{align}
    where $\widetilde{\delta}_{n, \widetilde{\scrS}_{\Theta_n}} = \max_{S \in \widetilde{\scrS}_{\Theta_n}} \widetilde{\delta}_{n, S}$, and $\widetilde{\Theta}_n$ and $\widetilde{\scrS}_{\Theta_n}$ are defined in \eqref{def:tilde_Theta_n} and \eqref{def:A_4}, respectively.
\end{lemma}
\begin{proof}
For $\theta \in \widetilde{\Theta}_{n}$, we have
\begin{align} \label{eqn:valid_quad_eq1}
    L_{n, \theta} - L_{n, \theta_0} = (\theta - \theta_0)^{\top} \dot{L}_{n, \theta_0} - \dfrac{1}{2} (\theta - \theta_0)^{\top} \bF_{n, \theta_0} (\theta - \theta_0) + r_n(\theta),
\end{align}
and Taylor's theorem gives
\begin{align} \label{eqn:valid_quad_eq2}
    L_{n, \theta} - L_{n, \theta_0} = (\theta - \theta_0)^{\top} \dot{L}_{n, \theta_0} - \dfrac{1}{2} (\theta - \theta_0)^{\top} \bF_{n, \overline{\theta}} (\theta - \theta_0)
\end{align}
for some $\overline{\theta} \in \bbR^{p}$ with $\| \bF_{n, \theta_0}^{1/2} ( \overline{\theta} - \theta_0)  \|_{2}^{2} \leq M_n^2 s_0 \log p$.
Combining \eqref{eqn:valid_quad_eq1} and \eqref{eqn:valid_quad_eq2}, we have 
\begin{align*}
    |r_n (\theta)| 
    &= \dfrac{1}{2} \left| (\theta - \theta_0)^{\top} \left[ \bF_{n, \theta_0} - \bF_{n, \overline{\theta}}  \right] (\theta - \theta_0) \right| \\
    &= \dfrac{1}{2} \left| ( \theta_{S_{\texttt{+}}} - \thetaBest[S_{\texttt{+}}])^{\top} \left[ \bF_{n, \thetaBest[S_{\texttt{+}}] }  - \bF_{n, \overline{\theta}_{S_{\texttt{+}}}}  \right] (  \theta_{S_{\texttt{+}}} - \thetaBest[S_{\texttt{+}}]) \right|,
\end{align*}
where $S_{\texttt{+}} = S_{\theta} \cup S_0$ and the second equality holds because $\thetaBest[S_{\texttt{+}}] = \theta_{0, S_{\texttt{+}}}$ and $S_{\overline{\theta}} \subseteq S_{\texttt{+}}$.
Note also that $\overline{\theta}_{S_{\texttt{+}}} \in \localSetRn[S_{\texttt{+}}]{ \widetilde{r}_{p, |S_{\texttt{+}}|}}$ because
\begin{align*}
\left\| \bF_{n, \thetaBest[S_{\texttt{+}}]}^{1/2} \left( \overline{\theta}_{S_{\texttt{+}}} - \thetaBest[S_{\texttt{+}}]   \right) \right\|_{2}^{2} 
= 
\left\| \bF_{n, \theta_0}^{1/2} \left( \overline{\theta} - \theta_0 \right)  \right\|_{2}^{2} 
\leq M_n^2 s_0 \log p     
\leq M_n^2 |S_{\texttt{+}}| \log p.
\end{align*}
Therefore, we have
\begin{align*}
    |r_n(\theta)| 
    &\leq
    \dfrac{\widetilde{\delta}_{n, S_{\texttt{+}}}}{2} \left\| \bF_{n, \thetaBest[S_{\texttt{+}}] } ^{1/2} (  \theta_{S_{\texttt{+}}} - \thetaBest[S_{\texttt{+}}]) \right\|_{2}^{2} \quad (\because \text{ Lemma \ref{lemma:extended_Fisher_smooth}}) \\
    &\leq
    \dfrac{\widetilde{\delta}_{n, S_{\texttt{+}}}}{2}  M_n^2 s_0 \log p \quad (\because \theta \in \widetilde{\Theta}_n),
\end{align*}
which completes the proof.
%Let $M_n\widetilde{\Theta}_n = \left\{ \widetilde{\theta}_{S_{\texttt{+}}} : \theta \in \widetilde{\Theta}_{n}, S_{\texttt{+}} = S_{\theta} \cup S_0 \right\}$, where $\widetilde{\theta}_{S_{\texttt{+}}}$ is defined in a similar way.
\end{proof}
\begin{remark}[Valid quadratic expansion on $\widetilde{\Theta}_{n}$] \label{remark:bounding_remainder}
    Note that the right hand side of \eqref{eqn:remainder_claim} can be simplified under certain conditions. 
    Specifically, if $\widetilde{\delta}_{n, \widetilde{\scrS}_{\Theta_n}} \lesssim M_n (s_0 \log p /n)^{1/2}$, then we have
    \begin{align*}
     \sup_{\theta \in \widetilde{\Theta}_{n}} \left| r_n(\theta) \right|
     \lesssim
     M_n^3 \sqrt{ \dfrac{\left( s_0 \log p \right)^3}{n} }.  
    \end{align*}
    In Theorem \ref{thm:no_superset}, it is required that
    \begin{align*}
     \sup_{\theta \in \widetilde{\Theta}_{n}} \left| r_n(\theta) \right|
     \lesssim
     \log p.
    \end{align*}    
    To satisfy this condition, a sufficient condition can be summarized as:
    \begin{align*} 
     M_n^6 s_{0}^{3} \log p = o(n).
    \end{align*}
\end{remark}

\subsection*{No superset}
Note that our goal is to show the model selection consistency, say $\bbE \, \Pi_{\alpha}^n(\theta: S_{\theta} = S_0) \rightarrow 1$. In order to show this consistency, our first goal is to prove that the posterior assigns zero probability mass on the over-fitted ($S \supsetneq S_0$) model set, that is,
\begin{align*} 
	\bbE \, \Pi_{\alpha}^n (\theta: S_{\theta} \in \scrS_{\rm sp}) \rightarrow 0,
\end{align*}
where $\scrS_{\rm sp} = \{S \in \scrS_{\Theta_n} : S \supsetneq S_0\}$. %S_0 \subsetneq S \right\}$. 

\begin{theorem}[No superset] \label{thm:no_superset}
    Suppose that conditions in Theorems \ref{thm:consistency_parameter} and \ref{thm:LA_marginal_likelihood} hold. 
    Also, assume that
    \begin{align}
    \begin{aligned} \label{assume:no_super_conditions2}
        A_4 + A_7/2 
        &\geq 
        \alpha(16 C_{\rm dev} + \varepsilon_{\rm fp}) + \delta_{1} + \log_{p}(s_0)  
        + \log_{p} \left( K_{\rm dim} A_2 \sqrt{\alpha^{-1} A_6} \right), \\
        M_n^2 \widetilde{\delta}_{n, \widetilde{\scrS}_{\Theta_n}} s_0 &\leq 1,
    \end{aligned}
    \end{align}
    where $\delta_{1} \in (0, 1)$ is small enough constant and $\varepsilon_{\rm fp} = M_n^2 \widetilde{\delta}_{n, \widetilde{\scrS}_{\Theta_n}} s_0/2$. Assume further that
    \begin{align} \label{assume:no_super_conditions} \tag{E.AS.4}
        &2C_{\rm radius} K_{\rm dim} \vee K_{\rm theta} \leq M_n^2, \quad 3^{1/\delta_{1}} \leq p.
    \end{align}
    Then,
     \begin{align} \label{eqn:no_superset}
        \bbE \, \Pi_{\alpha}^n (\theta: S_{\theta} \in \scrS_{\rm sp}) \leq  2(s_0 \log p)^{-1} + 5p^{-1} + 2p^{-s_0} + 3p^{-\delta_{1}}.
    \end{align}
\end{theorem}
\begin{proof}
Recall that for $\theta_S \in \bbR^{|S|}$, $\widetilde\theta_S$ is defined as \eqref{def:tilde_theta}.
Let
\begin{align*} 
    \widetilde{\Theta}_{n, S} = \left\{ \theta_S \in \bbR^{|S|} : \widetilde{\theta}_S \in \widetilde{\Theta}_{n}  \right\}.
\end{align*} 
Throughout this proof, for a $|S|$-dimensional vector $h_S \in \bbR^{|S|}$, the corresponding $p$-dimensional vector $\widetilde h_S \in \bbR^{p}$ is defined in the same way.

By Lemmas \ref{lemma:projection_score_vec} and \ref{lemma:concentration_mle_score}, there exists an event $\Omega_n$ such that $\bbP_0^{(n)} \left(\Omega_n\right) \geq 1 - p^{-1}$ and 
\begin{align*}
    \left\| \operatorname{Proj}_{\scrC(S, S_0)^{\perp}}  \left( \xi_{n, S} \right) \right\|_{2}^{2} \leq 32 C_{\rm dev} |S \setminus S_0| \log p,
    \quad 
    \thetaMLE \in \localSetRn[S]{r_{p, S}}.
\end{align*}
for all $S \in \scrS_{\rm sp} = \left\{S \in \scrS_{\Theta_n} : S_0 \subsetneq S \right\}$ on $\Omega_n$.
Note that
\begin{align*}
    \bbE \, \Pi_{\alpha}^n (\theta: S_{\theta} \supsetneq S_0)
    &\leq 
    \bbE \, \bigl\{ \Pi_{\alpha}^n(\theta: S_{\theta} \in \scrS_{\rm sp}) \, \mathds{1}_{\Omega_n} \bigr\}
    + \bbE \, \Pi_{\alpha}^n(\Theta_n^{\rm c}) 
    + \bbP_{0}^{(n)}(\Omega_n^{\rm c} ) \\
    &\leq 
    \bbE \, \bigl\{ \Pi_{\alpha}^n(\theta: S_{\theta} \in \scrS_{\rm sp}) \, \mathds{1}_{\Omega_n} \bigr\} +
    2(s_0 \log p)^{-1} + 5p^{-1} + 2p^{-s_0},
\end{align*}
where the second inequality holds by Theorem \ref{thm:consistency_parameter}. 
Hence, it remains to prove that
\begin{align*}
    \bbE \, \bigl\{ \Pi_{\alpha}^n(\theta: S_{\theta} \in \scrS_{\rm sp}) \, \mathds{1}_{\Omega_n} \bigr\}
    \leq
    3p^{-\delta_{1}}.
\end{align*}
In the remainder of this proof, we work on the event $\Omega_n$. 

Note that $\Pi_{\alpha}^n (\theta: S_{\theta} \in \scrS_{\rm sp}) = \sum_{S \in \scrS_{\rm sp}} \pi_\alpha^n(S)$ is bounded by
\begin{align} \label{eqn:no_sup_eqn0}
\begin{aligned} 
    \sum_{S \in \scrS_{\rm sp}} \dfrac{ \pi_\alpha^n(S) }{ \pi_\alpha^n(S_0)}
    &\leq \sum_{S \in \scrS_{\rm sp}} 2\dfrac{\pi_{n}(S) }{\pi_{n} (S_0)} (1 + \alpha \lambda^{-1})^{-\left(|S| - s_0\right)/2}
    \exp \bigg(\alpha L_{n, \thetaMLE[S]} - \alpha L_{n, \thetaMLE[S_0]} \bigg)
\end{aligned}
\end{align}
by Theorem \ref{thm:LA_marginal_likelihood}. For $S \in \scrS_{\rm sp}$, we next prove the following inequality:
\begin{align*}
    L_{n, \thetaMLE} - L_{n, \thetaMLE[S_0]} \leq \left( 16C_{\rm dev} + \varepsilon_{\rm fp} \right) |S \setminus S_0| \log p.
\end{align*}
Let $S \in \scrS_{\rm sp}$ and $h_S = \thetaMLE[S] - \thetaBest$. 
Since $\widetilde{\theta}_S^{\ast} = \theta_{0}$ and $\thetaMLE \in \localSetRn[S]{r_{p, S}}$ imply that 
\begin{align*}
    \left\| \bF_{n, \theta_0}^{1/2} \widetilde{h}_S \right\|_2^2 
    = \left\| \bF_{n, \thetaBest}^{1/2} \left( \thetaMLE - \thetaBest \right) \right\|_2^2 
    \leq 
    2C_{\rm radius} K_{\rm dim} s_0 \log p
    \leq M_n^2 s_0 \log p,
\end{align*}
we have $\thetaBest + h_S \in \widetilde{\Theta}_{n, S}$.
Let $h^{\circ}_{S} = \bF_{n, \thetaBest}^{-1/2} \operatorname{Proj}_{\scrC(S, S_0)} ( \bF_{n, \thetaBest}^{1/2} h_S )$,
where $\scrC(S, S_0)$ is defined in \eqref{def:proj_coordinate_subspace}.
Also,
\begin{align*}
    \left\| \bF_{n, \theta_0}^{1/2} \widetilde{h}^{\circ}_S \right\|_{2}^{2} = 
    \left\| \bF_{n, \thetaBest}^{1/2} h^{\circ}_{S} \right\|_{2}^2 
    = \left\| \operatorname{Proj}_{\scrC(S, S_0)} \left( \bF_{n, \thetaBest}^{1/2} h_S \right) \right\|_{2}^{2} 
    \leq \left\| \bF_{n, \thetaBest}^{1/2} h_S \right\|_{2}^{2} 
    \leq M_n^2 s_0 \log p,
\end{align*}
implying $\thetaBest + h^{\circ}_{S} \in \widetilde{\Theta}_{n, S}$.
Therefore, by the above results, we can apply Lemma \ref{lemma:quad_expansion_Theta} for $h_S$ and $h^{\circ}_S$. 
Let $\cR_{n} = \sup_{\theta \in \widetilde{\Theta}_{n}} \left| r_n(\theta) \right|$, where $r_n(\theta)$ is defined as in Lemma \ref{lemma:quad_expansion_Theta}.
Then, by Lemma \ref{lemma:quad_expansion_Theta}, 
\begin{align*}
    L_{n, \thetaBest + h_S} - L_{n, \thetaBest} &\leq \dot{L}_{n, \thetaBest}^{\top}h_S - \dfrac{1}{2} h_S^{\top} \bF_{n, \thetaBest} h_S + \cR_n \\
    L_{n, \thetaBest + h^{\circ}_S} - L_{n, \thetaBest} &\geq \dot{L}_{n, \thetaBest}^{\top} h^{\circ}_S - \dfrac{1}{2} h_S^{\circ \top} \bF_{n, \thetaBest} h^{\circ}_S - \cR_n.
\end{align*}
Note that $\bF_{n, \thetaBest}^{1/2} h^{\circ}_S \in \scrC(S, S_0)$ and $\bF_{n, \thetaBest}^{1/2} (h_S - h^{\circ}_S) \in \scrC(S, S_0)^{\perp}$, where $\scrC(S, S_0)^{\perp}$ denotes the orthogonal complement of $\scrC(S, S_0)$.
Since the orthogonality gives
\begin{align*}
    \left\| \bF_{n, \thetaBest}^{1/2} h_S  \right\|_{2}^{2} 
    = \left\| \bF_{n, \thetaBest}^{1/2} h^{\circ}_S  \right\|_{2}^{2} 
    + \left\| \bF_{n, \thetaBest}^{1/2} (h_S - h^{\circ}_S)  \right\|_{2}^{2}, 
\end{align*}
we have
\begin{align*} 
    &L_{n, \thetaBest + h_S} - L_{n, \thetaBest + h^{\circ}_S} \\
    &\leq \dot{L}_{n, \thetaBest}^{\top}(h_S - h^{\circ}_S) - \dfrac{1}{2} (h_S - h^{\circ}_S)^{\top} \bF_{n, \thetaBest} (h_S - h^{\circ}_S) + 2\cR_n \\
    &= \xi_{n, S}^{\top} \bF_{n, \thetaBest}^{1/2} (h_S - h^{\circ}_S) - \dfrac{1}{2}\left\| \bF_{n, \thetaBest}^{1/2} \left(h_S - h^{\circ}_S \right) \right\|_{2}^{2} + 2\cR_n \\
    &\leq \sup_{z \in \scrC(S, S_0)^{\perp}} \left[ \xi_{n, S}^{\top}z - \dfrac{1}{2}\|z\|_{2}^{2} \right] + 2\cR_n
    = \dfrac{1}{2}\left\| \operatorname{Proj}_{\scrC(S, S_0)^{\perp}} (\xi_{n, S}) \right\|_{2}^{2} + 2\cR_n \\
    &\leq 16C_{\rm dev}|S \setminus S_0| \log p + 2\cR_n.
\end{align*}
Also, $S_{\thetaBest + h^{\circ}_S} \subseteq S_0$ because $\thetaBest = \theta_{0, S}$ and $\bF_{n, \thetaBest}^{1/2} h^{\circ}_S \in \scrC(S, S_0)$. Hence, we have
\begin{align} 
\begin{aligned}
\label{eqn:overfit_MLE}
     &L_{n, \thetaMLE} - L_{n, \thetaMLE[S_0]} 
     \leq L_{n, \thetaBest + h_S} - L_{n, \thetaBest + h^{\circ}_S} 
     \leq 16C_{\rm dev}|S \setminus S_0| \log p + 2\cR_n \\
     &\leq 16C_{\rm dev}|S \setminus S_0| \log p + \dfrac{M_n^2}{2} \widetilde{\delta}_{n, \widetilde{\scrS}_{\Theta_n}} s_0 \log p
     \quad (\because \text{ Lemma \ref{lemma:quad_expansion_Theta}} ) \\
     &= (16C_{\rm dev} + \varepsilon_{\rm fp}) |S \setminus S_0| \log p. 
     \quad (\because \varepsilon_{\rm fp} = M_n^2 \widetilde{\delta}_{n, \widetilde{\scrS}_{\Theta_n}} s_0/2)   
\end{aligned}
\end{align}

By \eqref{eqn:overfit_MLE}, \eqref{eqn:no_sup_eqn0} can be bounded as
\begin{align} \label{eqn:no_sup_eqn3}
\begin{aligned} 
    &\sum_{S \in \scrS_{\rm sp}} \dfrac{ \pi_\alpha^n(S) }{ \pi_\alpha^n(S_0) } \\
    &\leq 
    2\sum_{S \in \scrS_{\rm sp}} 
    \dfrac{\pi_n(S) }{\pi_{n}(S_0)} (1 + \alpha \lambda^{-1})^{-\left(|S| - s_0\right)/2}
    e^{ \alpha(16C_{\rm dev} + \varepsilon_{\rm fp}) (|S|-s_0) \log p} \\
    & \leq 
    2\sum_{s = s_0 + 1}^{s_n} 
    \dfrac{ \binom{p}{s_0} \binom{p-s_0}{s-s_0}}{\binom{p}{s}}  \dfrac{w_n(s)}{w_n(s_0)} (1 + \alpha \lambda^{-1})^{-(s- s_0)/2}
    e^{ \alpha(16C_{\rm dev} + \varepsilon_{\rm fp}) (s-s_0) \log p},    
\end{aligned}
\end{align}
where the last equality holds because the number of models $S$ containing $S_0$ with $|S| = s$ is given by $\binom{p - s_0}{s - s_0}$. For $s > s_0$, note that
\begin{align*} 
    \dfrac{ \binom{p}{s_0} \binom{p-s_0}{s-s_0}}{\binom{p}{s}} = \binom{s}{s-s_0} &\leq s^{s-s_0}, \\
    \dfrac{w_n(s)}{w_n(s_0)} &\leq A_2^{s - s_0} p^{-A_4(s-s_0)}, \\
    (1 + \alpha \lambda^{-1})^{-(s- s_0)/2} &\leq \left( \alpha^{-1} A_6 \right)^{(s-s_0)/2} p^{-A_7(s-s_0)/2}.
\end{align*}
Let $\omega_{p} = \log_{p} \left( K_{\rm dim} A_2 \sqrt{\alpha^{-1} A_6} \right)$ in this proof.
Hence, the right hand side of \eqref{eqn:no_sup_eqn3} is bounded by
\begin{align*} 
    &2\sum_{s = s_0 + 1}^{s_n} 
    \left(\dfrac{s A_2 \sqrt{\alpha^{-1} A_6}  }{p^{A_4 + A_7/2}} \right)^{s - s_0} 
    e^{ \alpha(16C_{\rm dev} + \varepsilon_{\rm fp}) (s-s_0) \log p} \\
    &\leq
    2\sum_{s = s_0 + 1}^{s_n} 
    \left(\dfrac{ (K_{\rm dim} s_0) A_2 \sqrt{\alpha^{-1} A_6}  }{p^{A_4 + A_7/2}} \right)^{s - s_0} 
    e^{ \alpha(16C_{\rm dev} + \varepsilon_{\rm fp}) (s-s_0) \log p} \\
    &=
    2\sum_{s = s_0 + 1}^{s_n} 
    \exp \Bigg( \bigg[ \omega_{p}
    + \log_{p}(s_0) 
    + \alpha(16C_{\rm dev} + \varepsilon_{\rm fp}) - A_4 - A_7/2 \bigg] (s - s_0) \log p \Bigg) \\
    &\leq
    2\sum_{s = s_0 + 1}^{s_n} 
    \exp \big( - \delta_{1} (s - s_0) \log p \big) 
    =
    2\sum_{t = 1}^{s_n} 
    \exp \big( - \delta_{1} t\log p \big) 
    \leq
    3p^{-\delta_{1}},
\end{align*}
where the second inequality holds by \eqref{assume:no_super_conditions2}.
This completes the proof of \eqref{eqn:no_superset}.
\end{proof}

\begin{remark}
    Under the conditions \eqref{A5:b} and $p \rightarrow \infty$, the following hold
    \begin{align*}
        \epsilon_{\rm fp} = o(1), \quad 
        \log_{p} \left( K_{\rm dim} A_2 \sqrt{\alpha^{-1} A_6} \right) = o(1),
    \end{align*}
    where $\epsilon_{\rm fp}$ is defined in \eqref{assume:no_super_conditions}.
\end{remark}

\subsection*{Beta-min condition} 
Recall the following definition:
\begin{align*}
\scrS_{\rm fp} = \left\{ S \cup S_0 : S \nsupseteq S_0,  S \in \scrS_{\Theta_n} \right\}.
\end{align*}
\begin{theorem}[$\ell_{\infty}$-estimation error] \label{thm:L_infty_estimation}
    
    Suppose that conditions in Lemma \ref{lemma:extended_Fisher_smooth2} hold.
    Also, assume that conditions in Lemma \ref{lemma:eps_deviation_ineq} hold for some constant $C_{\rm col} > 1$, and
    there exists $\kappa_{n} > 1$ such that
    \begin{align*}
        \max_{S \in \scrS_{\rm fp}} \left\| \bF_{n, \thetaBest}^{-1} \right\|_{\infty} \leq \kappa_{n} n^{-1}.
    \end{align*}
    Then, with $\bbP_0^{(n)}$-probability at least $1 - 3p^{-1}$,
    \begin{align*}
        \max_{S \in \scrS_{\rm fp}} \left\| \thetaMLE - \thetaBest  \right\|_{\infty} 
        \leq 
        \left[ \dfrac{C_{\rm radius}(K_{\rm dim} + 1)}{\phi_{2}^{2}(\widetilde{s}_n; \bW_0)} \right]^{1/2}
        \left( \dfrac{s_0 \log p}{n} \right)^{1/2} \delta_{n, \scrS_{\rm fp}}
        + 
        4\sqrt{2 C_{\rm col}} \nu_n \kappa_{n} \sqrt{\dfrac{\log p}{n}},
    \end{align*}
    where $\delta_{n, \scrS_{\rm fp}} = \max_{S \in \scrS_{\rm fp}} \delta_{n, S}$.
\end{theorem}
\begin{proof}
By conditions in Lemma \ref{lemma:extended_Fisher_smooth2}, we have $\scrS_{\rm fp} \subset \widetilde{\scrS}_{s_{\max}}$, where $\widetilde{\scrS}_{s_{\max}}$ is defined in \eqref{def:supp_set_MLE_concentration}.
By Lemma \ref{lemma:concentration_mle_score}, there exists an event $\Omega_n$ such that $\bbP_0^{(n)}(\Omega_n) \geq 1 - p^{-1}$, and on $\Omega_n$, the following inequalities hold uniformly for all $S \in \scrS_{\rm fp}$:
\begin{align*}
    \thetaMLE[S] \in \localSetRn[S]{r_{p, S}}, \quad
    \left\| \bF_{n, \thetaBest}^{1/2} \left[ \thetaMLE - \thetaBest \right] - \xi_{n, S} \right\|_{2} \leq r_{p, S} \delta_{n, S}.
\end{align*}
Let $S \in \scrS_{\rm fp}$. Note that 
\begin{align} \label{eqn:L_inf_eq1}
    \left\| \thetaMLE - \thetaBest \right\|_{\infty} 
    \leq \left\| \thetaMLE - \thetaBest - \bF_{n, \thetaBest}^{-1/2} \xi_{n, S} \right\|_{\infty} + 
    \left\| \bF_{n, \thetaBest}^{-1/2} \xi_{n, S} \right\|_{\infty}.
\end{align}
Let $e_j$ be $j$th unit vector in $\bbR^{|S|}$.
For the first term in \eqref{eqn:L_inf_eq1}, note that
\begin{align} \label{eqn:L_inf_eq3}
\begin{aligned}
    &\left\| \thetaMLE - \thetaBest + \bF_{n, \thetaBest}^{-1/2} \xi_{n, S} \right\|_{\infty} 
    = \max_{j \in [|S|]} \left| e_j^{\top} \left[ \thetaMLE - \thetaBest - \bF_{n, \thetaBest}^{-1/2} \xi_{n, S} \right] \right| \\
    &= \max_{j \in [|S|]} \left| \left( \bF_{n, \thetaBest}^{-1/2} e_j \right)^{\top} \left[ \bF_{n, \thetaBest}^{1/2} \left(\thetaMLE - \thetaBest \right) - \xi_{n, S} \right] \right| \\
    &\leq \max_{j \in [|S|]} \left\| \bF_{n, \thetaBest}^{-1/2} e_j  \right\|_2 \left\| \bF_{n, \thetaBest}^{1/2} \left(\thetaMLE - \thetaBest \right) - \xi_{n, S} \right\|_2 \\
    &\leq \rho_{\min, S}^{-1/2} r_{p, S} \delta_{n, S} \quad (\because \eqref{eqn:Fisher_theory}) \\
    &\leq ( \phi_{2}^{2}(\widetilde{s}_n; \bW_0) n )^{-1/2} ( C_{\rm radius}|S| \log p )^{1/2} \delta_{n, S}. \quad (\because S \supseteq S_0)
\end{aligned}
\end{align}
For the second term in \eqref{eqn:L_inf_eq1}, note that
\begin{align*}
    \left\| \bF_{n, \thetaBest}^{-1/2} \xi_{n, S} \right\|_{\infty} 
    &= \max_{j \in [|S|]} \left|  e_j^{\top} \bF_{n, \thetaBest}^{-1/2} \xi_{n, S} \right|
    = \max_{j \in [|S|]} \left|  e_j^{\top} \bF_{n, \thetaBest}^{-1} \bX_S^{\top} \cE \right| \\
    &\leq \max_{j \in [|S|]} \left\| \bF_{n, \thetaBest}^{-1} e_j \right\|_1 \left\| \bX_S^{\top} \cE \right\|_{\infty}
    = \left\| \bF_{n, \thetaBest}^{-1} \right\|_{\infty} \left\|  \bX_S^{\top} \cE \right\|_{\infty},
\end{align*}
where $\cE = (\epsilon_i)_{i \in [n]}$.
Also, $\left\|  \bX_S^{\top} \cE \right\|_{\infty} = \max_{j \in [S]} | \bx_j^{\top} \cE | \leq \max_{j \in [p]} | \bx_j^{\top} \cE |$.
By Lemma \ref{lemma:eps_deviation_ineq}, 
\begin{align*}
    \bbP_0^{(n)} \left\{  \max_{j \in [p]} \left| \bx_j^{\top} \cE \right| > 4\sqrt{2 C_{\rm col}} \nu_n \left( n \log p \right)^{1/2} \right\} \leq 2p^{-1}, 
\end{align*}
where $\nu_n$ is defined in Lemma \ref{lemma:eps_deviation_ineq}.
Therefore, we have, with $\bbP_0^{(n)}$-probability at least $1 - 2p^{-1}$,
\begin{align} \label{eqn:L_inf_eq5}
    \max_{S \in \scrS_{\rm fp}} \left\| \bF_{n, \thetaBest}^{-1/2} \xi_{n, S} \right\|_{\infty} 
    \leq \max_{S \in \scrS_{\rm fp}}  \left\| \bF_{n, \thetaBest}^{-1} \right\|_{\infty} \left\|  \bX_S^{\top} \cE \right\|_{\infty} 
    \leq 4\sqrt{2 C_{\rm col}} \nu_n \kappa_{n} \sqrt{\dfrac{\log p}{n}}.
\end{align}
Let $\Omega_n^{'}$ be the intersection of $\Omega_n$ and the event where \eqref{eqn:L_inf_eq5} holds. Then, $\bbP_0^{(n)}(\Omega_n^{'}) \geq 1 - 3p^{-1}$.
Combining \eqref{eqn:L_inf_eq3} and \eqref{eqn:L_inf_eq5}, \eqref{eqn:L_inf_eq1} is further bounded by, on $\Omega_n^{'}$,
\begin{align*}
    \left\| \thetaMLE - \thetaBest \right\|_{\infty} 
    &\leq  
    \left[ \phi_{2}^{2}(\widetilde{s}_n; \bW_0) n \right]^{-1/2} (C_{\rm radius}|S| \log p )^{1/2} \delta_{n, S}
    + 
    4\sqrt{2 C_{\rm col}} \nu_n \kappa_{n} \sqrt{\dfrac{\log p}{n}} \\
    &\leq 
    \left[ \dfrac{C_{\rm radius}(K_{\rm dim} + 1)}{\phi_{2}^{2}(\widetilde{s}_n; \bW_0)} \right]^{1/2}
    \left( \dfrac{s_0 \log p}{n} \right)^{1/2} \delta_{n, \scrS_{\rm fp}}
    + 
    4\sqrt{2 C_{\rm col}} \nu_n \kappa_{n} \sqrt{\dfrac{\log p}{n}},
\end{align*}
which completes the proof.
\end{proof}

We now demonstrate that the posterior includes all necessary covariates, that is,
\begin{align} \label{eqn:no_underfitting}
	\bbE \, \Pi_{\alpha}^n(\theta: S_{\theta} \nsupseteq  S_0) = o(1).
\end{align}
Combining with Theorem \ref{thm:no_superset}, \eqref{eqn:no_underfitting} 
ensures that
$$
\bbE \, \Pi_{\alpha}^n (\theta: S_{\theta} \neq S_0)
    = 
    \bbE \, \Pi_{\alpha}^n(\theta: S_{\theta} \supsetneq S_0)
    + 
    \bbE \, \Pi_{\alpha}^n (\theta: S_{\theta} \nsupseteq S_0)
    = o(1),
$$
leading to model selection consistency:
$$
\bbE \, \Pi_{\alpha}^n (\theta: S_{\theta} = S_0 ) \rightarrow 1.
$$
To show \eqref{eqn:no_underfitting}, it is required that all non-zero variables in the correct model $S_0$ possess sufficiently large magnitude. Specifically, recall the condition in \eqref{A7:a}:
\begin{align*} 
	\vartheta_{n, p} = \min_{j \in S_0} | \theta_{0, j}  | \geq K_{\rm min} \Bigg[
    \left( \nu_n \kappa_{n} \sqrt{\frac{\log p}{ n }} \right) \wedge \left( \phi_2^{-1}\left( \widetilde{s}_n ; \bW_0 \right) \sqrt{\frac{s_0 \log p}{ n }} \right) \Bigg].
\end{align*} 
The above display is often called \textit{beta-min} condition in the variable selection literature. 

\begin{theorem}[Selection consistency] \label{thm:selection_consistency}
    Suppose that conditions in Theorems \ref{thm:no_superset}, \ref{thm:L_infty_estimation}, and equation \eqref{A7:a} hold for some constant $K_{\rm min} > 0$. 
    Also, assume that
    \begin{align} \label{assume:selection_conditions} \tag{E.AS.5}
    \begin{aligned}
        2K_{\rm dim} s_0 \leq p, \quad
        C \leq K_{\rm min}, \quad
        8C_{\rm radius}K_{\rm dim} + 16K_{\rm theta} \leq M_n^2,  \\
    \end{aligned}
    \end{align}
    where $C = C(\alpha, A_{1}, A_{2}, A_{3}, A_{4}, A_5, A_6, \alpha, C_{\rm dev}, K_{\rm dim})$ is large enough constant.
    Assume further that
    \begin{align} \label{assume:selection_conditions2}
    \begin{aligned}
        \left[ \dfrac{C_{\rm radius} (K_{\rm dim} + 1) }{32 C_{\rm col} \phi_2^2\left( \widetilde{s}_n ; \bW_0 \right) \nu_n^2 \kappa_{n}^2  } \right]
        s_0 \delta_{n, \scrS_{\rm fp}}^2 &\leq 1, \\
        \dfrac{ K_{\rm theta} }{ \nu_n \kappa_{n} \phi_2\left( \widetilde{s}_n ; \bW_0 \right)} \vee \dfrac{ 16 }{ \nu_{n}^2 \kappa_{n}^2  \phi_2^2\left( \widetilde{s}_n ; \bW_0 \right)} &< K_{\rm min}
    \end{aligned}
    \end{align}
    Then, 
    \begin{align} \label{eqn:selection_consistency}
    \bbE \, \Pi_{\alpha}^n (\theta: S_{\theta} = S_0 )
    \geq 
    1 - \left[ 4(s_0 \log p)^{-1} + 25p^{-1} + 4p^{-s_0} + 3p^{-\delta_{1}} \right].
    \end{align}
\end{theorem}
\begin{proof}
To obtain \eqref{eqn:selection_consistency}, combining with \eqref{eqn:no_superset}, 
we will prove that 
\begin{align*}
    \bbE \, \Pi_{\alpha}^n (\theta: S_{\theta} \nsupseteq S_0) 
    \leq 2(s_0 \log p)^{-1} + 20p^{-1} + 2p^{-s_0}.
\end{align*}
Let $\widetilde{\Omega}_n$ denote the event defined in Theorem \ref{thm:no_superset}. Furthermore, let $\Omega_n$ be the intersection of $\widetilde{\Omega}_n$ and the event where the result of Lemma \ref{thm:L_infty_estimation} holds. Then, we have $\bbP_0^{(n)}( \Omega_n ) \geq 1 - 4p^{-1}$.
Let $\scrS_{\rm omit} = \left\{ S \in \scrS_{\Theta_n} : S \nsupseteq S_0 \right\}$.
Since
\begin{align*}
    \bbE \, \Pi_{\alpha}^n(\theta: S_{\theta} \nsupseteq S_0)
    &\leq 
    \bbE \, \bigl\{ \Pi_{\alpha}^n(\theta: S_{\theta} \in \scrS_{\rm omit}) \, \mathds{1}_{\Omega_n} \bigr\} 
    + \bbE \, \Pi_{\alpha}^n(\Theta_n^{\rm c}) 
    + \bbP_{0}^{(n)} (\Omega_n^{\rm c}) \\
    &\leq 
    \bbE \, \bigl\{ \Pi_{\alpha}^n(\theta: S_{\theta} \in \scrS_{\rm omit}) \, \mathds{1}_{\Omega_n} \bigr\} +
    2(s_0 \log p)^{-1} + 8p^{-1} + 2p^{-s_0},
\end{align*}
we need to prove that
\begin{align*}
\bbE \, \bigl\{ \Pi_{\alpha}^n(\theta: S_{\theta} \in \scrS_{\rm omit}) \, \mathds{1}_{\Omega_n} \bigr\}
\leq 12p^{-1}.
\end{align*}
In the remainder of this proof, we work on the event $\Omega_n$.
Note that
\begin{align} 
\begin{aligned} \label{eqn:beta_min_claim1}
&\Pi_{\alpha}^n(\theta: S_{\theta} \in \scrS_{\rm omit}) \\
&= \sum_{S \in \scrS_{\rm omit}} \pi_\alpha^n(S) 
\leq \sum_{S \in \scrS_{\rm omit}} \dfrac{ \pi_\alpha^n(S) }{ \pi_\alpha^n(S_0)} \\
&\leq 2 \Bigg[ \sum_{S \in \scrS_{\rm omit}} 
    \dfrac{\pi_n(S)}{\pi_n(S_0)} (1 + \alpha \lambda^{-1})^{-(|S| - s_0)/2} \exp\bigl(\alpha L_{n, \thetaMLE[S]} - \alpha L_{n, \thetaMLE[S_0]} \bigr) \Bigg].
\end{aligned}
\end{align}
Here, our focus is on non-empty support sets $S$ because 
$K_{\rm theta}/ \left[ \nu_n \kappa_{n} \phi_2\left( \widetilde{s}_n ; \bW_0 \right) \right] < K_{\rm min}$ implies
$\varnothing \notin \scrS_{\Theta_n}$. Consequently, this allows us to apply Theorem \ref{thm:LA_marginal_likelihood} for the second inequality in \eqref{eqn:beta_min_claim1}.

We will obtain the upper bound of the likelihood ratio in \eqref{eqn:beta_min_claim1}.
Let $S \in \scrS_{\rm omit}$. Denote $S_{\texttt{+}} = S \cup S_0$, $r_1 = |S_0 \cap S^{\rm c}|$ and $r_2 = |S_0^{\rm c} \cap S|$.
By \eqref{eqn:overfit_MLE}, we have
\begin{align*}
    L_{n, \thetaMLE[S]} - L_{n, \thetaMLE[S_0]}
    &= L_{n, \thetaMLE[S]} - L_{n, \thetaMLE[S_{\texttt{+}}]} + L_{n, \thetaMLE[S_{\texttt{+}}]} - L_{n, \thetaMLE[S_0]} \\
    &\leq 
    L_{n, \thetaMLE[S]} - L_{n, \thetaMLE[S_{\texttt{+}}]} 
    + (16C_{\rm dev} + \varepsilon_{\rm fp}) r_2 \log p,
\end{align*}
where the inequality holds by Theorem \ref{thm:no_superset} and $\varepsilon_{\rm fp}$ is defined in \eqref{assume:no_super_conditions2}.

Next, we will prove that $L_{n, \thetaMLE[S]} - L_{n, \thetaMLE[S_{\texttt{+}}]} \leq -K_{\min} r_1 \log p$.
Given a suitable ordering of indices, let $\overline{\theta}_S = (\overline{\theta}_{j})_{j=1}^{|S_{\texttt{+}}|}$, where $\overline{\theta}_{j} = \widehat{\theta}_{S, j}^{ \operatorname{\texttt{MLE}}}$ for $j \in S$ and $\overline{\theta}_{j} = 0$ for $j \in S_{\texttt{+}} \setminus S$. 
Since $\dot{L}_{n, \thetaMLE[S_{\texttt{+}}]} = 0$, Taylor's theorem gives
\begin{align*}
    L_{n, \thetaMLE[S]} - L_{n, \thetaMLE[S_{\texttt{+}}]}
    &=
    L_{n, \overline{\theta}_S} - L_{n, \thetaMLE[S_{\texttt{+}}]} \\
    &= \dot{L}_{n, \thetaMLE[S_{\texttt{+}}]}^{\top} \left( \overline{\theta}_S - \thetaMLE[S_{\texttt{+}}] \right) 
    - \dfrac{1}{2} \left( \overline{\theta}_S - \thetaMLE[S_{\texttt{+}}] \right)^{\top} \bF_{n, \theta_{S_{\texttt{+}}}^{\circ}}  \left( \overline{\theta}_S - \thetaMLE[S_{\texttt{+}}] \right) \\
    &= - \dfrac{1}{2} \left( \overline{\theta}_S - \thetaMLE[S_{\texttt{+}}] \right)^{\top} \bF_{n, \theta_{S_{\texttt{+}}}^{\circ}}  \left( \overline{\theta}_S - \thetaMLE[S_{\texttt{+}}] \right)
\end{align*}
for some $\theta_{S_{\texttt{+}}}^{\circ}$ on the line segment between $\overline{\theta}_S$ and $\thetaMLE[S_{\texttt{+}}]$.

To apply Lemma \ref{lemma:extended_Fisher_smooth} for $\theta_{S_{\texttt{+}}}^{\circ}$, we need to verify $\thetaMLE[S_{\texttt{+}}], \overline{\theta}_S \in \widetilde{\Theta}_{n, S_{\texttt{+}}}$.
Firstly, note that $\thetaMLE[S_{\texttt{+}}] \in \widetilde{\Theta}_{n, S_{\texttt{+}}}$ because
\begin{align*}
    \left\| \bF_{n, \theta_0}^{1/2} 
    \left( \widetilde{\theta}_{S_{\texttt{+}}}^{\texttt{MLE}} - \theta_0 \right)
    \right\|_{2}^{2} 
    = 
    \left\| \bF_{n, \thetaBest[S_{\texttt{+}}]}^{1/2} 
    \left( \thetaMLE[S_{\texttt{+}}] - \thetaBest[S_{\texttt{+}}] \right)
    \right\|_{2}^{2} 
    \leq
    C_{\rm radius} (K_{\rm dim} + 1) s_0 \log p
    \leq 
    M_n^2 s_0 \log p,
\end{align*}
where the second inequality holds by \eqref{assume:no_super_conditions}.
For $\overline{\theta}_S$, note that
\begin{align}
\begin{aligned} \label{eqn:no_false_eq0}
   &\left\| \bF_{n, \theta_0}^{1/2} \left( \widetilde{\theta}_{S}^{\texttt{MLE}} - \theta_0 \right) \right\|_{2}^{2}
   = 
   \left\| \bF_{n, \thetaBest[S_{\texttt{+}}]}^{1/2} \left( \overline{\theta}_S - \thetaBest[S_{\texttt{+}}] \right) \right\|_{2}^{2} \\
   & \leq 
   2\left\| \bF_{n, \theta_0}^{1/2} \left( \widetilde{\theta}_{S}^{\texttt{MLE}} - \widetilde{\theta}_{S}^{\ast} \right) \right\|_{2}^{2}  + 
   2\left\| \bF_{n, \theta_0}^{1/2} \left( \widetilde{\theta}_{S}^{\ast} - \theta_0 \right) \right\|_{2}^{2} \\
   &= 
   2\left\| \bV_{n, S}^{1/2} \bF_{n, \thetaBest}^{-1/2} \bF_{n, \thetaBest}^{1/2} \left( \thetaMLE - \thetaBest \right) \right\|_{2}^{2}  + 
   2\left\| \bF_{n, \theta_0}^{1/2} \left( \widetilde{\theta}_{S}^{\ast} - \theta_0 \right) \right\|_{2}^{2} \\
   &\leq 
   2 \left\| \bF_{n, \thetaBest}^{-1/2} \bV_{n, S} \bF_{n, \thetaBest}^{-1/2} \right\|_{2}
   \left\| \bF_{n, \thetaBest}^{1/2} \left( \thetaMLE - \thetaBest \right) \right\|_{2}^{2}  + 
   2\left\| \bF_{n, \theta_0}^{1/2} \left( \widetilde{\theta}_{S}^{\ast} - \theta_0 \right) \right\|_{2}^{2} \\
   &\leq 
   4\left\| \bF_{n, \thetaBest}^{1/2} \left( \thetaMLE - \thetaBest \right) \right\|_{2}^{2}  + 
   2\left\| \bF_{n, \theta_0}^{1/2} \left( \widetilde{\theta}_{S}^{\ast} - \theta_0 \right) \right\|_{2}^{2} \quad (\because \text{ Lemma \ref{lemma:mis_on_posterior_concentration_set}})\\
   &\leq 
   4\left( 2 C_{\rm radius}|S|\log p \right)  + 2 \left( 8 K_{\rm theta} s_0 \log p \right) \quad (\because \text{ Lemmas \ref{lemma:concentration_mle_score}, \ref{lemma:mis_on_posterior_concentration_set}}) \\
   &\leq 
   4\left( 2 C_{\rm radius}K_{\rm dim} s_0 \log p \right)  + 2 \left( 8 K_{\rm theta} s_0 \log p \right) \\
   & = \left( 8C_{\rm radius}K_{\rm dim} + 16K_{\rm theta}  \right)s_0 \log p \\
   &\leq M_n^2 s_0 \log p \quad (\because \eqref{assume:selection_conditions}),    
\end{aligned}
\end{align}
which shows $\overline{\theta}_S \in \widetilde{\Theta}_{n, S_{\texttt{+}}}$.
Accordingly, we can apply Lemma \ref{lemma:extended_Fisher_smooth} for $\theta_{S_{\texttt{+}}}^{\circ} \in \widetilde{\Theta}_{n, S_{\texttt{+}}}$. 
Therefore, $L_{n, \thetaMLE[S]} - L_{n, \thetaMLE[S_{\texttt{+}}]}$ is further bounded by
\begin{align} \label{eqn:no_false_eq1}
    - \dfrac{1 - \widetilde{\delta}_{n, S_{\texttt{+}}} }{2} \left( \overline{\theta}_S - \thetaMLE[S_{\texttt{+}}] \right)^{\top} 
    \bF_{n, \thetaBest[S_{\texttt{+}}]}
    \left( \overline{\theta}_S - \thetaMLE[S_{\texttt{+}}] \right)
    \leq -\dfrac{n}{4} \phi_2^2(\widetilde{s}_n; \bW_0) \left\| \overline{\theta}_S - \thetaMLE[S_{\texttt{+}}] \right\|_2^{2},
\end{align}
where the inequality holds by $\widetilde{\delta}_{n, S_{\texttt{+}}} \leq 1/2$. 

Now, we need to obtain the lower bound of $\| \overline{\theta}_S - \thetaMLE[S_{\texttt{+}}] \|_2$.
Given a suitable ordering of indices, let $\check{\theta}_{S_{\texttt{+}}} = (\check{\theta}_{j})_{j \in S_{\texttt{+}}}$ with 
\begin{align*}
	\check{\theta}_{j} = 
	\begin{cases}
		\theta_{0, j}, \quad &\text{ if } j \in S_0 \cap S^{\rm c}, \\
		\widehat{\theta}_{S_{\texttt{+}}, j}^{\texttt{MLE}}, \quad &\text{ if } j \in S
	\end{cases},
\end{align*}
and $\widehat{\theta}_{S_{\texttt{+}}, S'}^{\texttt{MLE}} = ( \widehat{\theta}_{S_{\texttt{+}}, j}^{\texttt{MLE}} )_{j \in S'}$,
where $S' \subset S_{\texttt{+}}$.
Since $S_{\bar{\theta}_{S}} = S$ and $S_0 \subseteq S_{\texttt{+}}$, we have
\begin{align*}
    \left\| \overline{\theta}_{S} -  \widehat{\theta}_{S_{\texttt{+}}}^{\texttt{MLE}}  \right\|_{2}
    &\geq \left\| \overline{\theta}_{S} -  \check{\theta}_{S_{\texttt{+}}}  \right\|_{2} - \left\|\check{\theta}_{S_{\texttt{+}}} - \widehat{\theta}_{S_{\texttt{+}}}^{\texttt{MLE}} \right\|_{2} \\
    &= \left\| \theta_{0, S_0 \cap S^{\rm c}}  \right\|_{2} +  
    \left\| \widehat{\theta}_{S}^{\texttt{MLE}} - \widehat{\theta}_{S_{\texttt{+}}, S}^{\texttt{MLE}}  \right\|_2 -
    \left\| \widehat{\theta}_{S_{\texttt{+}}, S_0 \cap S^{\rm c}}^{\texttt{MLE}} - \theta_{0, S_0 \cap S^{\rm c}} \right\|_{2} \\
    &\geq \left\| \theta_{0, S_0 \cap S^{\rm c}}  \right\|_{2} - 
    \left\| \widehat{\theta}_{S_{\texttt{+}}, S_0 \cap S^{\rm c}}^{\texttt{MLE}} - \theta_{0, S_0 \cap S^{\rm c}} \right\|_{2} 
    \geq \sqrt{r_1} \bigg[
        \vartheta_{n, p} - \left\| \widehat{\theta}_{S_{\texttt{+}}}^{\texttt{MLE}} - \theta_{S_{\texttt{+}}}^{\ast} \right\|_{\infty}
    \bigg],
\end{align*}
where $\vartheta_{n, p} = \min_{j \in S_0} |\theta_{0, j}|$. 
By Lemma \ref{thm:L_infty_estimation}, we have
\begin{align} \label{eqn:beta_min_ell_infty}
\begin{aligned}
    &\left\| \thetaMLE[S_{\texttt{+}}] - \theta_{S_{\texttt{+}}}^{\ast} \right\|_{\infty} \\
    &\leq 
    \left[ \dfrac{C_{\rm radius}(K_{\rm dim} + 1)}{\phi_{2}^{2}(\widetilde{s}_n; \bW_0)} \right]^{1/2}
    \left( \dfrac{s_0 \log p}{n} \right)^{1/2} \delta_{n, \scrS_{\rm fp}}
        + 4\sqrt{2 C_{\rm col}} \nu_n \kappa_{n} \sqrt{\dfrac{\log p}{n}} \\
    &\leq 
    8\sqrt{2 C_{\rm col}} \nu_n \kappa_{n} \sqrt{\dfrac{\log p}{n}},
\end{aligned}
\end{align}
where the second inequality holds by \eqref{assume:selection_conditions2}.
We firstly consider the following case:
\begin{align*}
    \nu_n \kappa_{n} \sqrt{\dfrac{\log p}{n}} \leq \phi_2^{-1}\left( \widetilde{s}_n ; \bW_0 \right) \sqrt{\dfrac{s_0 \log p}{n}}.
\end{align*}
Combining \eqref{A7:a} and \eqref{eqn:beta_min_ell_infty}, we have
\begin{align*}
    \left\| \overline{\theta}_{S} -  \widehat{\theta}_{S_{\texttt{+}}}^{\texttt{MLE}}  \right\|_{2}
    \geq 
    \sqrt{r_1}  \left( K_{\rm min} - 8\sqrt{2 C_{\rm col}} \right) \nu_n \kappa_{n} \sqrt{\dfrac{\log p}{n}}
    \geq 
    \dfrac{K_{\rm min} \nu_n \kappa_{n}}{2} \sqrt{\dfrac{r_1 \log p}{n}},
\end{align*}
where the second inequality holds by \eqref{assume:selection_conditions}.
It follows that 
\begin{align} 
\begin{aligned} \label{eqn:no_false_eq2}
    L_{n, \thetaMLE[S]} - L_{n, \thetaMLE[S_{\texttt{+}}]}
    &\leq 
    -\dfrac{n}{4} \phi_2^2(\widetilde{s}_n; \bW_0)  
    \left( \dfrac{K_{\rm min}^2 \nu_{n}^2 \kappa_{n}^2 }{4}
    \dfrac{r_1 \log p}{n}\right) \\
    &= - \left( \dfrac{\phi_2^2(\widetilde{s}_n; \bW_0) K_{\rm min}^2  \nu_{n}^2 \kappa_{n}^2}{16} \right) r_1 \log p \\
    &\leq -K_{\rm min} r_1 \log p,     
\end{aligned}
\end{align}
where the second inequality holds by \eqref{assume:selection_conditions2}.
Secondly, suppose that we have the following:
\begin{align*}
    \nu_n \kappa_{n} \sqrt{\dfrac{\log p}{n}} > \phi_2^{-1}\left( \widetilde{s}_n ; \bW_0 \right) \sqrt{\dfrac{s_0 \log p}{n}}.
\end{align*}
For large enough $K_{\min}$ and $S \in \scrS_{\rm omit}$, we have
\begin{align*}
    \left\| \widetilde{\theta}_{S}^{\texttt{MLE}} - \theta_0 \right\|_{2} 
    \geq 
    \vartheta_{n, p}
    =
    \dfrac{K_{\rm min}}{\phi_2\left( \widetilde{s}_n ; \bW_0 \right)}  \sqrt{\dfrac{s_0 \log p}{n}}
    > 
    \dfrac{8C_{\rm radius}K_{\rm dim} + 16K_{\rm theta}}{\phi_2\left( \widetilde{s}_n ; \bW_0 \right)}  \sqrt{\dfrac{s_0 \log p}{n}},
\end{align*}
which contradicts \eqref{eqn:no_false_eq0}. Therefore, we only need to consider the first case.

Combining the upper bound in \eqref{eqn:no_false_eq2}, the bracket term in \eqref{eqn:beta_min_claim1} is bounded by
\begin{align} \label{eqn:underfit_eqn7}
\begin{aligned}
   \sum_{r_1 = 1}^{s_0} \sum_{r_2 = 0}^{s_n} \binom{s_0}{r_1} \binom{p - s_0}{r_2} \dfrac{\binom{p}{s_0}}{\binom{p}{s}} \dfrac{w_n(s)}{w_n(s_0)} (1 + \alpha \lambda^{-1})^{-(s - s_0)/2} 
   e^{ - c_1 r_1 \log p + c_2 r_2 \log p},
\end{aligned}
\end{align}
where $c_1 = \alpha K_{\rm min}$ and $c_2 = \alpha (16C_{\rm dev} + \varepsilon_{\rm fp})$.

We decompose our analysis based on the size of the model, $|S|$, divided into three separate cases. First, consider $|S| = S_0$ case, implying $r_1 = r_2$. Then, \eqref{eqn:underfit_eqn7} is bounded by
\begin{align*}
    &\sum_{r = 1}^{\infty}  \binom{s_0}{r} \binom{p - s_0}{r} e^{ (16 \alpha C_{\rm dev} + \alpha \varepsilon_{\rm fp} - \alpha K_{\rm min})r \log p} \\
    &\leq    
    \sum_{r = 1}^{\infty}  e^{ (2 + 16 \alpha C_{\rm dev} + \alpha \varepsilon_{\rm fp} - \alpha K_{\rm min})r \log p}
    \leq 
    \sum_{r = 1}^{\infty}  p^{-r} 
    \leq 2p^{-1}
\end{align*}
because $\binom{s_0}{r}, \binom{p - s_0}{r} \leq p^r$ and \eqref{assume:selection_conditions}. Second, consider $|S| > s_0$ case, implying $r_2 > r_1$.
Then, the following inequalities hold:
\begin{align*}
    &\dfrac{w_n(|S|)}{w_n(s_0)} \leq A_2^{|S| - s_0} p^{-A_4(|S| - s_0)} = A_2^{r_2 - r_1} p^{-A_4(r_2 - r_1)}, \\
    &\left(1 + \alpha \lambda^{-1}\right)^{-(|S| - s_0)/2} \leq \left( \alpha^{-1} A_6 \right)^{(r_2 - r_1)/2} p^{-A_7(r_2 - r_1)/2}, \\
    &\binom{s_0}{r_1} \leq p^{r_1}, \quad \binom{p - s_0}{r_2} \leq p^{r_2}, \quad
    \dfrac{\binom{p}{s_0}}{\binom{p}{|S|}} \leq (2 K_{\rm dim})^{r_2 - r_1} s_0^{r_2 - r_1} p^{-(r_2 - r_1)},
\end{align*}
where the last inequality holds by $p \geq 2K_{\rm dim}s_0$.
Let $\omega_{p} = \log_{p}(2A_2 K_{\rm dim} \sqrt{\alpha^{-1} A_6})$ in this proof.
Hence, \eqref{eqn:underfit_eqn7} is bounded by
\begin{align*}
    &\sum_{r_1 = 1}^{s_0} \sum_{r_2 > r_1}^{s_n}
     \left( 2A_2 K_{\rm dim} \sqrt{\alpha^{-1} A_6} s_0 \right)^{r_2 - r_1} 
     e^{(A_4 + 2 - \alpha K_{\rm min})r_1\log p + (16 \alpha C_{\rm dev} + \alpha \varepsilon_{\rm fp} - A_4 -A_7/2 )r_2 \log p} \\
    &=
    \sum_{r_1 = 1}^{s_0} \sum_{r_2 > r_1}^{s_n}
     e^{(A_4 + 1 - \omega_{p} - \alpha K_{\rm min})r_1\log p + (16 \alpha C_{\rm dev} + \alpha \varepsilon_{\rm fp} + \log_{p}(s_0) + \omega_{p} - A_4 -A_7/2 )r_2 \log p} \\
    &\leq 
    \sum_{r_1 = 1}^{s_0} \sum_{r_2 > r_1}^{s_n}
    e^{(A_4 + 1 - \omega_{p} - \alpha K_{\rm min}) r_1\log p - (\log_{p}(2) + \delta_{1}) r_2 \log p} 
    \quad \left(\because \eqref{assume:no_super_conditions} \right)
    \\
    &\leq 
    \sum_{r_1 = 1}^{s_0} \sum_{r_2 > r_1}^{s_n}
    e^{(A_4 + 1 - \omega_{p} - \alpha K_{\rm min})r_1\log p}  \\
    & \leq
    \sum_{r_1 = 1}^{\infty}
    p^{-r_1} 
    \leq 
    2p^{-1},
\end{align*}
where the last two inequalities hold by $s_n \leq p$, \eqref{assume:selection_conditions} and $p \geq 2$.

Third, consider $|S| < s_0$ case, yielding $r_1 > r_2$.
Then, the following inequalities hold:
\begin{align*}
    &\dfrac{w_n(|S|)}{w_n(s_0)} 
    \leq A_1^{-(s_0-|S|)} p^{A_3(s_0 - |S|)} = A_1^{-(r_1 - r_2)} p^{A_3(r_1 - r_2)}
    = e^{(A_3 + \log_{p}(A_1^{-1}))(r_1 - r_2)\log p}, \\
    &\dfrac{\binom{p}{s_0}}{ \binom{p}{|S|}} \leq \dfrac{\binom{s_0}{|S|} \binom{p}{s_0}}{ \binom{p}{|S|}} = \binom{p - |S|}{s_0- |S|} \leq p^{s_0 - |S|} = e^{(r_1 - r_2) \log p}, \\
    &\binom{s_0}{r_1} \leq p^{r_1}, \quad \binom{p - s_0}{r_2} \leq p^{r_2}, 
\end{align*}
and
\begin{align*}
    (1 + \alpha \lambda^{-1})^{-(|S| - s_0)/2} 
    &\leq (2\lambda^{-1})^{-(|S| - s_0)/2} 
    \leq (2p^{A_5})^{-(|S| - s_0)/2} = (2p^{A_5})^{(r_1 - r_2)/2} \\
    &= e^{(A_5/2 + \log_{p}(2)/2)(r_1 - r_2) \log p},
\end{align*}
where the second holds by \eqref{A2:b}.
Let $\widetilde{\omega}_{p} = \log_{p}(A_{1}^{-1}) + \log_{p}(2)/2$.
Therefore, \eqref{eqn:underfit_eqn7} is bounded by
\begin{align*}
    &\sum_{r_1 = 1}^{s_0} \sum_{r_2 < r_1}^{s_n}
    e^{(2 + A_3 + A_5/2 + \widetilde{\omega}_{p} - \alpha K_{\rm min})r_1\log p + (16 \alpha C_{\rm dev} + \alpha \varepsilon_{\rm fp} -A_{3} -A_5/2 - \widetilde{\omega}_{p} )r_2 \log p} \\
    &\leq 
    \sum_{r_1 = 1}^{s_0} \sum_{r_2 < r_1}^{s_n}
    e^{(2 + 2A_3 + A_5 + 2|\widetilde{\omega}_{p}| + 16 \alpha C_{\rm dev} + \alpha \varepsilon_{\rm fp} - \alpha K_{\rm min})r_1\log p} \quad (\because r_1 > r_2) \\
    & \leq 
    \sum_{r_1 = 1}^{\infty} 
    p^{-r_1} \\
    & \leq 
    2p^{-1},
\end{align*}
where the last two inequalities holds by \eqref{assume:selection_conditions} and $p \geq 2$, respectively.
Therefore, we have
\begin{align*}
\bbE \, \bigl\{ \Pi_{\alpha}^n (\theta: S_{\theta} \in \scrS_{\rm omit}) \, \mathds{1}_{\Omega_n} \bigr\}
&\leq 12p^{-1},
\end{align*}
which completes the proof.
\end{proof}

\section{Proofs for Section \ref{sec:examples}}

\begin{proof}[Proof of Corollary \ref{coro:example_design_matrix}]
    This corollary directly follows from Lemmas \ref{lemma:design_row_norm}, \ref{lemma:maximum_value_lowerbound}, \ref{lemma:design_column_norm}, \ref{lemma:natural_parameter}, \ref{lemma:matrix_infty_norm} and \ref{lemma:cubic_poly_Gaussian}.
\end{proof}

\begin{proof}[Proof of Corollary \ref{coro:example_logit_1}]
    By Lemma \ref{lemma:least_eigenvalue_logit} and $s_0 \log p  = o(n)$, we have
    \begin{align*}
        \phi_{2}^{2}\left( \widetilde{s}_{n}; \bW_0 \right) \geq \dfrac{1}{216}e^{-2\| \theta_0 \|_{2}}
    \end{align*}
    with $\bbP$-probability at least $1-5e^{-n/36}$. Since the Cauchy--Schwarz inequality implies that $\phi_1(s; \bW) \geq \phi_2(s; \bW)$ for any $s \in \bbN$, this completes the proof of the first assertion in \eqref{claim:logit_0}. 
    The second and third assertions in \eqref{claim:logit_0} directly follow from Lemmas \ref{lemma:variance_bound_example_logit} and \ref{lemma:least_eigenvalue_logit}, respectively.
    Also, \eqref{claim:logit_1} follows from Theorem \ref{thm:logit_mis_example}.
    The condition that $\| \theta_0 \|_{2} \leq C$ for some constant $C > 0$ and the assertions in \eqref{claim:logit_0} complete the proofs of the first and second assertions in \eqref{claim:logit_2}. Combining the second assertion in \eqref{claim:logit_2} and \eqref{eqn:example_design_matrix_2}, one can easily check that the third assertion is satisfied. Finally, the fourth assertion directly follows from Lemma \ref{lemma:orlic_bound}.
\end{proof}

\begin{proof}[Proof of Corollary \ref{coro:model_selection_logit}]
    By Corollaries \ref{coro:example_design_matrix} and \ref{coro:example_logit_1}, the assumptions in \eqref{assume:model_selection_logit} and those stated above imply all conditions required for Theorem \ref{thm:selection_main} under the random design $\bX$.
    Conditioning on an event where \eqref{eqn:example_design_matrix_1}, \eqref{eqn:example_design_matrix_2},  \eqref{claim:logit_0}, \eqref{claim:logit_1} and \eqref{claim:logit_2} hold,
    all remaining proofs are identical to those of Theorem \ref{thm:selection_main}.
\end{proof}

\begin{proof}[Proof of Corollary \ref{coro:example_poisson_1}]
    By Corollaries \ref{coro:example_design_matrix} and \ref{coro:example_poisson_1}, the assumptions in \eqref{assume:model_selection_Poisson} and those stated above imply all the conditions required for Theorem \ref{thm:selection_main} under the random design $\bX$.
    Conditioning on an event where \eqref{eqn:example_design_matrix_1}, \eqref{eqn:example_design_matrix_2}, \eqref{eqn:example_poisson_1_1}, \eqref{claim:Poisson_1} and \eqref{claim:Poisson_2} hold,
    all remaining proofs are identical to those of Theorem \ref{thm:selection_main}.
\end{proof}

\begin{proof}[Proof of Corollary \ref{coro:model_selection_poisson}]
    By Lemma \ref{lemma:Poisson_least_eigenvalue_V} and $s_0 \log p  = o(n)$, we have
    \begin{align*}
        \phi_{2}^{2}\left( \widetilde{s}_{n}; \bW_0 \right) \geq \dfrac{1}{36}
    \end{align*}
    with $\bbP$-probability at least $1-5e^{-n/24}$. Since the Cauchy--Schwarz inequality implies that $\phi_1(s; \bW) \geq \phi_2(s; \bW)$ for any $s \in \bbN$, this completes the proof of the first assertion in \eqref{eqn:example_poisson_1_1}. 
    The second assertion in \eqref{eqn:example_poisson_1_1} directly follows from Lemma \ref{lemma:variance_bound_example_Poisson}.
    Also, \eqref{claim:Poisson_1} follows from Theorem \ref{thm:Poisson_mis_example} under the assumption \eqref{eqn:example_poisson_1_assume}.
    Moreover, the fourth assertion in \eqref{claim:Poisson_2} follows from the condition that $\| \theta_0 \|_{2} \leq C$ for some constant $C > 0$ and the second assertion in \eqref{eqn:example_poisson_1_1}.
    The second assertion in \eqref{claim:Poisson_2} followss from Lemma \ref{lemma:Poisson_largest_eigenvalue}.
    Combining the first assertion in \eqref{claim:Poisson_2} and \eqref{eqn:example_design_matrix_2}, one can easily check that the third assertion is satisfied. Finally, the fourth assertion directly a direct consequence of Lemma \ref{lemma:orlic_bound} and the first assertion in \eqref{claim:Poisson_2}.
\end{proof}

\section{The misspecified estimators under random design} \label{sec:misspecified_estimator_example_app}

Throughout this section, we assume that $\bX$ is a random matrix with independent components following the standard normal distribution. With slight abuse of notation, let $\bbP$ be the joint probability measure corresponding to $(\bX, \bY)$. In this section, we prove that there exists $\overline{\theta}_{S}$ satisfying \eqref{A1:a} with high probability for the Poisson and logistic regression model.

\subsection{Poisson regression}
Throughout this sub-section, we assume that $b(\cdot) = \exp(\cdot)$. 

\begin{lemma} \label{thm:Poisson_mis_example}
Suppose that there exists a constant $c_1 > 0$ such that
\begin{align*}
    \left\| \theta_0 \right\|_{2} \leq c_{1}.
\end{align*}
Also, assume that 
\begin{align*}
     n \geq C\big( s_{\max}\log (n \vee p) \big)^{2}, \quad  p \geq C,
\end{align*}
where $C = C(c_1) > 0$ is large enough constant. 
Then, with $\bbP$- probability at least $1 - 3n^{-1} - 12e^{-n/48} - 3e^{-n/240} - 9p^{-1}$,
the following inequalities hold uniformly for all $S \in \scrS_{s_{\max}}$: 
\begin{align}
\begin{aligned} \label{claim:Poisson_mis_example}
    \left\| \bF_{n, \thetaBest}^{-1/2} \bF_{n, \thetaMLE} \bF_{n, \thetaBest}^{-1/2} \right\|_{2} 
    &\leq K, \\
    \left\| \bF_{n, \thetaMLE}^{-1/2} \bF_{n, \thetaBest} \bF_{n, \thetaMLE}^{-1/2} \right\|_{2}
    &\leq K, \\
    \left\| \bF_{n, \thetaBest}^{1/2} \left( \thetaMLE - \thetaBest \right) \right\|_{2} 
    &\leq K |S| \log p,    
\end{aligned}
\end{align}
where $K = K(c_1) > 0$ is a constant.
\end{lemma}

\begin{proof}
By Lemmas \ref{lemma:design_row_norm}, \ref{lemma:Poisson_least_eigenvalue}, \ref{lemma:Poisson_least_eigenvalue_V}, \ref{lemma:Poisson_largest_eigenvalue} and \ref{lemma:score_vec_random_design_Poisson}, there exists an event $\Omega_{n, 1}$ such that
\begin{align*}
    \bbP \left( \Omega_{n}^{\rm c} \right) \leq 3n^{-1} + 12e^{-n/48} + 3e^{-n/240} + 9p^{-1} 
\end{align*}
and, on $\Omega_{n}$, the following inequalities hold uniformly for all $S \in \scrS_{s_{\max}}$:
\begin{align*}
    \left\| \bV_{n, S}^{-1/2} \dot{L}_{n, \thetaBest} \right\|_{2}
    &\leq c_2 \left( |S| \log p \right)^{1/2}, \\
    \lambda_{\min} \left( \bF_{n, \theta_S} \right) &\geq c_3 n, \quad \forall \theta_S \in \bbR^{|S|}, \\
    c_4 n 
    \leq \lambda_{\min} \left( \bV_{n, S} \right) 
    &\leq \lambda_{\max} \left( \bV_{n, S} \right) 
    \leq c_5 n, \\
    \max_{i \in [n]} \left\| X_{i, S} \right\|_{2}^{2} &\leq c_6 s_{\max} \log (n \vee p),
\end{align*}
where $c_2, c_3, c_4, c_6 > 0$ are universal constants and $c_5 > 0$ is a constant depending only on $c_1$.
In the remainder of this proof, we work on the event $\Omega_{n}$.

Let $S \in \scrS_{s_{\max}}$.
For $\theta_S \in \bbR^{|S|}$, let $\mathbb{L}_{n, \theta_S} = \bbE (L_{n, \theta_S} \mid \bX) = \sum_{i = 1}^{n} b'(X_{i}^{\top} \theta_0)X_{i, S}^{\top} \theta_S - b(X_{i, S}^{\top} \theta_S)$ and $\dot{\mathbb{L}}_{n, \theta_S} = \sum_{i = 1}^{n} \left[ b'(X_{i}^{\top} \theta_0) - b'(X_{i, S}^{\top} \theta_S) \right] X_{i, S}$. 
Note that
\begin{align*}
    L_{n, \theta_S} - \mathbb{L}_{n, \theta_S} &=  \sum_{i = 1}^{n} \left[  Y_i - b'(X_{i}^{\top} \theta_0) \right] X_{i, S}^{\top} \theta_S \\
    \dot{L}_{n, \thetaMLE} - \dot{\mathbb{L}}_{n, \thetaMLE} &=  \sum_{i = 1}^{n} \left[  Y_i - b'(X_{i}^{\top} \theta_0) \right] X_{i, S}
    = - \dot{\mathbb{L}}_{n, \thetaMLE} = \dot{L}_{n, \thetaBest},
\end{align*}
where the last equality in the second line holds by the proof in Lemma \ref{lemma:score_vec_random_design}.
By linearization of $\dot{\mathbb{L}}_{n, \thetaMLE}$ at $\thetaBest$, Taylor's theorem gives
\begin{align*}
    \dot{\mathbb{L}}_{n, \thetaMLE} = \dot{\mathbb{L}}_{n, \thetaBest} - \bF_{n, \theta_S^{\circ}} \left( \thetaMLE - \thetaBest \right) = - \bF_{n, \theta_S^{\circ}} \left( \thetaMLE - \thetaBest \right)
\end{align*}
for some $\theta_S^{\circ} \in \bbR^{|S|}$ on the line segment between $\thetaMLE$ and $\thetaBest$. 
By $- \dot{\mathbb{L}}_{n, \thetaMLE} = \dot{L}_{n, \thetaBest}$, we have
\begin{align*}
    \left\|  \bV_{n, S}^{-1/2} \bF_{n, \theta_S^{\circ}} \left( \thetaMLE - \thetaBest \right) \right\|_{2}
    &=
    \left\|  \bV_{n, S}^{-1/2} \dot{L}_{n, \thetaBest} \right\|_{2} \\
    &\leq
    c_2 \left( |S| \log p \right)^{1/2}.
\end{align*}
Also, 
\begin{align*}
    &\left\|  \bV_{n, S}^{-1/2} \bF_{n, \theta_S^{\circ}} \left( \thetaMLE - \thetaBest \right) \right\|_{2} 
    \geq
    \lambda_{\max}^{-1/2}\left( \bV_{n, S} \right)
    \lambda_{\min}\left( \bF_{n, \theta_S^{\circ}} \right)
    \left\|  \thetaMLE - \thetaBest  \right\|_{2} 
\end{align*}
Combining last two displays, it follows that
\begin{align*}
    \left\|  \thetaMLE - \thetaBest  \right\|_{2} 
    &\leq 
    \bigg[  
    \lambda_{\max}^{1/2}\left( \bV_{n, S} \right)
    \lambda_{\min}^{-1}\left( \bF_{n, \theta_S^{\circ}} \right)
    \bigg]
    c_2 \left( |S| \log p \right)^{1/2} \\
    &\leq 
    \left( c_2 c_3^{-1} c_5^{1/2} \right) \left( \dfrac{|S| \log p}{n}  \right)^{1/2}
\end{align*}
for all $S \in \scrS_{s_{\max}}$.
It follows that
\begin{align*}
    &\max_{S \in \scrS_{s_{\max}}}
    \left\| \bX_{S} \left( \thetaMLE - \thetaBest \right) \right\|_{\infty} 
    =
    \max_{S \in \scrS_{s_{\max}}}
    \max_{i \in [n]}
    \left| X_{i, S}^{\top} \left( \thetaMLE - \thetaBest \right) \right| \\
    &\leq 
    \bigg( \max_{S \in \scrS_{s_{\max}}} \max_{i \in [n]}  \left\| X_{i, S} \right\|_{2} \bigg)
    \bigg( \max_{S \in \scrS_{s_{\max}}} \left\| \thetaMLE - \thetaBest \right\|_{2} \bigg) \\
    &\leq 
    \left(  c_6 c_2 c_3^{-1} c_5^{1/2} \right)
    \big( s_{\max} \log (n \vee p) \big)^{1/2}
    \left( \dfrac{ s_{\max} \log p}{n} \right)^{1/2}  \\
    &= 
    \left(  c_6 c_2 c_3^{-1} c_5^{1/2} \right) n^{-1/2} s_{\max} \log (n \vee p) \eqqcolon \delta_n 
    \leq 1.
\end{align*}
Note that
\begin{align*}
  \bF_{n, \thetaMLE} - \bF_{n, \thetaBest} 
  &= \sum_{i=1}^{n} \left( e^{X_{i, S}^{\top}\thetaMLE} - e^{X_{i, S}^{\top}\thetaBest[S]} \right) X_{i, S} X_{i, S}^{\top},
\end{align*}    
By Taylor's theorem, there exists $\theta_{S}^{\circ}(i)$ on the line segment between $\thetaMLE$ and $\thetaBest$ such that
\begin{align*} 
\begin{aligned} 
    &\left| e^{X_{i, S}^{\top}\thetaMLE} - e^{X_{i, S}^{\top}\thetaBest[S]} \right| \\
    &= \exp \left( X_{i, S}^{\top} \theta_{S}^{\circ}(i) -  X_{i, S}^{\top} \thetaBest[S] \right) 
    \left| X_{i, S}^{\top}\thetaMLE - X_{i, S}^{\top}\thetaBest[S] \right| \exp \left( X_{i, S}^{\top} \thetaBest[S] \right) \\
    &\leq \exp \left( \left| X_{i, S}^{\top} \theta_{S}^{\circ}(i) -  X_{i, S}^{\top} \thetaBest[S] \right| \right) 
    \left| X_{i, S}^{\top}\thetaMLE - X_{i, S}^{\top}\thetaBest[S] \right| \exp \left( X_{i, S}^{\top} \thetaBest[S] \right) \\   
    &\leq \exp \left( \left| X_{i, S}^{\top} \thetaMLE -  X_{i, S}^{\top} \thetaBest[S] \right| \right) 
    \left| X_{i, S}^{\top}\thetaMLE - X_{i, S}^{\top}\thetaBest[S] \right| \exp \left( X_{i, S}^{\top} \thetaBest[S] \right) \\
    &\leq \exp \left( \max_{S \in \scrS_{s_{\max}}} \left\| \bX_{S} \left( \thetaMLE -  \thetaBest \right) \right\|_{\infty} \right) 
    \left\{ \max_{S \in \scrS_{s_{\max}}} \left\| \bX_{S} \left( \thetaMLE -  \thetaBest \right) \right\|_{\infty} \right\} 
    \exp \left( X_{i, S}^{\top} \thetaBest[S] \right) \\
    &\leq \delta_n \left( 1 + 2\delta_n \right) \exp \left( X_{i, S}^{\top} \thetaBest[S] \right).    
\end{aligned}
\end{align*}
Hence, we have
\begin{align*}
    \max_{i \in [n]} \left| \exp \left( X_{i, S}^{\top}\thetaMLE \right) - \exp \left( X_{i, S}^{\top}\thetaBest[S] \right) \right|
    \leq 
    \delta_n \left( 1 + 2\delta_n \right) \exp \left( X_{i, S}^{\top} \thetaBest[S] \right).
\end{align*}	
It follows that
\begin{align*}
    \bF_{n, \thetaMLE} - \bF_{n, \thetaBest[S]} 
    \preceq \delta_n \left( 1 + 2\delta_n \right) \sum_{i=1}^{n} e^{X_{i, S}^{\top}\thetaBest[S]}  X_{i, S}X_{i, S}^{\top} 
    = \delta_n \left( 1 + 2\delta_n \right) \bF_{n, \thetaBest},
\end{align*}
implying
\begin{align*}
\max_{S \in \scrS_{s_{\max}}} \left\| \bF_{n, \thetaBest}^{-1/2} \bF_{n, \thetaMLE} \bF_{n, \thetaBest}^{-1/2} \right\|_{2} 
\leq 1 + \delta_n\left( 1 + 2\delta_n \right),
\end{align*}
which completes the proof of the first assertion in \eqref{claim:Poisson_mis_example}.

The proof for $\| \bF_{n, \thetaMLE}^{-1/2} \bF_{n, \thetaBest} \bF_{n, \thetaMLE}^{-1/2} \|_{2}$ is similar. As in the previous bound, we have
\begin{align*}
    \left| e^{X_{i, S}^{\top}\thetaBest} - e^{X_{i, S}^{\top}\thetaMLE} \right| 
    \leq \delta_n (1 + 2\delta_n) \exp \left( X_{i, S}^{\top} \thetaMLE \right).       
\end{align*}
Similarly, we have
\begin{align*}
    \bF_{n, \thetaBest[S]} - \bF_{n, \thetaMLE}
    &\preceq \delta_n \left( 1 + 2\delta_n \right) \bF_{n, \thetaMLE}, \\
    \max_{S \in \scrS_{s_{\max}}} \left\| \bF_{n, \thetaMLE}^{-1/2} \bF_{n, \thetaBest} \bF_{n, \thetaMLE}^{-1/2} \right\|_{2} 
    &\leq 1 + \delta_n\left( 1 + 2\delta_n \right),
\end{align*}
which completes the proof of the second assertion in \eqref{claim:Poisson_mis_example}.

Next, we will prove the last assertion in \eqref{claim:Poisson_mis_example}.
Note that
\begin{align*}
&\left\|  \bV_{n, S}^{-1/2} \bF_{n, \theta_S^{\circ}} \left( \thetaMLE - \thetaBest \right) \right\|_{2} \\
&=
\left\|  \bV_{n, S}^{-1/2} \bF_{n, \theta_S^{\circ}} \bF_{n, \thetaBest}^{-1} \bF_{n, \thetaBest}^{1/2} \bF_{n, \thetaBest}^{1/2} \left( \thetaMLE - \thetaBest \right) \right\|_{2} \\
&\geq
\lambda_{\max}^{-1/2}\left( \bV_{n, S} \right)
\lambda_{\min}\left( \bF_{n, \theta_S^{\circ}} \bF_{n, \thetaBest}^{-1} \right)
\lambda_{\min}^{1/2}\left( \bF_{n, \thetaBest} \right)
\left\| \bF_{n, \thetaBest}^{1/2} \left( \thetaMLE - \thetaBest \right) \right\|_{2},
\end{align*}
which implies that
\begin{align*}
&\left\| \bF_{n, \thetaBest}^{1/2} \left( \thetaMLE - \thetaBest \right) \right\|_{2} \\
&\leq 
\lambda_{\max}^{1/2}\left( \bV_{n, S} \right)
\lambda_{\max}\left( \bF_{n, \theta_S^{\circ}}^{-1} \bF_{n, \thetaBest} \right)
\lambda_{\min}^{-1/2}\left( \bF_{n, \thetaBest} \right)
\left\|  \bV_{n, S}^{-1/2} \dot{L}_{n, \thetaBest} \right\|_{2} \\
&\leq 
\lambda_{\max}\left( \bF_{n, \theta_S^{\circ}}^{-1} \bF_{n, \thetaBest} \right)
\bigg( c_{3}^{-1/2} c_5^{1/2} \left( |S| \log p \right)^{1/2} \bigg).
\end{align*}
Hence, we only need to show that $\lambda_{\max} ( \bF_{n, \theta_S^{\circ}}^{-1} \bF_{n, \thetaBest} ) \leq C'$ for some $C' > 0$. By Taylor's theorem, there exists $\overline{\theta}_{S}^{\circ}$ on the line segment between $\theta_S^{\circ}$ and $\thetaBest$ such that
\begin{align*} 
\begin{aligned} 
    &\left| e^{X_{i, S}^{\top}\theta_S^{\circ}} - e^{X_{i, S}^{\top}\thetaBest[S]} \right| \\
    &= \exp \left( X_{i, S}^{\top} \overline{\theta}_{S}^{\circ} -  X_{i, S}^{\top} \theta_S^{\circ} \right) 
    \left| X_{i, S}^{\top}\theta_S^{\circ} - X_{i, S}^{\top}\thetaBest[S] \right| 
    \exp \left( X_{i, S}^{\top} \theta_S^{\circ} \right) \\
    &\leq \exp \left( \left| X_{i, S}^{\top} \overline{\theta}_{S}^{\circ} -  X_{i, S}^{\top} \theta_S^{\circ} \right| \right) 
    \left| X_{i, S}^{\top}\theta_S^{\circ} - X_{i, S}^{\top}\thetaBest[S] \right| 
    \exp \left( X_{i, S}^{\top} \theta_S^{\circ} \right) \\   
    &\leq \exp \left( \left| X_{i, S}^{\top} \thetaMLE -  X_{i, S}^{\top} \thetaBest[S] \right| \right) 
    \left| X_{i, S}^{\top}\thetaMLE - X_{i, S}^{\top}\thetaBest[S] \right| 
    \exp \left( X_{i, S}^{\top} \theta_S^{\circ} \right) \\
    &\leq \exp \left( \max_{S \in \scrS_{s_{\max}}} \left\| \bX_{S} \left( \thetaMLE -  \thetaBest \right) \right\|_{\infty} \right) 
    \left\{ \max_{S \in \scrS_{s_{\max}}} \left\| \bX_{S} \left( \thetaMLE -  \thetaBest \right) \right\|_{\infty} \right\} 
    \exp \left( X_{i, S}^{\top} \theta_S^{\circ} \right) \\
    &\leq \delta_n \left( 1 + 2\delta_n \right) \exp \left( X_{i, S}^{\top} \theta_S^{\circ} \right).    
\end{aligned}
\end{align*}
Hence, we have
\begin{align*}
    \max_{i \in [n]} \left| \exp \left( X_{i, S}^{\top}\theta_S^{\circ} \right) - \exp \left( X_{i, S}^{\top}\thetaBest[S] \right) \right|
    \leq 
    \delta_n \left( 1 + 2\delta_n \right) \exp \left( X_{i, S}^{\top} \theta_S^{\circ} \right).
\end{align*}	
It follows that
\begin{align*}
    \bF_{n, \theta_S^{\circ}} - \bF_{n, \thetaBest[S]} 
    \preceq \delta_n \left( 1 + 2\delta_n \right) \sum_{i=1}^{n} e^{ X_{i, S}^{\top} \theta_S^{\circ} }  X_{i, S}X_{i, S}^{\top} 
    = \delta_n \left( 1 + 2\delta_n \right) \bF_{n, \theta_S^{\circ}},
\end{align*}
implying
\begin{align*}
\max_{S \in \scrS_{s_{\max}}} \left\| \bF_{n, \theta_S^{\circ}}^{-1} \bF_{n, \thetaBest[S]} \right\|_{2} 
\leq 1 + \delta_n\left( 1 + 2\delta_n \right).
\end{align*}
Therefore, we have
\begin{align*}
    \left\| \bF_{n, \thetaBest}^{1/2} \left( \thetaMLE - \thetaBest \right) \right\|_{2}
    \leq 
    4\left( c_{3}^{-1/2} c_5^{1/2} \right)  \left( |S| \log p \right)^{1/2}.
\end{align*}
This completes the proof.
\end{proof}

\subsection{Logistic regression}
Throughout this sub-section, we assume that $b(\cdot) = \log(1 + \exp(\cdot))$. 
\begin{lemma} \label{lemma:logit_MLE_bound}
    Let $s_{\ast} = s_{\max} + s_0$.
    Suppose that 
    \begin{align} \label{eqn:logit_mis_cond}
        n \geq C \bigg[ (s_{\ast} \log p)^{3/2} \vee \left( e^{ 10 \| \theta_0 \|_2}  s_\ast \log p \right) \bigg], \quad 
        p \geq C,
    \end{align}  
    where $C > 0$ is a large enough constant.
    Then, with $\bbP$-probability at least $1 - 22n^{-n/36} - 7p^{-1}$,
    the following inequality holds uniformly for all $S \in \scrS_{s_{\max}}$:
    \begin{align} \label{eqn:logit_MLE_bound}
    \begin{aligned}
        \bigg( \big\| \thetaMLE \big\|_{2} \vee \big\| \thetaBest \big\|_{2} \bigg) \leq \left\| \theta_0 \right\|_{2} + Ke^{6\| \theta_0 \|_2},       
    \end{aligned}
    \end{align}
    where $K > 0$ is a constant.
\end{lemma}
\begin{proof}
Let $\Omega_{n, 1}$ be an event on which the results of Lemmas \ref{lemma:least_eigenvalue_logit}, \ref{lemma:Fisher_smooth_logit} and \ref{lemma:score_vec_random_design} hold for $s_{\ast} = s_{\max} + s_0$. 
Then, we have $\bbP(\Omega_{n, 1}) \geq 1 - 22n^{-n/36} - 7p^{-1}$. On $\Omega_{n, 1}$, for all $S \in \scrS_{s_\ast}$ with $S \supseteq S_0$, 
\begin{align*}
    &\left\| \xi_{n, S} \right\|_{2} \leq c_1 e^{\| \theta_0 \|_2} \left( |S| \log p \right)^{1/2}, \\
    &c_2 n 
    \leq \lambda_{\min} \left( \bF_{n, 0_S} \right) 
    \leq \lambda_{\max} \left( \bF_{n, 0_S} \right) 
    \leq c_3 n, \\ 
    &\dfrac{c_2}{e^{2 \| \theta_0 \|_{2}}} n 
    \leq \lambda_{\min} \left( \bF_{n, \thetaBest} \right) 
    \leq \lambda_{\max} \left( \bF_{n, \thetaBest} \right) 
    \leq c_3 n 
\end{align*}
for some universal constants $c_1, c_2, c_3 > 0$, where $\bF_{n, 0_S} = \sum_{i=1}^{n} b''(0) X_{i, S} X_{i, S}^{\top}$.
Note that $\bF_{n, \thetaBest} = \bV_{n, S}$ for $S \supseteq S_0$.
In the remainder of this proof, we work on $\Omega_{n, 1}$.

Let $S \in \scrS_{s_{\max}}$ and $S_{\texttt{+}} = S \cup S_0$. 
By Taylor's theorem, there exists some $\theta_{ S_{\texttt{+}} }^{\circ} \in \bbR^{ |S_{\texttt{+}}| }$ such that
\begin{align}
\begin{aligned} \label{eqn:logit_mis_eq3}
&L_{n, 0} - L_{n, \thetaBest[S_{\texttt{+}}]} 
= 
\dot{L}_{n, \thetaBest[S_{\texttt{+}}]}^{\top} \left( 0 - \thetaBest[S_{\texttt{+}}] \right) - \dfrac{1}{2} \left\| \bF_{n, \theta_{ S_{\texttt{+}} }^{\circ}  }^{1/2}  
\left( 0 - \thetaBest[S_{\texttt{+}}] \right) \right\|_{2}^2 \\
&= 
- \xi_{n, S_{\texttt{+}}}^{\top} \bF_{n, \thetaBest[S_{\texttt{+}}]}^{1/2} \thetaBest[S_{\texttt{+}}] - \dfrac{1}{2} \left\| \bF_{n, \theta_{ S_{\texttt{+}} }^{\circ} }^{1/2} \thetaBest[S_{\texttt{+}}] \right\|_{2}^2 \\
&\geq 
- \left\| \xi_{n, S_{\texttt{+}}} \right\|_{2} \left\| \bF_{n, \thetaBest[S_{\texttt{+}}]}^{1/2} \thetaBest[S_{\texttt{+}}] \right\|_{2} 
- \dfrac{1}{2} \left\| \bF_{n, 0_{S_{\texttt{+}}}}^{1/2} \thetaBest[S_{\texttt{+}}] \right\|_{2}^2 \\
&\geq 
- \left( c_1 e^{\| \theta_0 \|_2} \sqrt{|S_{\texttt{+}}| \log p} \right) 
    \left( c_3 \sqrt{n} \left\| \thetaBest[S_{\texttt{+}}] \right\|_2 \right) 
- \dfrac{1}{2} \left( c_3 n \left\| \thetaBest[S_{\texttt{+}}] \right\|_2^2 \right) \\
&= 
- \left( c_1 e^{\| \theta_0 \|_2} \sqrt{ s_{\ast} \log p} \right) 
    \left( c_3 \sqrt{n} \left\| \theta_0 \right\|_2 \right) 
- \dfrac{1}{2} \left( c_3 n \left\| \theta_0 \right\|_2^2 \right) \\
&\geq 
- c_3 n \big( \left\| \theta_0 \right\|_2 \vee 1 \big)^2
\end{aligned}
\end{align}
where the last inequality holds by \eqref{eqn:logit_mis_cond}.
Let 
\begin{align*}
    \widetilde{r}_n = c_4^{-1} e^{- 3 \| \theta_0 \|_2} \sqrt{n} 
\end{align*}
for some large constant $c_4 \geq (864 \widetilde{K}_{\rm cubic})^{1/2}$, where $\widetilde{K}_{\rm cubic}$ is the constant specified in Lemma \ref{lemma:cubic_poly_Gaussian}.
First, one may assume that
\begin{align*}
\left\| \widetilde{\theta}_S^{\texttt{MLE}} - \theta_0 \right\|_{2} 
>
\left( e^{ \| \theta_0 \|_{2} } c_{2}^{-1/2} n^{-1/2} \right) \widetilde{r}_n
=
c_4^{-1} c_{2}^{-1/2} e^{-2\| \theta_0 \|_{2}}.
\end{align*}
Note that
\begin{align*}
    864 \widetilde{K}_{\rm cubic} e^{6 \| \theta_0 \|_2} \left( c_4^{-1} e^{- 3 \| \theta_0 \|_2} \sqrt{n} \right)^{2}
    \leq n,
\end{align*}
which allows applying applying Lemma \ref{lemma:tail_concave_likelihood} with $r_n = \widetilde{r}_n$.
By Lemma \ref{lemma:tail_concave_likelihood}, we have
\begin{align}
\begin{aligned} \label{eqn:logit_mis_eq4}
&L_{n, \widetilde{\theta}_S^{\texttt{MLE}}} - L_{n, \thetaBest[S_{\texttt{+}}]} \\
&\leq
\left\| \bF_{n, \thetaBest[S_{\texttt{+}}]}^{-1/2} \dot{L}_{n, \thetaBest[S_{\texttt{+}}]}  \right\|_{2} 
\left\| \bF_{n, \thetaBest[S_{\texttt{+}}]}^{1/2} \right\|_{2}
\left\| \widetilde{\theta}_S^{\texttt{MLE}} - \theta_0  \right\|_{2}
-
\dfrac{\widetilde{r}_n}{4} 
\left\| \bF_{n, \thetaBest[S_{\texttt{+}}]}^{1/2} \right\|_{2}
\left\| \widetilde{\theta}_S^{\texttt{MLE}} - \theta_0  \right\|_{2} \\
&\leq 
\left\| \bF_{n, \thetaBest[S_{\texttt{+}}]}^{-1/2} \dot{L}_{n, \thetaBest[S_{\texttt{+}}]}  \right\|_{2} 
\left( c_3 n \right)^{1/2}
\left\| \widetilde{\theta}_S^{\texttt{MLE}} - \theta_0  \right\|_{2}
-
\dfrac{\widetilde{r}_n}{4} \left( \dfrac{c_2}{e^{2 \| \theta_0 \|_{2}}} n \right)^{1/2} 
\left\| \widetilde{\theta}_S^{\texttt{MLE}} - \theta_0  \right\|_{2} \\
&\leq
\left( c_1 e^{\| \theta_0 \|_2} \sqrt{s_\ast \log p} \right)
\left( c_3 n \right)^{1/2}
\left\| \widetilde{\theta}_S^{\texttt{MLE}} - \theta_0  \right\|_{2}
-
\dfrac{\widetilde{r}_n}{4} \left( \dfrac{c_2}{e^{2 \| \theta_0 \|_{2}}} n \right)^{1/2} 
\left\| \widetilde{\theta}_S^{\texttt{MLE}} - \theta_0  \right\|_{2} \\
&=
\bigg[ c_1 c_3^{1/2} e^{\| \theta_0 \|_2} (n s_{\ast} \log p)^{1/2}  - \dfrac{c_2^{1/2}}{4c_4} e^{-4\| \theta_0 \|_2} n \bigg] 
\left\| \widetilde{\theta}_S^{\texttt{MLE}} - \theta_0  \right\|_{2} \\
&\leq
- \dfrac{c_2^{1/2}}{8 c_4}e^{-4\| \theta_0 \|_2}  n
\left\| \widetilde{\theta}_S^{\texttt{MLE}} - \theta_0  \right\|_{2}.    
\end{aligned}
\end{align}
Note that $L_{n, 0} - L_{n, \thetaBest[S_{\texttt{+}}]} \leq L_{n, \widetilde{\theta}_S^{\texttt{MLE}} } - L_{n, \thetaBest[S_{\texttt{+}}]}$.
Combining \eqref{eqn:logit_mis_eq3} and \eqref{eqn:logit_mis_eq4}, we have
\begin{align*}
- c_3 n \big( \left\| \theta_0 \right\|_2 \vee 1 \big)^2
\leq L_{n, 0} - L_{n, \thetaBest[S_{\texttt{+}}]} 
&\leq L_{n, \widetilde{\theta}_S^{\texttt{MLE}} } - L_{n, \thetaBest[S_{\texttt{+}}]} \\
&\leq 
- \dfrac{c_2^{1/2}}{8 c_4}e^{-4\| \theta_0 \|_2} n
\left\| \widetilde{\theta}_S^{\texttt{MLE}} - \theta_0  \right\|_{2}.    
\end{align*}
which implies that
\begin{align*}
\left\| \thetaMLE \right\|_{2} 
\leq 8 c_4 c_3 c_2^{-1/2} e^{4\| \theta_0 \|_{2} } \big( \| \theta_0 \|_{2} \vee 1 \big)^2 + \| \theta_0 \|_{2}
\leq 8 c_4 c_3 c_2^{-1/2} e^{ 6 \| \theta_0 \|_{2} } + \| \theta_0 \|_{2}.
\end{align*}
Secondly, if 
\begin{align*}
\left\| \widetilde{\theta}_S^{\texttt{MLE}} - \theta_0 \right\|_{2} 
\leq c_4^{-1} c_{2}^{-1/2} e^{-2\| \theta_0 \|_{2}}
\leq c_4^{-1} c_{2}^{-1/2},
\end{align*}
we immediately obtain the following inequality:
\begin{align*}
\left\| \thetaMLE \right\|_{2}
\leq
c_4^{-1} c_{2}^{-1/2} + \left\| \theta_0 \right\|_{2},
\end{align*}
which completes the proof of the first assertion in \eqref{eqn:logit_MLE_bound}.

The proof for the second assertion is similar. Hence, we will provide a sketch of the proof.
By Taylor's theorem, there exists some $\theta_{ S_{\texttt{+}} }^{\circ} \in \bbR^{ |S_{\texttt{+}}| }$ such that
\begin{align*}
\begin{aligned} 
\mathbb{L}_{n, 0} - \mathbb{L}_{n, \thetaBest[S_{\texttt{+}}]} 
&= 
\dot{\mathbb{L}}_{n, \thetaBest[S_{\texttt{+}}]}^{\top} \left( 0 - \thetaBest[S_{\texttt{+}}] \right) - \dfrac{1}{2} \left\| \bF_{n, \theta_{ S_{\texttt{+}} }^{\circ}  }^{1/2}  
\left( 0 - \thetaBest[S_{\texttt{+}}] \right) \right\|_{2}^2 \\
&= 
- \dfrac{1}{2} \left\| \bF_{n, \theta_{ S_{\texttt{+}} }^{\circ} }^{1/2} \thetaBest[S_{\texttt{+}}] \right\|_{2}^2 
\geq 
- \dfrac{1}{2} \left\| \bF_{n, 0_{S_{\texttt{+}}}}^{1/2} \thetaBest[S_{\texttt{+}}] \right\|_{2}^2 \\
&\geq 
- \dfrac{1}{2} \left( c_3 n \left\| \thetaBest[S_{\texttt{+}}] \right\|_2^2 \right) 
= 
- \dfrac{c_3}{2} n \left\| \theta_0 \right\|_2^2.
\end{aligned}
\end{align*}
Also, if 
\begin{align*}
\left\| \widetilde{\theta}_S^{\ast} - \theta_0 \right\|_{2} 
>
c_4^{-1} c_{2}^{-1/2} e^{-2\| \theta_0 \|_{2}},
\end{align*}
then we have
\begin{align*}
\mathbb{L}_{n, \widetilde{\theta}_S^{\ast}} - \mathbb{L}_{n, \thetaBest[S_{\texttt{+}}]} 
\leq 
-\dfrac{\widetilde{r}_n}{4} \left( \dfrac{c_2}{e^{2 \| \theta_0 \|_{2}}} n \right)^{1/2} 
\left\| \widetilde{\theta}_S^{\ast} - \theta_0  \right\|_{2} 
=
\bigg[ - \dfrac{c_2^{1/2}}{4c_4} e^{-4\| \theta_0 \|_2} \bigg]  n
\left\| \widetilde{\theta}_S^{\ast} - \theta_0  \right\|_{2}.
\end{align*}
Similarly, we have
\begin{align*}
- c_3 n \left\| \theta_0 \right\|_2^2
\leq \mathbb{L}_{n, 0} - \mathbb{L}_{n, \thetaBest[S_{\texttt{+}}]} 
\leq \mathbb{L}_{n, \widetilde{\theta}_S^{\ast}} - \mathbb{L}_{n, \thetaBest[S_{\texttt{+}}]} 
\leq 
- \dfrac{c_2^{1/2}}{4c_4} e^{-4\| \theta_0 \|_2} n
\left\| \widetilde{\theta}_S^{\ast} - \theta_0  \right\|_{2},
\end{align*}
which implies that
\begin{align*}
\left\| \thetaBest \right\|_{2} 
\leq 4 c_4 c_3 c_2^{-1/2} e^{4\| \theta_0 \|_{2} } \| \theta_0 \|_{2}^2 + \| \theta_0 \|_{2}
\leq 4 c_4 c_3 c_2^{-1/2} e^{ 6 \| \theta_0 \|_{2} } + \| \theta_0 \|_{2}.
\end{align*}
Secondly, if 
\begin{align*}
\left\| \widetilde{\theta}_S^{\ast} - \theta_0 \right\|_{2} 
\leq c_4^{-1} c_{2}^{-1/2} e^{-2\| \theta_0 \|_{2}}
\leq c_4^{-1} c_{2}^{-1/2},
\end{align*}
we immediately obtain the following inequality:
\begin{align*}
\left\| \thetaBest \right\|_{2}
\leq
c_4^{-1} c_{2}^{-1/2} + \left\| \theta_0 \right\|_{2},
\end{align*}
which completes the proof of the second assertion in \eqref{eqn:logit_MLE_bound}.
\end{proof}

\begin{theorem} \label{thm:logit_mis_example}
Let $s_{\ast} = s_{\max} + s_0$. Suppose that there exists a constant $c_1 > 0$ such that $\left\| \theta_0 \right\|_{2} \leq c_{1}$.
Also, assume that
\begin{align*}
    n \geq C (s_{\ast} \log p)^{3/2}, \quad 
    p \geq C,
\end{align*}  
where $C = C(c_1)> 0$ is a large enough constant.
Then, with $\bbP$- probability at least $1 - 31e^{-n/40} - 9p^{-1}$,
the following inequalities hold uniformly for all $S \in \scrS_{s_{\max}}$: 
\begin{align}
\begin{aligned}
    \left\| \bF_{n, \thetaMLE}^{-1/2} \bF_{n, \thetaBest} \bF_{n, \thetaMLE}^{-1/2} \right\|_{2}
    &\leq K, \\
    \left\| \bF_{n, \thetaBest}^{-1/2} \bF_{n, \thetaMLE} \bF_{n, \thetaBest}^{-1/2} \right\|_{2} 
    &\leq K, \\
    \left\| \bF_{n, \thetaBest}^{1/2} \left( \thetaMLE - \thetaBest \right) \right\|_{2} 
    &\leq K |S| \log p,    
\end{aligned}
\end{align}
where $K = K(c_1) > 0$ is a constant.
\end{theorem}

\begin{proof}
Let $\Omega_{n, 1}$ be an event on which the result of Lemmas \ref{lemma:logit_MLE_bound}, \ref{lemma:least_eigenvalue_logit}, \ref{lemma:Fisher_smooth_logit} and \ref{lemma:score_vec_random_design} hold. 
By $\| \theta_0 \|_2 \leq c_1$, on $\Omega_{n, 1}$, we have
\begin{align*}
    \bigg( \big\| \thetaMLE \big\|_{2} \vee \big\| \thetaBest \big\|_{2} \bigg) \leq c_2, \quad 
    \text{ for all } S \in \scrS_{s_{\max}} 
\end{align*}
where $c_2 = c_2(c_1) > 0$ is a constant. Note that $\bbP(\Omega_{n, 1}) \geq 1 - 22n^{-n/36} - 7p^{-1}$.
Also, by Lemma \ref{lemma:eigenvalue_logit_uniform}, there exists an event $\Omega_{n, 2}$ such that, on $\Omega_{n, 2}$, the following inequalities hold:
\begin{align*} 
    \dfrac{n}{1030 e^{2(M + 1)}} 
    \leq \min_{S \in \scrS_{s_\ast}} \inf_{\theta_S \in \Theta_{S, M}}  \lambda_{\min} \left( \bF_{S, \theta_S} \right) 
    \leq \max_{S \in \scrS_{s_\ast}} \sup_{\theta_S \in \bbR^{|S|}} \lambda_{\max} \left( \bF_{S, \theta_S} \right) 
    \leq \dfrac{9}{4} n,
\end{align*}
where $\Theta_{S, M} = \{ \theta_S \in \bbR^{|S|} : \| \theta_S \|_{2} \leq M \}$ for $M > 0$,
and $\bbP(\Omega_{n, 2}) \geq 1 - 9e^{-n/40} -2(np)^{-1}$.
Then, 
\begin{align*}
    \bbP(\Omega_n) \geq 1 - 31e^{-n/40} -9p^{-1},
\end{align*}
where $\Omega_n = \Omega_{n, 1} \cap \Omega_{n, 2}$.
In the remainder of this proof, we work on the event $\Omega_n$.

Let $S \in \scrS_{s_{\max}}$.
For $\theta_S \in \bbR^{|S|}$, let $\mathbb{L}_{n, \theta_S} = \bbE (L_{n, \theta_S} \mid \bX) = \sum_{i = 1}^{n} b'(X_{i}^{\top} \theta_0)X_{i, S}^{\top} \theta_S - b(X_{i, S}^{\top} \theta_S)$ and $\dot{\mathbb{L}}_{n, \theta_S} = \sum_{i = 1}^{n} \left[ b'(X_{i}^{\top} \theta_0) - b'(X_{i, S}^{\top} \theta_S) \right] X_{i, S}$. 
Note that
\begin{align*}
    L_{n, \theta_S} - \mathbb{L}_{n, \theta_S} &=  \sum_{i = 1}^{n} \left[  Y_i - b'(X_{i}^{\top} \theta_0) \right] X_{i, S}^{\top} \theta_S \\
    \dot{L}_{n, \thetaMLE} - \dot{\mathbb{L}}_{n, \thetaMLE} &=  \sum_{i = 1}^{n} \left[  Y_i - b'(X_{i}^{\top} \theta_0) \right] X_{i, S}
    = - \dot{\mathbb{L}}_{n, \thetaMLE} = \dot{L}_{n, \thetaBest},
\end{align*}
where the last equality in the second line holds by the proof in Lemma \ref{lemma:score_vec_random_design}.
By linearization of $\dot{\mathbb{L}}_{n, \thetaMLE}$ at $\thetaBest$, Taylor's theorem gives
\begin{align*}
    \dot{\mathbb{L}}_{n, \thetaMLE} = \dot{\mathbb{L}}_{n, \thetaBest} - \bF_{n, \theta_S^{\circ}} \left( \thetaMLE - \thetaBest \right) = - \bF_{n, \theta_S^{\circ}} \left( \thetaMLE - \thetaBest \right)
\end{align*}
for some $\theta_S^{\circ} \in \bbR^{|S|}$ on the line segment between $\thetaMLE$ and $\thetaBest$. 
By $- \dot{\mathbb{L}}_{n, \thetaMLE} = \dot{L}_{n, \thetaBest}$, we have
\begin{align*}
    \left\|  \bV_{n, S}^{-1/2} \bF_{n, \theta_S^{\circ}} \left( \thetaMLE - \thetaBest \right) \right\|_{2}
    =
    \left\|  \bV_{n, S}^{-1/2} \dot{L}_{n, \thetaBest} \right\|_{2} 
    \leq
    c_3 \left( |S| \log p \right)^{1/2},
\end{align*}
where the last inequaility holds by Lemma \ref{lemma:score_vec_random_design} and $c_3 = c_3(c_1)$
Also, 
\begin{align*}
    &\left\|  \bV_{n, S}^{-1/2} \bF_{n, \theta_S^{\circ}} \left( \thetaMLE - \thetaBest \right) \right\|_{2} \\
    &=
    \left\|  \bV_{n, S}^{-1/2} \bF_{n, \theta_S^{\circ}} \bF_{n, \thetaBest}^{-1/2} \bF_{n, \thetaBest}^{1/2} \left( \thetaMLE - \thetaBest \right) \right\|_{2} \\
    &\geq
    \lambda_{\max}^{-1/2}\left( \bV_{n, S} \right)
    \lambda_{\min}\left( \bF_{n, \theta_S^{\circ}} \right)
    \lambda_{\max}^{-1/2}\left( \bF_{n, \thetaBest} \right)
    \left\| \bF_{n, \thetaBest}^{1/2} \left( \thetaMLE - \thetaBest \right) \right\|_{2} 
\end{align*}
Combining last two displays, it follows that
\begin{align*}
    \left\| \bF_{n, \thetaBest}^{1/2} \left( \thetaMLE - \thetaBest \right) \right\|_{2} 
    \leq 
    \bigg[  
    \lambda_{\max}^{1/2}\left( \bV_{n, S} \right)
    \lambda_{\min}^{-1}\left( \bF_{n, \theta_S^{\circ}} \right)
    \lambda_{\max}^{1/2}\left( \bF_{n, \thetaBest} \right)
    \bigg]
    c_3 \left( |S| \log p \right)^{1/2}.
\end{align*}
By Lemma \ref{lemma:least_eigenvalue_logit}, we have
\begin{align*}
    \lambda_{\max} \left( \bV_{n, S} \right) \leq c_4 n, \quad 
    \lambda_{\max} \left( \bF_{n, \thetaBest} \right) \leq c_4 n,
\end{align*}
for some universal constant $c_4 > 0$.
Since $\| \thetaMLE \|_{2} \vee \| \thetaBest \|_2 \leq c_2$ for all $S \in \scrS_{s_{\max}}$, we have $\| \theta_S^{\circ} \|_{2} \leq c_2$ for all $S \in \scrS_{s_{\max}}$.
By Lemma \ref{lemma:eigenvalue_logit_uniform}, the following inequalities hold uniformly for all $S \in \scrS_{s_{\max}}$:
\begin{align*}
    \lambda_{\min} \left( \bF_{n, \theta_S^{\circ}} \right)
    \geq 
    c_5 n, \quad
    \lambda_{\min} \left( \bF_{n, \thetaBest} \right)
    \geq 
    c_5 n,
\end{align*}
where $c_5 = c_5(c_2) > 0$ is a constant.
Therefore, we have
\begin{align} \label{eqn:logit_random_desing_res_1}
    \left\| \bF_{n, \thetaBest}^{1/2} \left( \thetaMLE - \thetaBest \right) \right\|_{2} 
    \leq 
    \left( c_3 c_4 c_5^{-1} \right)  \left( |S| \log p \right)^{1/2}, \quad \forall S \in \scrS_{s_{\max}},
\end{align}
which implies that
\begin{align*}
    \left\| \thetaMLE - \thetaBest \right\|_{2} 
    \leq 
    \left( c_3 c_4 c_5^{-2} \right) \left( \dfrac{|S| \log p}{n} \right)^{1/2}.
\end{align*}
By Lemma \ref{lemma:Fisher_smooth_logit} and $\lambda_{\min}(\bF_{n, \thetaBest}) \geq c_5 n$, we have
\begin{align*}
    \left\| \bF_{n, \thetaBest}^{-1/2} \bF_{n, \thetaMLE} \bF_{n, \thetaBest}^{-1/2} - \bI_{|S|} \right\|_{2}
    \leq 
    c_6 \left( \dfrac{|S| \log p}{n} \right)^{1/2} \eqqcolon \delta_n \leq 1/2,
\end{align*}
where $c_6 = c_6(c_3, c_4, c_5, \widetilde{K}_{\rm cubic})$ is a constant and $\widetilde{K}_{\rm cubic}$ is the (universal) constant specified in Lemma \ref{lemma:cubic_poly_Gaussian}.
It follows that
\begin{align}
\begin{aligned} \label{eqn:logit_random_desing_res_2}
    \left\| \bF_{n, \thetaBest}^{-1/2} \bF_{n, \thetaMLE} \bF_{n, \thetaBest}^{-1/2}  \right\|_{2}
    &\leq (1 + \delta_n), \\
    \left\| \bF_{n, \thetaMLE}^{-1/2} \bF_{n, \thetaBest} \bF_{n, \thetaMLE}^{-1/2}  \right\|_{2}
    &\leq (1 - \delta_n)^{-1}.
\end{aligned}    
\end{align}
Combining \eqref{eqn:logit_random_desing_res_1} and \eqref{eqn:logit_random_desing_res_2}, 
we complete the proof.
\end{proof}

%%%%%%%%%%%%%%%%%%%%%%%%%%%%%%%%%%%%%%%%%%%%%%%%%% Technical lemmas
\section{Technical lemmas} \label{sec:technical_lemmas_app}
Throughout this section (except for Lemma \ref{lemma:orlic_bound}), we assume that $\mathbf{X} \in \mathbb{R}^{n\times p}$ is a random matrix with independent rows, where the $i$th row $X_i$ follows $\mathcal{N}\left(0, \bI_p\right)$ distribution. Let $\bbP$ be the corresponding probability measure, $\scrS_{s} = \left\{ S \subset [p] : 0 < |S| \leq s \right\}$ and $s_* \leq p$ be a positive integer.  Constants $c_1, c_2, \ldots$ used in the proofs may vary according to their contexts.

\begin{lemma} \label{lemma:extreme_eigenvalue}
Suppose that
\begin{align} \label{assume:least_eigenvalue_lemma}
    p \geq 3, \quad 4 s_{\ast} \log p \leq n.
\end{align}
Then,
\begin{align} \label{eqn:least_eigenvalue_claim}
    \bbP \left\{ 
    \lambda_{\min} \left( \sum_{i=1}^{n}  X_{i, S} X_{i, S}^{\top} \right) 
    \leq \dfrac{1}{9} n \quad 
    \text{ for some } S \in \scrS_{s_{\ast}}
    \right\}
    \leq 3e^{-n/4}
\end{align}
and 
\begin{align} \label{eqn:least_eigenvalue_claim2}
    \bbP \left\{ 
    \lambda_{\max} \left( \sum_{i=1}^{n}  X_{i, S} X_{i, S}^{\top} \right) 
    \geq 9 n \quad
    \text{ for some } S \in \scrS_{s_{\ast}}
    \right\}
    \leq 3e^{-n/4}.
\end{align}
\end{lemma}

\begin{proof}
By the equation (60) in \cite{wainwright2009sharp} and $s_{\ast} \leq n$, we have, for $S \in \scrS_{s_{\ast}}$,
\begin{align*}
    \bbP \left\{ 
    \lambda_{\min} \left( \sum_{i=1}^{n}  X_{i, S} X_{i, S}^{\top} \right) 
    \leq \dfrac{1}{9} n
    \right\}
    \leq 2e^{-n/2}.
\end{align*}
Since $\binom{p}{s} \leq p^{s}$ and $p \geq 3$,
\begin{align*}
    &\bbP \left\{ 
    \lambda_{\min} \left( \sum_{i=1}^{n}  X_{i, S} X_{i, S}^{\top} \right) 
    \leq \dfrac{1}{9} n \: \:
    \text{ for some } S \in \scrS_{s_{\ast}}
    \right\} \\
    &\leq  
    \left| \scrS_{s_{\ast}} \right|
    \max_{S \in \scrS_{s_{\ast}}}  \bbP \left\{ 
    \lambda_{\min} \left( \sum_{i=1}^{n}  X_{i, S} X_{i, S}^{\top} \right) 
    \leq \dfrac{1}{9} n
    \right\} \\
    &= \left[ \sum_{s = 1}^{s_{\ast}} \binom{p}{s} \right] 
    \max_{S \in \scrS_{s_{\ast}}}  \bbP \left\{ 
    \lambda_{\min} \left( \sum_{i=1}^{n}  X_{i, S} X_{i, S}^{\top} \right) 
    \leq \dfrac{1}{9}n 
    \right\} \\
    &\leq \left[ \sum_{s = 1}^{s_{\ast}} p^{s} \right] \times 2e^{-n/2} 
    \leq 3p^{s_{\ast}}e^{-n/2}
    = 3\exp \left( -\dfrac{n}{2} + s_{\ast} \log p \right)
    \quad \big(\because \eqref{assume:least_eigenvalue_lemma} \big) \\
    &\leq 3e^{-n/4},
\end{align*}
completing the proof of \eqref{eqn:least_eigenvalue_claim}. 

The proof of \eqref{eqn:least_eigenvalue_claim2} is similar. 
By the equation (59) in \cite{wainwright2009sharp} and $s_{\ast} \leq n$, we have, for $S \in \scrS_{s_{\ast}}$,
\begin{align*}
    \bbP \left\{ 
    \lambda_{\max} \left( \sum_{i=1}^{n}  X_{i, S} X_{i, S}^{\top} \right) 
    \geq 9 n
    \right\}
    \leq 2e^{-n/2}.
\end{align*}
Since $\binom{p}{s} \leq p^{s}$ and $p \geq 3$,
\begin{align*}
    &\bbP \left\{ 
    \lambda_{\max} \left( \sum_{i=1}^{n}  X_{i, S} X_{i, S}^{\top} \right) 
    \geq 9 n \: \:
    \text{ for some } S \in \scrS_{s_{\ast}}
    \right\} \\
    &\leq  
    \left| \scrS_{s_{\ast}} \right|
    \max_{S \in \scrS_{s_{\ast}}}  \bbP \left\{ 
    \lambda_{\max} \left( \sum_{i=1}^{n}  X_{i, S} X_{i, S}^{\top} \right) 
    \geq 9 n
    \right\}  \\
    &\leq 3p^{s_{\ast}}e^{-n/2}
    = 3\exp \left( -\dfrac{n}{2} + s_{\ast} \log p \right) \leq 3e^{-n/4},
\end{align*}
which completes the proof of \eqref{eqn:least_eigenvalue_claim2}.

\end{proof}

\begin{lemma} \label{lemma:design_row_norm}
We have
\begin{align} \label{eqn:design_row_norm_claim}
    \bbP \biggl\{ 
    \max_{i \in [n], j \in [p]} |X_{i,j}|
    >
    2 \sqrt{\log(np)}
    \biggr\}
    \leq 2(np)^{-1}
\end{align}
and
\begin{align} \label{eqn:design_row_norm_claim3}
    \bbP \biggl\{ 
    \max_{i \in [n], S \in \scrS_{s_{\ast}}} \left\| X_{i, S} \right\|_{2}^{2}
    > 
    4s_{\ast}\log(np)
    \biggr\}
    &\leq 2(np)^{-1}, \\ \label{eqn:design_row_norm_claim4}
    \bbP \biggl\{ 
    \left\| \bX_{S_0} \right\|_{\infty}
    > 
    2s_{0}\sqrt{\log(np)}
    \biggr\}
    &\leq 2(np)^{-1}.
\end{align}
Also, for $S \in \scrS_{s_\ast}$ and $u_S \in \cU_{S} = \left\{ u_S \in \bbR^{|S|} : \| u_S \|_{2} = 1 \right\}$, 
\begin{align}
\label{eqn:design_row_norm_claim2}
    \bbP \biggl\{ 
    \max_{i \in [n]} \left| X_{i, S}^{\top} u_{S} \right|
    > 
    2\sqrt{\log n} 
    \biggr\}
    \leq 2n^{-1}.
\end{align}
\end{lemma}

\begin{proof}
Since $X_{ij} \sim \cN \left( 0, 1 \right)$, we have, for all $t \geq 0$,
\begin{align*}
    \bbP \bigg( 
    |X_{ij}| > t
    \bigg) \leq 2 \exp \left( - \dfrac{t^2}{2} \right).
\end{align*}
It follows that 
\begin{align*}
    \bbP \bigg( 
    \max_{i \in [n], j \in [p]} |X_{ij}| > t
    \bigg) \leq  2 np \exp \left( - \dfrac{t^2}{2} \right).
\end{align*}
By taking $t = 2\sqrt{\log (np)}$, we complete the proof of \eqref{eqn:design_row_norm_claim}.
Let $u_S \in \cU_{S}$. Since $X_{i, S}^{\top}u_{S} \sim \cN\left( 0, 1\right)$ and
\begin{align*}
    \bbP \biggl\{ 
    \max_{i \in [n]} \left| X_{i, S}^{\top} u_{S} \right|
    > 
    t 
    \biggr\}
    \leq n \max_{i \in [n]} \bbP \biggl\{ \left| X_{i, S}^{\top} u_{S} \right| >  t \biggr\}
    \leq n \times 2e^{-t^2/2} = 2e^{-t^2/2 + \log n},
\end{align*}
the proof of \eqref{eqn:design_row_norm_claim2} is complete by taking $t = 2\sqrt{\log n}$.
Also,
\begin{align*}
    \max_{i \in [n], S \in \scrS_{s_{\ast}}} \left\| X_{i, S} \right\|_{2}^{2} \leq s_{\ast} \left\| \bX \right\|_{\max}^2.
\end{align*}
This completes the proof of \eqref{eqn:design_row_norm_claim3}. The proof of \eqref{eqn:design_row_norm_claim4} is similar. Note that
\begin{align*}
    \left\| \bX_{S_0} \right\|_{\infty} = \max_{i \in [n]} \sum_{j \in S_0} |X_{ij}| 
    \leq s_0 \max_{i \in [n], j \in [p]} |X_{ij}|
    \leq 2s_0 \sqrt{\log(np)}
\end{align*}
with $\bbP$-probability at least $1-2(np)^{-1}$, where the second inequality holds by \eqref{eqn:design_row_norm_claim}.
\end{proof}

\begin{lemma} \label{lemma:maximum_value_lowerbound}
We have
\begin{align} \label{eqn:maximum_value_lowerbound_claim}
    \bbP \biggl\{ 
    \max_{i \in [n], j \in S_0} X_{ij} \geq 1
    \biggr\}
    \leq 1 - (0.88)^{ns_0}.
\end{align}
\end{lemma}
\begin{proof}
    For $t \geq 1$, note that
    \begin{align*}
        \bbP \left( \max_{i \in [n], j \in S_0} X_{ij} \geq t \right) 
        &= 1 - \bbP \left( \max_{i \in [n], j \in S_0} X_{ij} \leq t \right)
        = 1 - \bigg[ \bbP \left( X_{ij} \leq t \right) \bigg]^{n s_0} \\
        &\geq 1 - \bigg[ 1 - \dfrac{1}{2\sqrt{2\pi}} t^{-1} e^{-t^{2}/2 } \bigg]^{n s_0},
    \end{align*}
    where the last inequality holds by the standard inequality known as Mills' ratio. By taking $t = 1$, the right-hand side of the last display is equal to
    \begin{align*}
        1 - \bigg[ 1 - \dfrac{1}{2\sqrt{2\pi}} e^{-1/2 } \bigg]^{n s_0}
        \geq 
        1 - (0.88)^{n s_0},
    \end{align*}
    which completes the proof.
\end{proof}

\begin{lemma} \label{lemma:design_column_norm}
We have
\begin{align} \label{eqn:design_column_norm_claim}
    \bbP \biggl\{ 
    \max_{j \in [p]} \left\| \bX_{j} \right\|_{2} \geq \sqrt{n} + 2\sqrt{\log p}
    \biggr\}
    \leq p^{-1}.
\end{align}
\end{lemma}
\begin{proof}
    For $j \in [p]$, note that $\bX_j \sim \cN(0, \bI_{n})$.
    By Theorem B.1 in \cite{spokoiny2023deviation}, the Gaussian quadratic deviation inequality gives
    \begin{align*}
        \bbP \biggl\{ 
        \left\| \bX_{j} \right\|_{2}^{2} \geq \operatorname{tr}(\bI_{n}) + 2\left\| \bI_{n}^{2} \right\|_{\rm F}\sqrt{t} + 2\left\| \bI_{n} \right\|_{2} t
        \biggr\}
        \leq e^{-t}
    \end{align*}    
    for any $t \geq 0$.
    It follows that 
    \begin{align*}
        \bbP \biggl\{ 
        \left\| \bX_{j} \right\|_{2}^{2} \geq n + 2\sqrt{nt} + 2t
        \biggr\}
        \leq e^{-t}.
    \end{align*}    
    Since $(\sqrt{n} + \sqrt{2t})^{2} \geq n + 2\sqrt{nt} + 2t$ for any $n, t \geq 0$, we have
    \begin{align*}
        \bbP \biggl\{ 
        \left\| \bX_{j} \right\|_{2} \geq \sqrt{n} + \sqrt{2t}
        \biggr\}
        \leq e^{-t},
    \end{align*}  
    which further implies that 
    \begin{align*}
        \bbP \biggl\{ 
        \max_{j \in [p]} \left\| \bX_{j} \right\|_{2} \geq \sqrt{n} + \sqrt{2t}
        \biggr\}
        \leq e^{-t + \log p}.
    \end{align*}
    By taking $t = 2\log p$, we complete the proof of \eqref{eqn:design_column_norm_claim}.
\end{proof}

\begin{lemma} \label{lemma:natural_parameter}
We have
\begin{align} \label{eqn:natural_parameter_claim}
    \bbP \biggl\{ 
    \max_{i \in [n]} \left| X_{i}^{\top} \theta_0 \right| \geq 2 \left\| \theta_0 \right\|_{2} \sqrt{\log n} 
    \biggr\}
    \leq n^{-1}.
\end{align}
\end{lemma}
\begin{proof}
Since $X_{i}^{\top} \theta_0 \sim \cN \left( 0, \| \theta_0 \|_{2}^{2} \right)$, we have, for all $t \geq 0$,
\begin{align*}
    \bbP \bigg( 
    \left| X_{i}^{\top} \theta_0 \right| > t \| \theta_0 \|_{2}
    \bigg) \leq 2 \exp \left( - \dfrac{t^2}{2} \right).
\end{align*}
It follows that 
\begin{align*}
    \bbP \bigg( 
    \max_{i \in [n]} \left| X_{i}^{\top} \theta_0 \right| > t \| \theta_0 \|_{2}
    \bigg) \leq  2 n \exp \left( - \dfrac{t^2}{2} \right).
\end{align*}
By taking $t = 2\sqrt{\log n}$, we complete the proof of \eqref{eqn:natural_parameter_claim}.
\end{proof}

\begin{lemma} \label{lemma:GLM_b_ratio}
For the logistic and Poisson regression models, we have
\begin{align*}
    \dfrac{b''\left( \eta_{1} \right)}{b''\left( \eta_{2} \right)}
    \leq e^{ 3\left| \eta_1 - \eta_2 \right|}
\end{align*}
for all $\eta_{1}, \eta_{2} \in \bbR$.
\end{lemma}
\begin{proof}
Let $\eta_1, \eta_2 \in \bbR$. For Poisson regression, the proof is trivial since $b''\left( \eta_{1} \right) / b''\left( \eta_{2} \right) = e^{\eta_1 - \eta_2}$. 
Hence, we consider the logistic regression case where $b(\eta) = \log \left( 1 + e^{\eta} \right)$. 
Since $b''(\eta) = e^{\eta} / \left( 1 + e^{\eta} \right)^2$ for $\eta \in \bbR$, note that
\begin{align*}
\dfrac{b''\left( \eta_{1} \right)}{b''\left( \eta_{2} \right)}
= 
e^{\eta_1 - \eta_{2}}
\bigg(
\dfrac{
    1 + e^{\eta_2}
}{
    1 + e^{\eta_1}
}
\bigg)^2.
\end{align*}
Also,
\begin{align*}
\dfrac{ 1 + e^{\eta_2} }{ 1 + e^{\eta_1} }
= 1 + \dfrac{ e^{\eta_2} - e^{\eta_1} }{ 1 + e^{\eta_1} }
= 1 +  \dfrac{ e^{\eta_1} \left( e^{\eta_2 - \eta_1} - 1 \right) }{ 1 + e^{\eta_1} }
\leq  1 +  e^{\eta_2 - \eta_1} - 1 \leq e^{|\eta_1 - \eta_2|}.
\end{align*}
It follows that 
\begin{align*}
\dfrac{b''\left( \eta_{1} \right)}{b''\left( \eta_{2} \right)} 
\leq  e^{\eta_1 - \eta_{2}} \times e^{2|\eta_1 - \eta_{2}|}  
\leq  e^{3|\eta_1 - \eta_{2}|},
\end{align*}
which completes the proof.
\end{proof}

\begin{lemma} \label{lemma:matrix_infty_norm}
Suppose that $s_{\ast}^2 \log p \leq n$ and $p \geq 3$.
Then, there exists a constant $K_{\infty} > 0$ such that
\begin{align} \label{eqn:matrix_infty_norm}
    \bbP \biggl\{ 
        \max_{S \in \scrS_{s_{\ast}}} \left\| \left( \bX_S^{\top} \bX_{S} \right)^{-1} \right\|_{\infty}
        \leq K_{\infty} n^{-1}
    \biggr\}
    \geq 1- 6p^{-s_{\ast}}.
\end{align}
\end{lemma}

\begin{proof}
By Lemma 5 in \cite{wainwright2009sharp}, we have, for $S \in \scrS_{\ast}$,
\begin{align*}
\bbP \bigg( 
\left\| n\left( \bX_S^{\top} \bX_{S} \right)^{-1} - \bI_{|S|} \right\|_{\infty} > 8 \left( \dfrac{|S|}{n} \right)^{1/2} + t
\bigg)
\leq 2 \exp \bigg( - c_1 \dfrac{n t^2}{128 |S|} + \log |S| + |S|\log 2 \bigg)
\end{align*}
for some universal constant $c_1 > 0$.
It follows that
\begin{align*}
&\bbP \biggl\{ 
    \max_{S \in \scrS_{s_{\ast}}}
    \left\| n\left( \bX_S^{\top} \bX_{S} \right)^{-1} - \bI_{|S|} \right\|_{\infty} > 8 \left( \dfrac{s_{\ast}}{n} \right)^{1/2} + t
\biggr\} \\
&\leq \left| \scrS_{s_\ast} \right| \max_{S \in \scrS_{s_\ast}} \left[ 2 \exp \bigg( - c_1 \dfrac{n t^2}{128 |S|} + \log |S| + |S| \log 2 \bigg) \right] \\
&\leq 3p^{s_\ast} \times 2 \exp \bigg( - c_1 \dfrac{n t^2}{128 s_{\ast}} + \log s_\ast + s_{\ast} \log 2 \bigg) \\
&\leq 6\exp \bigg( - c_1 \dfrac{n t^2}{128 s_{\ast}} + \log s_\ast + 2s_\ast \log p \bigg).
\end{align*}
By taking
\begin{align*}
    t = \left[ \dfrac{128 s_{\ast} }{c_1 n} \bigg( \log s_{\ast} + 3s_{\ast} \log p \bigg) \right]^{1/2},
\end{align*}
we have
\begin{align*}
&\bbP \Bigg(
    \max_{S \in \scrS_{s_{\ast}}}
    \left\| n\left( \bX_S^{\top} \bX_{S} \right)^{-1} - \bI_{\ast} \right\|_{\infty} > 8 \left( \dfrac{s_{\ast}}{n} \right)^{1/2} + \left[ \dfrac{128 s_{\ast} }{c_1 n} \bigg( \log s_{\ast} + 3s_{\ast} \log p \bigg) \right]^{1/2}
\Bigg) \\
&\leq 6p^{-s_{\ast}}.
\end{align*}
Since $p \geq 3$ and $s_{\ast} \in [1, p]$, we have
\begin{align*}
&\bbP \Bigg(
    \max_{S \in \scrS_{s_{\ast}}}
    \left\| n\left( \bX_S^{\top} \bX_{S} \right)^{-1} - \bI_{s_\ast} \right\|_{\infty} 
    > c_2 \left( \dfrac{s_{\ast}^2 \log p}{n} \right)^{1/2}
\Bigg) 
\leq 6p^{-s_{\ast}},
\end{align*}
for some constant $c_2 = c_2(c_1) > 0$. Therefore, 
\begin{align*}
    &\max_{S \in \scrS_{s_{\ast}}} \left\| \left( \bX_S^{\top} \bX_{S} \right)^{-1} \right\|_{\infty} 
    \leq \max_{S \in \scrS_{s_{\ast}}} \bigg[ \left\| \left( \bX_S^{\top} \bX_{S} \right)^{-1} - n^{-1}\bI_{s_\ast} \right\|_{\infty} + \left\| n^{-1}\bI_{s_\ast} \right\|_{\infty} \bigg] \\
    &\leq \left[ c_2 \left( \dfrac{s_{\ast}^2 \log p}{n} \right)^{1/2} + 1 \right] n^{-1}
    \leq \big( c_2 + 1 \big) n^{-1}
\end{align*}
with $\bbP$-probability at least $1 - 6p^{-s_{\ast}}$. This completes the proof of \eqref{eqn:matrix_infty_norm}.
\end{proof}

\begin{lemma} \label{lemma:cubic_poly_Gaussian}
Suppose that $(s_{\ast} \log p)^{3/2} \leq n$ and $p \geq 12$.
Then, 
\begin{align} \label{eqn:cubic_poly_Gaussian}
    \bbP \bigg(
    \max_{S \in \scrS_{s_{\ast}}} \sup_{u_S \in \cU_S}
    \dfrac{1}{n} \sum_{i = 1}^{n} \left| X_{i, S}^{\top} u_{S} \right|^{3} \leq \widetilde{K}_{\rm cubic}
    \bigg)
    \geq 1- 6p^{-s_{\ast}},
\end{align}
where $\widetilde{K}_{\rm cubic} > 0$ is a constant.
\end{lemma}

\begin{proof}
    Let $\widehat{\cU}_{S, 1/4}$ be a $1/4$-cover of $\cU_S$. 
    By the Proposition 1.3 of Section 15 in \cite{lorentz1996constructive}, one can choose $\widehat{\cU}_{S, 1/4}$ so that $|\widehat{\cU}_{S, 1/4}| \leq 12^{|S|}$.
    Let $u_S \in \cU_S$ and $u_{S}' \in \widehat{\cU}_{S, 1/4}$ with $\| u_{S} - u_{S}' \|_2 \leq 1/4$.
    Let $f(u_{S}) = n^{-1} \sum_{i = 1}^{n} \left| X_{i, S}^{\top} u_{S} \right|^{3}$. Note that 
    \begin{align*}
        &f(u_{S}) - f(u_{S}') \\
        &= \dfrac{1}{n} \sum_{i=1}^{n} \left[ \bigg( \left| X_{i, S}^{\top} u_{S} \right| - \left| X_{i, S}^{\top} u_{S}' \right| \bigg) \bigg( \left| X_{i, S}^{\top} u_{S} \right|^2 + \left| X_{i, S}^{\top} u_{S} \right| \left| X_{i, S}^{\top} u_{S}' \right| + \left| X_{i, S}^{\top} u_{S}' \right|^2 \bigg) \right] \\
        &\leq
        \dfrac{1}{n} \sum_{i=1}^{n} \left[ \left| X_{i, S}^{\top} \left[u_{S} - u_{S}' \right] \right| \bigg( \left| X_{i, S}^{\top} u_{S} \right|^2 + \left| X_{i, S}^{\top} u_{S} \right| \left| X_{i, S}^{\top} u_{S}' \right| + \left| X_{i, S}^{\top} u_{S}' \right|^2 \bigg) \right] \\
        &\leq
        \left\| u_{S} - u_{S}' \right\|_2  3 \sup_{u_{1}, u_{2}, u_{3} \in \cU_S} \biggl\{ \dfrac{1}{n} \sum_{i=1}^{n} \left[
           \left| X_{i, S}^{\top} u_{1} \right| \times \left| X_{i, S}^{\top} u_{2} \right| \times \left| X_{i, S}^{\top} u_{3} \right| \right] \biggr\} \\        
        &\leq
        \left\| u_{S} - u_{S}' \right\|_2  \sup_{u_{1}, u_{2}, u_{3} \in \cU_S}  \biggl\{ \dfrac{1}{n} \sum_{i=1}^{n} \left[
         \left| X_{i, S}^{\top} u_{1} \right|^3 +  \left| X_{i, S}^{\top} u_{2} \right|^3 + \left| X_{i, S}^{\top} u_{3} \right|^3  \right] \biggr\} \\                
        &\leq
        3\left\| u_{S} - u_{S}' \right\|_2  \sup_{u_{1} \in \cU_S}  \biggl\{ \dfrac{1}{n} \sum_{i=1}^{n} \left| X_{i, S}^{\top} u_{1} \right|^3  \biggr\} 
        \leq \dfrac{3}{4} \sup_{u_{1} \in \cU_S} f(u_1),
    \end{align*}
    where the third inequality holds by arithmetic mean-geometric inequality. It follows that
    \begin{align*}
        \sup_{u_{S} \in \cU_S} f(u_S) \leq \max_{u_{S}' \in \widehat{\cU}_{S, 1/4}}f(u_S') + \dfrac{3}{4} \sup_{u_{1} \in \cU_S} f(u_1),
    \end{align*}
    implying
    \begin{align} \label{eqn:cubic_poly_proof1}
        \sup_{u_{S} \in \cU_S} f(u_S) \leq 4 \max_{u_{S}' \in \widehat{\cU}_{S, 1/4}}f(u_S').
    \end{align}
    We will use a concentration inequality for polynomials of sub-Gaussian variables
    (see page 11 of the supplementary material in  \cite{loh2017statistical} and Theorem 1.4 in \cite{adamczak2015concentration}). For $u_{S} \in \cU_{S}$ and $t \geq 0$, we have
    \begin{align*}
        \bbP \Bigg( 
            \left| f(u_S) - \bbE f(u_S) \right| \geq c_1 \left[ \left( \dfrac{t}{n} \right)^{1/2} + \dfrac{t^{3/2}}{n} \right]
        \Bigg) \leq 2e^{-t}
    \end{align*}
    for some universal constant $c_1 > 0$. It follows that
    \begin{align*}
        \bbP \Bigg( 
            \max_{u_{S} \in \widehat{\cU}_{S, 1/4}} f(u_S)
            \geq \bbE f(u_S) + c_1 \left[ \left( \dfrac{t}{n} \right)^{1/2} + \dfrac{t^{3/2}}{n} \right]
        \Bigg) \leq 2e^{-t + |S|\log (12)},
    \end{align*}
    where the inequality holds by $|\widehat{\cU}_{S, 1/4}| \leq 12^{|S|}$. Also, \eqref{eqn:cubic_poly_proof1} implies that
    \begin{align*}
        \bbP \Bigg( 
            \sup_{u_{S} \in \cU_{S}} f(u_S)
            \geq 4\bbE f(u_S) + 4c_1 \left[ \left( \dfrac{t}{n} \right)^{1/2} + \dfrac{t^{3/2}}{n} \right]
        \Bigg) \leq 2e^{-t + |S|\log (12)}.
    \end{align*}
    By taking $t = 3s_{\ast}\log p$, $|\scrS_{s_{\ast}}| \leq 3p^{s_{\ast}}$ and $\bbE f(u_S) = \sqrt{8/\pi}$ give
    \begin{align*}
        &\bbP \Bigg( 
            \max_{S \in \scrS_{s_{\ast}}} \sup_{u_{S} \in \cU_S} f(u_S)
            \geq \sqrt{\dfrac{128}{\pi}} + 4c_1 \left[ \left( \dfrac{3s_{\ast}\log p}{n} \right)^{1/2} + \dfrac{(3s_{\ast}\log p)^{3/2}}{n} \right]
        \Bigg) \\
        &\leq 6e^{-s_{\ast}\log p} = 6p^{-s_{\ast}},
    \end{align*}
    where the inequality holds by $p \geq 12$. Therefore, 
    \begin{align*}
        \bbP \bigg( 
            \max_{S \in \scrS_{s_{\ast}}} \sup_{u_{S} \in \cU_S}
            f(u_S) \geq K
        \bigg) \leq 6p^{-s_{\ast}},
    \end{align*}
    where $K = \sqrt{128/\pi} + 16c_1\sqrt{3}$.    
\end{proof}

\begin{lemma} \label{lemma:orlic_bound}
    Suppose that $\bX$ is non-random. Then, we have
    \begin{align} \label{eqn:orlicz_norm_claim}
        \max_{i \in [n]} \left\| \epsilon_{i}  \right\|_{\psi_{1}} \leq 
            \big(1 + 2/(e\log 2) \big) \big( 1 + \sigma_{\max}^{2} (\log 2)^{-1} \big),
    \end{align}
    where $\sigma_{\max}^2$ is defined in Lemma \ref{lemma:eps_deviation_ineq}.
\end{lemma}
\begin{proof}
    To prove \eqref{eqn:orlicz_norm_claim}, we utilize the result of Lemma A.3 in \cite{Gotze2021subE}. By taking $K_{\alpha} = 1$ and $d_{\alpha} = e/2$ in
    Lemma A.3 in \cite{Gotze2021subE}, we have, for $i \in [n]$,
    \begin{align} \label{eqn:logit_Orlicz_bound}
        \left\| \epsilon_{i}  \right\|_{\psi_{1}} = \left\| Y_i - \bbE Y_i  \right\|_{\psi_{1}}
        \leq \left( 1 + \left[\dfrac{e \log 2}{2} \right]^{-1} \right) \left\| Y_i \right\|_{\psi_{1}}.
    \end{align}
    First, we consider the logistic regression case. Note that
    \begin{align*}
        \left\| Y_i \right\|_{\psi_{1}} \leq \left( \log 2 \right)^{-1/2} \left\| Y_i \right\|_{\psi_{2}} \leq \dfrac{1}{4\sqrt{\log 2}},
    \end{align*}
    where the inequalities hold by the standard result of the exponential Orlicz norms (see page 145 in \cite{van2023weak}) and $Y_i \in [0, 1]$. Therefore, \eqref{eqn:orlicz_norm_claim} holds because
    $3 \geq (4\sqrt{\log 2})^{-1}$. 
    
    Secondly, we consider the Poisson regression case. Let $\sigma_i^2 = \Var(Y_i)$.
    By $\log(1+x) \geq x/(1+x)$ for $x \geq 0$, we have
    \begin{align*}
        \left\| Y_i \right\|_{\psi_{1}} = \dfrac{1}{ \log\left[ \sigma_{i}^{-2} \log 2 + 1 \right]}
        \leq
        \dfrac{\sigma_{i}^{-2} \log 2 + 1 }{ \sigma_{i}^{-2} \log 2 }
        =
        1 + \sigma_{i}^{2} (\log 2)^{-1},
    \end{align*}
    which completes the proof of \eqref{eqn:orlicz_norm_claim}.
\end{proof}

\subsection{Poisson regression}

\begin{lemma} \label{lemma:log_normal_concentration}
We have
\begin{align} \label{eqn:log_normal_concentration}
    \bbP \biggl\{ 
    \dfrac{1}{4}
    \leq
    \dfrac{1}{n} \sum_{i=1}^{n} e^{X_i^{\top} \theta_0}
    \leq 
    2 e^{\| \theta_0 \|_{2}^{2}}
    \biggr\}
    \geq 1 - n^{-1} - e^{-n/24}.
\end{align}
\end{lemma}

\begin{proof}
The assertion is trivial for $\theta_0 = 0$; hence assume that $\theta_0 \neq 0$.
Note that $X_i^{\top} \theta_0 \overset{\iid}{\sim} \cN (0, \| \theta_0 \|_{2}^2 )$ for all $i \in [n]$. By the definition of log-normal distribution, note that
\begin{align*}
    \exp\left( X_i^{\top} \theta_0 \right) \overset{\iid}{\sim} \operatorname{logNormal} \left( 0, \| \theta_0 \|_{2} \right),
\end{align*}
where $\operatorname{logNormal} \left( \mu, \sigma \right)$ denotes the log-normal distribution which has probability density function  $f(x)$ defined as
\begin{align*}
    f(x) =  \dfrac{1}{x \sigma \sqrt{2\pi}} \exp \left( - \dfrac{\left(\log x - \mu \right)^2}{2 \sigma^2} \right) \mathds{1}_{ \left\{x \geq 0 \right\} }.
\end{align*}
By Chebyshev inequality, we have
\begin{align*}
    \bbP \Bigg( 
    \left| \dfrac{1}{n} \sum_{i=1}^{n} e^{X_i^{\top} \theta_0} - \bbE e^{X_i^{\top} \theta_0} \right|
    \geq 
    t
    \Bigg)
    \leq 
    \dfrac{\Var\big( e^{X_1^{\top} \theta_0} \big)}{nt^{2}}.
\end{align*}
By taking $t = \sqrt{ \Var\big( e^{X_i^{\top} \theta_0} \big) }$,
\begin{align*}
    \bbE e^{X_i^{\top} \theta_0} = e^{ \| \theta_0 \|_{2}^2 / 2}, \quad 
    \Var\big( e^{X_i^{\top} \theta_0} \big) = \big( e^{\| \theta_0 \|_{2}^2} - 1 \big) e^{\| \theta_0 \|_{2}^2} \leq  \bigg( e^{\| \theta_0 \|_{2}^2} - \dfrac{1}{2} \bigg)^2
\end{align*}
implies that
\begin{align*}
\dfrac{1}{n} \sum_{i=1}^{n} e^{X_i^{\top} \theta_0} \leq e^{ \| \theta_0 \|_{2}^2 / 2} + e^{\| \theta_0 \|_{2}^2} - \dfrac{1}{2} \leq 2e^{\| \theta_0 \|_{2}^2}
\end{align*}
with $\bbP$-probability at least $1 - n^{-1}$. This completes the proof of the upper bound in \eqref{eqn:log_normal_concentration}

Next, we will prove the lower bound of $n^{-1} \sum_{i=1}^{n} e^{X_i^{\top} \theta_0}$. 
We will utilize the Chernoff-type left tail inequality (see Section 2.3 in \cite{vershynin2018high}). Let $S_n = \sum_{i=1}^{n} Z_i$,
where $Z_i \overset{\iid}{\sim} \operatorname{Bernoulli}(\omega)$. Then, 
\begin{align*}
    \bbP \biggl\{ S_n \leq (1 - \delta) \omega n \biggr\} \leq \exp \left( - \dfrac{\delta^2}{3} \omega n \right).
\end{align*}
Note that $\bbP \left( X_i^{\top} \theta_0 \geq 0 \right) = 1/2$.
By taking $\delta = 1/2$ and $\omega = 1/2$, $\bbP \left( |\cI| \leq n/4 \right) \leq e^{-n/24}$, where $\cI =  \left\{ i \in [n] : X_{i}^{\top} \theta_0 \geq 0  \right\}$.
Since each $e^{X_i^{\top} \theta_0}$ is positive,
\begin{align*}
    \dfrac{1}{n} \sum_{i=1}^{n} e^{X_i^{\top} \theta_0} \geq \dfrac{1}{n} \sum_{i \in \cI} e^{X_i^{\top} \theta_0} \geq \dfrac{|\cI|}{n} \geq \dfrac{1}{4}
\end{align*}
with $\bbP$-probability at least $1 - e^{-n/24}$. This completes the proof of the lower bound in \eqref{eqn:log_normal_concentration}
\end{proof}

\begin{lemma} \label{lemma:variance_bound_example_Poisson}
Suppose that $b(\cdot) = \exp(\cdot)$ and $\| \theta_0 \|_{2} \leq c_1$ for some constant $c_1 > 0$. Then, for any $k > 0$, there exists a constant $K > 0$, depending only on $k$ and $c_1$, such that
\begin{align*} %\label{eqn:variance_bound_example_Poisson}
\sigma_{\min}^{-2} \vee \sigma_{\max}^{2} \leq \exp\left( 2 \| \theta_0 \|_{2} \sqrt{\log n} \right) \leq K n^{k}.
\end{align*}
with $\bbP$-probability at least $1- n^{-1}$. 
\end{lemma}
\begin{proof}
By Lemma \ref{lemma:natural_parameter}, we have
\begin{align*}
    \bbP \biggl\{ 
    \max_{i \in [n]} \left| X_{i}^{\top} \theta_0 \right| \geq 2 \left\| \theta_0 \right\|_{2} \sqrt{\log n} 
    \biggr\}
    \geq 1 - n^{-1}.    
\end{align*}
Since $b''(\cdot) = \exp(\cdot)$, it follows that
\begin{align*}
    \sigma_{\max}^{2} = \max_{i \in [n]} \exp\left( X_{i}^{\top} \theta_0 \right) 
    \leq \exp \left( 2 \left\| \theta_0 \right\|_{2} \sqrt{\log n} \right)
    \leq \exp \left( 2c_1 \sqrt{\log n} \right)
\end{align*}
with $\bbP$-probability at least $1- n^{-1}$. 
For any $k > 0$, we have
\begin{align*}
    \lim_{n \rightarrow \infty} \dfrac{e^{\sqrt{\log n}}}{n^{k}} = 0.
\end{align*}
Hence, we have, for any $k' > 0$, there exists some constant $K > 0$ such that
\begin{align*}
    \exp \left( 2c_1 \sqrt{\log n} \right) \leq K n^{k'}.
\end{align*}
Also,
\begin{align*}
    \sigma_{\min}^{2} = \min_{i \in [n]} \exp\left( X_{i}^{\top} \theta_0 \right) 
    \geq \exp \left( -2 \left\| \theta_0 \right\|_{2} \sqrt{\log n} \right)
    \geq \exp \left( -2c_1 \sqrt{\log n} \right),
\end{align*}
Therefore, the upper bound of $\sigma_{\min}^{-2}$ can be proved similarly. %This completes the proof of \eqref{eqn:variance_bound_example_Poisson}.
\end{proof}

\begin{lemma} \label{lemma:Poisson_least_eigenvalue}
Suppose that $b(\cdot) = \exp(\cdot)$ and 
\begin{align*}
    n \geq C s_{\ast}\log (n \vee p), \quad  p \geq C,
\end{align*}
where $C > 0$ is a large enouch constant.
Then, 
\begin{align} \label{eqn:Poisson_least_eigenvalue}
\begin{aligned}
    &\bbP \bigg( \lambda_{\min} \big( \bF_{n, \theta_S} \big) \leq \dfrac{n}{540} \: \: \text{ for some } S \in \scrS_{s_{\ast}}  \text{ and } \theta_S \in \bbR^{|S|} \bigg)  \\
    &\leq 
    2(np)^{-1} + 3e^{-n/50} + 3e^{-n/30} + 3e^{-n/240}.    
\end{aligned}
\end{align}
\end{lemma}

\begin{proof}
Let $S \in \scrS_{s_{\ast}}$ and $\theta_S \in \bbR^{S}$. 
For $i \in [n]$, note that $\exp\left( X_{i, S}^{\top} \theta_S \right) \geq 1$ is equivalent to $\exp\left( X_{i, S}^{\top} \theta_S / \| \theta_S \|_2 \right) \geq 1$. For $\delta > 0$, let
\begin{align*}
    \cI_S(u_S, \delta) &=  \left\{ i \in [n] : \exp\left( X_{i, S}^{\top} u_S \right) \geq  \delta \right\}, \\
    \cU_S &= \left\{ u_S \in \bbR^{|S|} : \| u_S \|_{2} = 1 \right\}.
\end{align*}
Let $\nu_S = \theta_S / \| \theta_S \|_2$. Note that
\begin{align*}
    \lambda_{\min} \big( \bF_{n, \theta_S} \big)
    &= \lambda_{\min} \Bigg( \sum_{i=1}^{n} \exp\big( X_{i, S}^{\top} \theta_S \big) X_{i, S} X_{i, S}^{\top}  \Bigg) \\
    &\geq \lambda_{\min} \Bigg( \sum_{i \in \cI_S(\nu_S, 1)} \exp\big( X_{i, S}^{\top} \theta_S \big) X_{i, S} X_{i, S}^{\top} \Bigg) \\
    &\geq \lambda_{\min} \Bigg( \sum_{i \in \cI_S(\nu_S, 1)} X_{i, S} X_{i, S}^{\top} \Bigg).
\end{align*}
If $| \cI_S(u_S, 1) | \geq C' n$ for all $u_S \in \cU_S$ and $S \in \scrS_{s_\ast}$ with some constant $C' \in (0, 1)$ on an event $\Omega$, then Lemma \ref{lemma:extreme_eigenvalue} implies that
\begin{align*}
    \lambda_{\min} \Bigg( \sum_{i \in \cI_S(\nu_S, 1)} X_{i, S} X_{i, S}^{\top} \Bigg)
    \geq 
    \dfrac{1}{9} C' n 
\end{align*}
on $\Omega \cap \Omega'$, where $\Omega'$ is an event with $\bbP(\Omega') \geq 1 - 3e^{-C'n/4}$.
To complete the proof, therefore, we need to show that $| \cI_S(u_S, 1) |$ is sufficiently large for all $u_S \in \cU_S$ and $S \in \scrS_{s_\ast}$.

For $\epsilon > 0$, let $\widehat{\cU}_{S, \epsilon}$ be an $\epsilon$-cover of $\cU_S$ with $|\widehat{\cU}_{n, \epsilon}| \leq (3/\epsilon)^{|S|}$. 
One can choose such a cover by Proposition 1.3 of Section 15 in \cite{lorentz1996constructive}.
Then, for $u_S \in \cU_S$, we can choose $\widehat{u}_S \in \widehat{\cU}_{S, \epsilon}$ satisfying 
$\| u_S - \widehat{u}_S \|_{2} \leq \epsilon$.
Note that
\begin{align*}
    \exp \left( X_{i, S}^{\top} u_S \right) 
    &= 
    \exp \left( X_{i, S}^{\top} \left[ u_S - \widehat{u}_S \right] + X_{i, S}^{\top}\widehat{u}_S \right) \\
    &\geq 
    \exp \left( -\left\| X_{i, S} \right\|_{2} \left\| u_S - \widehat{u}_S \right\|_{2} + X_{i, S}^{\top}\widehat{u}_S \right) \\
    &\geq 
    \exp \left( -\epsilon \left\| X_{i, S} \right\|_{2} + X_{i, S}^{\top}\widehat{u}_S \right).
\end{align*}
Hence, if
\begin{align*}
    \exp \left( X_{i, S}^{\top}\widehat{u}_S \right) 
    \geq 
    \exp \left( \epsilon \max_{i \in [n]} \max_{S \in \scrS_{s_\ast}} \left\| X_{i, S} \right\|_{2} \right)
\end{align*}
then $\exp \left( X_{i, S}^{\top} u_S \right) \geq 1$. 
Let $\delta(\epsilon) = \epsilon \max_{i \in [n]} \max_{S \in \scrS_{s_\ast}} \left\| X_{i, S} \right\|_{2}$.
By the last display, for $u_S \in \cU_S$ and $S \in \scrS_{s_\ast}$, we have
\begin{align*}
    \left| \cI_{S}(u_S, 1) \right|
    \geq 
    \left| \cI_{S}(\widehat{u}_S, e^{\delta(\epsilon)}) \right|
    \geq
    \min_{S \in \scrS_{s_\ast}} \min_{\widehat{u}_S \in \widehat{\cU}_{S, \epsilon}}
    \left| \cI_{S}(\widehat{u}_S, e^{\delta(\epsilon)}) \right|.
\end{align*}
By Lemma \ref{lemma:design_row_norm}, there exists an event $\Omega_{n, 1}$ such that, on $\Omega_{n, 1}$, 
\begin{align*}
    \max_{i \in [n]} \max_{S \in \scrS_{s_\ast}} \left\| X_{i, S} \right\|_{2}
    \leq 
    2\sqrt{2} \sqrt{s_\ast \log (n \vee p)},
\end{align*}
and $\bbP(\Omega_{n, 1}) \geq 1 - 2(np)^{-1}$.
By taking $\epsilon_0 = ( 4\sqrt{2} \sqrt{s_\ast \log (n \vee p)} )^{-1}$, on $\Omega_{n, 1}$, we have
\begin{align*}
    \delta(\epsilon_0) = 
    \epsilon_0 \max_{i \in [n]} \max_{S \in \scrS_{s_\ast}} \left\| X_{i, S} \right\|_{2}
    \leq 
    \left( 4\sqrt{2} \sqrt{s_\ast \log (n \vee p)} \right)^{-1}
    2\sqrt{2} \sqrt{s_\ast \log (n \vee p)}
    = 1/2.
\end{align*}
Also, by $X_{i, S}^{\top} \widehat{u}_S \sim \cN(0, 1)$, we have
\begin{align*}
    \bbP \left\{ \exp\left( X_{i, S}^{\top} \widehat{u}_S \right) \geq e^{1/2} \right\} 
    =
    \bbP \left\{ X_{i, S}^{\top} \widehat{u}_S \geq 1/2 \right\} 
    \geq 3/10.
\end{align*}
We will utilize the Chernoff-type left tail inequality (see Section 2.3 in \cite{vershynin2018high}). Let $S_n = \sum_{i=1}^{n} Z_i$,
where $Z_i \overset{\iid}{\sim} \operatorname{Bernoulli}(\omega)$. Then, 
\begin{align*}
    \bbP \biggl\{ S_n \leq (1 - \delta) \omega n \biggr\} \leq \exp \left( - \dfrac{\delta^2}{3} \omega n \right).
\end{align*}
By taking $\delta = 1/2$ and $\omega = 3/10$, for $S \in \scrS_{s_{\ast}}$ and $\widehat{u}_S \in \widehat{\cU}_{S, \epsilon_0}$, we have
\begin{align*}
    \bbP \left( \left| \cI_{S}(\widehat{u}_S, e^{1/2}) \right| \leq \dfrac{3}{20}n \right) \leq e^{-n/40},
\end{align*}
Let
\begin{align*}
\Omega_{n, 2} &= \biggl\{ \left| \cI_{S}(\widehat{u}_S, e^{1/2}) \right| \geq \dfrac{1}{4} n \: \: \text{ for all } S \in \scrS_{s_{\ast}} \text{ and } \ \widehat{u}_S \in \widehat{\cU}_{S, \epsilon_0} \biggr\}, \\
\Omega_{n, 3} &= \Biggl\{ 
\lambda_{\min} \left( \sum_{i \in \cI_{S}(\widehat{u}_S, e^{1/2})}  X_{i, S} X_{i, S}^{\top} \right) \geq \dfrac{1}{9} \left| \cI_{S}(\widehat{u}_S, e^{1/2}) \right| \text{ for all } S \in \scrS_{s_{\ast}} \text{ and } \ \widehat{u}_S \in \widehat{\cU}_{S, \epsilon_0}
\Biggr\}
\end{align*}
Note that
\begin{align*}
    \bbP \big( \Omega_{n, 2} \mid \Omega_{n, 1} \big)
    &\leq \left| \scrS_{s_{\ast}} \right| \left| \widehat{\cU}_{S, \epsilon_0} \right|   
    \max_{S \in \scrS_{s_{\ast}}} \max_{\widehat{u}_S \in \widehat{\cU}_{S, \epsilon_0}} 
    \bbP \bigg(  \left| \cI_{S}(\widehat{u}_S, e^{1/2}) \right| \leq \dfrac{3}{20}n \bigg) \\
    &\leq 3p^{s_{\ast}} \left( 12\sqrt{2} \sqrt{s_\ast \log (n \vee p)} \right)^{|S|}  e^{-n/40} \\
    &= 3 \exp \bigg( -\dfrac{n}{40} + s_{\ast} \log p + s_{\ast} \log \big( 12\sqrt{2} \big) + \dfrac{s_{\ast}}{2} \log \big( s_\ast \log (n \vee p) \big) \bigg) \\
    &\leq 
    3 e^{-n/50}.
\end{align*}
Also, Lemma \ref{lemma:extreme_eigenvalue} gives
\begin{align*}
    \bbP \big( \Omega_{n, 3} \mid \Omega_{n, 1}, \Omega_{n, 2} \big)
    &\leq 3| \widehat{\cU}_{S, \epsilon_0} | e^{-3n/80} \leq 3e^{-n/30}.
\end{align*}
It follows that
\begin{align*}
\bbP \bigl\{ \Omega_n \bigr\}    
\geq 1 - \big( 2(np)^{-1} + 3e^{-n/50} + 3e^{-n/30} \big), 
\end{align*}
where $\Omega_n = \Omega_{n, 1} \cap \Omega_{n, 2} \cap \Omega_{n, 3}$.
Therefore, on $\Omega_n$, we have
\begin{align*}
    \left| \cI_{S}(u_S, 1) \right|
    \geq
    \min_{S \in \scrS_{s_\ast}} \min_{\widehat{u}_S \in \widehat{\cU}_{S, \epsilon_0}}
    \left| \cI_{S}(\widehat{u}_S, e^{1/2}) \right|
    \geq
    \dfrac{1}{60} n.
\end{align*}
Therefore, we can conclude that
\begin{align*}
    \lambda_{\min} \big( \bF_{n, \theta_S} \big)
    \geq 
    \dfrac{1}{540}n, \quad
    \forall S \in \scrS_{s_\ast}, \text{ and } \forall \theta_S \in \bbR^{|S|}
\end{align*}
with $\bbP$-probability at least $1 - 2(np)^{-1} - 3e^{-n/50} - 3e^{-n/30} - 3e^{-n/240}$.
\end{proof}

%%%%%%%%%%%%%%%%%% 1

\begin{lemma} \label{lemma:Poisson_least_eigenvalue_V}
Suppose that $b(\cdot) = \exp(\cdot)$, $4s_{\ast}\log p \leq n$ and $p \geq 3$.
Then, 
\begin{align} \label{eqn:Poisson_least_eigenvalue_V}
    \bbP \bigg( \lambda_{\min} \big( \bV_{n, S} \big) \leq \dfrac{n}{36} \ \ \text{ for some } S \in \scrS_{s_{\ast}} \bigg) \leq 5e^{-n/24}.
\end{align}
\end{lemma}

\begin{proof}
Since the proof of this Lemma is similar to Lemma \ref{lemma:Poisson_least_eigenvalue}, we provide the sketch of the proof. 
Since $ X_{i}^{\top} \theta_0 \sim \cN(0, \| \theta_{0} \|_{2}^{2} )$, we have 
\begin{align*}
    \bbP \left\{ \exp\left( X_{i}^{\top} \theta_0 \right) \geq 1 \right\} \geq 1/2
\end{align*}
for all $i \in [n]$.
By the similar argument in Lemma \ref{lemma:Poisson_least_eigenvalue}, we have
$\bbP \left( |\cI| \leq n/4 \right) \leq e^{-n/24}$, where $\cI =  \left\{ i \in [n] : \exp\left( X_{i}^{\top} \theta_0 \right) \geq 1  \right\}$. 
Let
\begin{align*}
\Omega_{n, 1} &= \biggl\{ \left| \cI \right| \geq \dfrac{1}{4} n \biggr\}, \quad
\Omega_{n, 2} = \biggl\{ 
\lambda_{\min} \left( \sum_{i \in \cI}  X_{i, S} X_{i, S}^{\top} \right) \geq \dfrac{1}{9} \left| \cI \right| \text{ for all } S \in \scrS_{s_{\ast}} 
\biggr\}.
\end{align*}
By Lemma \ref{lemma:extreme_eigenvalue},
\begin{align*}
\bbP \bigl\{ \Omega_{n, 1}^{\rm c} \bigr\} \leq e^{-n/24}, \quad     
\bbP \bigl\{ \Omega_{n, 2}^{\rm c} \mid \Omega_{n, 1}  \bigr\} \leq 3e^{-n/16}.
\end{align*}
Note that
\begin{align*}
\bbP \bigl\{ \Omega_{n, 1}^{\rm c} \cup \Omega_{n, 2}^{\rm c}  \bigr\}    
&\leq 
\bbP \bigl\{ \Omega_{n, 1}^{\rm c} \bigr\} + \bbP \bigl\{ \Omega_{n, 2}^{\rm c}  \bigr\} 
=
\bbP \bigl\{ \Omega_{n, 1}^{\rm c} \bigr\} + \bbP \bigl\{ \Omega_{n, 2}^{\rm c} \cap \Omega_{n, 1} \bigr\} + \bbP \bigl\{ \Omega_{n, 2}^{\rm c} \cap \Omega_{n, 1}^{\rm c} \bigr\} \\
&\leq 
\bbP \bigl\{ \Omega_{n, 1}^{\rm c} \bigr\} + \bbP \bigl\{ \Omega_{n, 2}^{\rm c} \mid \Omega_{n, 1} \bigr\} + \bbP \bigl\{ \Omega_{n, 1}^{\rm c} \bigr\} 
= 2\bbP \bigl\{ \Omega_{n, 1}^{\rm c} \bigr\} + \bbP \bigl\{ \Omega_{n, 2}^{\rm c} \mid \Omega_{n, 1} \bigr\} \\
&\leq 5e^{-n/24}. 
\end{align*}
It follows that $\bbP \bigl\{ \Omega_{n} \bigr\} \geq 1 - 5e^{-n/24}$, where $\Omega_{n} = \Omega_{n, 1} \cap \Omega_{n, 2}$. 
On $\Omega_n$, note that
\begin{align*}
    \lambda_{\min} \big( \bV_{n, S} \big)
    &= \lambda_{\min} \Bigg( \sum_{i=1}^{n} \exp\big( X_{i}^{\top} \theta_0 \big) X_{i, S} X_{i, S}^{\top}  \Bigg)
    \geq \lambda_{\min} \Bigg( \sum_{i \in \cI} \exp\big( X_{i}^{\top} \theta_0 \big) X_{i, S} X_{i, S}^{\top} \Bigg) \\
    &\geq \lambda_{\min} \Bigg( \sum_{i \in \cI} X_{i, S} X_{i, S}^{\top} \Bigg)
    \geq \frac{1}{36} n
\end{align*}
for all $S \in \scrS_{s_{\ast}}$. This completes the proof.
\end{proof}

%%%%%%%%%%%%%%%%%%%%%% 2

\begin{lemma} \label{lemma:Poisson_largest_eigenvalue}
Suppose that $b(\cdot) = \exp(\cdot)$ and 
\begin{align*}
    n \geq C (s_{\ast}\log p)^{3/2}, \quad p \geq C,
\end{align*}
where $C > 0$ is large enough constant.
Then, 
\begin{align} \label{eqn:Poisson_largest_eigenvalue_V}
    \bbP \bigg( \lambda_{\max} \big( \bV_{n, S} \big) \geq K e^{3\| \theta_0 \|_{2}^{2}} n  \: \: \text{ for some } S \in \scrS_{s_{\ast}} \bigg) \leq n^{-1} + e^{-n/24} + 6p^{-s_\ast},
\end{align}
where $K > 0$ is a constant.
\end{lemma}

\begin{proof}
    Let $\cU_S = \{ u_S \in \bbR^{|S|} : \| u_S \|_{2} = 1 \}$.
    By Lemmas \ref{lemma:cubic_poly_Gaussian} and \ref{lemma:log_normal_concentration}, there exists an event $\Omega_n$ such that the following inequalities hold on $\Omega_n$:  
    \begin{align*}
        \sum_{i=1}^{n} e^{X_{i}^{\top} \theta} \leq 2n e^{ \| \theta \|_{2}^{2}}, \quad 
        \max_{S \in \scrS_{s_{\ast}}} \sup_{u_S \in \cU_S} \sum_{i=1}^{n} \left| X_{i, S}^{\top} u_S \right|^{3} \leq \widetilde{K}_{\rm cubic}
    \end{align*}
    for any $\theta \in \bbR^p$ and some constant $\widetilde{K}_{\rm cubic} > 0$, and 
    \begin{align*}
        \bbP(\Omega_n) \geq 1 - n^{-1} - e^{-n/24} - 6p^{-s_\ast}.
    \end{align*}
    It follows that, on $\Omega_n$,
    \begin{align*}
        \max_{S \in \scrS_{s_{\ast}}}
        \lambda_{\max} \big( \bV_{n, S} \big)
        &= 
        \max_{S \in \scrS_{s_{\ast}}}
        \sup_{u_S \in \cU_S} \sum_{i=1}^{n} \exp( X_{i}^{\top} \theta_0 ) \left( X_{i, S}^{\top} u_S \right)^{2} \\
        &\leq  
        \bigg[ \sum_{i=1}^{n} \left( e^{X_{i}^{\top} \theta_0} \right)^{3} \bigg]^{1/3}
        \bigg[ \max_{S \in \scrS_{s_{\ast}}} \sup_{u_S \in \cU_S} 
        \sum_{i=1}^{n} \left| X_{i, S}^{\top} u_S \right|^{3} \bigg]^{2/3} \\
        &\leq 
        \bigg[ 2n e^{9 \| \theta_0 \|_{2}^{2} } \bigg]^{1/3}
        \bigg[ \widetilde{K}_{\rm cubic} n \bigg]^{2/3}
        = \left( 2^{1/3} \widetilde{K}_{\rm cubic}^{2/3} e^{3 \| \theta_0 \|_{2}^{2}} \right) n.
    \end{align*}
    This completes the proof of \eqref{eqn:Poisson_largest_eigenvalue_V}.
\end{proof}

\begin{lemma} \label{lemma:score_vec_random_design_Poisson}
Let $\widetilde{\xi}_{n, S} = \bV_{n, S}^{-1/2} \dot{L}_{n, \thetaBest[S]}$.
Suppose that $b(\cdot) = \exp(\cdot)$ and
\begin{align*}
    n \geq C \big( s_{\ast}\log (n \vee p) \big)^{2},
\end{align*}
where $C > 0$ is large enough constant.
Then,
\begin{align} \label{claim:score_vec_random_design_Poisson}
    &\bbP \bigg(  \left\| \widetilde{\xi}_{n, S} \right\|_{2} > K (|S| \log p)^{1/2}  \: \:
    \text{ for some } S \in \scrS_{s_\ast}  \bigg) \leq 5n^{-n/24} + 3p^{-1},
\end{align}
where $K > 0$ is a constant.
\end{lemma}

\begin{proof}
Let $1 \leq s_\ast \leq p$. 
By Lemmas \ref{lemma:design_row_norm} and \ref{lemma:Poisson_least_eigenvalue_V}, there exists an event 
$\Omega_{n, 1}$ such that the following inequalities hold on $\Omega_{n, 1}$
\begin{align*} 
    \min_{S \in \scrS_{s_{\ast}}} \lambda_{\min} \left( \bV_{n, S} \right) \geq \dfrac{n}{36}, \quad 
    \max_{i \in [n]} \max_{S \in \scrS_{s_{\ast}}} \left\| X_{i, S} \right\|_{2}^{2} \leq 8 s_{\ast} \log(n \vee p),
\end{align*}
and $\bbP \left( \Omega_{n, 1} \right) \geq 1 - 5n^{-n/24} - 2(np)^{-1}$. It follows that
\begin{align*}
    \max_{i \in [n]} \max_{S \in \scrS_{s_{\ast}}}
    \left\| \bV_{n, S}^{-1/2} X_{i, S} \right\|_{2}
    &\leq
    \bigg[ \min_{S \in \scrS_{s_{\ast}}} \lambda_{\min}^{-1/2} \left( \bV_{n, S} \right) \bigg]
    \bigg[ \max_{i \in [n]} \max_{S \in \scrS_{s_{\ast}}} \left\| X_{i, S} \right\|_{2} \bigg] \\
    &\leq 
    \left( 6 n^{-1/2} \right) \left( 2 \sqrt{2} \sqrt{s_{\ast} \log (n \vee p)} \right) \\
    &\leq 12\sqrt{2} (n^{-1} s_{\ast} \log (n \vee p))^{1/2},
\end{align*}
where $c_1 > 0$ is a constant depending only on $C$ and $C'$.

Conditioning on $\bX$, for $S \in \scrS_{s_\ast}$, note that $\bbE \nabla L_{n, \thetaBest} = 0$ implies $\sum_{i=1}^n (\epsilon_i - \epsilon_{i, \theta_S^*}) X_{i, S} = 0$.
It follows that
\begin{align*}
  \widetilde{\xi}_{n, S} 
  = \sum_{i=1}^{n} \bV_{n, S}^{-1/2} (\epsilon_i + \epsilon_{i, \theta_S^*} - \epsilon_i ) X_{i, S}
  = \sum_{i=1}^{n} \bV_{n, S}^{-1/2} \epsilon_i X_{i, S}.
\end{align*}
Let $\widetilde{\omega} > 0$.
For $u \in \bbR^{|S|}$ with $\| u \|_{2} = 1$ and $t > 0$, note that
\begin{align}
\begin{aligned} \label{eqn:dev_ineq_eqn1_random_Poisson}
  &\bbP \left\{ u^{\top} \widetilde{\xi}_{n, S} > \widetilde{\omega} \: \big\rvert \: \bX \right\} 
  = \bbP \left\{ u^{\top} \bV_{n, S}^{-1/2}  \sum_{i=1}^{n}  \left[ Y_i - b'(X_i^{\top} \theta_0) \right] X_{i, S}   > \widetilde{\omega} \: \bigg\rvert \: \bX  \right\} \\
  &= \bbP \left\{ t \sum_{i=1}^{n} u^{\top} \bV_{n, S}^{-1/2} X_{i, S} Y_i > t\sum_{i=1}^{n}  u^{\top} \bV_{n, S}^{-1/2} b'(X_i^{\top} \theta_0) X_{i, S} + t\widetilde{\omega} \: \bigg \rvert \: \bX \right\}.
\end{aligned}
\end{align}
By conditional Markov inequality and \eqref{assume:mgf_finite}, the logarithm of the probability in \eqref{eqn:dev_ineq_eqn1_random_Poisson} is bounded by, on $\Omega_{n, 1}$,
\begin{align*} 
\begin{aligned} 
    & -\sum_{i=1}^{n} \left[ t u^{\top} \bV_{n, S}^{-1/2} b'(X_i^{\top} \theta_0) X_{i, S} \right] - t \widetilde{\omega}
    + \sum_{i=1}^{n} \left[ b\left( X_i^{\top} \theta_0  + t u^{\top} \bV_{n, S}^{-1/2} X_{i, S} \right) - b(X_i^{\top} \theta_0)
    \right]
    \\
    &= \sum_{i=1}^{n} \left[ 
        b\left( X_i^{\top} \theta_0  + t u^{\top} \bV_{n, S}^{-1/2} X_{i, S} \right) - b(X_i^{\top} \theta_0)
        - b'(x_i^{\top} \theta_0) t u^{\top} \bV_{n, S}^{-1/2} x_{i, S}
    \right] - t \widetilde{\omega} \\
    &= \dfrac{t^2}{2} 
    u^{\top} \bV_{n, S}^{-1/2} 
    \left[ \sum_{i=1}^{n} 
    b'' \left( X_i^{\top} \theta_0  + \eta t u^{\top} \bV_{n, S}^{-1/2} X_{i, S}\right) X_{i, S}X_{i, S}^{\top} \right] 
    \bV_{n, S}^{-1/2} u
    - t \widetilde{\omega} \\
    &= \dfrac{t^2}{2} 
    \exp \left( \eta t  \left\| \bV_{n, S}^{-1/2} X_{i, S} \right\|_{2} \right) 
    u^{\top} \bV_{n, S}^{-1/2} 
    \left[ \sum_{i=1}^{n} 
    b'' \left( X_i^{\top} \theta_0 \right) X_{i, S}X_{i, S}^{\top} \right] 
    \bV_{n, S}^{-1/2} u 
    - t \widetilde{\omega}  \\
    &\leq \dfrac{t^2}{2} \exp \left( 12t \sqrt{2} \sqrt{\dfrac{s_\ast \log (n \vee p)}{n}} \right)
    - t \widetilde{\omega},
\end{aligned}
\end{align*}
where the second equality holds for some $\eta \in (0,1)$ by Taylor's theorem. 
Suppose that $t = \omega$ for some $\omega > 0$ and
\begin{align} \label{eqn:omega_constaint}
    \exp \left( 12 \omega \sqrt{2} \sqrt{\dfrac{s_\ast \log (n \vee p)}{n}} \right)
    \leq 
    2.
\end{align}
By taking $\widetilde{\omega} = 2 \omega$, we have, for $u \in \bbR^{|S|}$ with $\| u \|_{2} = 1$ and $\omega$ satisfying \eqref{eqn:omega_constaint}, on $\Omega_{n, 1}$,
\begin{align} \label{eqn:score_vec_random_design_eq1_Poisson}
\bbP \left( u^{\top} \widetilde{\xi}_{n, S} > 2\omega \: \bigg \rvert \: \bX \right) \leq e^{-\omega^2}.
\end{align}
Let 
\begin{align*}
\omega_{p, s} = \left[ (2s + 1) \log p + s \log(6) \right]^{1/2}.
\end{align*}
Note that
\begin{align*}
    \omega_{p, s} = \left[ (2s + 1) \log p + s \log(6) \right]^{1/2} \leq  2 (s \log p)^{1/2},
\end{align*}
which, combining with the assumption, implies \eqref{eqn:omega_constaint} holds with $\omega = \omega_{p, s}$ provided that $C$ is large enough.

For $S \in \scrS_{s_\ast}$, let $\cU_{S} = \left\{ u \in \bbR^{|S|} :  \| u \|_{2} = 1 \right\}$ and $\widehat{\cU}_{S, 1/2}$ be the $1/2$-cover of $\cU_{S}$. One can choose $\widehat{\cU}_{S, \epsilon}$ so that $| \widehat{\cU}_{S, \epsilon} | \leq (6)^{|S|}$; see Proposition 1.3 of Section 15 in \cite{lorentz1996constructive}.
    For $y \in \bbR^{|S|}$, we can choose $x \in \widehat{\cU}_{S, 1/2}$ such that
    \begin{align*} 
        x^{\top} \dfrac{y}{\| y \|_2} = \left( \dfrac{y}{\| y \|_2} \right)^{\top} \dfrac{y}{\| y \|_2} + \left( x-\dfrac{y}{\| y \|_2} \right)^{\top} \dfrac{y}{\| y \|_2} \geq 1/2,
    \end{align*}    
    so we have $x^{\top}y \geq \| y \|_{2}/2$.
    It follows that, on $\Omega_{n, 1}$,
    \begin{align*} 
    \begin{aligned} 
        &\bbP \bigg( \| \widetilde{\xi}_{n, S} \|_{2} > 2\omega_{p, |S|} \: \big \rvert \: \bX \bigg) \\
        &\leq \bbP \left\{  \max_{u \in \widehat{\cU}_{S, \epsilon} }  u^{\top} \widetilde{\xi}_{n, S}  > \omega_{p, |S|}  \: \Bigg \rvert \: \bX \right\} \\
        &\leq \left| \widehat{\cU}_{S, 1/2} \right|
        \max_{u \in \widehat{\cU}_{S, 1/2}}
        \bbP \left\{ u^{\top} \widetilde{\xi}_{n, S}  > \omega_{p, |S|} \: \bigg \rvert \: \bX \right\}
        \\
        &\leq \left( 6 \right)^{|S|} e^{-\omega_{p, |S|}^2}
        = \left( 6 \right)^{|S|} 
        \exp \left[ 
            - \log p  - |S| \left\{ 2\log p + \log  \left( 6 \right) \right\} 
        \right] \\
        &= p^{- (1 +2|S|)}        
    \end{aligned}
    \end{align*}
    where the last inequality holds by \eqref{eqn:score_vec_random_design_eq1_Poisson}.  
    On $\Omega_{n, 1}$, we have
    \begin{align*}
        \bbP \left( \| \widetilde{\xi}_{n, S} \|_{2} > 2\omega_{p, |S|} \ \text{ for some } S \in \scrS_{s_\ast} \: \bigg \rvert \: \bX \right) 
        \leq \sum_{s = 1}^{\infty} \binom{p}{s} p^{-1 - 2s}
        \leq p^{-1} \sum_{s = 1}^{\infty} p^{- s} \leq p^{-1},
    \end{align*}
    where the second inequality holds because $\binom{p}{s} \leq p^{s}$.
    Therefore,
    \begin{align*}
        &\bbP \bigg(\| \widetilde{\xi}_{n, S} \|_{2} > 2\omega_{p, |S|} \ \text{ for some } S \in \widetilde{\scrS}_{s_\ast} \bigg)  \\
        &\leq 
        \bbE \left[ \bbP \bigg(\| \widetilde{\xi}_{n, S} \|_{2} > 2\omega_{p, |S|} \ \text{ for some } S \in \widetilde{\scrS}_{s_\ast} \: \bigg \rvert \: \bX \bigg) \mathds{1}_{\Omega_{n, 1}} \right]  + \bbP \bigg( \Omega_{n, 1}^{\rm c} \bigg) \\
        &\leq 5e^{-n/24} + 2(np)^{-1} + p^{-1} \leq 5e^{-n/24} + 3p^{-1},
    \end{align*}
    which conclude the proof of \eqref{claim:score_vec_random_design_Poisson}.      
\end{proof}

\subsection{Logistic regression}

\begin{lemma} \label{lemma:variance_bound_example_logit}
Suppose that $b(\cdot) = \log(1 + \exp(\cdot))$ and $\| \theta_0 \|_{2} \leq c_1$ for some constant $c_1 > 0$. Then, for any $k > 0$, there exists a constant $K > 0$, depending only on $k$ and $c_1$, such that
\begin{align*} %\label{eqn:variance_bound_example_logit}
\sigma_{\min}^{-2} \leq 4\exp\left( 2 \| \theta_0 \|_{2} \sqrt{\log n} \right) \leq K n^{k}
\end{align*}
with $\bbP$-probability at least $1- n^{-1}$. Furthermore, it holds that $\sigma_{\max}^{2} \leq 1/4$.
\end{lemma}
\begin{proof}
By Lemma \ref{lemma:natural_parameter}, we have
\begin{align*}
    \bbP \biggl\{ 
    \max_{i \in [n]} \left| X_{i}^{\top} \theta_0 \right| \geq 2 \left\| \theta_0 \right\|_{2} \sqrt{\log n} 
    \biggr\}
    \geq 1 - n^{-1}.    
\end{align*}
Note that $b''(\eta) = e^{\eta}/(1 + e^{\eta})^{2} \geq e^{-|\eta|}/4$ for all $\eta \in \bbR$.
It follows that
\begin{align*}
    \sigma_{\min}^{2} &= \min_{i \in [n]} \exp\left( X_{i}^{\top} \theta_0 \right) 
    = \min_{i \in [n]} \dfrac{\exp(X_{i}^{\top} \theta_0)}{\left[ 1 + \exp\left( X_{i}^{\top} \theta_0 \right) \right]^{2} }
    \geq \dfrac{1}{4} \exp \left( - \max_{i \in [n]} \left| X_{i}^{\top} \theta_0 \right| \right)  \\
    &\geq 
    \dfrac{1}{4} \exp \left( - 2 \left\| \theta_0 \right\|_{2} \sqrt{\log n} \right)
    \geq 
    \dfrac{1}{4} \exp \left( - 2 c_1 \sqrt{\log n} \right)
\end{align*}
with $\bbP$-probability at least $1- n^{-1}$. 
For any $k > 0$, we have
\begin{align*}
    \lim_{n \rightarrow \infty} \dfrac{e^{\sqrt{\log n}}}{n^{k}} = 0.
\end{align*}
Hence, we have, for any $k' > 0$, there exists some constant $K > 0$ such that
\begin{align*}
    \exp \left( 2c_1 \sqrt{\log n} \right) \leq K n^{k'}.
\end{align*}
Since $b''(\cdot) \leq b''(0) = 1/4$, we have
\begin{align*}
    \sigma_{\max}^{2} = \max_{i \in [n]} \exp\left( X_{i}^{\top} \theta_0 \right) 
    \leq \dfrac{1}{4}
\end{align*}
This completes the proof.
\end{proof}

\begin{lemma} \label{lemma:least_eigenvalue_logit}
Suppose that $4s_{\ast}\log p \leq n$, $p \geq 3$ and $b(\cdot) = \log(1 + \exp(\cdot))$.
Then, 
\begin{align} \label{eqn:least_eigenvalue_logit}
\begin{aligned}
    &\min_{S \in \scrS_{s_{\ast}}} \lambda_{\min} \left( \bV_{n, S} \right) \geq \dfrac{n}{216 e^{2 \| \theta_0 \|_{2}}}, \\
    &\min_{S \in \scrS_{s_{\ast}}} \lambda_{\min} \left( \bF_{n, 0_S} \right) \geq \dfrac{n}{36}, \quad
    \max_{S \in \scrS_{s_{\ast}}}  \lambda_{\max} \left( \bF_{n, 0_S} \right) \leq \dfrac{9}{4} n, \\    
    &\max_{S \in \scrS_{s_{\ast}}} \lambda_{\max} \left( \bF_{n, \thetaBest} \right) \leq \dfrac{9}{4} n, \quad
    \max_{S \in \scrS_{s_{\ast}}} \lambda_{\max} \left( \bV_{n, S} \right) \leq \dfrac{9}{4} n
\end{aligned}
\end{align}
with $\bbP$-probability at least $1 - 11e^{-n/36}$, where $0_S = (0, 0, ..., 0)^{\top} \in \bbR^{|S|}$.
\end{lemma}

\begin{proof}
For $S \in \scrS_{s_{\max}}$, we have
\begin{align*}
    \bV_{n, S} = \sum_{i = 1}^{n} \left[ b''\left( X_i^{\top} \theta_0 \right) X_{i, S} X_{i, S}^{\top} \right].
\end{align*}
Let $\cI_\omega = \left\{ i \in [n] : |X_i^{\top} \theta_0| \leq \omega \| \theta_0 \|_{2} \right\}$. Note that
\begin{align}
\begin{aligned} \label{eqn:least_eigenvalue_logit_eq1}
    &\lambda_{\min} \left( \bV_{n, S} \right) \\
    &= \lambda_{\min} \left( \sum_{i = 1}^{n} \left[ b''\left( X_i^{\top} \theta_0 \right) X_{i, S} X_{i, S}^{\top} \right] \right)
    \geq \lambda_{\min} \left( \sum_{i \in \cI_{\omega}} \left[ b''\left( X_i^{\top} \theta_0 \right) X_{i, S} X_{i, S}^{\top} \right] \right) \\
    &\geq b''\left( \omega \left\| \theta_0  \right\|_{2} \right)  \lambda_{\min} \left( \sum_{i \in \cI_{\omega}} X_{i, S} X_{i, S}^{\top} \right),    
\end{aligned}
\end{align}
where the second inequality holds by the symmetry and monotonicity of $b''(\cdot)$ in the logistic regression case. 
First, we will prove that $|\cI_{2}| \geq n/6$ with high probability.
Since $X_i^{\top} \theta_0 \sim \cN(0, \| \theta_0 \|_{2}^2)$, 
\begin{align*}
    \bbP \left( \left| X_i^{\top} \theta_0 \right| > t \left\| \theta_0 \right\|_{2} \right) \leq 2e^{-t^2/2}.
\end{align*}
By taking $t = 2$, we have
\begin{align*}
    \bbP \left( \left| X_i^{\top} \theta_0 \right| \leq 2\left\| \theta_0 \right\|_{2} \right) \geq 1 - 2e^{-2} \geq \dfrac{1}{3}.
\end{align*}
We will utilize the Chernoff-type left tail inequality (see Section 2.3 in \cite{vershynin2018high}). Let $S_n = \sum_{i=1}^{n} Z_i$,
where $Z_i \overset{\iid}{\sim} \operatorname{Bernoulli}(\eta)$. Then, 
\begin{align*}
    \bbP \biggl\{ S_n \leq (1 - \delta) \eta n \biggr\} \leq \exp \left( - \dfrac{\delta^2}{3} \eta n \right).
\end{align*}
By taking $\delta = 1/2$ and $\eta = 1/3$, 
\begin{align} \label{eqn:least_eigenvalue_logit_eq2}
    \bbP \bigg( |\cI_{2}| \leq \dfrac{n}{6} \bigg) \leq e^{-n/36}.
\end{align}
Let
\begin{align*}
\Omega_{n, 1} &= \biggl\{ \left| \cI_{2} \right| \geq \dfrac{1}{6} n  \biggr\}, \quad
\Omega_{n, 2} = \biggl\{ 
\lambda_{\min} \left( \sum_{i \in \cI_{2}}  X_{i, S} X_{i, S}^{\top} \right) \geq \dfrac{1}{9} \left| \cI_{2} \right| \text{ for all } S \in \scrS_{s_{\max}} 
\biggr\}.
\end{align*}
By the equation \eqref{eqn:least_eigenvalue_logit_eq2} and Lemma \ref{lemma:extreme_eigenvalue},
\begin{align*}
\bbP \bigl\{ \Omega_{n, 1}^{\rm c} \bigr\} \leq e^{-n/36}, \quad     
\bbP \bigl\{ \Omega_{n, 2}^{\rm c} \mid \Omega_{n, 1}  \bigr\} \leq 3e^{-n/24}.
\end{align*}
Note that
\begin{align*}
\bbP \bigl\{ \Omega_{n, 1}^{\rm c} \cup \Omega_{n, 2}^{\rm c}  \bigr\}    
\leq 
2\bbP \bigl\{ \Omega_{n, 1}^{\rm c} \bigr\} + \bbP \bigl\{ \Omega_{n, 2}^{\rm c} \mid \Omega_{n, 1} \bigr\} 
\leq 
5e^{-n/36}. 
\end{align*}
It follows that $\bbP \bigl\{ \Omega_{n} \bigr\} \geq 1 - 5e^{-n/36}$, where $\Omega_{n} = \Omega_{n, 1} \cap \Omega_{n, 2}$. 
On $\Omega_n$, therefore, we have
\begin{align*}
\min_{S \in \scrS_{s_{\max}}} \lambda_{\min} \left( \bV_{n, S} \right) 
&\geq 
b''\left(  2 \left\| \theta_0  \right\|_{2}  \right)  \dfrac{n}{54}
= \left[ \dfrac{ \exp\left( 2 \left\| \theta_0  \right\|_{2} \right) }{ 54 \left\{ 1 + \exp\left( 2 \left\| \theta_0  \right\|_{2} \right) \right\}^2 } \right] n \\
&\geq 
\dfrac{n}{216 e^{2 \| \theta_0 \|_{2}}},
\end{align*}
where the second inequality holds by $e^{x}/(1 + e^{x})^2 \geq 1/(4e^{x})$ for $x \geq 0$. 

The remaining proofs for \eqref{eqn:least_eigenvalue_logit} are simple. 
Let $0_S = (0, 0, ..., 0)^{\top} \in \bbR^{|S|}$.
Since $b''(\cdot) \leq b''(0) = 1/4$, with $\bbP$-probability at least $1 - 6e^{-n/4}$, for all $S \in \scrS_{s_{\max}}$,
\begin{align*}
    \lambda_{\max} \left( \bF_{n, \thetaBest} \right)
    = \lambda_{\max} \left( \sum_{i = 1}^{n} \left[ b''\left( X_{i, S}^{\top} \thetaBest \right) X_{i, S} X_{i, S}^{\top} \right] \right) 
    \leq \dfrac{1}{4} \lambda_{\max} \left( \sum_{i = 1}^{n} X_{i, S} X_{i, S}^{\top} \right) 
    &\leq \dfrac{9}{4} n, \\
    \lambda_{\max} \left( \bV_{n, S} \right)
    = \lambda_{\max} \left( \sum_{i = 1}^{n} \left[ b''\left( X_{i}^{\top} \theta_0 \right) X_{i, S} X_{i, S}^{\top} \right] \right) 
    \leq \dfrac{1}{4} \lambda_{\max} \left( \sum_{i = 1}^{n} X_{i, S} X_{i, S}^{\top} \right)
    &\leq \dfrac{9}{4} n, \\
    \lambda_{\max} \left( \bF_{n, 0_S} \right)
    = \lambda_{\max} \left( \sum_{i = 1}^{n} \left[ b''\left( X_{i, S}^{\top} 0_S \right) X_{i, S} X_{i, S}^{\top} \right] \right) 
    = \dfrac{1}{4} \lambda_{\max} \left( \sum_{i = 1}^{n} X_{i, S} X_{i, S}^{\top} \right)
    &\leq \dfrac{9}{4} n, \\
    \lambda_{\min} \left( \bF_{n, 0_S} \right)
    = \lambda_{\min} \left( \sum_{i = 1}^{n} \left[ b''\left( X_{i, S}^{\top} 0_S \right) X_{i, S} X_{i, S}^{\top} \right] \right) 
    = \dfrac{1}{4} \lambda_{\min} \left( \sum_{i = 1}^{n} X_{i, S} X_{i, S}^{\top} \right)
    &\geq \dfrac{1}{36} n
\end{align*}
by Lemma \ref{lemma:extreme_eigenvalue}. This completes the proof of \eqref{eqn:least_eigenvalue_logit}.
\end{proof}

\begin{lemma} \label{lemma:Fisher_smooth_logit}
    Suppose that $b(\cdot) = \log(1 + \exp(\cdot))$ and 
    \begin{align*}
        (s_{\ast} \log p)^{3/2} \vee 4(s_\ast \log p) \leq n, \quad p \geq 12.
    \end{align*}
    Then, with $\bbP$-probability at least $1 - 6p^{-s_{\ast}} - 11e^{-n/36}$, the following inequalities hold uniformly for all $S \in \scrS_{s_{\ast}}$:
    \begin{align} \label{eqn:Fisher_smooth_logit_1}
    \left\| \bF_{n, \theta_S} - \bF_{n, \thetaBest} \right\|_{2} &\leq K \left\| \theta_S - \thetaBest[S]  \right\|_{2} n, \quad \forall \theta_S \in \bbR^{|S|}.
    \end{align}
    Furthermore, if $\lambda_{\min} \big( \bF_{n, \thetaBest} \big)$ is nonsingular for all $S \in \scrS_{s_{\ast}}$, then
    \begin{align}
    \label{eqn:Fisher_smooth_logit_2}
    \left\| \bF_{n, \thetaBest}^{-1/2} \bF_{n, \theta_S} \bF_{n, \thetaBest}^{-1/2} - \bI_{|S|} \right\|_{2} &\leq \lambda_{\min}^{-1} \big( \bF_{n, \thetaBest} \big) \big( K \left\| \theta_S - \thetaBest \right\|_{2} n \big),
    \end{align}
    where $K > 0$ is a constant.
\end{lemma}

\begin{proof}
Let $\Omega_{n, 1}$ be an event on which the results of Lemmas \ref{lemma:cubic_poly_Gaussian} and \ref{lemma:least_eigenvalue_logit} hold. Then, 
\begin{align*}
    \bbP \left( \Omega_{n, 1} \right) \geq 1 - 6p^{-s_{\ast}} - 11e^{-n/36}.
\end{align*}
In the remainder of this proof, we work on the event $\Omega_{n, 1}$.

Let $S \in \scrS_{s_{\ast}}$ with $S \supseteq S_0$ and $\cU_S = \left\{ u_S \in \bbR^{|S|} : \| u_S \|_2 = 1 \right\}$.
For given $\theta_{S} \in \bbR^{|S|}$ and $u_S \in \cU_S$,
\begin{align} \label{eqn:Fisher_smooth_logit_eq1} 
    u_{S}^{\top} \left( \bF_{n ,\theta_S} - \bF_{n, \thetaBest} \right) u_{S} = \sum_{i=1}^{n} \left[ b''(X_{i, S}^{\top}\theta_S) - b''(X_{i, S}^{\top}\thetaBest[S]) \right] \left( X_{i, S}^{\top} u_{S} \right)^{2}
\end{align}
By Taylor's theorem, note that for some $t \in [0, 1]$
\begin{align*} 
    \left| b''(X_{i, S}^{\top}\theta_S) - b''(X_{i, S}^{\top}\thetaBest[S]) \right| 
    &= \left| b'''\left( x_{i, S}^{\top} \thetaBest + t X_{i, S}^{\top} \left[ \theta_S - \thetaBest \right]  \right) \right| 
    \left| X_{i, S}^{\top}\theta_S - X_{i, S}^{\top}\thetaBest[S]  \right| \\
    &\leq  \left| X_{i, S}^{\top}\theta_S - X_{i, S}^{\top}\thetaBest[S]  \right|
\end{align*}
where the inequality holds by $|b'''(\cdot)| \leq 1$ in the logistic regression case.
Let $\nu_S = (\theta_S - \thetaBest)/\left\| \theta_S - \thetaBest[S]  \right\|_{2}$. 
Hence, the right hand side of \eqref{eqn:Fisher_smooth_logit_eq1} is bounded by
\begin{align*}
&\sum_{i=1}^{n} 
\left| X_{i, S}^{\top}\theta_S - X_{i, S}^{\top}\thetaBest[S]  \right|
\left( X_{i, S}^{\top} u_{S} \right)^{2}  
\leq
\left\| \theta_S - \thetaBest  \right\|_{2}
\sum_{i=1}^{n} 
\left| X_{i, S}^{\top} \nu_S \right|
\left( X_{i, S}^{\top} u_{S} \right)^{2}  \\
&\leq 
\left\| \theta_S - \thetaBest[S]  \right\|_{2}
n 
\bigg( \dfrac{1}{n} \sum_{i=1}^{n} \left| X_{i, S}^{\top} u_S \right|^3 \bigg)^{2/3}
\bigg( \dfrac{1}{n} \sum_{i=1}^{n} \left| X_{i, S}^{\top} \nu_S \right|^3 \bigg)^{1/3} \\
&\leq 
\left\| \theta_S - \thetaBest[S]  \right\|_{2}
n 
\left[ \max_{S \in \scrS_{s_\ast}} \sup_{u_S \in \cU_{S}} \bigg( \dfrac{1}{n} \sum_{i=1}^{n} \left| X_{i, S}^{\top} u_S \right|^3 \bigg) \right] 
\leq 
\widetilde{K}_{\rm cubic} \left\| \theta_S - \thetaBest[S]  \right\|_{2} n, 
\end{align*}
where the last inequality holds by Lemma \ref{lemma:cubic_poly_Gaussian}. This completes the proof of \eqref{eqn:Fisher_smooth_logit_1}. 

Also,
\begin{align*}
\left\| \bF_{n, \thetaBest}^{-1/2} \bF_{n, \theta_S} \bF_{n, \thetaBest}^{-1/2} - \bI_{|S|} \right\|_{2}
&\leq
\left[ \lambda_{\min} \bigg( \bF_{n, \thetaBest} \bigg) \right]^{-1} \left\| \bF_{n, \theta_S} - \bF_{n, \thetaBest} \right\|_{2} \\
&\leq 
\lambda_{\min}^{-1} \bigg( \bF_{n, \thetaBest} \bigg) \times \widetilde{K}_{\rm cubic} \left\| \theta_S - \thetaBest[S]  \right\|_{2} n,
\end{align*}
which completes the proof of \eqref{eqn:Fisher_smooth_logit_2}.
\end{proof}

\begin{lemma} \label{lemma:tail_concave_likelihood}
Let $S \in \scrS_{s_\ast}$ with $S \supseteq S_0$, $u \in \bbR^{|S|}$ and $r_n > 0$.
Suppose that $b(\cdot) = \log(1 + \exp(\cdot))$. Also, assume that
\begin{align*}
    n \geq \left( C(s_{\ast}\log p)^{3/2} \right) \vee \left( 864 \widetilde{K}_{\rm cubic} e^{6\| \theta_0 \|_2} r_n^{2} \right), \quad 
    p \geq C, \quad 
    \left\| \bF_{n, \thetaBest}^{1/2} u\right\|_{2} > r_n,
\end{align*}
where $C > 0$ is large enough constant and $\widetilde{K}_{\rm cubic}$ is the constant specified in Lemma \ref{lemma:cubic_poly_Gaussian}. 
Then, 
with $\bbP$-probability at least $1 - 6p^{-s_{\ast}} - 11e^{-n/36}$, the following inequalities hold uniforly
for all $S \in \scrS_{s_\ast}$ with $S \supseteq S_0$:
\begin{align}
\begin{aligned} \label{claim:tail_concave_likelihood}
    L_{n, \thetaBest + u} - L_{n, \thetaBest} - \dot{L}_{n, \thetaBest}^{\top} u
    &\leq 
    -\dfrac{1}{4} r_n \left\| \bF_{n, \thetaBest}^{1/2} u \right\|_{2}, \\
    \mathbb{L}_{n, \thetaBest + u} - \mathbb{L}_{n, \thetaBest} 
    &\leq 
    -\dfrac{1}{4} r_n \left\| \bF_{n, \thetaBest}^{1/2} u \right\|_{2},
\end{aligned}    
\end{align}
where $\mathbb{L}_{n, \theta_S} = \bbE(L_{n, \theta_S} \mid \bX)$ for $\theta_S \in \bbR^{|S|}$.
\end{lemma}
\begin{proof}
By Lemmas \ref{lemma:least_eigenvalue_logit} and \ref{lemma:Fisher_smooth_logit}, there exists an event $\Omega_n$ such that, on $\Omega_n$, the following inequalities hold uniformly for all $S \in \scrS_{s_\ast}$ with $S \supseteq S_0$:
\begin{align*}
    \lambda_{\min} \left( \bF_{n, \thetaBest} \right) &\geq \dfrac{n}{216 e^{2 \| \theta_0 \|_2}}, \\
    \left\| \bF_{n, \thetaBest}^{-1/2} \bF_{n, \theta_S} \bF_{n, \thetaBest}^{-1/2} - \bI_{|S|} \right\|_{2}
    &\leq \left( 216 e^{2 \| \theta_0 \|_2} \right) \widetilde{K}_{\rm cubic} \left\| \theta_S - \thetaBest \right\|_{2},  \quad \forall \theta_{S} \in \bbR^{|S|},
\end{align*}
and $\bbP(\Omega_n) \geq 1 - 6p^{-s_{\ast}} - 11e^{-n/36}$. In the remainder of this proof, we work on the event $\Omega_n$.

Let $S \in \scrS_{s_\ast}$ with $S \supseteq S_0$.
By the assumption, we have $\thetaBest + u \notin \Theta_{S}(r_n)$. Let 
\begin{align*}
    \partial \Theta_{S}(r_n) = \left\{ \theta_{S} \in \bbR^{|S|} : \left\| \bF_{n, \thetaBest}^{1/2} 
    \left( \theta_S - \thetaBest \right) \right\|_{2} = r_n \right\}.
\end{align*}
Also, let
\begin{align*}
    u^{\circ} = 4r_n \left\| \bF_{n, \thetaBest}^{1/2} u \right\|_{2}^{-1} u,
\end{align*}
which implies that $\thetaBest + u^{\circ} \in \partial \Theta_{S}(r_n)$. 
It follows that
\begin{align*}
    \left\| \bF_{n, \thetaBest}^{1/2} u\right\|_{2} > r_n = \left\| \bF_{n, \thetaBest}^{1/2} u^{\circ} \right\|_{2}.
\end{align*}
For any $\theta_S \in \Theta_{S}(r_n)$, note that
\begin{align*}
    \left\| \bF_{n, \thetaBest}^{-1/2} \bF_{n, \theta_S} \bF_{n, \thetaBest}^{-1/2} - \bI_{|S|} \right\|_{2} &\leq \left( 216 \widetilde{K}_{\rm cubic} \right) e^{2 \| \theta_0 \|_{2}} \left\| \theta_S - \thetaBest \right\|_{2} \\
    &= \left( 216 \widetilde{K}_{\rm cubic} \right) e^{2 \| \theta_0 \|_{2}} \left\| \bF_{n, \thetaBest}^{-1/2} \bF_{n, \thetaBest}^{1/2} (\theta_S - \thetaBest) \right\|_{2} \\
    &\leq (\sqrt{216} \widetilde{K}_{\rm cubic}) e^{3 \| \theta_0 \|_{2}} n^{-1/2} r_n \eqqcolon \delta_n \leq 1/2,
\end{align*}
where the last inequality holds by the assumption.
By Taylor's theorem, the last display implies that
\begin{align*}   
    \left( \dot{L}_{\thetaBest + u^{\circ}} - \dot{L}_{n, \thetaBest} \right)^{\top} \left( u - u^{\circ}  \right)
    &\leq 
    \sup_{\theta_S^{\circ} \in \Theta_{S}(r_n)} \bigg[ - \left( \bF_{n, \theta_S^{\circ}} u^{\circ} \right)^{\top} \left( u - u^{\circ}  \right) \bigg] \\
    &\leq 
    - \left( 1 - \delta_n \right) \left( \bF_{n, \thetaBest} u^{\circ} \right)^{\top} \left( u - u^{\circ}  \right) \\
    &\leq 
    - \dfrac{1}{2} \left( \bF_{n, \thetaBest} u^{\circ} \right)^{\top} \left( u - u^{\circ}  \right),
\end{align*}
and 
\begin{align*}
L_{\thetaBest + u^{\circ}} - L_{\thetaBest} -\dot{L}_{\thetaBest}^{\top} u^{\circ}
&\leq 
\sup_{\theta_S^{\circ} \in \Theta_{S}(r_n)} \bigg[ - \dfrac{1}{2} \left\| \bF_{n, \theta_S^{\circ}}^{1/2} u^{\circ} \right\|_{2}^{2} \bigg] \\
&\leq 
-\dfrac{1}{2} \left( 1 - \delta_n \right)
\left\| \bF_{n, \thetaBest}^{1/2} u^{\circ} \right\|_{2}^{2}.
\end{align*}
Also, by the concavity of the map $\theta \mapsto L_{n, \theta}$, we have
\begin{align*}
    L_{\thetaBest + u} \leq 
    L_{\thetaBest + u^{\circ}} + \dot{L}_{\thetaBest + u^{\circ}}^{\top} \left( u - u^{\circ} \right).
\end{align*}
By the last three displays, we have
\begin{align*}
    &L_{\thetaBest + u} - L_{\thetaBest} -\dot{L}_{\thetaBest}^{\top} u \\
    &= 
    \bigg( 
    L_{\thetaBest + u} - L_{\thetaBest + u^{\circ}} -\dot{L}_{\thetaBest + u^{\circ}}^{\top} (u - u^{\circ})
    \bigg)
    +
    L_{\thetaBest + u^{\circ}} - L_{\thetaBest} -\dot{L}_{\thetaBest}^{\top} u^{\circ} \\
    &\qquad +
    \left( \dot{L}_{\thetaBest + u^{\circ}} - \dot{L}_{n, \thetaBest} \right)^{\top} \left( u - u^{\circ}  \right) \\
    &\leq
    L_{\thetaBest + u^{\circ}} - L_{\thetaBest} -\dot{L}_{\thetaBest}^{\top} u^{\circ}
    +
    \left( \dot{L}_{\thetaBest + u^{\circ}} - \dot{L}_{n, \thetaBest} \right)^{\top} \left( u - u^{\circ}  \right) \\
    &\leq 
    - \dfrac{1}{2}(1 - \delta_n) \left\| \bF_{n, \thetaBest}^{1/2} u^{\circ} \right\|_{2}^{2}
    - \left( 1 - \delta_n \right) \left( \bF_{n, \thetaBest} u^{\circ} \right)^{\top} \left( u - u^{\circ}  \right) \\
    &= 
    - \dfrac{1}{2}(1 - \delta_n) \left\| \bF_{n, \thetaBest}^{1/2} u^{\circ} \right\|_{2}^{2}
    + \left( 1 - \delta_n \right) 
    \bigg[ \left\| \bF_{n, \thetaBest}^{1/2} u^{\circ} \right\|_{2}^{2} - 
    \left\| \bF_{n, \thetaBest}^{1/2} u \right\|_{2}
    \left\| \bF_{n, \thetaBest}^{1/2} u^{\circ} \right\|_{2}
    \bigg] \\
    &= 
    \dfrac{1}{2}(1 - \delta_n) \left\| \bF_{n, \thetaBest}^{1/2} u^{\circ} \right\|_{2}^{2}
    - \left( 1 - \delta_n \right) 
    \left\| \bF_{n, \thetaBest}^{1/2} u \right\|_{2}
    \left\| \bF_{n, \thetaBest}^{1/2} u^{\circ} \right\|_{2} \\
    &\leq
    -\dfrac{1}{2}(1 - \delta_n) \left\| \bF_{n, \thetaBest}^{1/2} u \right\|_{2}
    \left\| \bF_{n, \thetaBest}^{1/2} u^{\circ} \right\|_{2} \\
    &\leq 
    -\dfrac{1}{4}
    \left\| \bF_{n, \thetaBest}^{1/2} u \right\|_{2}
    \left\| \bF_{n, \thetaBest}^{1/2} u^{\circ} \right\|_{2} 
    = 
    -\dfrac{1}{4} r_n
    \left\| \bF_{n, \thetaBest}^{1/2} u \right\|_{2},        
\end{align*}
which completes the proof of the first assertion in \eqref{claim:tail_concave_likelihood}. 
The proof for the second assertion in \eqref{claim:tail_concave_likelihood} follows a similar structure to that of the first assertion.
\end{proof}

%%%%%%%%%%%%%%%%%%%%%%%%%%%%%% 1

\begin{lemma} \label{lemma:eigenvalue_logit_uniform}
For $M > 0$ and $S \in \scrS_{s_{\ast}}$, let 
\begin{align*}
    \Theta_{S, M} = \left\{ \theta_{S} \in \bbR^{|S|} :  \left\| \theta_{S} \right\|_{2} \leq M \right\}.
\end{align*}
Suppose that
\begin{align*}
    n \geq C \bigg[ \left\{ s_{\ast} \log p \right\} \vee \left\{ s_{\ast} \log ( M ) \right\}  \bigg], \quad 
    p \geq C,
\end{align*}
where $C > 0$ is large enough constant.
Then, 
\begin{align} \label{claim:eigenvalue_logit_uniform}
    \dfrac{n}{1030 e^{2(M + 1)}} 
    \leq \min_{S \in \scrS_{s_\ast}} \inf_{\theta_S \in \Theta_{S, M}}  \lambda_{\min} \left( \bF_{S, \theta_S} \right) 
    \leq \max_{S \in \scrS_{s_\ast}} \sup_{\theta_S \in \bbR^{|S|}} \lambda_{\max} \left( \bF_{S, \theta_S} \right) 
    \leq \dfrac{9}{4} n
\end{align}
with $\bbP$-probability at least $1 - 9e^{-n/40} -2(np)^{-1}$.
\end{lemma}

\begin{proof}
Let $S \in \scrS_{s_\ast}$.
For $M > 0$ and $\epsilon \in (0, 1)$, let $\widehat{\Theta}_{S, M}(\epsilon)$ be the $\epsilon$-cover of $\Theta_{S, M}$. One can choose $\widehat{\Theta}_{S, M}(\epsilon)$ so that $|\widehat{\Theta}_{S, M}(\epsilon)| \leq (3M/\epsilon)^{p}$; see Proposition 1.3 of Section 15 in \cite{lorentz1996constructive}. Let $\theta_S \in \Theta_{S, M}$. By the definition of $\widehat{\Theta}_{S, M}(\epsilon)$, there exists $\widehat{\theta}_S \in \widehat{\Theta}_{S, M}(\epsilon)$ such that $\| \theta_S - \widehat{\theta}_S \|_{2} \leq \epsilon$.
For $\omega \geq 0$, let 
\begin{align*}
    \cI_{\omega}(S, \theta_S) = \cI(S, \theta_S, \omega, M) = \left\{ i \in [n] : \left| X_{i, S}^{\top} \theta_S \right| \leq \omega (M + 1) \right\}.
\end{align*}
Note that
\begin{align*}
\begin{aligned} 
    &\lambda_{\min} \left( \bF_{S, \theta_S} \right) \\
    &= 
    \lambda_{\min} \left( \sum_{i = 1}^{n} b''( X_{i, S}^{\top} \theta_S ) X_{i, S} X_{i, S}^{\top} \right)
    = 
    \lambda_{\min} \left( \sum_{i = 1}^{n} 
    \dfrac{b''( X_{i, S}^{\top} \theta_S )}{b''( X_{i, S}^{\top} \widehat{\theta}_S )}
    b'' ( X_{i, S}^{\top} \widehat{\theta}_S ) X_{i, S} X_{i, S}^{\top} \right) \\    
    &\geq 
    \Bigg[ \min_{i \in [n]} \dfrac{b''( X_{i, S}^{\top} \theta_S )}{b''( X_{i, S}^{\top} \widehat{\theta}_S )} \Bigg]
    \lambda_{\min} \left( \sum_{i \in \cI_{\omega}(S, \widehat{\theta}_S)} b''( X_{i, S}^{\top} \widehat{\theta}_S ) X_{i, S} X_{i, S}^{\top} \right) \\
    &\geq
    \exp\left( -3 \left\| \theta_S - \widehat{\theta}_S \right\|_{2} 
    \max_{i \in [n]} \max_{S \in \scrS_{s_\ast}} \left\| X_{i, S} \right\|_{2}  \right)
    \lambda_{\min} \left( \sum_{i \in \cI_{\omega}(S, \widehat{\theta}_S)} b'' ( X_{i, S}^{\top} \widehat{\theta}_S ) X_{i, S} X_{i, S}^{\top} \right) \\    
    &\geq
    \exp\left( -3 \epsilon \max_{i \in [n]} \max_{S \in \scrS_{s_\ast}} \left\| X_{i, S} \right\|_{2}  \right)
    b''\big( \omega (M + 1) \big) \lambda_{\min} \left( \sum_{i \in \cI_{\omega}(S, \widehat{\theta}_S)} X_{i, S} X_{i, S}^{\top} \right)
\end{aligned}
\end{align*}
where the second inequality holds by Lemma \ref{lemma:GLM_b_ratio}, and the last inequality follows from
the symmetry and monotonicity of $b''(\cdot)$ in the logistic regression model. 

First, for $\widehat{\theta}_S \in \widehat{\Theta}_{S, M}(\epsilon)$ and $S \in \scrS_{s_\ast}$, we will prove that $| \cI_{2}(S, \widehat{\theta}_S) | \geq n/6$ with high probability.
Since $X_i^{\top} \widehat{\theta}_S \sim \cN(0, \| \widehat{\theta}_S \|_{2}^2)$ and 
\begin{align*}
    \| \widehat{\theta}_S \|_{2} 
    \leq \left\| \theta_S \right\|_{2} + \| \theta_S - \widehat{\theta}_S \|_{2}   
    \leq M + \epsilon \leq M + 1,
\end{align*}
we have, for $i \in [n]$,
\begin{align*}
    \bbP \bigg( \left| X_{i, S}^{\top} \widehat{\theta}_S \right| > t (M + 1) \bigg)
    \leq 
    \bbP \bigg( \left| X_{i, S}^{\top} \widehat{\theta}_S \right| > t \| \widehat{\theta}_S \|_{2} \bigg) \leq 2e^{-t^2/2}, \quad \forall t \geq 0.
\end{align*}
By taking $t = 2$, we have
\begin{align*}
    \bbP \bigg( \left| X_{i, S}^{\top} \widehat{\theta}_S \right| \leq 2(M + 1) \bigg) \geq 1 - 2e^{-2} \geq \dfrac{1}{3}.
\end{align*}
We will utilize the Chernoff-type left tail inequality (see Section 2.3 in \cite{vershynin2018high}). Let $S_n = \sum_{i=1}^{n} Z_i$,
where $Z_i \overset{\iid}{\sim} \operatorname{Bernoulli}(\eta)$ for some $\eta \in (0, 1)$. Then, for any $\delta \in (0, 1)$,
\begin{align*}
    \bbP \biggl\{ S_n \leq (1 - \delta) \eta n \biggr\} \leq \exp \left( - \dfrac{\delta^2}{3} \eta n \right).
\end{align*}
By taking $\delta = 1/2$ and $\eta = 1/3$ in the above display, we have, for $\widehat{\theta}_S \in \widehat{\Theta}_{S, M}(\epsilon)$ and $S \in \scrS_{s_\ast}$,
\begin{align*} 
    \bbP \bigg( \left| \cI_{2}(S, \widehat{\theta}_S) \right| \leq \dfrac{n}{6} \bigg) \leq e^{-n/36}.
\end{align*}
By taking $\widehat{\Theta}_{S, M} = \widehat{\Theta}_{S, M}(\epsilon_0)$ with $\epsilon_0 = (4\sqrt{2} \sqrt{s_\ast \log (n \vee p)})^{-1}$, it follows that
\begin{align} 
\begin{aligned} \label{eqn:eigenvalue_logit_uniform_eq1}
    &\bbP \left( \min_{\widehat{\theta}_S \in \widehat{\Theta}_{S, M}} \min_{S \in \scrS_{s_\ast}} 
    \left| \cI_{2}(S, \widehat{\theta}_S) \right| \leq \dfrac{n}{6} \right) 
    \leq (3M/\epsilon_0)^{|S|} (3p^{s_\ast}) e^{-n/36} \\
    &\leq 
    3\exp \bigg( 
        s_{\ast} \log \big( 12\sqrt{2} M \big) + \dfrac{s_\ast}{2} \log \big( s_{\ast} \log(n \vee p) \big)
        + s_{\ast} \log p - \dfrac{n}{36} \bigg) \\
    &\leq 3e^{-n/40}. 
\end{aligned}
\end{align}
Let
\begin{align*}
    \Omega_{n, 1} &= \biggl\{ \left| \cI_{2}(S, \widehat{\theta}_S) \right| \geq \dfrac{1}{6} n \: \:
        \text{ for all } S \in \scrS_{s_\ast}  \text{ and } \widehat{\theta}_S \in \widehat{\Theta}_{S, M} \biggr\}, \\
    \Omega_{n, 2} &= \Biggl\{ 
        \lambda_{\min} \left( \sum_{i \in \cI_{2}(S, \widehat{\theta}_S) }  X_{i, S} X_{i, S}^{\top} \right) \geq \dfrac{1}{9} \left| \cI_{2}(S, \widehat{\theta}_S) \right| \: \:
        \text{ for all } S \in \scrS_{s_\ast}  \text{ and } \widehat{\theta}_S \in \widehat{\Theta}_{S, M}
        \Biggr\}, \\
    \Omega_{n, 3} &= \Biggl\{ 
        \max_{i \in [n]} \max_{S \in \scrS_{s_\ast}} \left\| X_{i, S} \right\|_{2} \leq 2\sqrt{2} \sqrt{ s_\ast \log(n \vee p)}   
    \Biggr\}.
\end{align*}
By equation \eqref{eqn:eigenvalue_logit_uniform_eq1}, Lemmas \ref{lemma:extreme_eigenvalue} and \ref{lemma:design_row_norm}, we have
\begin{align*}
\bbP \bigl\{ \Omega_{n, 1}^{\rm c} \bigr\} &\leq 3e^{-n/40}, \\
\bbP \bigl\{ \Omega_{n, 2}^{\rm c} \mid \Omega_{n, 1}  \bigr\} 
&\leq (3M/\epsilon_0)^{s_\ast}  3e^{-n/24} \leq 3e^{-n/40}, \\
\bbP \bigl\{ \Omega_{n, 3}^{\rm c} \bigr\} &\leq 2(np)^{-1}
\end{align*}
By $1 - x \geq e^{-2x}$ and $e^{-y} \geq 1 - y$ for $x \in [0, 0.797]$ and $y \in \bbR$, we have
\begin{align*}
\bbP \bigl\{ \Omega_{n} \bigr\}
\geq 1 - 6e^{-n/40} -2(np)^{-1},
\end{align*}
where $\Omega_{n} = \Omega_{n, 1} \cap \Omega_{n, 2} \cap \Omega_{n, 3}$. 
On $\Omega_n$, therefore, we have
\begin{align*}
&\min_{S \in \scrS_{s_\ast}} \min_{\theta_S \in \Theta_{S, M}} \lambda_{\min} \left( \bF_{S, \theta_S} \right)  \\
&\geq
    \exp\left( -3 \epsilon_0 \max_{i \in [n]} \min_{S \in \scrS_{s_\ast}} \left\| X_{i, S} \right\|_{2}  \right)
    b''\big( 2 (M + 1) \big)  
    \left( \dfrac{1}{9} \min_{S \in \scrS_{s_\ast}} \min_{ \widehat{v}_S \in \widehat{\Theta}_{S, M}}   
    \left| \cI_{2}\left(S, \widehat{\theta}_S \right) \right|
    \right) \\
&\geq 
e^{-3/2} \times
\dfrac{ \exp\left( 2 (M + 1) \right) }{ \big[ 1 + \exp\left( 2 (M + 1) \right) \big]^2 } \times
\dfrac{n}{54} \\
&\geq 
\dfrac{n}{1030 e^{2(M + 1)}},
\end{align*}
where the third inequality holds by $e^{-3/2} \geq 1/5$ and $e^{x}/(1 + e^{x})^2 \geq 1/(4e^{x})$ for $x \geq 0$. 

The proof of the third inequality in \eqref{claim:eigenvalue_logit_uniform} is simple. Since $b''(\cdot) \leq b''(0) = 1/4$, with $\bbP$-probability at least $1 - 3e^{-n/4}$, 
\begin{align*}
    \max_{S \in \scrS_{s_\ast}} \sup_{\theta_S \in \bbR^{|S|}}
    \lambda_{\max} \left( \bF_{S, \theta} \right)
    &= \max_{S \in \scrS_{s_\ast}} \sup_{\theta_S \in \bbR^{|S|}}
    \lambda_{\max} \left( \sum_{i = 1}^{n} \left[ b''\left( X_{i, S}^{\top} \theta_S \right) X_{i} X_{i}^{\top} \right] \right) \\
    &\leq \dfrac{1}{4} \max_{S \in \scrS_{s_\ast}} \lambda_{\max} \left( \sum_{i = 1}^{n} X_{i, S} X_{i, S}^{\top} \right)
    \leq \dfrac{9}{4} n,
\end{align*}
where the second inequality holds by Lemma \ref{lemma:extreme_eigenvalue}. This completes the proof.
\end{proof}

%%%%%%%%%%%%%%%%%%%%%%%%%%%%%%%% 2

\begin{lemma} \label{lemma:score_vec_random_design}
Let $\widetilde{\xi}_{n, S} = \bV_{n, S}^{-1/2} \dot{L}_{n, \thetaBest[S]}$.
Suppose that $b(\cdot) = \log(1 + \exp(\cdot))$ and
\begin{align*}
    n \geq C s_{\ast}\log p, \quad p \geq C, 
\end{align*}
where $C > 0$ is a large enough constant. 
Then,
\begin{align} \label{eqn:score_vec_random_design_2}
    &\bbP \bigg(  \left\| \widetilde{\xi}_{n, S} \right\|_{2} > K e^{\| \theta_0 \|_{2}} (|S| \log p)^{1/2}  \: \:
    \text{ for some } S \in \scrS_{s_\ast}  \bigg) \leq 11^{-n/36} + p^{-1},
\end{align}
where $K > 0$ is a constant.
\end{lemma}

\begin{proof}
Let $1 \leq s_\ast \leq p$. 
By Lemmas \ref{lemma:extreme_eigenvalue} and \ref{lemma:least_eigenvalue_logit}, there exists an event 
$\Omega_{n, 1}$ such that the following inequalities hold on $\Omega_{n, 1}$
\begin{align} \label{eqn:score_vec_random_design_eq0} 
    \max_{S \in \scrS_{s_{\ast}}} \lambda_{\max} \left( \sum_{i=1}^{n}  X_{i, S} X_{i, S}^{\top} \right) \leq 9 n, \quad
    \min_{S \in \scrS_{s_{\ast}}} \lambda_{\min} \left( \bV_{n, S} \right) \geq \dfrac{n}{216 e^{2 \| \theta_0 \|_{2}}},
\end{align}
and 
\begin{align*}
    \bbP \left( \Omega_{n, 1} \right) \geq 1 - 11^{-n/36}.
\end{align*}
Conditioning on $\bX$, for $S \in \scrS_{s_\ast}$, note that $\bbE \dot{L}_{n, \thetaBest} = 0$ implies $\sum_{i=1}^n (\epsilon_i - \epsilon_{i, \theta_S^*}) X_{i, S} = 0$.
It follows that
\begin{align*}
  \widetilde{\xi}_{n, S} 
  = \sum_{i=1}^{n} \bV_{n, S}^{-1/2} (\epsilon_i + \epsilon_{i, \theta_S^*} - \epsilon_i ) X_{i, S}
  = \sum_{i=1}^{n} \bV_{n, S}^{-1/2} \epsilon_i X_{i, S}.
\end{align*}
Let $\widetilde{\omega} = 2\sqrt{ 243 e^{2 \| \theta_0 \|_{2}} \omega^2} = 18\sqrt{3} e^{\| \theta_0 \|_{2}} \omega$.
For $u \in \bbR^{|S|}$ with $\| u \|_{2} = 1$ and $t > 0$, note that
\begin{align}
\begin{aligned} \label{eqn:dev_ineq_eqn1_random}
  &\bbP \left\{ u^{\top} \widetilde{\xi}_{n, S} > \widetilde{\omega} \: \big\rvert \: \bX \right\} 
  = \bbP \left\{ u^{\top} \bV_{n, S}^{-1/2}  \sum_{i=1}^{n}  \left[ Y_i - b'(X_i^{\top} \theta_0) \right] X_{i, S}   > \widetilde{\omega} \: \bigg\rvert \: \bX  \right\} \\
  &= \bbP \left\{ t \sum_{i=1}^{n} u^{\top} \bV_{n, S}^{-1/2} X_{i, S} Y_i > t\sum_{i=1}^{n}  u^{\top} \bV_{n, S}^{-1/2} b'(X_i^{\top} \theta_0) X_{i, S} + t\widetilde{\omega} \: \bigg \rvert \: \bX \right\}.
\end{aligned}
\end{align}
By conditional Markov inequality and \eqref{assume:mgf_finite}, the logarithm of the probability in \eqref{eqn:dev_ineq_eqn1_random} is bounded by, on $\Omega_{n, 1}$,
\begin{align*} 
\begin{aligned} 
    & -\sum_{i=1}^{n} \left[ t u^{\top} \bV_{n, S}^{-1/2} b'(X_i^{\top} \theta_0) X_{i, S} \right] - t \widetilde{\omega}
    + \sum_{i=1}^{n} \left[ b\left( X_i^{\top} \theta_0  + t u^{\top} \bV_{n, S}^{-1/2} X_{i, S} \right) - b(X_i^{\top} \theta_0)
    \right]
    \\
    &= \sum_{i=1}^{n} \left[ 
        b\left( X_i^{\top} \theta_0  + t u^{\top} \bV_{n, S}^{-1/2} X_{i, S} \right) - b(X_i^{\top} \theta_0)
        - b'(x_i^{\top} \theta_0) t u^{\top} \bV_{n, S}^{-1/2} x_{i, S}
    \right] - t \widetilde{\omega} \\
    &= \dfrac{t^2}{2} 
    u^{\top} \bV_{n, S}^{-1/2} 
    \left[ \sum_{i=1}^{n} 
    b'' \left( X_i^{\top} \theta_0  + \eta t u^{\top} \bV_{n, S}^{-1/2} X_{i, S}\right) X_{i, S}X_{i, S}^{\top} \right] 
    \bV_{n, S}^{-1/2} u
    - t \widetilde{\omega} \\
    &\leq \dfrac{t^2}{8} 
    u^{\top} \bV_{n, S}^{-1/2} 
    \left[ \sum_{i=1}^{n}  X_{i, S}X_{i, S}^{\top} \right] 
    \bV_{n, S}^{-1/2} u 
    - t \widetilde{\omega} \quad ( \because b''(\cdot) \leq 1/4 ) \\
    &\leq \dfrac{t^2}{8} \left( \dfrac{216 e^{2 \| \theta_0 \|_{2}}}{n} \right)
    \left( 9n \right)
    - t \widetilde{\omega}  \quad \left( \because \eqref{eqn:score_vec_random_design_eq0} \right)  \\
    &= 243 e^{2 \| \theta_0 \|_{2}} t^2 - t \widetilde{\omega}   
\end{aligned}
\end{align*}
where the second equality holds for some $\eta \in (0,1)$ by Taylor's theorem. By taking $t = \omega / \sqrt{ 243 e^{2 \| \theta_0 \|_{2}} }$, therefore, the right hand side of the last display is equal to 
\begin{align*}
    243 e^{2 \| \theta_0 \|_{2}} \dfrac{ \omega^2 }{243 e^{2 \| \theta_0 \|_{2}}} - \dfrac{ \omega }{ \sqrt{243 e^{2 \| \theta_0 \|_{2}}} } 2\sqrt{ 243 e^{2 \| \theta_0 \|_{2}} \omega^2}
     = -\omega^2.
\end{align*}
Therefore, for $u \in \bbR^{|S|}$ with $\| u \|_{2} = 1$, on $\Omega_{n, 1}$,
\begin{align} \label{eqn:score_vec_random_design_eq1}
\bbP \left( u^{\top} \widetilde{\xi}_{n, S} > 18\sqrt{3} e^{\| \theta_0 \|_2 } \omega \: \bigg \rvert \: \bX \right) \leq e^{-\omega^2}.
\end{align}
Let 
\begin{align*}
\omega_{\epsilon, p, s} = \left[ (2s + 1) \log p + s \log(3/\epsilon) \right]^{1/2}, \quad 
 z_{\epsilon, p, S} = 18\sqrt{3} e^{\| \theta_0 \|_{2}} (1-\epsilon)^{-1} \omega_{\epsilon, p, |S|}.
\end{align*}
For $S \in \scrS_{s_\ast}$ and $\epsilon \in (0, 1)$, let $\cU_{S} = \left\{ u \in \bbR^{|S|} :  \| u \|_{2} = 1 \right\}$ and $\widehat{\cU}_{S, \epsilon}$ be the $\epsilon$-cover of $\cU_{S}$. One can choose $\widehat{\cU}_{S, \epsilon}$ so that $| \widehat{\cU}_{S, \epsilon} | \leq (3/\epsilon)^{|S|}$; see Proposition 1.3 of Section 15 in \cite{lorentz1996constructive}.
    For $y \in \bbR^{|S|}$, we can choose $x \in \widehat{\cU}_{S, \epsilon}$ such that
    \begin{align*} 
        x^{\top} \dfrac{y}{\| y \|_2} = \left( \dfrac{y}{\| y \|_2} \right)^{\top} \dfrac{y}{\| y \|_2} + \left( x-\dfrac{y}{\| y \|_2} \right)^{\top} \dfrac{y}{\| y \|_2} \geq 1 - \epsilon,
    \end{align*}    
    so we have $x^{\top}y \geq (1-\epsilon)\| y \|_{2}$.
    It follows that, on $\Omega_{n, 1}$,
    \begin{align*} 
    \begin{aligned} 
        &\bbP \bigg( \| \widetilde{\xi}_{n, S} \|_{2} > z_{\epsilon, p, S} \: \big \rvert \: \bX \bigg) \\
        &\leq \bbP \left\{  \max_{u \in \widehat{\cU}_{S, \epsilon} }  u^{\top} \widetilde{\xi}_{n, S}  > (1 - \epsilon) z_{\epsilon, p, S}  \: \Bigg \rvert \: \bX \right\} \\
        &\leq \left| \widehat{\cU}_{S, \epsilon} \right|
        \max_{u \in \widehat{\cU}_{S, \epsilon}}
        \bbP \left\{ u^{\top} \widetilde{\xi}_{n, S}  > (1 - \epsilon) z_{\epsilon, p, S} \: \bigg \rvert \: \bX \right\}
        \\
        &\leq \left( \dfrac{3}{\epsilon}\right)^{|S|} e^{-\omega_{\epsilon, p, |S|}^2}
        = \left( \dfrac{3}{\epsilon}\right)^{|S|} 
        \exp \left[ 
            - \log p  - |S| \left\{ 2\log p + \log  \left(\frac{3}{\epsilon} \right) \right\} 
        \right] \\
        &= p^{- (1 +2|S|)}        
    \end{aligned}
    \end{align*}
    where the last inequality holds by \eqref{eqn:score_vec_random_design_eq1}.  
    On $\Omega_{n, 1}$, we have
    \begin{align*}
        \bbP \left( \| \widetilde{\xi}_{n, S} \|_{2} > z_{\epsilon, p, S} \ \text{ for some } S \in \scrS_{s_\ast} \: \bigg \rvert \: \bX \right) 
        \leq \sum_{s = 1}^{\infty} \binom{p}{s} p^{-1 - 2s}
        \leq p^{-1} \sum_{s = 1}^{\infty} p^{- s} \leq p^{-1},
    \end{align*}
    where the second inequality holds because $\binom{p}{s} \leq p^{s}$.
    Therefore,
    \begin{align*}
        &\bbP \bigg(\| \widetilde{\xi}_{n, S} \|_{2} > z_{\epsilon, p, S} \ \text{ for some } S \in \widetilde{\scrS}_{s_\ast} \bigg)  \\
        &\leq 
        \bbE \left[ \bbP \bigg(\| \widetilde{\xi}_{n, S} \|_{2} > z_{\epsilon, p, S} \ \text{ for some } S \in \widetilde{\scrS}_{s_\ast} \: \bigg \rvert \: \bX \bigg) \mathds{1}_{\Omega_{n, 1}} \right]  + \bbP \bigg( \Omega_{n, 1}^{\rm c} \bigg) \\
        &\leq 11n^{-n/36} + p^{-1},
    \end{align*}
    By taking $\epsilon = 1/2$, we conclude the proof of \eqref{eqn:score_vec_random_design_2}.      
\end{proof}

%%%%%%%%%%%%%%%%%%%%%%%%%%%%%%%%%%%%%%%%%%%%%%%%%% Design regularity for Poisson regression

\section{Design regularity for Poisson regression} \label{sec:design_regularity_app}

In this section, we provide an example satisfying the design regularity condition $\designRegular = O(n^{-1/2})$ for the Poisson regression model. Throughout this section, we assume that $\mathbf{X} \in \mathbb{R}^{n\times p}$ is a random matrix with independent rows, where the $i$th row $X_i$ follows a $\mathcal{N}\left(0, \bI_p\right)$ distribution. Let $\bbP$ be the corresponding probability measure and $\scrS_{s} = \left\{ S \subset [p] : 0 < |S| \leq s \right\}$.

\begin{lemma} \label{lemma:lowerbound_prob}
For $\beta > 1$, $\omega \in (0, 1/2)$ and $\theta_0 \in \bbR^p$, suppose that 
\begin{align} \label{assume:lowerbound_prob}
    \dfrac{\sqrt{2}}{1 - 2\omega} \log \beta \leq \left\| \theta_0 \right\|_{2}.
\end{align}
Then,
\begin{align} \label{eqn:lowerbound_prob_claim1}
    \bbP \left\{ 
    \exp\left( X_i^{\top} \theta_0 \right) \geq \beta
    \right\}
    \geq \omega,
\end{align}
and 
\begin{align} \label{eqn:lowerbound_prob_claim2}
    \bbP \Biggl(
    \bigg| \left\{ i \in [n] : \exp\left( X_i^{\top} \theta_0 \right) \geq \beta  \right\} \bigg|
    \geq \dfrac{\omega}{2} n
    \Biggr)
    \geq 1 - e^{-\omega n/12}.
\end{align}
\end{lemma}

\begin{proof}
Note that $X_i^{\top} \theta_0 \overset{\iid}{\sim} \cN (0, K_n^2)$ for all $i \in [n]$, where $K_n = \| \theta_0 \|_{2}$. By the definition of log-normal distribution, note that
\begin{align*}
    \exp\left( X_i^{\top} \theta_0 \right) \overset{\iid}{\sim} \operatorname{logNormal} \left( 0, K_n \right),
\end{align*}
where $\operatorname{logNormal} \left( \mu, \sigma \right)$ denotes the log-normal distribution which has probability density function  $f(x)$ and cumulative distribution function $\Phi(x)$ defined as
\begin{align*}
    f(x) =  \dfrac{1}{x \sigma \sqrt{2\pi}} \exp \left( - \dfrac{\left(\log x - \mu \right)^2}{2 \sigma^2} \right), \quad
    \Phi(x) = \dfrac{1}{2} \left\{1 + \operatorname{erf} \left( \dfrac{\log x - \mu}{\sigma \sqrt{2}}  \right)\right\}
\end{align*}
for $x \in \bbR_{+}$. Here, for $z \in \bbR$, the error function $\operatorname{erf}(\cdot)$ is defined by
\begin{align*}
    \operatorname{erf}(z) = \dfrac{2}{\sqrt{\pi}} \int_{0}^{z} \exp\left(-t^2\right) \rmd t.
\end{align*}
It follows that
\begin{align*}
\bbP \left\{ \exp\left( X_i^{\top} \theta_0 \right) \geq \beta \right\}        
&= 1 - \Phi \left( \beta \right)
= 1 - \dfrac{1}{2} \left\{1 + \operatorname{erf} \left( \dfrac{\log \beta}{K_n \sqrt{2}}  \right)\right\} \\
&= 1 - \dfrac{1}{2} \left\{1 - \operatorname{erf} \left( -\dfrac{\log \beta}{K_n \sqrt{2}}  \right)\right\} 
\quad \big(\because \operatorname{erf}(\cdot) \text{ is odd function} \big)
\\
&= 1- \dfrac{1}{2} \operatorname{erfc} \left( -\dfrac{\log \beta}{K_n \sqrt{2}} \right),
\end{align*}
where $\operatorname{erfc}(z) = 1 - \operatorname{erf}(z)$ denotes the complementary error function. From the last display, it suffices to show that
\begin{align*}
\operatorname{erfc} \left( -\dfrac{\log \beta}{K_n \sqrt{2}} \right)
\leq 2(1 - \omega).
\end{align*}
By the fact that $\operatorname{erfc}(x) \leq 1 - 2x$ for $x \leq 0$ and \eqref{assume:lowerbound_prob}, we have
\begin{align*}
\operatorname{erfc} \left( -\dfrac{\log \beta}{K_n \sqrt{2}} \right)
\leq 1 +\sqrt{2}\dfrac{\log \beta}{K_n} 
\leq 1 +\sqrt{2} \left( \dfrac{1- 2\omega}{\sqrt{2}} \right)
= 2 -2\omega,
\end{align*}
which completes the proof of \eqref{eqn:lowerbound_prob_claim1}.

To prove \eqref{eqn:lowerbound_prob_claim2}, we will utilize the Chernoff-type left tail inequality (see Section 2.3 in \cite{vershynin2018high}). Let $S_n = \sum_{i=1}^{n} Z_i$,
where $Z_i \overset{\iid}{\sim} \operatorname{Bernoulli}(\omega)$. Then, 
\begin{align*}
    \bbP \biggl\{ S_n \leq (1 - \delta) \omega n \biggr\} \leq \exp \left( - \dfrac{\delta^2}{3} \omega n \right).
\end{align*}
By taking $\delta = 1/2$ in the last display, we complete the proof of \eqref{eqn:lowerbound_prob_claim2}.
\end{proof}

\begin{theorem}[Design regularity] \label{thm:design_regualr_example}
Suppose that 
\begin{align*}
    4s_{\ast} \log p \leq n, \quad p \geq 3, \quad 2\sqrt{2} \log \left[ 4s_{\ast}\log (np) \right] \leq \left\| \theta_0 \right\|_{2}.
\end{align*}
Then, 
\begin{align} \label{eqn:design_regular_claim}
\bbP \biggl\{
    \max_{S \in \scrS_{s_\ast} : S \supseteq S_0} \designRegular \leq 6\sqrt{2} n^{-1/2}
\biggr\} 
\geq
1 - 5e^{-n/48} - 2(np)^{-1}.
\end{align}
\end{theorem}
\begin{proof}
Let
\begin{align*}
\Omega_{n, 1} &= \biggl\{ \left| \cI \right| \geq \dfrac{1}{8} n \biggr\}, \quad
\Omega_{n, 2} = \biggl\{ 
\lambda_{\min} \left( \sum_{i \in \cI}  X_{i, S} X_{i, S}^{\top} \right) \geq \dfrac{1}{9} \left| \cI \right| \text{ for all } S \in \scrS_{s_{\ast}} 
\biggr\}, \\
\Omega_{n, 3} &= \biggl\{ 
    \max_{i \in [n], S \in \scrS_{s_{\ast}}} \left\| X_{i, S} \right\|_{2}^{2}
        \leq 
        4 s_{\ast} \log (np)
\biggr\},
\end{align*}
where $\cI =  \left\{ i \in [n] : \exp\left( X_i^{\top} \theta_0 \right) \geq 4s_{\ast}\log (np) \right\}$. By Lemmas \ref{lemma:lowerbound_prob}, \ref{lemma:extreme_eigenvalue} and \ref{lemma:design_row_norm}, we have
\begin{align*}
\bbP \bigl\{ \Omega_{n, 1}^{\rm c} \bigr\} \leq e^{-n/48}, \quad     
\bbP \bigl\{ \Omega_{n, 2}^{\rm c} \mid \Omega_{n, 1}  \bigr\} \leq 3e^{-n/32}, \quad 
\bbP \bigl\{ \Omega_{n, 3}^{\rm c} \bigr\} \leq 2(np)^{-1}.
\end{align*}
Note that
\begin{align*}
\bbP \bigl\{ \Omega_{n, 1}^{\rm c} \cup \Omega_{n, 2}^{\rm c}  \bigr\}    
&\leq 
\bbP \bigl\{ \Omega_{n, 1}^{\rm c} \bigr\} + \bbP \bigl\{ \Omega_{n, 2}^{\rm c}  \bigr\} \\   
&=
\bbP \bigl\{ \Omega_{n, 1}^{\rm c} \bigr\} + \bbP \bigl\{ \Omega_{n, 2}^{\rm c} \cap \Omega_{n, 1} \bigr\} + \bbP \bigl\{ \Omega_{n, 2}^{\rm c} \cap \Omega_{n, 1}^{\rm c} \bigr\} \\
&\leq 
\bbP \bigl\{ \Omega_{n, 1}^{\rm c} \bigr\} + \bbP \bigl\{ \Omega_{n, 2}^{\rm c} \mid \Omega_{n, 1} \bigr\} + \bbP \bigl\{ \Omega_{n, 1}^{\rm c} \bigr\} \\      
&= 2\bbP \bigl\{ \Omega_{n, 1}^{\rm c} \bigr\} + \bbP \bigl\{ \Omega_{n, 2}^{\rm c} \mid \Omega_{n, 1} \bigr\} \leq 5e^{-n/48}. 
\end{align*}
It follows that 
\begin{align*}
\bbP \bigl\{ \Omega_{n} \bigr\} \geq 1 - 5e^{-n/48} - 2(np)^{-1},
\end{align*}
where $\Omega_{n} = \Omega_{n, 1} \cap \Omega_{n, 2} \cap \Omega_{n, 3}$. 
In the remainder of this proof, we work on the event $\Omega_n$.

Note that
\begin{align*}
    \lambda_{\min} \big( \bV_{n, S} \big)
    &= \lambda_{\min} \Bigg( \sum_{i=1}^{n} \exp\big( X_i^{\top} \theta_0 \big) X_{i, S} X_{i, S}^{\top}  \Bigg)
    \geq \lambda_{\min} \Bigg( \sum_{i \in \cI} \exp\big( X_i^{\top} \theta_0 \big) X_{i, S} X_{i, S}^{\top} \Bigg) \\
    &\geq 4s_{\ast}\log (np)  \lambda_{\min} \Bigg( \sum_{i \in \cI} X_{i, S} X_{i, S}^{\top} \Bigg)
    \geq \frac{n}{72} \times  4s_{\ast}\log (np)
\end{align*}
for any $S \in \scrS_{s_\ast}$. Hence, for any $S \in \scrS_{s_\ast}$ with $S \supseteq S_0$,
\begin{align*}
    \lambda_{\min}^{-1} \big( \bF_{n, \thetaBest} \big)
    = \lambda_{\min}^{-1} \big( \bV_{n, S} \big) 
    \leq 72 \left[ n \times 4s_{\ast}\log (np) \right]^{-1},
\end{align*}
where $\Delta_{{\rm mis}, S}$ is defined in Lemma \ref{lemma:dev_ineq_score_func}. 
It follows that
\begin{align*}
    \lambda_{\min} \big( \bF_{n, \thetaBest} \big) 
    \geq 
    \dfrac{1}{72} n \left[ 4s_{\ast}\log (np) \right].
\end{align*}
By the definition of $\designRegular$, we have
\begin{align*}
    \max_{S \in \scrS_{s_\ast} : S \supseteq S_0} \designRegular
     &= \max_{S \in \scrS_{s_\ast} S \supseteq S_0} \max_{i \in [n]} \left\| \bF_{n, \thetaBest}^{-1/2} X_{i, S} \right\|_2
     \leq \max_{S \in \scrS_{s_\ast} S \supseteq S_0} \max_{i \in [n]} \left\| \bF_{n, \thetaBest}^{-1/2} \right\|_2 \left\| X_{i, S} \right\|_2 \\
     &\leq \bigg( \dfrac{1}{72} n \left[ 4s_{\ast}\log (np) \right] \bigg)^{-1/2} \big( 4s_{\ast}\log (np) \big)^{1/2} \\
     &= 6\sqrt{2} n^{-1/2},
\end{align*}
which completes the proof of \eqref{eqn:design_regular_claim}.
\end{proof}

\section{General sub-exponential tail case} \label{sec:general_sub-exponential_tail_app}

Recall the definition of $\epsilon_{i} = Y_i - \bbE Y_i$ and $\sigma_i = \Var(Y_i)$. Since our main focus is on the sub-exponential random behavior of $\epsilon_{i}$ (e.g., Poisson regression), suppose that
\begin{align} \label{assume:sub-exp_mgf}
	\log \bbE \exp\left(t \sigma_{i}^{-1} \epsilon_{i} \right) \leq \dfrac{1}{2} \nu_0^{2} t^{2}, \quad \forall i \in [n], |t| \leq t_0,
\end{align}
for some fixed constants $\nu_0, t_0 > 0$. This condition is equivalent to the definition of the sub-exponential random variable since $\bbE \epsilon_{i} = 0$ for all $i \in [n]$ (Section 2.7 in \cite{vershynin2018high}). 

The lemma presented below is a modification of Lemma 3.9 in \cite{spokoiny2017penalized} and serves as a more general version of Lemma \ref{lemma:dev_ineq_score_func}. In particular, Lemma \ref{lemma:dev_ineq_score_func} leverages the closed-form solution of the moment-generating function for the exponential family. This eliminates the necessity to bound the maximal variance, represented as $\sigma_{\max} = \max_{i \in [n]} \sigma_{i}$. It should be noted that, except for Lemma \ref{lemma:dev_ineq_score_func}, all other lemmas in Section \ref{sec:param_estimation_app} remain valid as long as Lemma \ref{lemma:mgf_score_function} holds.

\begin{lemma}[Exponential moment of normalized score function] \label{lemma:mgf_score_function}
  Suppose that \eqref{assume:sub-exp_mgf} holds for some constants $t_0$ and $\nu_0$. For $S \subset [p]$, assume that $\bF_{n, \theta^*_S}$ is nonsingular and
  \begin{align} \label{assume:identification}
      \lambda_{\max}( \bF_{n, \thetaBest}^{-1/2} \bV_{n, S} \bF_{n, \thetaBest}^{-1/2}) \leq C_{\operatorname{mis}}
  \end{align}
  for some constant $C_{\operatorname{mis}} > 0$.
  Then, for $S \subset [p]$ and $\|u\|_2 \leq t_{n, S}$,
	\begin{align} \label{def:mgf_tn}
		\log \bbE \exp \left\{ u^{\top} \xi_{n, S} \right\} \leq \dfrac{\widetilde{\nu}^2 }{2} \| u  \|_{2}^{2}.		
	\end{align}
 where $\widetilde{\nu}^2 = \nu_0^2 C_{\operatorname{mis}}$ and $t_{n, S} = t_0 ( \designRegular \sigma_{\max})^{-1}$.
\end{lemma}
\begin{proof}
Note that
\begin{align*}
  \xi_{n, S} 
  &= \bF_{n, \thetaBest}^{-1/2} \nabla L_{n, \thetaBest} 
  = \sum_{i=1}^n \bF_{n, \thetaBest}^{-1/2} \score_{i, \theta^*_S}
  = \sum_{i=1}^{n} \bF_{n, \thetaBest}^{-1/2} \epsilon_{i, \thetaBest[S]} x_{i, S}
  \\
  &= \sum_{i=1}^{n} \bF_{n, \thetaBest}^{-1/2} \big\{\epsilon_i + b'(x_{i, S_0}^\top \theta_{S_0}^*) - b'(x_{i, S}^\top \theta_S^*) \big\} x_{i, S}
  = \sum_{i=1}^{n} \bF_{n, \thetaBest}^{-1/2} \epsilon_i x_{i, S},
\end{align*}
where the last equality holds because $\bbE \nabla L_{n, \thetaBest} = 0$.
For given $u \in \bbR^{|S|}$ with $\| u \|_{2} \leq t_{n, S}$,
\begin{align*}
  \log \bbE \exp \left\{ u^{\top} \xi_{n, S} \right\} 
  &= \log \bbE \exp \left\{
  u^{\top} \bF_{n, \thetaBest}^{-1/2} \sum_{i=1}^{n} \epsilon_{i} x_{i, S} \right\}
  = \sum_{i=1}^{n} \log \bbE \exp \left\{ \eta_i \sigma_{i}^{-1} \epsilon_{i} \right\},
\end{align*}	
where $\eta_i = \sigma_{i} u^{\top} \bF_{n, \thetaBest}^{-1/2} x_{i, S}$. 
Since $\| u \|_{2} \leq t_{n, S}$, we have
\begin{align*}
  |\eta_i| 
  &= \sigma_{i} \left|u^{\top} \bF_{n, \thetaBest}^{-1/2} x_{i, S}\right|
  \leq  \sigma_{i} t_{n, S} \left\|\bF_{n, \thetaBest}^{-1/2} x_{i, S}\right\|_{2}
  \leq t_{n, S} \designRegular \sigma_{\operatorname{max}}  = t_0.
\end{align*}
Hence,
\begin{align*}
  \sum_{i=1}^{n} \log \bbE \exp \left\{ \eta_i \sigma_{i}^{-1} \epsilon_{i} \right\} 
  &\leq \dfrac{\nu_0^2}{2} \sum_{i=1}^{n} |\eta_i|^2
  = \dfrac{\nu_0^2}{2} 
  u^{\top} \bF_{n, \thetaBest}^{-1/2} 
  \sum_{i=1}^{n} 
  \left[ \sigma_{i}^2 x_{i, S}x_{i, S}^{\top}\right] \bF_{n, \thetaBest}^{-1/2} u \\
  &\leq 
  \dfrac{\nu_0^2 C_{\operatorname{mis}}}{2} \| u \|_{2}^2,
\end{align*}
where the first and last inequalities hold by \eqref{assume:sub-exp_mgf} and \eqref{assume:identification}, respectively.
\end{proof}

\end{appendix}
\end{document}

%% file: sty/paper-preamble2.tex
\usepackage{amsmath}
\usepackage{amsfonts}
\usepackage{amssymb}
\usepackage{amsthm}
\usepackage{mathrsfs}
\usepackage{graphicx}
\usepackage{color}
\usepackage[authoryear]{natbib}
\usepackage[margin=1in]{geometry}% Just for this example
\usepackage[titletoc]{appendix}
\usepackage{bm}
\usepackage{algorithm}
\usepackage{algpseudocode}
\usepackage{algpascal}
\usepackage{url}
\usepackage{afterpage}
\usepackage{enumitem} 
\usepackage{multirow}
\usepackage{booktabs}
\usepackage{subfigure}

\numberwithin{equation}{section}

\newtheorem{theorem}{Theorem}[section]
\newtheorem{lemma}[theorem]{Lemma}

\newtheorem{corollary}[theorem]{Corollary}
\newtheorem*{remark}{\bf Remark}

\newcommand{\scrC}{{\mathscr C}}
\newcommand{\scrD}{{\mathscr D}}

\newcommand{\scrK}{{\mathscr K}}

\newcommand{\scrS}{{\mathscr S}}

\def\cA{\mathcal A}

\def\cE{\mathcal E}

\def\cI{\mathcal I}

\def\cM{\mathcal M}
\def\cN{\mathcal N}

\def\cR{\mathcal R}

\def\cU{\mathcal U}
\def\cV{\mathcal V}

\newcommand{\bA}{{\bf A}}

\newcommand{\bF}{{\bf F}}

\newcommand{\bH}{{\bf H}}
\newcommand{\bI}{{\bf I}}

\newcommand{\bV}{{\bf V}}

\newcommand{\bW}{{\bf W}}
\newcommand{\bx}{{\bf x}}
\newcommand{\bX}{{\bf X}}

\newcommand{\bY}{{\bf Y}}

\newcommand{\bbE}{{\mathbb E}}
\newcommand{\bbG}{{\mathbb G}}

\newcommand{\bbN}{{\mathbb N}}
\newcommand{\bbP}{{\mathbb P}}
\newcommand{\bbR}{{\mathbb R}}

\newcommand{\E}{\mathbb{E}}

\newcommand{\Var}{\mathbb{V}}

\newcommand{\bSigma}{\bm{\Sigma}}

\newcommand{\iid}{{i.i.d.}}

\newcommand{\bc}{\begin{center}}
\newcommand{\ec}{\end{center}}
\newcommand{\be}{\begin{equation}}
\newcommand{\ee}{\end{equation}}
\newcommand{\ba}{\begin{array}}
\newcommand{\ea}{\end{array}}
\newcommand{\bean}{\setlength\arraycolsep{1pt}\begin{eqnarray*}}
\newcommand{\eean}{\end{eqnarray*}}
\newcommand{\bea}{\setlength\arraycolsep{1pt}\begin{eqnarray}}
\newcommand{\eea}{\end{eqnarray}}
\newcommand{\ben}{\begin{enumerate}}
\newcommand{\een}{\end{enumerate}}
\newcommand{\bed}{\begin{itemize}}
\newcommand{\eed}{\end{itemize}}

\DeclareMathOperator*{\argmax}{argmax}

\def\d{\mbox{d}}

\def\defeq{ \stackrel{\rm def}{=} }

\newcommand{\score}{\dot{\ell}}